\documentclass[final,3p,12pt]{elsarticle}
\usepackage{amsmath,amsthm,amscd,amssymb,latexsym,upref,stmaryrd}
\usepackage{mathtools}

\numberwithin{equation}{section}
\usepackage{hyperref}

\newtheorem{theorem}{Theorem}[section]
\newtheorem{proposition}[theorem]{Proposition}
\newtheorem{corollary}[theorem]{Corollary}
\newtheorem{lemma}[theorem]{Lemma}
\newtheorem{remark}[theorem]{Remark}
\newtheorem{definition}[theorem]{Definition}
\newtheorem{example}[theorem]{Example}

\let\Im\relax
\DeclareMathOperator{\Im}{Im}
\let\Re\relax
\DeclareMathOperator{\Re}{Re}

\DeclareMathOperator{\AC}{AC}

\DeclareMathOperator{\col}{col}
\DeclareMathOperator{\const}{const}

\DeclareMathOperator{\card}{card}
\DeclareMathOperator{\sign}{sign}
\DeclareMathOperator{\diag}{diag}

\DeclareMathOperator{\dist}{dist}
\DeclareMathOperator{\dom}{dom}
\DeclareMathOperator*{\esssup}{ess\,sup}
\DeclareMathOperator{\GL}{GL}
\DeclareMathOperator{\inter}{int}
\DeclareMathOperator{\Ker}{Ker}
\DeclareMathOperator{\Lip}{Lip}
\DeclareMathOperator{\loc}{loc}

\DeclareMathOperator{\per}{per}

\DeclareMathOperator{\rank}{rank}
\DeclareMathOperator{\separ}{sep}
\DeclareMathOperator{\Span}{span}

\DeclareMathOperator{\Tim}{Tim}
\DeclareMathOperator{\tr}{tr}

\renewcommand{\le}{\leqslant}
\renewcommand{\ge}{\geqslant}
\newcommand{\ol}{\overline}
\newcommand{\wt}{\widetilde}
\newcommand{\wh}{\widehat}
\renewcommand{\(}{\left(}
\renewcommand{\)}{\right)}

\newcommand{\eps}{\varepsilon}
\newcommand{\alp}{\alpha}
\newcommand{\gam}{\gamma}
\newcommand{\Gam}{\Gamma}
\renewcommand{\phi}{\varphi}
\renewcommand{\l}{\lambda}
\renewcommand{\L}{\Lambda}

\def\cA{\mathcal{A}}
\def\cB{\mathcal{B}}
\def\cC{\mathcal{C}}
\def\cD{\mathcal{D}}
\def\cE{\mathcal{E}}
\def\cF{\mathcal{F}}

\def\cH{\mathcal{H}}

\def\cK{\mathcal{K}}
\def\cL{\mathcal{L}}

\def\cP{\mathcal{P}}
\def\cQ{\mathcal{Q}}
\def\cR{\mathcal{R}}
\def\cS{\mathcal{S}}
\def\cU{\mathcal{U}}
\def\cW{\mathcal{W}}

\def\fa{\mathfrak{a}}
\def\fB{\mathfrak{B}}

\def\fE{\mathfrak{E}}
\def\fF{\mathfrak{F}}

\def\fH{\mathfrak{H}}

\def\fM{\mathfrak{M}}

\def\fP{\mathfrak{P}}
\def\fQ{\mathfrak{Q}}
\def\fp{\mathfrak{p}}
\def\fq{\mathfrak{q}}
\def\fr{\mathfrak{r}}

\def\bC{\mathbb{C}}
\def\bD{\mathbb{D}}
\def\bN{\mathbb{N}}
\def\bO{\mathbb{O}}
\def\bQ{\mathbb{Q}}
\def\bR{\mathbb{R}}

\def\bZ{\mathbb{Z}}

\def\bfD{\mathbf{D}}
\def\bfL{\mathbf{L}}

\newcommand{\MatrixSpace}[2]{{#1}({#2}; \bC^{2 \times 2})}
\newcommand{\LL}[1]{\MatrixSpace{L^{#1}}{[0,\ell]}}

\newcommand{\bigabs}[1]{\bigl|{#1}\bigr|}
\newcommand{\abs}[1]{\left|{#1}\right|}

\newcommand{\norm}[1]{\left\|{#1}\right\|}

\newcommand{\angnorm}[1]{\left\langle{#1}\right\rangle}

\newcommand{\oneto}[1]{\{1, \ldots, {#1}\}}
\newcommand{\onetom}{\oneto{m}}

\newcommand{\oneton}{\oneto{n}}
\newcommand{\onetoN}{\oneto{N}}
\newcommand{\onetor}{\oneto{r}}

\newcommand{\DD}[1]{\frac{\partial}{\partial{#1}}}
\newcommand{\Dx}{\DD{x}}
\newcommand{\Dt}{\DD{t}}

\begin{document}

\sloppy

\begin{frontmatter}

\title
{On transformation operators and \\
Riesz basis property of root vectors system \\
for $n \times n$ Dirac type operators. \\
Application to the Timoshenko beam model}

\author{Anton~A.~Lunyov}
\ead{A.A.Lunyov@gmail.com}
\address{
Facebook, Inc.,
1 Hacker Way, Menlo Park,
California, 94025,
United States of America}

\author{Mark~M.~Malamud}
\ead{malamud3m@gmail.com}
\address{
Peoples Friendship University of Russia (RUDN University),
6 Miklukho-Maklaya St.
Moscow, 117198,
Russian Federation}

\begin{abstract}
The paper is concerned with
the following $n \times n$ Dirac type equation
\begin{equation*}
 L y = -i B(x)^{-1} \bigl(y' + Q(x) y\bigr) = \l y , \quad
 B(x) = B(x)^*, \quad y= \col(y_1, \ldots, y_n), \quad x \in [0,\ell],
\end{equation*}
on a finite interval $[0,\ell]$. Here $Q \in L^1([0,\ell]; \bC^{n \times n})$ is a potential matrix and $B \in L^{\infty}([0,\ell]; \bR^{n \times n})$ is an invertible self-adjoint diagonal ``weight'' matrix.
If $n=2m$ and $B(x) = \diag(-I_m, I_m)$ this equation is equivalent to Dirac equation of order $n$.

We show the existence of triangular transformation operators for such equation under additional uniform separation conditions on the entries of the matrix function $B$.
Here we apply this result to study direct spectral properties of the boundary value problem (BVP) associated with the above equation subject to the general boundary conditions $U(y)=Cy(0)+Dy(\ell) = 0$, $\rank(C \ D) = n$.

As a first application of this result, we show that the deviation of the characteristic determinants of this BVP and the unperturbed BVP (with $Q=0$) is a Fourier transform of some summable function explicitly expressed via kernels of the transformation operators. In turn, this representation yields asymptotic behavior of the spectrum in the case of regular boundary conditions. Namely, $\l_m = \l_m^0 + o(1)$ as $m \to \infty$, where $\{\l_m\}_{m \in \bZ}$ and $\{\l_m^0\}_{m \in \bZ}$ are sequences of eigenvalues of perturbed and unperturbed ($Q=0$) BVP, respectively.

Further, we prove that the system of root vectors of the above BVP constitutes a Riesz basis in a certain weighted $L^2$-space, provided that the boundary conditions are \emph{strictly regular}. Along the way, we also establish completeness, uniform minimality and asymptotic behavior of root vectors.

The main results are applied to establish asymptotic behavior of eigenvalues and eigenvectors, and the Riesz basis property for the dynamic generator of spatially non-homogenous damped Timoshenko beam model. We also found a new case when eigenvalues have an explicit asymptotic, which to the best of our knowledge is new even in the case of constant parameters of the model.
\end{abstract}


\end{frontmatter}

\renewcommand{\contentsname}{Contents}
\tableofcontents

\section{Introduction} \label{sec:intro}
In this paper we continue our investigation~\cite{LunMal15JST},~\cite{LunMal16JMAA} of the spectral properties
of non-self-adjoint boundary value problems (BVP) for the following
first order system of ordinary differential equations (ODE):
\begin{equation} \label{eq:LQ.def.intro}
 \cL(Q) y := -i B(x)^{-1} \bigl(y' + Q(x) y\bigr) = \l y , \qquad
 y= \col(y_1, \ldots, y_n), \qquad x \in [0,\ell],
\end{equation}
subject to the following boundary conditions with $n\times n$ matrices $C, D \in \bC^{n \times n}$:
\begin{equation} \label{eq:Uy=0.intro}
 U(y) := C y(0) + D y(\ell) = 0, \quad\text{and}\quad \rank(C \ D) = n.
\end{equation}
Here
\begin{equation} \label{eq:Bx.def.intro}
 B = \diag(\beta_1, \ldots, \beta_n), \qquad
 \beta_k \in L^1([0,\ell]; \bR),
 \qquad k \in \oneton,
\end{equation}
is a self-adjoint invertible diagonal summable matrix function, and
\begin{equation} \label{eq:Q=Qjk.def.intro}
 Q = (Q_{jk})_{j,k=1}^n, \qquad Q_{jk} \in L^1[0,\ell] := L^1([0,\ell]; \bC),
 \qquad j, k \in \oneton,
\end{equation}
is a summable (generally non-self-adjoint) potential matrix.

Next we associate with the BVP~\eqref{eq:LQ.def.intro}--\eqref{eq:Uy=0.intro}
an operator $L_U(Q)$. To this end we denote by $\fH_k := L^2_{|\beta_k|}[0,\ell]$ the weighted $L^2$-space
with the weight $|\beta_k|$, $k \in \oneton$, and set
$\fH := \fH_1 \oplus \ldots \oplus \fH_n.$ Now the operator $L_U(Q)$ in $\fH$ is defined
as a restriction of the maximal operator $L_{\max}(Q)$ generated in $\fH$ by the differential expression $\cL(Q)$ to the domain
\begin{equation} \label{eq:dom.LU.Intro}
\dom(L_U(Q)) := \{y \in \AC([0,\ell]; \bC^n) \ : \ \cL(Q) y \in \fH, \ \
 \ U(y) = C y(0) + D y(\ell) = 0 \}.
\end{equation}

Systems~\eqref{eq:LQ.def.intro} are of significant interest in some
theoretical and practical questions. For instance, if $n=2m$,
$B(x)=\diag(-I_m, I_m)$, $Q=\begin{pmatrix} 0 & Q_{12} \\ Q_{21}
& 0 \end{pmatrix}$,
system~\eqref{eq:LQ.def.intro} is equivalent to
the Dirac system (see~\cite{LevSar88},~\cite[{Section 1.2}]{Mar77}). Note also that
equation~\eqref{eq:LQ.def.intro} {with arbitrary
constant not necessary self-adjoint matrix $B(x)= \diag(b_1, \ldots, b_n) \in \bC^{n\times n}$}
is used to integrate the $N$-waves problem arising in nonlinear
optics~\cite[{Sec.III.4}]{ZMNovPit80}.

The spectral problem~\eqref{eq:LQ.def.intro}--\eqref{eq:Uy=0.intro} (the operator $L_U(Q)$) has first been investigated by G.D.~Birkhoff and R.E.~Langer~\cite{BirLan23}. Namely, they have extended some previous results of~Birkhoff and Tamarkin on non-self-adjoint boundary value problem for ODE to the case of BVP~\eqref{eq:LQ.def.intro}--\eqref{eq:Uy=0.intro}. More precisely,
they introduced the concepts of \emph{regular and strictly
regular boundary conditions}~\eqref{eq:Uy=0.intro} and
investigated the asymptotic behavior of eigenvalues and
eigenfunctions of the corresponding operator $L_U(Q)$.
Moreover, they proved \emph{a pointwise convergence result} on
spectral decompositions of the operator $L_U(Q)$ corresponding
to the BVP~\eqref{eq:LQ.def.intro}--\eqref{eq:Uy=0.intro} with
regular boundary conditions.

To the best of our knowledge, the problem of the completeness of the system of root vectors
\emph{of general BVP}~\eqref{eq:LQ.def.intro}--\eqref{eq:Uy=0.intro} \emph{with arbitrary
constant not necessary self-adjoint matrix
$B(x)= \diag(b_1, \ldots, b_n) \in \bC^{n\times n}$},
was first investigated in the recent papers~\cite{MalOri00,MalOri12} by
one of the authors and L.L.~Oridoroga. In these papers the
concept of \emph{weakly regular} boundary conditions for the
system~\eqref{eq:LQ.def.intro} was introduced and the completeness of
root vectors for this class of BVP was proved.
Completeness property for general
BVP~\eqref{eq:LQ.def.intro}--\eqref{eq:Uy=0.intro} with non-weakly
regular and even with degenerate boundary conditions was obtained in~\cite{LunMal15JST}.
 \emph{Emphasize that in the case of
non-weakly regular boundary conditions the completeness property
substantially depends on the values $Q(0)$ and $Q(1)$.}

Going over to the basis property note that
during the last decade there appeared numerous papers devoted mainly to the
Riesz basis property for $2\times 2$ Dirac system subject to the
\emph{regular or strictly regular} boundary conditions
(see~\cite{TroYam02,Mit03,Mit04,HasOri09,Bask11,DjaMit10,DjaMit12UncDir,DjaMit12TrigDir,DjaMit12Equi,DjaMit12Crit,DjaMit13CritDir}).
The most complete result on the Riesz basis property for $2\times 2$ Dirac and Dirac-type systems with $Q \in L^1$ and strictly regular boundary conditions was obtained independently by different methods and at the same time by A.M.~Savchuk and A.A.~Shkalikov~\cite{SavShk14} and by the authors~\cite{LunMal14Dokl, LunMal16JMAA}. The case of regular boundary conditions is treated in~\cite{SavShk14} for the first time. Other proofs were obtained later in~\cite{SavSad15DAN},~\cite{SavSad15} (see also their recent survey~\cite{SavSad20} and references therein).

To the best of our knowledge the first result on the Riesz basis property for
BVP~\eqref{eq:LQ.def.intro}--\eqref{eq:Uy=0.intro} generated by general $n \times n$ system~\eqref{eq:LQ.def.intro}
with $B(x) = B = \diag(b_1, \ldots, b_n) \in \bC^{n\times n}\not = B^*$
and bounded $Q \in L^{\infty}([0,1]; \bC^{n \times n})$
was obtained by the authors in~\cite{LunMal15JST}. Treated boundary conditions form rather broad class that covers, in particular, periodic, antiperiodic, and regular separated (not necessarily self-adjoint) boundary conditions.
Note also that BVP for $2m \times 2m$ Dirac equation ($B=\diag(-I_m, I_m)$) was investigated in~\cite{MykPuy13} (Bari-Markus property for Dirichlet BVP with $Q \in L^2([0,1]; \bC^{2m \times 2m})$ and in~\cite{KurAbd18, KurGad20} (Bessel and Riesz basis properties on abstract level).

Note also that \emph{periodic and antiperiodic (necessarily non-strictly regular) BVP} for $2 \times 2$ Dirac and Sturm-Liouville equations have also attracted certain attention during the last decade. For instance, in~\cite[Theorem~13]{DjaMit12TrigDir},~\cite[Theorem~19]{DjaMit12Crit} and~\cite{DjaMit13CritDir}, it is established a \emph{criterion} for the system of root vectors to contain a Riesz basis for periodic (resp., antiperiodic) $2 \times 2$ Dirac operator in terms of the Fourier coefficients of $Q$ as well as in terms of periodic (resp., antiperiodic) and Dirichlet spectra.
It is also worth mentioning that F.~Gesztesy and V.A.~Tkachenko~\cite{GesTka09,GesTka12} for $q \in L^2[0,\pi]$ and P.~Djakov and B.S.~Mityagin~\cite{DjaMit12Crit} for $q \in W^{-1,2}[0,\pi]$ established by different methods a \emph{criterion} for the system of root vectors to contain a Riesz basis for Sturm-Liouville operator $-\frac{d^2}{dx^2} + q(x)$ on $[0,\pi]$. See also recent survey~\cite{DjaMit20UMNper} by P. Djakov and B. Mityagin, surveys~\cite{Mak12,Mak21Ito} by A.S. Makin, and the references therein.

Note in conclusion, that the Riesz basis property for abstract operators is investigated in numerous papers. Let us mention~\cite{Katsn67,MarMats84,Markus88,Agran99,Shk10,BarYak15,BarYak16}, the recent survey by A.A.~Shkalikov~\cite{Shk16}, and the references therein.

\vskip 10pt

Let us formulate our main results. To this end we need to impose certain conditions on the entries of the matrix function $B(\cdot)$. We assume that there exists $\theta \in (0,1)$ and $n_- \in \{0, 1, \ldots, n\}$, such that
\begin{align}
\label{eq:beta-}
 -\infty < -\theta^{-1} < \beta_1(x) & \le \ldots \le \beta_{n_-}(x)
 < -\theta < 0, \qquad x \in [0, \ell], \\
\label{eq:beta+}
 0 < \theta < \beta_{n_-+1}(x) & \le \ldots \le \beta_n(x)
 < \theta^{-1} < \infty, \qquad x \in [0, \ell],
\end{align}
and for each $k \in \oneto{n-1}$
\begin{equation}
\label{eq:betak-betaj<-eps.intro}
 \text{either} \quad \beta_k \equiv \beta_{k+1}
 \quad \text{or} \quad \beta_k(x) + \theta < \beta_{k+1}(x),
 \quad x \in [0, \ell].
\end{equation}

Let us recall the definition of \emph{regular boundary conditions} (see~\cite[{p.89}]{BirLan23})
confining ourselves to the case of $B(x) = B(x)^*$. Set
\begin{equation} \label{eq:S.def}
S := \diag(s_1, \ldots, s_n), \qquad s_k := \sign(\beta_k(\cdot)), \quad
k \in \oneton.
\end{equation}
Conditions~\eqref{eq:beta-}--\eqref{eq:beta+} guarantee that, in fact, $S = \diag(-I_{n_-}, I_{n-n_-})$. Denote by $P_+$ and $P_-$ the spectral {projections} in $\bC^n$ onto ``positive'' and ``negative'' parts of the spectrum of $S = S^*$, respectively.
Now the concept of regularity of boundary conditions for the problem~\eqref{eq:LQ.def.intro}--\eqref{eq:Uy=0.intro} reads as follows:
\begin{equation} \label{eq:JP+.JP-.intro}
\det(CP_{+} +\ DP_{-}) \ne 0 \quad\text{and} \quad \det(CP_{-} +\ DP_{+}) \ne 0.
\end{equation}
Additionally, boundary conditions~\eqref{eq:Uy=0.intro} are called strictly regular if eigenvalues of the unperturbed operator $L_U(0)$ are asymptotically separated (see Definition~\ref{def:strict.regular} for details).

Our first main result establishes existence of triangular transformation operators for equation~\eqref{eq:LQ.def.intro} with
$Q \in L^1$ and \emph{non-constant} self-adjoint $n \times n$ matrix function $B(\cdot)$, satisfying conditions~\eqref{eq:beta-}--\eqref{eq:betak-betaj<-eps.intro} (see Theorem~\ref{th:transform.oper}). Namely, assuming for simplicity that $\beta_1 < \beta_2$ let $A = \col(a_1, \ldots, a_n)$ have non-zero entries. Then we show that the vector solution $y_A(\cdot,\l)$ of the equation~\eqref{eq:LQ.def.intro} subject to the initial condition $y_A(0,\l) = A$ admits a triangular representation
\begin{equation} \label{eq:trans.opera.intro}
 y_A(x,\l)= (I + \cK_A) e_A(x,\l)
 = e_A(x, \l) + \int^x_0 K_A(x,t) B(t) e_A(t, \l) \,dt,
\end{equation}
where
$$
e_A(x, \l) = \col(a_1 e^{i \l \rho_1(x)}, \ldots, a_n e^{i \l \rho_n(x)}),
\qquad \rho_k(x) = \int_0^x \beta_k(t)\,dt,
$$
is the solution to the unperturbed equation $\cL(0) f = \l f$ subject to the same initial condition.

This result is substantial advancement comparing to the corresponding result by one of the authors in~\cite{Mal99}, where this was established for
$Q \in L^{\infty}$ and \emph{constant} self-adjoint $n \times n$ matrix $B = B^*$, and the corresponding result by the authors in~\cite{LunMal16JMAA}, where this was established for
$Q \in L^1$ and \emph{constant} self-adjoint $2 \times 2$ matrix $B = \diag(b_1, b_2) = B^*$. This result has a wide area of applications. Here we apply triangular transformation operators only to investigation
of the spectral properties of BVP~\eqref{eq:LQ.def.intro}--\eqref{eq:Uy=0.intro}
(operator $L_U(Q)$).

As an immediate application of triangular transformation operators, we obtain formulas relating the fundamental matrix solutions $\Phi(\cdot,\l)$ and $\Phi_0(\cdot,\l)$ to equation~\eqref{eq:LQ.def.intro} with $Q \not =0$ and $Q=0$, respectively (see Proposition~\ref{prop:Phip=Phip0+int.sum} and formula~\eqref{eq:Phixl.wtA}).
In fact, it leads to a representation of the deviation $\Phi(\cdot,\l) - \Phi_0(\cdot,\l)$ as a Fourier transform of linear combinations of several transformation operators kernels (see formula~\eqref{eq:Phip=Phi0p+int1n}).

In turn, starting with this Fourier representation of $\Phi(\cdot,\l) - \Phi_0(\cdot,\l)$
we establish an important identity relating characteristic determinants $\Delta_Q(\cdot)$ and $\Delta_0(\cdot)$ of the operators $L_U(Q)$ and $L_U(0)$ (see~\eqref{eq:A.def}--\eqref{eq:Delta.def} for exact definitions). Namely, letting
\begin{equation} \label{eq:def_b_-,b_+,_b_k}
 b_- := b_1 + \ldots + b_{n_-} \quad\text{and}\quad
 b_+ := b_{n_- + 1} + \ldots + b_n, \quad\text{where}\quad
 b_k := \rho_k(\ell) = \int_0^{\ell} \beta_k(x)\, dx.
\end{equation}
\emph{we show that the characteristic determinant $\Delta_Q(\cdot)$ admits the following representation:}
\begin{equation} \label{eq:DeltaQ=Delta0+int.intro}
 \Delta_Q(\l) = \Delta_0(\l) + \int_{b_-}^{b_+} g(u) e^{i \l u} \,du
 \qquad\text{with}\qquad g \in L^1[b_-, b_+], \quad \l \in \bC.
\end{equation}
It is worth mentioning that the second key ingredient in the proof of formula~\eqref{eq:DeltaQ=Delta0+int.intro} is an extension of the classical Liouville formula for the determinant of a fundamental matrix $\Phi(x,\l)$ to the case of its $m$-th exterior powers $\bigwedge^m \Phi(x,\l)$ obtained in Proposition~\ref{prop:minor.form}.

For a special case of $2 \times 2$ Dirac type operator (\emph{constant} $B \equiv \diag(b_1, b_2) = B^*\in \bC^{2\times 2}$) representation~\eqref{eq:DeltaQ=Delta0+int.intro} was earlier established in~\cite{LunMal14Dokl}, \cite[Lemma 4.1]{LunMal16JMAA}. In recent papers by A.~Makin~\cite{Mak20,Mak21,Mak21DE} this representation was applied to establish Riesz basis property of periodic BVP (regular but not strictly regular) for $2 \times 2$ Dirac equation under certain explicit algebraic assumptions on a potential matrix.

Formula~\eqref{eq:DeltaQ=Delta0+int.intro} gives a bridge between the spectral theory of the operator $L_U(Q)$
and the theory of entire functions of exponential type due to the simple fact: the spectrum $\sigma(L_U(Q)) = \{\l_m\}_{m \in \bZ}$
of $L_U(Q)$ coincides with the set of zeros (counting multiplicity) of the entire function $\Delta_Q(\cdot)$.
Assuming \emph{boundary conditions to be regular} we
easily obtain from~\eqref{eq:DeltaQ=Delta0+int.intro} that $\Delta_Q(\cdot)$ is an entire sine-type function of
the same types $b_{\pm}$ in $\bC_{\mp}$ as the determinant $\Delta_0(\cdot)$.
Further, following the schema of the proof of~\cite[Proposition 4.7]{LunMal16JMAA},
we extract the following asymptotic formula from representation~\eqref{eq:DeltaQ=Delta0+int.intro}
\begin{equation} \label{eq:ln=ln0+o.intro)}
\l_m = \l_m^0 + o(1) \qquad \text{as} \qquad m \to \infty,
\end{equation}
relating the eigenvalues $\{\l_m\}_{m \in \bZ}$ of the operator $L_U(Q)$ and eigenvalues $\{\l_m^0\}_{m \in \bZ}$ of the unperturbed operator $L_U(0)$. Note, that for $2\times 2$ Dirac equation formula~\eqref{eq:ln=ln0+o.intro)} was first established in~\cite{LunMal14Dokl} and~\cite{SavShk14} independently and by different methods.

Moreover, \emph{assuming boundary conditions to be strictly regular}
we complete formula~\eqref{eq:ln=ln0+o.intro)} by establishing similar formula for the normalized eigenvectors $f_m(\cdot)$ and $f_m^0(\cdot)$ of the operators $L_U(Q)$ and $L_U(0)$, respectively. Namely,
using formula relating $\Phi(\cdot,\l)$ and $\Phi^0(\cdot,\l)$ as well as a simple abstract formula for
simple eigenvectors of the operator $L_U(Q)$, we
establish the following formula for their deviation which is valid \emph{uniformly in $x \in [0,\ell]$}:
\begin{equation} \label{eq:Yp=Yp0+o_Intro}
 f_m(x) = f_m^0(x) + o(1) \quad\text{as}\quad m \to \infty, \quad m \in \bZ,
\end{equation}
In turn, this relationship and formula~\eqref{eq:ln=ln0+o.intro)} are substantially involved in the proof of the Riesz basis property of the operator $L_U(Q)$ provided that boundary conditions are \emph{strictly regular}.
This proved to be challenging even on the algebraic level (the case of $Q=0$) and required establishing a new algebraic identity (see Proposition~\ref{prop:y0p.y0*q}) for the inner product of the eigenvectors of the unperturbed operator $L_U(0)$ and its adjoint $L_U^*(0)$.

It is worth mentioning that in Section~\ref{subsec:examples.strict} we find necessary and sufficient conditions for quasi-periodic boundary conditions
\begin{equation} \label{eq:yell=C.y0_Intro}
 y_k(\ell) = c_k y_k(0), \qquad c_k \ne 0, \qquad k \in \oneton,
\end{equation}
to be strictly regular. In this case conditions~\eqref{eq:Uy=0.intro} hold with invertible
$C = \diag(c_1, \ldots, c_n)$ and $D = -I_n$. In accordance with~\eqref{eq:JP+.JP-.intro}, conditions~\eqref{eq:yell=C.y0_Intro} are always regular but not necessary strictly regular.
Morover, \emph{antiperiodic} boundary conditions $(c_1 = \ldots = c_n = -1)$ are strictly regular if and only if
for some $b_0 > 0$ the numbers $b_1, \ldots, b_n$ given by~\eqref{eq:def_b_-,b_+,_b_k}
can be ordered in such a way that the following representation holds,
\begin{equation}
 b_k = 2^{a_k} (2 u_k + 1) b_0, \qquad a_k, u_k \in \bZ, \quad k \in \oneton,
 \qquad 0 \le a_1 < a_2 < \ldots < a_n.
\end{equation}
In particular, antiperiodic boundary conditions are strictly regular if
$b_k = 2^k$, $k \in \oneton$.

We also obtain completeness property in the case of regular boundary conditions extending the corresponding
result from~\cite{MalOri12} to the case of non-constant matrix function $B(\cdot) \not = \const$.
In Section~\ref{sec:uniform.minim.riesz} we also establish the Riesz basis property with parentheses for the operator $L_U(Q)$ provided that boundary conditions~\eqref{eq:Uy=0.intro} are only regular (but not strictly regular). For the proof we use the perturbation idea which goes back to A.A.~Shkalikov~\cite{Shk79} and was applied later on to $2 \times 2$ Dirac systems in~\cite{SavShk14,LunMal16JMAA}.

Finally, we apply our main abstract results with $B(x) = B(x)^* \in
\bC^{4 \times 4}$ to the Timoshenko beam model investigated
under different restrictions in numerous papers
(see~\cite{Tim55,KimRen87,MenZua00,Shub02,XuYung04,XuHanYung07,WuXue11,Shub11,LunMal15JST,LunMal16JMAA,Akian22} and the references therein). In our previous papers~\cite{LunMal15JST, LunMal16JMAA} we studied the Timoshenko beam model with relaxed smoothness assumptions on the coefficients, when the beam is fixed at one end and with the most general boundary condition at the other end. For this general model, we established completeness and Riesz basis property with parentheses, assuming certain identity for the coefficients of the model: the ratio of wave speeds $\frac{K(\cdot)}{\rho(\cdot)}$ and $\frac{EI(\cdot)}{I_\rho(\cdot)}$ is constant.
This assumption has to be added because in our previous papers~\cite{LunMal15JST, LunMal16JMAA} we treated BVP~\eqref{eq:LQ.def.intro}--\eqref{eq:Uy=0.intro} with
a \emph{constant matrix} $B(x) = B$. However, the dynamic generator of the general Timoshenko beam model is similar to the operator $L_U(Q)$ with $B(x) = (-\beta_1(x), \beta_1(x), -\beta_2(x), \beta_2(x))$ and \emph{functions $\beta_1(x), \beta_2(x)$ with non-constant ratio}, and cannot be reduced to Dirac-type operator with a constant matrix $B$.
Since in this paper, we treat more general BVP~\eqref{eq:LQ.def.intro}--\eqref{eq:Uy=0.intro} with arbitrary non-constant matrix $B(\cdot)$, this allows us to remove this algebraic assumption.
Moreover, we establish asymptotic behavior of the eigenvalues and eigenvectors of the dynamic generator $L_{\Tim}$ of the Timoshenko beam model as well as the Riesz basis property (without parentheses) of the root vectors system of the operator $L_{\Tim}$, provided that the eigenvalues of $L_{\Tim}$ are asymptotically separated. We also provide comprehensive list of explicit conditions that guarantee this property.

When our preprint was almost ready we became aware of the short communication~\cite{Shk21} by A.A.~Shkalikov, where results on Riesz basis property of regular BVP~\eqref{eq:LQ.def.intro}--\eqref{eq:Uy=0.intro} with $Q \in L^1$ were announced under the similar assumptions on matrix function $B(\cdot)$.

The paper is organized as follows. Section~\ref{sec:prelim} is devoted to some preliminaries. In particular, we list
some identities of determinants of sums and products of matrices.

In Section~\ref{sec:transform.oper} we prove our first main result, Theorem~\ref{th:transform.oper}, establishing existence of triangular transformation operators for equation~\eqref{eq:LQ.def.intro}.
The case of \emph{non-constant} matrix function $B(\cdot)$ poses significant difficulties even for $Q \in C^1$.

In Section~\ref{sec:fundament} we apply transformation operators
to establish an important identity for the fundamental matrix $\Phi(x,\l)$ of the equation~\eqref{eq:LQ.def.intro}. Namely, in Proposition~\ref{prop:Phip=Phip0+int.sum} we show that the deviation of the fundamental matrices of equation~\eqref{eq:LQ.def.intro} with $Q \ne 0$ and $Q = 0$ admits a Fourier transform representation involving the kernels of the transformation operators.
In Proposition~\ref{prop:minor.form} we generalize a classical Liouville's formula and show that $m$-th exterior power $\bigwedge^m \Phi(x,\l)$
of the fundamental matrix $\Phi(x,\l)$
satisfies equation similar to~\eqref{eq:LQ.def.intro}. This result, in turn, implies similar Fourier transform representation for the minors of $\Phi(x,\l)$.

In Section~\ref{sec:regular.bc} we establish some general properties of BVP~\eqref{eq:LQ.def.intro}--\eqref{eq:Uy=0.intro} and introduce concepts of regular and strictly regular boundary conditions. In particular, we apply Jacobi's formula to establish certain important uniform estimates from below for eigenvectors of this BVP with $Q=0$, provided that boundary conditions are strictly regular (see Proposition~\ref{prop:eigenvec0}).

In Section~\ref{sec:asymp.eigen} we establish key identity~\eqref{eq:DeltaQ=Delta0+int.intro}
and similar identity related to eigenvectors of the operator $L_U(Q)$. These identities are used to establish the asymptotic behavior of eigenvalues and eigenvectors in Theorems~\ref{th:ln=ln0+o} and~\ref{th:eigenvec}, respectively.

In Section~\ref{sec:compl}, following~\cite{MalOri12}, we establish completeness of the root vectors system of the operator $L_U(Q)$ (see Subsection~\ref{subsec:srv} for exact definition), provided that $Q \in L^1$, matrix function $B(\cdot)$ meats conditions~\eqref{eq:beta-}--\eqref{eq:betak-betaj<-eps.intro}, and boundary conditions~\eqref{eq:Uy=0.intro} are regular.

In Section~\ref{sec:adjoint}, we study adjoint operator $L_U(Q)^*$. In Proposition~\ref{prop:y0p.y0*q} we establish an important identity
for the inner product of eigenvectors of the unperturbed operator $L_U(0)$ and its adjoint $L_U^*(0)$, which is essential for proving uniform minimality property.

In Section~\ref{sec:uniform.minim.riesz}, we prove our main results on uniform minimality and Riesz basis property of the root vectors system of the operator $L_U(Q)$ with strictly regular boundary conditions (see Theorems~\ref{th:uniform.minim} and~\ref{th:Riesz_basis}).
Here we also establish Riesz basis property with parentheses provided that boundary conditions~\eqref{eq:Uy=0.intro} are regular (see Theorem~\ref{th:riesz.paren}).

In Section~\ref{sec:Timoshenko} we apply our abstract results with $B(x) = B(x)^* \in \bC^{4 \times 4}$ to the dynamic generator $L_{\Tim}$ of the general model~\eqref{eq:Tim.Ftt}--\eqref{eq:Tim.WLFLa2} of spatially non-homogenous Timoshenko beam with both boundary and locally distributed damping. By reducing this dynamic generator $L_{\Tim}$ to the special $4 \times 4$ Dirac type operator $L_U(Q)$, we show that the root vectors system of $L_{\Tim}$ forms a Riesz basis in the suitable energy space, when the corresponding operator $L_U(Q)$ is equipped with the strictly regular boundary conditions (see Theorems~\eqref{eq:Tim.riesz} and~\eqref{th:Tim.asymp.riesz}(ii)). We also apply results of Section~\ref{sec:asymp.eigen} to establish the asymptotic behavior of the eigenvalues and the eigenvectors of the operator $L_{\Tim}$ (see Theorems~\ref{th:Tim.asymp} and~\ref{th:Tim.asymp.riesz}(i)). In particular, we found an interesting case when eigenvalues of $L_{\Tim}$ have an explicit asymptotical formula, which to the best of our knowledge is new even in the case of constant parameters of the model (see Theorem~\ref{th:Tim.asymp}(iii)).
\section{Preliminaries} \label{sec:prelim}
\subsection{Definition of the system of root vectors}
\label{subsec:srv}
Let us also recall the notion of the system of root vectors of an operator with compact resolvent. First, we recall a few basic facts regarding the eigenvalues of a compact, linear operator $T \in \cB_{\infty}(\fH)$ in a
separable complex Hilbert space $\fH$. The {\it geometric
multiplicity}, $m_g(\mu_0,T)$, of an eigenvalue $\mu_0
 \in \sigma_p (T)$ of $T$ is given by
$
m_g(\mu_0,T) := \dim(\ker(T - \mu_0)).
$

The {\it root subspace} of $T$ corresponding to $\mu_0 \in
\sigma_p(T)$ is given by
\begin{equation} \label{root.subspace}
\cR_{\mu_0}(T) = \big\{f \in \fH\,:\, (T - \mu_0)^k f = 0 \
\ \text{for some}\ \ k \in \bN \big\}.
\end{equation}
Elements of $\cR_{\mu_0}(T)$ are called {\it root vectors}.
For $\mu_0 \in \sigma_p (T) \backslash \{0\}$, the set
$\cR_{\mu_0}(T)$ is a closed linear subspace of
$\fH$ whose dimension equals to the {\it algebraic
multiplicity}, $m_a(\mu_0,T)$, of $\mu_0$,
$
m_a(\mu_0,T) := \dim\big(\cR_{\mu_0}(T)\big)<\infty.
$

Denote by $\{\mu_j\}_{j=1}^{\infty}$ the sequence of non-zero
eigenvalues of $T$ ($\mu_j \ne \mu_k$) and let $m_j$ be the algebraic multiplicity
of $\mu_j$. By the {\it system of root vectors} of the operator
$T$ we mean any sequence of the form
$
\cup_{j=1}^{\infty}\{e_{jk}\}_{k=1}^{m_j},
$
where $\{e_{jk}\}_{k=1}^{m_j}$ is a basis in $\cR_{\mu_j}(T)$,
$m_j = m_a(\mu_j,T) < \infty$.
The system or root vectors of the operator $T$ is called \emph{normalized} if $\|e_{jk}\|_{\fH} = 1$, $j \in \bN$, $k \in \oneto{m_j}$.

We are particularly interested in the case where $A$ is a
densely defined, closed, linear operator in $\fH$ whose
resolvent is compact, that is,
$
R_A(\l):=(A - \l)^{-1} \in \cB_{\infty}(\fH), \ \l \in \rho(A).
$
Via the spectral mapping theorem all eigenvalues of $A$
correspond to eigenvalues of its resolvent $R_A(\l)$, $\l \in
\rho(A)$, and vice versa. Hence, we use the same notions of
root vectors, root subspaces, geometric and algebraic
multiplicities associated with the eigenvalues of $A$, and the
system of root vectors of $A$.
\subsection{Properties of Lipshitz functions} \label{subsec:Lip}
Recall that $\Lip_1(\cS)$ for $\cS \subset X$ in any normed space $X$ is the class of functions $f$ acting from $\cS$ to $\bC$ and satisfying the condition
$$
|f(u) - f(u')| \le \alp \|u - u'\|_X, \quad u,u' \in \cS \quad\text{for some}\quad \alp = \alp_f > 0.
$$
It is well known that for any finite segment $[a,b] \in \bR$,
\begin{equation} \label{eq:Lip1ab}
 \Lip_1[a,b] = \{f \in \AC[a,b]: f' \in L^{\infty}[a,b]\} = W^{1,\infty}[a,b].
\end{equation}
We also denote by $L_{1,\loc}(\cS)$ a set of functions $f$ that are Lipshitz on any compact subset of $\cS$.

Our main target use case will be $X = \bR^2$. To this end we denote by
$$
|u-u'| := \|u-u'\|_{\bR^2} := |x - x'| + |t - t'|,
\qquad u = (x,t), \ u' = (x',t') \in \bR^2,
$$
a Manhattan distance between points $u$ and $u'$.
Further, a simply connected, closed bounded set $\cS \subset \bR^2$ is said to have a Lipshitz boundary if its boundary $\partial S$ can be parametrized as
\begin{equation} \label{eq:dS}
 \partial S = \{(\gam_1(t), \gam_2(t)): t \in [a,b]\},
\end{equation}
for some $-\infty < a < b < \infty$ where $\gam_j \in \Lip_1[a,b]$ and $\gam_j(a) = \gam_j(b)$, $j \in \{1,2\}$.
Throughout the paper we will denote for $f : \bR^2 \to \bC$,
\begin{equation} \label{eq:D1f.D2f}
 (D_1 f)(x,t) := D_1 f(x,t) := \DD{x} f(x,t) , \qquad
 (D_2 f)(x,t) := D_2 f(x,t) := \DD{t} f(x,t),
\end{equation}
whenever corresponding partial derivatives exist.

In the sequel we will need the following simple properties of Lipshitz functions.
\begin{lemma} \label{lem:C.Lip}
\textbf{(i)} Let $\cS \subset \bR^2$ and let $f_m \in \Lip_1(\cS)$, $m \in \bN$, be such that
\begin{align} \label{eq:fm.to.x}
 & f_m(x,t) \to f(x,t) \ \ \text{as}\ \ m \to \infty, \quad u = (x,t) \in \cS, \\
\label{eq:fm.unilip}
 & |f_m(x,t) - f_m(x',t')| \le \alp (|x-x'| + |t-t'|), \quad (x,t), (x',t') \in \cS, \quad m \in \bN,
\end{align}
for some $\alp>0$ and $f: \cS \to \bC$. Then
$$
f \in \Lip_1(\cS) \qquad\text{and}\qquad
|f(x,t) - f(x',t')| \le \alp (|x-x'| + |t-t'|), \quad (x,t), (x',t') \in \cS.
$$
Moreover, $\Lip_1(\cS)$ is a Banach space with the norm
$$
\|f\|_{\Lip_1(\cS)} := |f(x_0, t_0)|
+ \sup_{\genfrac{}{}{0pt}{2}{u, u' \in \cS}{u \ne u'}} \frac{|f(x,t) - f(x',t')|}{|u-u'|}, \qquad u = (x,t), \ u' = (x',t'),
$$
where $u_0 = (x_0, t_0) \in \cS$ is an arbitrary fixed point.

\textbf{(ii)} Let $\cS \subset \bR^2$ be a simply connected, closed bounded set with a Lipshitz boundary and let $f \in \Lip_1(\cS)$, be such that
\begin{equation} \label{eq:cDf.def}
 \cD(f) := \max\{\|D_1 f\|_{L^{\infty}(\cS)}, \|D_2 f\|_{L^{\infty}(\cS)}\}
 < \infty.
\end{equation}
Then
$$
|f(x,t) - f(x',t')| \le \alp (|x-x'| + |t-t'|), \quad (x,t), (x',t') \in \cS,
$$
where constant $\alp \in (0, \infty)$ depends only on $\cD(f)$ and the set $\cS$.

\textbf{(iii)} Let $\cS \subset \bR^2$ be a simply connected, closed bounded set with a Lipshitz boundary and let $f_m \in \Lip_1(\cS)$, $m \in \bN$, be such that
\begin{equation} \label{eq:cDfm<tau}
 f_m(x,t) \to f(x,t) \ \ \text{as}\ \ m \to \infty, \quad (x,t) \in \cS,
 \qquad \|D_1 f_m\|_{L^{\infty}(\cS)},
 \ \|D_2 f_m\|_{L^{\infty}(\cS)} \le \tau, \quad m \in \bN,
\end{equation}
for some $\tau>0$ and $f: \cS \to \bC$. Then $f \in \Lip_1(\cS)$.
\end{lemma}
\begin{proof}
Parts (i) and (ii) are well-known and are of folklore nature. Let us only mention that the set $\cS$ as a simply connected, closed and bounded set with a Lipshitz boundary has the following important property: for any interior points $u,u' \in \inter S$ there exists a ``Manhattan'' curve $\Gam_{u,u'}$ (a finite sequence of alternating vertical and horizontal segments) connecting $u$ and $u'$ such the length of $\Gam_{u,u'}$
is bounded by $C_{\cS} |u-u'|$ for some $C_{\cS}$ that only depend on $\cS$ (more precisely it only depends on Lipshitz constants of the parametric curves $\gam_1$, $\gam_2$ of the boundary $\partial S$). This property allows to utilize fundamental representation~\eqref{eq:Lip1ab} for Lipshitz space on a finite segment.

Let us also comment on part (iii). Part (ii) and condition~\eqref{eq:cDfm<tau} imply uniform Lipshitz condition~\eqref{eq:fm.to.x}--\eqref{eq:fm.unilip} with some $\alp$ that only depends on $\tau$ and the set $\cS$. Part (i) finishes the proof.
\end{proof}
\subsection{The Banach spaces \texorpdfstring{$X_1$ and $X_\infty$}{X1 and Xinf}}
\label{subsec:X1inf}
Following~\cite{Mal94,LunMal16JMAA} denote by $X_1:= X_1(\Omega)$ and $X_\infty := X_\infty(\Omega)$ the linear spaces composed of (equivalent classes of) measurable functions defined on
\begin{equation} \label{eq:Omega.def}
 \Omega := \{(x,t) : 0 \le t \le x \le \ell\}
\end{equation}
satisfying
\begin{align}
\label{eq:B2.norm.def}
 \|f\|_{X_1} &:= \esssup_{t \in [0,\ell]} \int_t^\ell |f(x,t)| dx < \infty,\\
\label{eq:B1.norm.def}
 \|f\|_{X_{\infty}} &:= \esssup_{x \in [0,\ell]} \int_0^x |f(x,t)| dt < \infty,
\end{align}
respectively. It can easily be shown that the spaces $X_1$ and
$X_\infty$ equipped with the norms~\eqref{eq:B2.norm.def} and
\eqref{eq:B1.norm.def} form Banach spaces that are not
separable. Denote by $X_{1,0}$ and $X_{\infty,0}$ the subspaces
of $X_1$ and $X_\infty$, respectively, obtained by taking the
closure of continuous functions $f \in C(\Omega)$. Clearly, the
set $C^1(\Omega)$ of smooth functions is also dense in both
spaces $X_{1,0}$ and $X_{\infty,0}$.

The following simple property of the class
$X_{\infty,0}(\Omega)$ established in~\cite{LunMal16JMAA} will be essential in the sequel.
\begin{lemma}[Lemma 2.2 in~\cite{LunMal16JMAA}] \label{lem:trace}
For each $a \in [0,\ell]$ the trace mapping
\begin{equation}
 i_a:\ C(\Omega) \to C[0,a], \qquad
 (i_a R)(t) := R(a, t), \quad R \in C(\Omega),
\end{equation}
originally defined on $C(\Omega)$ admits a continuous extension (also denoted by $i_a$) as a mapping
from $X_{\infty,0}(\Omega)$ onto $L^1[0,a]$.
\end{lemma}
We will also need a property of intersections $X_1(\Omega) \cap X_{\infty}(\Omega)$, $X_{1,0}(\Omega) \cap X_{\infty,0}(\Omega)$ established in~\cite{LunMal16JMAA}. To this end, for any measurable on $\Omega$ kernel $R(\cdot,\cdot)$ we define Volterra type operator $\cR$ as follows,
\begin{equation} \label{eq:cR.def}
 (\cR f)(x) = \int^x_0 R(x,t) f(t) dt.
\end{equation}
Denote by $\|\cR\|_p := \|\cR\|_{L^p[0,\ell]\to L^p[0,\ell]}$ the $L_p$-norm of the operator $\cR$, provided that it is bounded. Further, recall that a Volterra operator in a Banach space is a compact operator with zero spectrum.
\begin{lemma}[Lemma 2.1 in~\cite{LunMal16JMAA}] \label{lem.Volterra.operGeneral}
Let $R \in X_1(\Omega) \cap X_\infty(\Omega)$ and $\cR$ be a Volterra type operator given by~\eqref{eq:cR.def}. Then:

\textbf{(i)} The operator $\cR$ is bounded in $L^p[0,\ell]$ for
each $p \in [1,\infty]$ and
\begin{equation} \label{2.4Aop}
 \|\cR\|_p \le \|R\|_{X_1(\Omega)}^{1/p} \cdot \|R\|_{X_\infty(\Omega)}^{1-1/p}.
\end{equation}
Moreover,
\begin{equation} \label{2.4opNew}
 \|\cR\|_{1} = \|R\|_{X_1(\Omega)}, \qquad \|\cR\|_{\infty} = \|R\|_{X_\infty}.
\end{equation}

\textbf{(ii)} If $R \in X_{1,0}(\Omega) \cap X_{\infty,0}(\Omega)$, then $\cR$ is a Volterra operator in $L^p[0,\ell]$ for each $p \in [1,\infty]$.
\end{lemma}
In what follows, we will also systematically use notations $X_{1,0} \otimes \bC^{n \times m}$ and $X_{\infty,0} \otimes \bC^{n \times m}$. In general, for any set $S$ of complex-valued functions, notation $S \otimes \bC^{n \times m}$ means a set of all $n \times m$ matrices, where each entry of the matrix is an element of the set $S$.
\subsection{Properties of adjugate matrix}
\label{subsec:adjugate}
Denote by $\GL(n) := \GL(n, \bC)$ the set of invertible $n \times n$ matrices with complex entries. For a matrix $\cA \in \bC^{n \times n}$ denote by $\cA^a$ its adjugate matrix, i.e.
\begin{equation}
\cA \cA^a = \cA^a \cA = \det(\cA) I_n.
\end{equation}
Let us recall some properties of the adjugate matrix,
\begin{align}
\label{eq:Aa.inv}
 & \cA^a = \det(\cA) \cdot \cA^{-1}, \qquad \cA \in \GL(n), \\
\label{eq:Aa.adjont}
 & [\cA^a]^* = [\cA^*]^a, \qquad \cA \in \bC^{n \times n}, \\
\label{eq:ABa}
 & [\cA \cB]^a = \cB^a \cA^a, \qquad \cA, \cB \in \bC^{n \times n}, \\
\label{eq:adjug.prod}
 & [\cA_1 \cA \cA_2]^a = \cA_2^a \cA^a \cA_1^a =
 \det(\cA_1 \cA_2) \cdot \cA_2^{-1} \cA^a \cA_1^{-1},
 \qquad \cA_1, \cA_2 \in \GL(n, \bC).
\end{align}
Let $\cA(\cdot) = (\fa_{jk}(\cdot))_{j,k=1}^n$ be an
$n\times n$ matrix function differentiable at a point $\l \in \bC$ and let
$\cA^a(\cdot) =: (\cA_{jk}(\cdot))_{j,k=1}^n$ be its adjugate matrix function.
Then in accordance with the Jacobi identity,
\begin{equation} \label{eq:jacobi.def}
\frac{d}{d\l} \det(\cA(\l)) = \tr \( \cA^a(\l) \cA'(\l)\) = \sum_{j,k=1}^n \cA_{jk}(\l) \fa_{kj}'(\l).
\end{equation}
\subsection{Determinants of matrix sums and products}
\label{subsec:det.sum.prod}
If what follows we will need a few classical formulas for determinant of the sum and the product of matrices.
Further, assuming $n \in \bN$ is fixed throughout entire paper, we introduce the following set:
\begin{equation} \label{eq:cPm.def}
 \fP_m := \{\fp := (p_1, \ldots, p_m): 1 \le p_1 < \ldots < p_m \le n\},
 \qquad m \in \oneton,
\end{equation}
i.e.\ $\fP_m$ is the set of all increasing sequences with exactly $m$ elements from $1$ to $n$. Additionally, we define $\fP_0 := \{\fE_0\}$, where $\fE_0 := ()$ is an empty sequence. In what follows, we also denote $\sigma(\fp) := p_1 + \ldots + p_m$. Clearly $\sigma(\fE_0) = 0$.

Further, for any $n \times n$ matrix $\cA = (a_{jk})_{j,k=1}^n$ and elements $\fp = (p_1, \ldots, p_m)$ and $\fq = (q_1, \ldots, q_m)$ of $\fP_m$, $m \in \{0,1,\ldots,n\}$, we set
\begin{equation} \label{eq:Aqp.def}
 \cA[\fq,\fp] := \det (a_{q_jp_k})_{j,k=1}^m =
 \det \begin{pmatrix}
 a_{q_1p_1} & \ldots & a_{q_1p_m} \\
 \vdots & \ddots & \vdots \\
 a_{q_mp_1} & \ldots & a_{q_mp_m} \\
 \end{pmatrix}, \qquad \fq, \fp \in \fP_m,
\end{equation}
i.e.\ $\cA[\fq, \fp]$ is a minor of the matrix $\cA$ generated by the rows with indexes $q_1 < \ldots < q_m$ and columns with indexes $p_1 < \ldots < p_m$. If $m=0$ then $\cA[\fq,\fp] := \det(I_0) := 1$, where $I_0$ is an empty matrix, where $\fq = \fp = \fE_0$ are empty sequences.

Further, for $\fp \in \fP_m$ we denote by $\wh{\fp} \in \fP_{n-m}$ the complement of $\fp$ in the set $\oneton$. Namely, let $\fp = (p_1, \ldots, p_m) \in \fP_m$, i.e.\ $1 \le p_1 < \ldots < p_m \le n$, and let
$$
\oneton \setminus \{p_1, \ldots, p_m\} =: \{r_1, \ldots, r_{n-m}\},
$$
where $1 \le r_1 < \ldots < r_{n-m} \le n$. Then by definition $\wh{\fp} := (r_1, \ldots, r_{n-m})$.

Now we are ready to formulate a classical ``folklore'' formula for the determinant of the sum of matrices (see e.g.~\cite{Mar90}).
\begin{lemma} \label{lem:det.A+B}
Let $\cA, \cB \in \bC^{n \times n}$. Then
\begin{equation} \label{eq:det.A+B}
 \det(\cA+\cB) = \det(\cA) + \sum_{m=1}^n \sum_{\fq, \fp \in \fP_m}
 (-1)^{\sigma(\fp) + \sigma(\fq)} \cA[\wh{\fq}, \wh{\fp}]
 \cdot \cB[\fq, \fp].
\end{equation}
\end{lemma}
Here we utilized all of the above notations, including a complement notation $\wh{\fp}$. It is clear, that the summand in r.h.s of~\eqref{eq:det.A+B} for $m=n$ equals to $\det(\cB)$. Indeed, if $m = n$ then $\fP_m = \{\fp_0\}$, where $\fp_0 := (1, \ldots, n)$, and inner sum degenerates to $\cA[\wh{\fp_0}, \wh{\fp_0}] \cdot \cB[\fp_0, \fp_0]$. Complement $\wh{\fp_0} = \fE_0$ is an empty sequence. Hence $\cA[\wh{\fp_0}, \wh{\fp_0}] = \det(I_0) = 1$. It is also clear that $\cB[\fp_0, \fp_0] = \det(\cB)$.

Next, we formulate a straightforward extension of the classical Cauchy–Binet formula.
\begin{lemma}[Subsection 1.2.6 in~\cite{Gan59}] \label{lem:det.AB.qp}
Let $\cB, \cC \in \bC^{n \times n}$, $m \in \oneton$ and $\fp, \fq \in \fP_m$. Then
\begin{equation} \label{eq:det.AB.qp}
 (\cB \cC)[\fq,\fp] = \sum_{\fr \in \fP_m}
 \cB[\fq,\fr] \cdot \cC[\fr,\fp].
\end{equation}
\end{lemma}
Some remarks. If $m=1$, this formula is nothing more than a definition of the matrix product. If $m=n$, then this formula turns into $\det(\cB\cC) = \det(\cB)\cdot \det(\cC)$.

Finally, combining Lemmas~\ref{lem:det.A+B} and~\ref{lem:det.AB.qp} we arrive at the following formula that will be useful for studying characteristic determinant of the BVP~\eqref{eq:LQ.def.intro}--\eqref{eq:Uy=0.intro}.
\begin{lemma} \label{lem:det.A+B.C}
Let $\cA, \cB, \cC \in \bC^{n \times n}$. Then
\begin{equation} \label{eq:det.A+B.C}
 \det(\cA + \cB \cC) = \det(\cA) + \sum_{m=1}^n \sum_{\fq, \fp, \fr \in \fP_m}
 (-1)^{\sigma(\fp) + \sigma(\fq)} \cA[\wh{\fq}, \wh{\fp}]
 \cdot \cB[\fq,\fr] \cdot \cC[\fr,\fp].
\end{equation}
\end{lemma}
To estimate root vectors of the operator $L_U(Q)$ we will also need version of Lemma~\ref{lem:det.A+B.C} for cofactors of $\cA + \cB \cC$. Let us recall the corresponding definition. To this end, let $\cA = (\fa_{jk})_{j,k=1}^n \in \bC^{n \times n}$. Then by definition, cofactor $\cA\{j,k\}$ of the element $\fa_{jk}$ of the matrix $\cA$ is the element at the $j$-th row and $k$-th column of the matrix $\cA^a$ adjugate to $\cA$ (introduced in Subsection~\ref{subsec:adjugate}), i.e.\ $\cA^a =: (\cA\{j,k\})_{j,k=1}^n$. Let us express it via our notation $\cA[\fq,\fp]$.
It is easily seen that
\begin{equation} \label{eq:cAa.A.fpk.fpj}
 \cA^a = \(\cA\{j,k\}\)_{j,k=1}^n = \((-1)^{j+k}
 \cA[\fp_k, \fp_j]\)_{j,k=1}^n,\end{equation}
where
\begin{equation} \label{eq:fpk.def}
 \fp_k := (1, \ldots, k-1, k+1, \ldots n) = \wh{(k)} \in \fP_{n-1}.
\end{equation}
With this observation we can easily derive the following versions of Lemmas~\ref{lem:det.A+B} and~\ref{lem:det.A+B.C} for cofactors.
\begin{lemma} \label{lem:det.A+B.C.jk}
Let $\cA, \cB \in \bC^{n \times n}$ and let $j, k \in \oneton$. Then
\begin{align}
\label{eq:det.A+B.jk}
 (\cA+\cB)\{j,k\} &= \cA\{j,k\} + (-1)^{j+k} \sum_{m=1}^{n-1}
 \sum_{\genfrac{}{}{0pt}{2}{\fq, \fp \in \fP_m}
 {k \not \in \fq, \, j \not \in \fp}} (-1)^{\sigma(\fp) + \sigma(\fq)}
 \cA[\wh{\fq}, \wh{\fp}] \cdot \cB[\fq, \fp], \\
\label{eq:det.A+B.C.jk}
 (\cA + \cB \cC)\{j,k\} &= \cA\{j,k\} + (-1)^{j+k} \sum_{m=1}^{n-1}
 \sum_{\genfrac{}{}{0pt}{2}{\fq, \fp, \fr \in \fP_m}
 {k \not \in \fq, \, j \not \in \fp}} (-1)^{\sigma(\fp) + \sigma(\fq)}
 \cA[\wh{\fq}, \wh{\fp}] \cdot \cB[\fq,\fr] \cdot \cC[\fr,\fp].
\end{align}
\end{lemma}
\section{Transformation operators} \label{sec:transform.oper}
In this section we prove the existence of triangular transformation operators for the system
\begin{equation} \label{eq:LQy.def.transform}
 \cL(Q) y := -i B(x)^{-1} (y' + Q(x) y) = \l y,
 \qquad y = \col(y_1, \ldots, y_n), \quad x \in [0,\ell], \\
\end{equation}
expressing solution to the certain Cauchy problem for equation~\eqref{eq:LQy.def.transform} via the solution to the same Cauchy problem for the simplest equation
\begin{equation} \label{eq:L0.def.transform}
 \cL(0) y := -i B(x)^{-1} y' = \l y,
 \qquad y = \col(y_1, \ldots, y_n), \quad x \in [0,\ell].
\end{equation}
This result extends Theorem 1.2 from~\cite{Mal99} and Theorem 2.5 from~\cite{LunMal16JMAA} to the case of non-constant matrix $B(\cdot)$. Following the scheme of reasonings of Theorem 1.2 from~\cite{Mal99} we first establish the similarity of certain restrictions $\cL_0(Q)$ and $\cL_0(0)$ of the operators $\cL(Q)$ and $\cL(0)$, respectively.
\subsection{Similarity of operators \texorpdfstring{$\cL_0(Q)$}{cL0Q} and \texorpdfstring{$\cL_0(0)$}{cL00}}
First we introduce the main objects, the operators $\cL_0(Q)$ and $\cL_0(0)$.
To this end we need to change notation comparing to~\eqref{eq:Bx.def.intro}--\eqref{eq:Q=Qjk.def.intro} and work with a block-matrix decomposition for matrix functions $B(\cdot)$ and $Q(\cdot)$. Namely, let
\begin{align}
\label{eq:Bx.block.def}
 B(x) &= \diag(B_1(x), \ldots, B_r(x)) = B(x)^*, \qquad x \in [0, \ell], \\
\label{eq:Bkx.def}
 B_k(x) &= \beta_k(x) I_{n_k}, \quad x \in [0, \ell],
 \qquad \beta_k \in L^1([0,\ell]; \bR),
 \qquad k \in \onetor,
\end{align}
be a self-adjoint invertible diagonal summable matrix function, where $n_1 + \ldots + n_r = n$, and
\begin{equation} \label{eq:Q.block.L1}
 Q =: (Q_{jk})_{j,k=1}^r, \qquad
 Q_{jk} \in L^1([0,\ell]; \bC^{n_j \times n_k}), \qquad
 Q_{jj} \equiv 0, \qquad j, k \in \onetor,
\end{equation}
be a summable (generally non-self-adjoint) potential matrix with zero ``block diagonal'' with respect to the decomposition $\bC^n = \bC^{n_1} \oplus \ldots \oplus \bC^{n_r}$. It will be shown in Lemma~\ref{lem:gauge} that the case of arbitrary $Q$ can be reduced to it.
We deliberately reused existing notation for $\beta_k$ and $Q_{jk}$ to avoid introducing new notation. But notation~\eqref{eq:Bx.block.def}--\eqref{eq:Q.block.L1} will be used solely in this section, which should avoid any confusion.

Let us rewrite conditions~\eqref{eq:beta-}--\eqref{eq:betak-betaj<-eps.intro} on matrix function $B(\cdot)$ with a new notation~\eqref{eq:Bx.block.def}--\eqref{eq:Bkx.def} in mind. Namely, in this section we assume that for some $\theta > 0$ and $r_- \in \{0, 1, \ldots, r\}$ the following relations hold:
\begin{align}
\label{eq:betak.alpk.Linf.r}
 & \beta_k, \ 1/\beta_k \in L^{\infty}[0, \ell], \qquad
 \sign(\beta_k(\cdot)) \equiv \const \ne 0, \qquad
 k \in \onetor, \\
\label{eq:beta-+r}
 & \beta_1(x) < \ldots < \beta_{r_-}(x) < -\theta < 0
 < \theta < \beta_{r_-+1}(x) < \ldots < \beta_r(x),
 \qquad x \in [0, \ell], \\
\label{eq:beta.k+eps<beta.k+1.r}
 & \beta_k(x) + \theta < \beta_{k+1}(x), \quad x \in [0, \ell],
 \qquad k \in \oneto{r-1}.
\end{align}
Let us also set
\begin{equation} \label{eq:rho1r.def}
 b_j := \rho_j(\ell), \qquad \rho_j(x) := \int_0^x \beta_j(t) dt,
 \qquad j \in \onetor.
\end{equation}

Next we denote by $\cL_0(Q)$ the restriction of the maximal operator $\cL_{\max}(Q)$ in $\cH := L^2([0,\ell];\bC^n)$ generated by the expression
$\cL(Q)$ on the domain      
\begin{equation}
\dom \cL_0(Q) = \{f \in \dom \cL_{\max}(Q): \ f(0)=0\} \subset \wt{W}^{1,1}_0([0,\ell]; \bC^n).
\end{equation}
Here
$$
\wt{W}^{1,p}_0([0,\ell]; \bC^n) := \{f \in W^{1,p}([0,\ell]; \bC^n) : f(0) = 0\}.
$$
Moreover, if $Q \in L^2\bigl([0,\ell];\bC^{n\times n}\bigr)$, then 
\begin{equation}
\dom\cL_{\max}(Q)= W^{1,2}([0,\ell]; \bC^n) \qquad \text{and} \qquad \dom\cL_0(Q) = \wt{W}^{1,2}_0([0,\ell]; \bC^n).
\end{equation}
In particular, one has
$$
\cL_{0}(0) = B(x)^{-1}\otimes D_0, \quad \dom\cL_{0}(0) = \wt{W}^{1,2}_0([0,\ell]; \bC^n),
 \quad \text{where}\quad D_0 := -i\frac{d}{dx}\upharpoonright \wt{W}^{1,2}_0[0,\ell].
$$
 Note also that the operator $\cL_{0}(0)$
 is invertible and $\cL_{0}(0)^{-1} = B(x) \otimes (iJ)$ where $J$ is the Volterra integration
 operator, $J:\ f\to\int^x_0 f(t)\,dt$.

To state the main result of this subsection let us recall the  following definition.
\begin{definition}
Let $L_1$ and $L_2$ be closed densely defined operators in a Banach space $X$ with domains $\dom L_1$ and $\dom L_2$,
respectively. It is said that a bounded operator $T$ intertwines the operators $L_1$ and $L_2$ if:

a) $T$ maps $\dom L_1$ onto $\dom L_2$;

b) $L_2 T f = T L_1 f,\quad f \in \dom L_1$.

If in addition, $0 \in \rho(T)$, i.e.\ $T$ has a bounded inverse, then the operators $L_1$ and $L_2$ are called similar.
\end{definition}
To establish similarity of the operators $\cL_{0}(Q)$ and $\cL_{0}(0)$ for $Q \in L^1([0,\ell];\bC^{n \times n})$ we first establish this under additional smoothness assumptions of the potential $Q$ related to Lipshitz properties (see Subsection~\ref{subsec:Lip}).
\begin{proposition} \label{prop:similarity.LQ.L0}
Let matrix functions $B(\cdot)$ and $Q(\cdot)$ satisfy conditions~\eqref{eq:Bx.block.def}--\eqref{eq:beta.k+eps<beta.k+1.r}.
In particular, we assume that $Q_{jj}=0$ for $j \in \onetor$ and $|\beta_j(x) - \beta_k(x)| > \theta$ for a.e.\ $x \in [0,\ell]$ and $j \ne k$. Let also
\begin{equation} \label{eq:wtQjk.Lip}
 \wt{Q}_{jk} := \frac{Q_{jk}}{\beta_j - \beta_k} \in \Lip_1[0, \ell],
 \quad j \ne k, \qquad Q \in L^{\infty}([0,\ell]; \bC^{n \times n}).
\end{equation}
Then the operators $\cL_0(Q)$ and $\cL_0(0) = B(x)^{-1}\otimes D_0$ are similar in
$L^p([0,\ell];\bC^{n\times n})$, $p \in [1,\infty]$.
Moreover, there exists an $n \times n$ matrix kernel
\begin{equation} \label{eq:RLip}
 R \in \Lip_1(\Omega) \otimes \bC^{n \times n},
\end{equation}
where domain $\Omega$ is given by~\eqref{eq:Omega.def}, such that triangular Volterra type operator $I + \cR$,
\begin{equation} \label{eq:interwine.C1}
 (I + \cR)f := f(x) + \int^x_0 R(x,t) B(t) f(t)\,dt,
 \qquad f \in L^p([0,\ell];\bC^n),
\end{equation}
is bounded on $L^p([0,\ell];\bC^n)$, has a bounded inverse, and intertwines the operators $\cL_0(Q)$ and $\cL_0(0)$, i.e.
\begin{equation} \label{eq:LQ.I+R=I+R.L0.C1}
\cL_0(Q)(I + \cR)f = (I + \cR)\cL_0(0)f, \qquad f \in \dom\cL_{0}(0) = \wt{W}^{1,2}_0([0,\ell]; \bC^n).
\end{equation}
\end{proposition}
\begin{proof}
The proof will be divided into multiple steps.

\textbf{(i)} At this step we show that the intertwining property~\eqref{eq:LQ.I+R=I+R.L0.C1} is equivalent to a certain boundary value problem for the kernel $R(\cdot,\cdot)$ in the triangle $\Omega= \{0 \le t \le x \le \ell\}$. It is easily seen that
\begin{multline}
 \bigl(-i B \cL_0(Q)(I + \cR) f\bigr)(x) = \left[\frac{d}{dx}+Q\right]
 \bigl(f(x) + \int^x_0 R(x,t)B(t)f(t)\,dt\bigr) \\
 = f'(x) + Q(x)f(x) + \frac{d}{dx} \int^x_0 R(x,t) B(t)f(t)\,dt
 + Q(x) \int^x_0 R(x,t)B(t)f(t)\,dt \\
 = f'(x) + Q(x)f(x) + R(x,x)B(x)f(x) + \int^x_0 \Dx R(x,t)B(t)f(t)\,dt
 + Q(x) \int^x_0 R(x,t)B(t)f(t)\,dt.
\end{multline}
On the other hand, integrating by parts one derives
\begin{multline}
 \bigl(-i B (I+\cR) \cL_0(0)f\bigr)(x)
 = \bigl(B (I + \cR) (B^{-1} f')\bigr)(x) \\
 = f'(x) + B(x) \int^x_0 R(x,t)B(t)\cdot B^{-1}(t)f'(t)\,dt
 = f'(x) + B(x) \int^x_0 R(x,t) f'(t)\,dt \\
 = f'(x) + B(x) R(x,x)f(x) - B(x) R(x,0)f(0) - B(x) \int^x_0\(\Dt R(x,t)\)f(t)\,dt.
\end{multline}
Equating right hand sides of both equations and noting that $f \in \wt{W}^{1,2}_0([0,\ell]; \bC^n)$ is arbitrary satisfying $f(0)=0$, leads to the following boundary value problem for the matrix kernel $R(x,t):$
\begin{align}
\label{eq:DxRxtBt+BxDtRxt}
 \Dx R(x,t) B(t) + B(x) \Dt R(x,t) + Q(x) R(x,t)B(t) &= 0 , \\
\label{eq:RxxBx-BxRxx}
 R(x,x)B(x) - B(x) R(x,x) + Q(x) &= 0,
\end{align}
for a.e.\ $(x,t) \in \Omega$.
Let us write the matrix kernel $R(\cdot,\cdot)$ in the block-matrix form $R(x,t) = \(R_{jk}(x,t)\)_{j,k=1}^r$ with respect to the decomposition $\bC^n = \bC^{n_1} \oplus \ldots \oplus \bC^{n_r}$. Since $B(\cdot)$ is block-diagonal it follows that the problem~\eqref{eq:DxRxtBt+BxDtRxt}--\eqref{eq:RxxBx-BxRxx} splits into $r$ independent problems on columns $R_k(x,t) := (R_{jk}(x,t))_{j=1}^r$ of the matrix kernel $R(\cdot,\cdot)$. Fixing $k \in \onetor$, using the block-matrix representation of matrix functions $B(\cdot)$ and $Q(\cdot)$ and taking into account condition $Q_{jj} \equiv 0$, $j \in \onetor$, corresponding problem for the $k$-th column of the matrix kernel $R(\cdot,\cdot)$ takes the following form for a.e.\ $(x,t) \in \Omega$,
\begin{align} \label{eq:DxRjk+DtRjk}
 \Dx R_{jk}(x,t) + \frac{\beta_j(x)}{\beta_k(t)} \Dt R_{jk}(x,t)
 & = -\sum^r_{p=1} Q_{jp}(x) R_{pk}(x,t),
 \quad j \in \onetor, \\
\label{eq:Rjk.xx}
 R_{jk}(x,x) = \frac{Q_{jk}(x)}{\beta_j(x) - \beta_k(x)} & = \wt{Q}_{jk}(x),
 \qquad j \ne k, \quad j \in \onetor.
\end{align}
Emphasize that $\wt{Q}_{jk}$, $j \ne k$, is well-defined and summable on $[0,\ell]$ due to conditions~\eqref{eq:Q.block.L1}--\eqref{eq:beta.k+eps<beta.k+1.r}. It is clear now that to finish the proof it is sufficient to show that for each $k \in \onetor$ there exists a vector kernel $R_k(\cdot,\cdot)$ that satisfies the incomplete Cauchy problem~\eqref{eq:DxRjk+DtRjk}--\eqref{eq:Rjk.xx}.

\textbf{(ii)} To prepare for the next step, we need to extend functions $\beta_j(\cdot)$ and $\rho_j(\cdot)$ to be defined on $\bR$ and satisfy conditions~\eqref{eq:betak.alpk.Linf.r}--\eqref{eq:beta.k+eps<beta.k+1.r} (for each $x \in \bR$). By definition, $\beta_j \in L^{\infty}$ is a class of functions equivalent to a certain base function and conditions~\eqref{eq:betak.alpk.Linf.r}--\eqref{eq:beta.k+eps<beta.k+1.r} are valid only for a.e.\ $x \in [0,\ell]$. It is clear that for each $j \in \onetor$ we can select an appropriate representative from this class of equivalence (and call it $\beta_j$ for simplicity) such that the selected functions $\beta_j(\cdot)$ are defined for each $x \in [0,\ell]$ and satisfy conditions~\eqref{eq:betak.alpk.Linf.r}--\eqref{eq:beta.k+eps<beta.k+1.r} for each $x \in [0,\ell]$ (and not just for a.e.\ $x \in [0,\ell]$). With this remark in mind, it is clear that the following extension satisfies conditions~\eqref{eq:betak.alpk.Linf.r}--\eqref{eq:beta.k+eps<beta.k+1.r} for each $x \in \bR$,
\begin{equation} \label{eq:betaj.ext}
 \beta_j(x) := \begin{cases}
 \beta_j(0), & x < 0, \\
 \beta_j(x), & x \in [0,\ell], \\
 \beta_j(\ell), & x > \ell,
 \end{cases},
 \qquad \rho_j(x) := \int_0^x \beta_j(t) dt, \quad x \in \bR,
 \qquad j \in \onetor.
\end{equation}
It is clear from~\eqref{eq:betak.alpk.Linf.r} that $\rho_j(\bR) = \bR$ and $\rho_j(\cdot)$ is a strictly monotonous function, $j \in \onetor$. Hence there exists strictly monotonous inverse $\rho_j^{-1}$ that also maps $\bR$ onto $\bR$, $j \in \onetor$. Summarizing this we have,
\begin{equation} \label{eq:rhok.rho-1k}
 \rho_j(\bR) = \rho_j^{-1}(\bR) = \bR, \qquad\text{and}\qquad
 \rho_j \ \ \text{and} \ \ \rho_j^{-1} \quad\text{are strictly monotonous},
 \qquad j \in \onetor.
\end{equation}
More importantly, condition~\eqref{eq:beta.k+eps<beta.k+1.r} implies the same property for differences $\rho_j - \rho_k$,
\begin{equation} \label{eq:rho.j-k.rho-1.j-k}
 (\rho_j - \rho_k)(\bR) = (\rho_j - \rho_k)^{-1}(\bR) = \bR,
 \quad\text{and}\quad
 \rho_j - \rho_k, \ (\rho_j - \rho_k)^{-1} \ \ \text{are strictly monotonous},
 \quad j \ne k.
\end{equation}
Note also that conditions~\eqref{eq:betak.alpk.Linf.r}--\eqref{eq:beta.k+eps<beta.k+1.r} implie the following important Lipshitz property,
\begin{equation} \label{eq:rho.Lip}
 \rho_j, \ \rho_j^{-1}, \ \ \rho_j - \rho_k, \ (\rho_j - \rho_k)^{-1}
 \in \Lip_1(\bR), \qquad j \ne k, \quad j, k \in \onetor.
\end{equation}
Local Lipshitz property is implied directly by conditions~\eqref{eq:betak.alpk.Linf.r}--\eqref{eq:beta.k+eps<beta.k+1.r}. Global Lipshitz property holds because functions $\rho_j$ are linear outside of $[0,\ell]$.

We also need to extend matrix function $Q(\cdot)$ to be defined on $\bR$. For simplicity we will use the same notation for it and the same notation for $\wt{Q}_{jk}$, $j \ne k$, given by~\eqref{eq:Rjk.xx}. In the future we will need to work with a certain system of integral equations without the assumption $\wt{Q}_{jk} \in \Lip_1[0,\ell]$. Hence, we start by extending $\wt{Q}_{jk}$ (that outside of this Proposition might be only summable) in any way such that the following properties hold,
\begin{align}
\label{eq:wtQjk.bR}
 & \wt{Q}_{jk}(x) = 0, \qquad x \notin (-\delta,\ell+\delta),
 \qquad j \ne k, \\
\label{eq:wtQjk.Lp}
 & \wt{Q}_{jk} \in L^p(\bR) \quad\text{whenever}\quad
 \wt{Q}_{jk} \in L^p[0,\ell], \quad p \in [1, \infty], \qquad j \ne k, \\
\label{eq:wtQjj.bR}
 & \wt{Q}_{jk} \in \Lip_1(\bR) \quad\text{whenever}\quad
 \wt{Q}_{jk} \in \Lip_1[0,\ell], \qquad j \ne k,
\end{align}
with some $\delta > 0$. Then we naturally define $Q_{jk}(x)$ for $x \in \bR$ by formula~\eqref{eq:Rjk.xx} and also extend block-diagonal entries of the matrix function $Q(\cdot)$ to be zero,
\begin{align}
\label{eq:Qjk.bR}
 Q_{jk}(x) & := (\beta_j(x) - \beta_k(x)) \cdot \wt{Q}_{jk}(x),
 \qquad x \in \bR, \qquad j \ne k, \\
\label{eq:Qjj.bR}
 Q_{jj}(x) & := \wt{Q}_{jj}(x) := 0, \qquad x \in \bR, \qquad j \in \onetor.
\end{align}
It is clear from the construction and conditions~\eqref{eq:Q.block.L1}--\eqref{eq:beta.k+eps<beta.k+1.r} that
\begin{equation} \label{eq:Qx=0}
 Q(x) = 0, \quad x \notin (-\delta,\ell+\delta), \qquad\text{and}\qquad
 Q \in L^p(\bR) \quad\text{whenever}\quad Q \in L^p[0,\ell],
 \quad p \in [1, \infty].
\end{equation}

\textbf{(iii)} At this step assuming $k \in \onetor$ to be fixed, we consider the incomplete Cauchy problem~\eqref{eq:DxRjk+DtRjk}--\eqref{eq:Rjk.xx} and complete it up to a special Goursat problem in a certain extended domain $\Omega_k \supset \Omega$ assuming functions $\beta_j$, $\rho_j$ and $Q_{jp}$ to be extended on $\bR$ as constructed in step (ii). Along the way, we establish equivalent system of integral equations.

The corresponding characteristic system is ${\beta_j(x)}\,dx = {\beta_k(t)}\,dt$, $j \in \onetor$. Since $\rho_j'(x) = \beta_j(x)$ for a.e.\ $x \in [0,\ell]$, this characteristic system defines the system of characteristic curves of the equation~\eqref{eq:DxRjk+DtRjk}:
\begin{equation} \label{eq:characteristic_curves}
 \rho_j(x) - \rho_k(t) = c = \const, \qquad j \in \onetor.
\end{equation}
Condition~\eqref{eq:rhok.rho-1k} implies that we can present the characteristic at ``level'' $c \in \bR$ as
\begin{equation} \label{eq:Gamjk.def}
 \Gam_{jk}^c := \{(x, \gam_{jk}^c(x)): x \in \bR\},
 \qquad \gam_{jk}^c(x) := \rho_k^{-1}(\rho_j(x) - c), \quad x \in \bR,
 \qquad j \in \onetor.
\end{equation}
It is clear that $\gam_{jk}^c(\bR) = \bR$, and $\gam_{jk}^c$ is strictly monotonous, linear outside of a certain interval (that depends on $j$, $k$ and $c$) and globally Lipshitz just like functions $\rho_k^{-1}$ and $\rho_j$.

It can be easily seen that the incomplete Cauchy problem~\eqref{eq:DxRjk+DtRjk}--\eqref{eq:Rjk.xx} is not characteristic. To integrate it we need to look for solution of this problem in the extended domain. To this end, let the column $R_k \in \Lip_{1,\loc}(\bR^2; \bC^n)$ satisfy~\eqref{eq:DxRjk+DtRjk}--\eqref{eq:Rjk.xx} for all $x,t \in \bR$. Let us obtain some important properties of $R_k$.
It follows from the formula for derivative of the inverse function of the absolutely continuous function that
\begin{equation} \label{eq:rho-1j'}
 \rho_j'(x) = \beta_j(x), \qquad
 (\rho_j^{-1})'(x) = \frac{1}{\beta_j(\rho_j^{-1}(x))},
 \qquad\text{for a.e.}\ \ x \in \bR, \qquad j \in \onetor.
\end{equation}
Combining~\eqref{eq:Gamjk.def} and~\eqref{eq:rho-1j'} we arrive at
\begin{equation} \label{eq:gamjkc'}
 (\gam_{jk}^c)'(u) = \frac{\beta_j(u)}{\beta_k(\gam_{jk}^c(u))},
 \qquad\text{for a.e.}\ \ u \in \bR, \qquad j \in \onetor.
\end{equation}
Recall that as per~\eqref{eq:D1f.D2f},
\begin{equation} \label{eq:D1R.D2R}
 D_1 R_{jk}(x,t) = \Dx R_{jk}(x,t), \qquad
 D_2 R_{jk}(x,t) = \Dt R_{jk}(x,t),
 \qquad x,t \in \bR, \quad j \in \onetor.
\end{equation}
With this notation in mind we can rewrite~\eqref{eq:DxRjk+DtRjk} as follows
\begin{equation} \label{eq:D1Rjk+D2Rjk}
 D_1 R_{jk}(u,v) + \frac{\beta_j(u)}{\beta_k(v)} D_2 R_{jk}(u,v)
 = -\sum^r_{p=1} Q_{jp}(u) R_{pk}(u,v),
 \qquad\text{for a.e.}\ \ u, v \in \bR, \quad j \in \onetor.
\end{equation}
Using standard rules for differentiating functions of two variables, and
combining~\eqref{eq:gamjkc'} and~\eqref{eq:D1Rjk+D2Rjk}, we obtain for $j \in \onetor$,
\begin{align}
\nonumber
 \frac{d}{d u} \bigl[R_{jk}(u, \gam_{jk}^c(u))\bigr]
 & = D_1 R_{jk}(u, \gam_{jk}^c(u)) + (\gam_{jk}^c)'(u)
 D_2 R_{jk}(u, \gam_{jk}^c(u)) \\
\nonumber
 & = D_1 R_{jk}(u, \gam_{jk}^c(u))
 + \frac{\beta_j(u)}{\beta_k(\gam_{jk}^c(u))}
 D_2 R_{jk}(u, \gam_{jk}^c(u)) \\
\label{eq:ddu.Rjku.gam}
 & = -\sum^r_{p=1} Q_{jp}(u) R_{pk}(u,\gam_{jk}^c(u)),
 \qquad\text{for}\ \ c \in \bR \quad\text{and for a.e.}\ \ u \in \bR.
\end{align}
It is clear that the function $R_{jk}(\cdot, \gam_{jk}^c(\cdot))$ is locally Lipshitz. Hence, integrating formula~\eqref{eq:ddu.Rjku.gam} from $a$ to $x$, we arrive at
\begin{equation} \label{eq:Rjx-Rjky}
 R_{jk}(x, \gam_{jk}^c(x)) - R_{jk}(a, \gam_{jk}^c(a)) =
 - \int_a^x \sum^r_{p=1} Q_{jp}(u) R_{pk}(u,\gam_{jk}^c(u)) du,
 \quad x,a,c \in \bR, \quad j \in \onetor.
\end{equation}
Let us fix $j \in \onetor$ and $-\infty < t \le x < \infty$. Consider the characteristic curve $\Gam_{jk}^c$ passing through the point $(x,t)$. It is clear from definition~\eqref{eq:Gamjk.def} of $\gam_{jk}^c(\cdot)$ that
\begin{equation} \label{eq:gamjkx=t}
 \gam_{jk}^c(x) = t \qquad\text{and}\qquad c = \rho_j(x) - \rho_k(t).
\end{equation}
Let us also set for brevity for $j \in \onetor$ and $x,t,u \in \bR$
\begin{equation} \label{eq:gamjkxt.def}
 \gam_{jk}^{x,t}(u) := \gam_{jk}^{c}(u)
 = \rho_k^{-1}(\rho_j(u) - \rho_j(x) + \rho_k(t)), \qquad x, t, u \in \bR,
 \qquad \gam_{jk}^{\cdot,\cdot}(\cdot) \in \Lip_1(\bR^3),
\end{equation}
where the last inclusion follows from the property~\eqref{eq:rho.Lip}.

First let $j \ne k$. Then there exists $a = a_{jk}(x,t) \in \bR$ such that $\gam_{jk}^{x,t}(a) = a$. Indeed, this equation is equivalent to $\rho_j(a) - \rho_j(x) + \rho_k(t) = \rho_k(a)$ and hence condition~\eqref{eq:rho.j-k.rho-1.j-k} implies that
\begin{equation} \label{eq:ajkdef}
 a_{jk}(x,t) :=
 (\rho_j - \rho_k)^{-1}(\rho_j(x) - \rho_k(t))
 = \Bigl.(\rho_j - \rho_k)^{-1}(u)\Bigr|_{u = \rho_j(x) - \rho_k(t)},
 \qquad a_{jk} \in \Lip_1(\bR^2).
\end{equation}
Setting $a = a_{jk}(x,t)$ in~\eqref{eq:Rjx-Rjky} and taking into account~\eqref{eq:gamjkx=t},~\eqref{eq:ajkdef} and~\eqref{eq:Rjk.xx}, we arrive at
\begin{multline} \label{eq:int=Rjxxt-wtQ}
 - \int_a^x \sum^r_{p=1} Q_{jp}(u) R_{pk}(u,\gam_{jk}^{x,t}(u)) du
 = R_{jk}(x, \gam_{jk}^{x,t}(x)) - R_{jk}(a, \gam_{jk}^{x,t}(a)) \\
 = R_{jk}(x, t) - R_{jk}(a, a) = R_{jk}(x, t) - \wt{Q}_{jk}(a),
 \quad a = a_{jk}(x,t).
\end{multline}
or
\begin{equation} \label{eq:Rjkxt=wtQ+int}
 R_{jk}(x, t) = \wt{Q}_{jk}(a_{jk}(x,t))
 - \int_{a_{jk}(x,t)}^x \sum^r_{p=1} Q_{jp}(u) R_{pk}(u,\gam_{jk}^{x,t}(u)) du,
 \qquad t \le x, \quad j \ne k.
\end{equation}

Now we are ready to define the domain $\Omega_k$ on which we will set the Goursat problem. For convenience, $[u,v]$ and $[v,u]$ will denote the same segment of the real line,
\begin{equation} \label{eq:uv=vu}
 [u,v] := [v,u] = \{t \in \bR: v \le t \le u\}, \qquad v \le u.
\end{equation}
Let us also define the part of the characteristic curve $\Gam_{jk}^c$ that starts at the point $(x,t)$ of the triangle $\Omega$ and ends on the diagonal $\bfD := \{(u,v) \in \bR^2: u = v\}$ of $\bR^2$,
\begin{equation} \label{eq:Gamjkct.def}
 \Gam_{jk}^{x,t} := \left\{\(u, \gam_{jk}^{x,t}(u)\) :
 \ u \in [x, a_{jk}(x,t)] \right\},
 \qquad (x,t) \in \Omega, \qquad j \ne k.
\end{equation}
With this notation in mind, we formally define
\begin{equation} \label{eq:Omegak'.def}
 \Omega_k := \bigcup_{j \ne k} \Omega_{jk}, \qquad
 \Omega_{jk} := \bigcup_{(x,t) \in \Omega} \Gam_{jk}^{x,t}, \qquad j \ne k.
\end{equation}

Let us fix $j \ne k$. It is clear from the continuity and monotonicity of $\gam_{jk}^{x,t}(\cdot)$ and $(\rho_j - \rho_k)^{-1}$, and general geometrical reasoning that we can only have three cases:
\begin{itemize}
\item $\Omega_{jk} = \Omega$ (when the characteristic curve segments $\Gam_{jk}^{x,t}$ do not ``go outside'' of the triangle $\Omega$);
\item $\Omega_{jk} = \{(u,v) : \ u \in [a_{jk}(\ell,0), \ell], \ v \in [\gam_{jk}^{\ell,0}(u), u]\}$ (when $\Gam_{jk}^{x,t}$ that ``go outside'' of the triangle $\Omega$ intersect the diagonal $\bfD$ at points with negative coordinate $u$);
\item $\Omega_{jk} = \Omega \cup \{(u,v) : \ u \in (\ell, a_{jk}(\ell,0)], \ v \in [\gam_{jk}^{\ell,0}(u), u]\}$ (when $\Gam_{jk}^{x,t}$ that ``go outside'' of $\Omega$ intersect the diagonal $\bfD$ at points with positive coordinate $u$).
\end{itemize}
It is clear that we can unify all three cases with a single representation of the following form,
\begin{equation} \label{eq:Omegajk.canon}
 \Omega_{jk} = \{(u,v) : \ u \in [a_{jk}^-, a_{jk}^+],
 \ \ v \in [\gam_{jk}^-(u), u]\},
\end{equation}
where numbers $a_{jk}^- \le 0 < \ell \le a_{jk}^+$ and function $\gam_{jk}^-(\cdot) \le 0$ are defined as follows in each of the three cases:
\begin{flalign} \label{eq:gamjk-}
 \quad \begin{array}{llll}
 \bullet & a_{jk}^- = 0, & a_{jk}^+ = \ell,
 & \qquad \gam_{jk}^-(u) = 0, \quad u \in [0,\ell], \\
 \bullet & a_{jk}^- = a_{jk}(\ell,0), & a_{jk}^+ = \ell,
 & \qquad \gam_{jk}^-(u) = \gam_{jk}^{\ell,0}(u), \quad u \in [0,\ell], \\
 \bullet & a_{jk}^- = 0, & a_{jk}^+ = a_{jk}(\ell,0),
 & \qquad \gam_{jk}^-(u) = \begin{cases}
 0, \quad & u \in [0,\ell], \\
 \gam_{jk}^{\ell,0}(u), \quad & u \in (\ell,a_{jk}(\ell,0)], \\
 \end{cases} \\
 \end{array}&&
\end{flalign}
(see Remark~\ref{rem:Omegajk} for more details).
It now follows from~\eqref{eq:Omegajk.canon}--\eqref{eq:gamjk-} that
\begin{align}
\label{eq:Omegak.canon}
 & \Omega_k = \{(u,v) : \ u \in [a_k^-, a_k^+],
 \ \ v \in [\gam_k^-(u), u]\}, \quad\text{where} \\
\label{eq:ak-+.def}
 & a_k^- := \min\{a_{jk}^- : j \ne k\} \le 0, \qquad
 a_k^+ := \max\{a_{jk}^+ : j \ne k\} \ge \ell, \\
 & \gam_k^{-}(u) := \min\{\gam_{jk}^-(u) :
 \ j \ne k \ \ \text{such that} \ \ u \in [a_{jk}^-, a_{jk}^+]\}, \qquad u \in [a_k^-, a_k^+], \\
 & \Gam_k^- := \{(u,\gam_k^{-}(u)) : \ u \in [a_k^-, a_k^+]\}
 \quad\text{is the ``lower boundary'' of the domain}\ \ \Omega_k.
\end{align}
Properties of $\gam_{jk}^c(\cdot)$ outlined after formula~\eqref{eq:Gamjk.def} and the fact, that the minimum of finite number of Lipshitz non-decreasing functions preserves these properties, imply that
\begin{equation} \label{eq:gamjk-.Lip}
 \gam_k^- \in \Lip_1[a_{k}^-, a_{k}^+],
 \qquad \gam_{jk}^- \in \Lip_1[a_{jk}^-, a_{jk}^+] \qquad\text{and}\qquad
 \gam_k^-, \gam_{jk}^- \quad\text{are non-decreasing}, \quad j \ne k.
\end{equation}
This in turn imply that $\Omega_k$ has a Lipshitz boundary. It is also simply connected, closed and bounded.

In order to formulate Goursat problem, observe that the definition of $\Omega_{jk}$ implies that the characteristic curve segment $\Gam_{jk}^{x,t}$, defined in~\eqref{eq:Gamjkct.def} and considered for points $(x,t) \in \Omega_{jk}$, does not go outside of $\Omega_{jk}$. The same is valid for the union $\Omega_k$ (though this needs a bit more considerations involving ``ordering'' of characteristic curves implied by~\eqref{eq:beta-+r}).
Hence we can restrict kernels $R_{jk}(x,t)$ to be only defined on $\Omega_k$ and the system of equations~\eqref{eq:Rjkxt=wtQ+int} will still be valid,
\begin{equation} \label{eq:Rjkxt=wtQ+int.Omegak}
 R_{jk}(x, t) = \wt{Q}_{jk}(a_{jk}(x,t))
 - \int_{a_{jk}(x,t)}^x \sum^r_{p=1} Q_{jp}(u) R_{pk}(u,\gam_{jk}^{x,t}(u)) du,
 \qquad (x,t) \in \Omega_k, \quad j \ne k.
\end{equation}

Let us now go back to our raw equation~\eqref{eq:Rjx-Rjky} and handle the diagonal entry $R_{kk}(\cdot,\cdot)$ (recall that $k \in \onetor$ is fixed). We don't have any initial data for it yet and are free to impose any appropriate initial condition. To this end note that the ``characteristic function'' $\gam_{kk}^{x,t}$ is strictly increasing for any given point $(x,t)$,
\begin{equation} \label{eq:gamkk.inc}
 \gam_{kk}^{x,t}(u) = \rho_k^{-1}(\rho_k(u) - \rho_k(x) + \rho_k(t)), \quad u \in \bR,
 \qquad \gam_{kk}^{x,t} \quad\text{is increasing}, \quad x,t \in \bR.
\end{equation}
Note also, that the characteristic curve $\Gam_{kk}^c$ never intersects the diagonal $\bfD$ of $\bR^2$ if $c \ne 0$, while $\Gam_{kk}^0 = \bfD$. With account of these observations and noting that $\Gam_k^-$ (the ``lower boundary'' of $\Omega_k$) intersects the diagonal $\bfD$ of $\bR^2$ at the point $a_k^-$,
we see that the characteristic curve $\Gam_{kk}^{x,t}$, that passes through any point $(x,t)$ of the domain $\Omega_k$, intersects with $\Gam_k^-$. Namely, there exists $a_{kk} = a_{kk}(x,t) \in [a_k^-, a_k^+]$ such that
\begin{equation} \label{eq:gamkkxt}
 \gam_{kk}^{x,t}(a_{kk}(x,t)) = \gam_k^-(a_{kk}(x,t)), \qquad (x,t) \in \Omega_k.
\end{equation}
Let us show that $a_{kk} \in \Lip(\Omega_k)$. Let $(x,t) \in \Omega_k$ and $j \ne k$ be fixed and let $a_{kk,j}(x,t)$ be the intersection of $\Gam_{kk}^{x,t}$ with the ``lower boundary'' of $\Omega_{jk}$, i.e.
\begin{equation} \label{eq:gamjkxt}
 \gam_{kk}^{x,t}(a_{kk,j}(x,t)) = \gam_{jk}^-(a_{kk,j}(x,t)), \qquad (x,t) \in \Omega_k, \quad j \ne k.
\end{equation}
In each of the three cases in~\eqref{eq:gamjk-} we can find explicit form of $a_{kk,j}(x,t)$. For the first two cases we have,
\begin{flalign} \label{eq:gamjk-forkk}
 \quad \begin{array}{llll}
 \bullet & \gam_{jk}^-(u) = 0, \quad u \in [0,\ell],
 & a_{kk,j}(x,t) = \rho_k^{-1}(\rho_k(x) - \rho_k(t)), \\
 \bullet & \gam_{jk}^-(u) = \gam_{jk}^{\ell,0}(u), \quad u \in [0,\ell], &
 a_{kk,j}(x,t) = (\rho_j - \rho_k)^{-1}(\rho_j(\ell) - \rho_k(x) + \rho_k(t)),\\
 \end{array}&&
\end{flalign}
while in the third $a_{kk,j}(x,t)$ is in a way ``a union'' of two cases. Here is the detailed proof for the second case for posterity,
\begin{align*}
 \gam_{kk}^{x,t}(a) &= \gam_{jk}^{\ell,0}(a) \quad \Leftrightarrow \\
 \rho_k^{-1}(\rho_k(a) - \rho_k(x) + \rho_k(t))
 &= \rho_k^{-1}(\rho_j(a) - \rho_j(\ell)) \quad \Leftrightarrow \\
 \rho_j(\ell) - \rho_k(x) + \rho_k(t)
 &= \rho_j(a) - \rho_k(a) \quad \Leftrightarrow \\
 a_{kk,j}(x,t) & = (\rho_j - \rho_k)^{-1}(\rho_j(\ell) - \rho_k(x) + \rho_k(t)).
\end{align*}
Note that for some points $(x,t) \in \Omega_k$ the found point $a_{kk,j}(x,t)$ intersects the curve $\Gam_{jk}^- := \{(u,\gam_{jk}^{-}(u)) : \ u \in \bR\}$ outside of $\Omega_{jk}$. Despite that, it can be shown that
\begin{equation} \label{eq:akk.def}
 a_{kk}(x,t) = \min\{a_{kk,j}(x,t) : j \ne k\}, \qquad (x,t) \in \Omega_k.
\end{equation}
Formulas~\eqref{eq:gamjk-forkk},~\eqref{eq:akk.def} and property~\eqref{eq:rho.Lip} now easily imply that
\begin{equation} \label{eq:akk.Lip}
 a_{kk}, a_{kk,j} \in \Lip(\Omega_k), \qquad j \ne k.
\end{equation}
With this preparation in mind, we can impose the following initial condition on $R_{kk}(\cdot,\cdot)$,
\begin{equation} \label{eq:Rkk.init}
 R_{kk}(x, \gam_k^-(x)) = 0, \qquad x \in [a_k^-, a_k^+].
\end{equation}
Setting $a = a_{kk}(x,t)$ in~\eqref{eq:Rjx-Rjky} and taking into account~\eqref{eq:Rkk.init} we get similar to~\eqref{eq:int=Rjxxt-wtQ},
\begin{equation} \label{eq:Rkkxt=int}
 R_{kk}(x, t) = - \int_{a_{kk}(x,t)}^x \sum^r_{p=1}
 Q_{kp}(u) R_{pk}(u,\gam_{kk}^{x,t}(u)) du, \qquad (x,t) \in \Omega_k.
\end{equation}
With account of convention $\wt{Q}_{kk} = 0$ (see~\eqref{eq:Qjj.bR}), we can combine~\eqref{eq:Rjkxt=wtQ+int.Omegak} and~\eqref{eq:Rkkxt=int} into a single formula for all $r$ equations,
\begin{equation} \label{eq:Rjkxt=wtQ+int.canon}
 R_{jk}(x, t) = \wt{Q}_{jk}(a_{jk}(x,t))
 - \int_{a_{jk}(x,t)}^x \sum^r_{p=1} Q_{jp}(u) R_{pk}(u,\gam_{jk}^{x,t}(u)) du,
 \qquad (x,t) \in \Omega_k, \quad j \in \onetor,
\end{equation}
where, as before, $k \in \onetor$ is fixed. It is clear from considerations in formulas~\eqref{eq:ddu.Rjku.gam}--\eqref{eq:Rjx-Rjky} that the system of integral equations~\eqref{eq:Rjkxt=wtQ+int.canon} is equivalent to the following Goursat problem whenever $R_k \in \Lip_1(\Omega_k)$,
\begin{align} \label{eq:Goursat}
 \Dx R_{jk}(x,t) + \frac{\beta_j(x)}{\beta_k(t)} \Dt R_{jk}(x,t)
 & = -\sum^r_{p=1} Q_{jp}(x) R_{pk}(x,t),
 \quad (x,t) \in \Omega_k, \quad j \in \onetor. \\
\label{eq:Rjk.xx.Omega}
 R_{jk}(x,x) & = \wt{Q}_{jk}(x), \qquad x \in [a_k^-, a_k^+],
 \qquad j \ne k, \quad j \in \onetor. \\
\label{eq:Rkk.x.gam}
 R_{kk}(x, \gam_k^-(x)) & = 0, \qquad x \in [a_k^-, a_k^+].
\end{align}
Moreover, it is obvious that any vector kernel solution $R_k(\cdot, \cdot)$ of this Goursat problem satisfies the desired incomplete Cauchy problem~\eqref{eq:DxRjk+DtRjk}--\eqref{eq:Rjk.xx}. Therefore, to finish the proof it is sufficient to show solvability of the system of integral equations~\eqref{eq:Rjkxt=wtQ+int.canon} in $\Lip_1(\Omega_k)$.

\textbf{(iv)} At this step, again assuming $k \in \onetor$ to be fixed, we apply the method of successive approximation to prove the existence of solution to the system~\eqref{eq:Rjkxt=wtQ+int.canon} in $C(\Omega_k)$. First, note that if the $k$-th block column $Q_k = (Q_{jk})_{j=1}^r$ of the matrix $Q$ is zero, then $R_k \equiv 0$ is a valid solution of the system~\eqref{eq:Rjkxt=wtQ+int.canon}. Going forward we assume that $Q_k \ne 0$. To this end we set
\begin{equation} \label{eq:sucsev_approx_R_{jk}^0}
 R_{jk}^{(0)}(x,t) := 0, \qquad (x,t) \in \Omega_k,
 \qquad j \in \onetor.
\end{equation}
Assuming functions $R_{jk}^{(m-1)}$ to be defined for a given $m \in \bN$, we set for $(x,t) \in \Omega_k$,
\begin{align}
\label{eq:Rjkm.def}
 R_{jk}^{(m)}(x,t) & := \wt Q_{jk}(a_{jk}(x,t)) - \int_{a_{jk}(x,t)}^x
 \sum^r_{p=1} Q_{jp}(u) R_{pk}^{(m-1)}(u, \gam_{jk}^{x,t}(u)) \,du,
 \quad j \ne k, \quad j \in \onetor. \\
\label{eq:Rkkm.def}
 R_{kk}^{(m)}(x,t) & := - \int_{a_{kk}(x,t)}^x \sum^r_{p=1}
 Q_{kp}(u) R_{pk}^{(m)}(u, \gam_{kk}^{x,t}(u)) \,du,
\end{align}
To explain the correctness of this definition let us mention that since $Q_{kk}(\cdot) =0$ by the assumption, the entries $R_{kk}^{(m)}$ of the vector function $\bigl(R_{jk}^{(m)}\bigr)_{j=1}^r$ are expressed by means of equalities~\eqref{eq:Rkkm.def} via the entries $R_{jk}^{(m)}$, $j \ne k$, defined on the previous step by equalities~\eqref{eq:Rjkm.def}. Note in particular, that
\begin{equation} \label{eq:sucsev_approx_R_{jk}^1}
 R_{jk}^{(1)}(x,t) = \wt{Q}_{jk}(a_{jk}(x,t)),
 \qquad (x,t) \in \Omega_k,
 \qquad j \ne k, \quad j \in \onetor,
\end{equation}

Note that since
\begin{equation} \label{eq:all.Lip}
 \wt{Q}_{jk} \in \Lip_1(\bR), \quad
 a_{jk} \in \Lip_1(\bR^2), \quad
 Q_{jp} \in L^{\infty}(\bR), \quad
 \gam_{jk}^{\cdot,\cdot}(\cdot) \in \Lip_1(\bR^3), \qquad j, p \in \onetor,
\end{equation}
then
\begin{equation} \label{eq:Rjkm.Lip}
 R_{jk}^{(m)} \in \Lip_1(\Omega_k), \qquad m \in \bN_0, \quad j \in \onetor.
\end{equation}
First we show that for any $j \in \onetor$ the (uniform) limit $\lim_{m\to\infty} R_{jk}^{(m)} =: R_{jk}$ exists in $C(\Omega_k)$ and defines a solution to the system~\eqref{eq:Rjkxt=wtQ+int.canon}. Clearly, this convergence
is equivalent to the convergence in $C(\Omega_k)$ of the series
\begin{equation} \label{eq:series_of_sucsev_approx}
 R_{jk}^{(0)} + \sum_{m=1}^{\infty} \left[R_{jk}^{(m)} - R_{jk}^{(m-1)} \right]
\end{equation}
Let us prove by induction on $N$ that
\begin{align} \label{eq:induction_hypothesis}
 \abs{R^{(N+1)}_{jk}(x,t) - R^{(N)}_{jk}(x,t)} \le
 C_{jk} \tau_k^N \frac{|\rho_k(x) - \rho_k(t)|^N}{N!},
 \quad (x,t) \in \Omega_k, \quad j \in \onetor, \quad N \in \bN_0,
\end{align}
with constants $C_{1k}, \ldots, C_{rk} \ge 0$ and $\tau_k > 0$ given by
\begin{align}
\label{eq:Cjk.wtqk.def}
 C_{jk} & := \wt{q}_k := \max\{\|\wt Q_{pk}\|_{\infty} : p \ne k\} > 0,
 \qquad j \ne k, \quad j \in \onetor, \\
\label{eq:Ckk.qk.def}
 C_{kk} & := \wt{q}_k \cdot q_k \cdot (a_k^+ - a_k^-), \qquad
 q_k := \sum^r_{p=1} \|Q_{kp}\|_{\infty} \ge 0, \\
\label{eq:tauk.def}
 \tau_k & := \theta^{-1} \max\Bigl\{q_k \|Q_{jk}\|_{\infty} + \sum^r_{p=1} \|Q_{jp}\|_{\infty} : j \ne k \Bigr\} > 0.
\end{align}
where $[a_k^-, a_k^+]$ is the projection of $\Omega_k$ on $\bR$ and is given by~\eqref{eq:ak-+.def} and parameter $\theta$ is from the condition~\eqref{eq:beta.k+eps<beta.k+1.r}. We also set for brevity, $\|f\|_{\infty} := \|f\|_{L^{\infty}(\Omega_k)}$ for any $f : \Omega_k \to \bC$. Note that $\wt{q}_k > 0$ and $\tau_k > 0$ because of the current assumption that the $k$-th block column $Q_k = (Q_{jk})_{j=1}^r$ of the matrix $Q$ is non-zero. In addition, $C_{kk} = 0 = q_k$ whenever the block row $(Q_{kp})_{p=1}^r$ of $Q$ is zero. In this case r.h.s.\ of equation~\eqref{eq:Rkkm.def} vanishes and $R_{kk}^{(m)}$ is necessarily zero for all $m \in \bN_0$.

Let $N=0$ and $j \ne k$. It follows from~\eqref{eq:sucsev_approx_R_{jk}^0},~\eqref{eq:sucsev_approx_R_{jk}^1} and~\eqref{eq:Cjk.wtqk.def} that
\begin{equation} \label{eq:Rjk1-Rjk0}
 \abs{R_{jk}^{(1)}(x,t) - R_{jk}^{(0)}(x,t)} = \abs{\wt{Q}_{jk}(a_{jk}(x,t))}
 \le \|\wt Q_{jk}\|_{\infty} \le \wt{q}_k = C_{jk},
 \qquad (x,t) \in \Omega_k,
 \qquad j \ne k,
\end{equation}
for $j \in \onetor$. Hence~\eqref{eq:induction_hypothesis} is valid for $N=0$ and $j \ne k$.

Assume that the estimate~\eqref{eq:induction_hypothesis} is valid for $N=m \in \bN_0$ and $j \ne k$ and let us prove it for the same $N=m$ and $j=k$. First observe that relation~\eqref{eq:gamjkxt.def} implies that
\begin{equation} \label{eq:rhok.u-rhok.gamkk}
 \rho_k(u) - \rho_k(\gam_{kk}^{x,t}(u)) = \rho_k(x) - \rho_k(t),
 \qquad x,t,u \in \bR.
\end{equation}
Taking into account relations~\eqref{eq:rhok.u-rhok.gamkk},~\eqref{eq:Cjk.wtqk.def}--\eqref{eq:Ckk.qk.def} and the fact that $Q_{kk} \equiv 0$, we subtract two equations~\eqref{eq:Rkkm.def} with $m$ and $m+1$, respectively, and insert the estimate~\eqref{eq:induction_hypothesis}, valid by induction hypothesis for $N=m$ and $p \ne k$, into this difference,
\begin{align}
\nonumber
 |R^{(m+1)}_{kk}(x,t) - R^{(m)}_{kk}(x,t)|
 & \le \sum^r_{p=1} \abs{\int_{a_{kk}(x,t)}^x |Q_{kp}(u)| \cdot
 \abs{R_{pk}^{(m+1)}(u, \gam_{kk}^{x,t}(u))
 - R_{pk}^{(m)}(u, \gam_{kk}^{x,t}(u))}\,du} \\
\nonumber
 & \le \tau_k^m \wt{q}_k \sum^r_{p=1} \|Q_{kp}\|_{\infty}
 \left|\int_{a_{kk}(x,t)}^x
 \frac{|\rho_k(u) - \rho_k(\gam_{kk}^{x,t}(u))|^{m}}{m!} \,du \right| \\
\nonumber
 & \le \wt{q}_k q_k \tau_k^m
 \left|\int_{a_{kk}(x,t)}^x
 \frac{|\rho_k(x) - \rho_k(t)|^{m}}{m!} \,du \right| \\
\label{eq:Rkk.m+1-Rkkm}
 & \le \wt{q}_k q_k \tau_k^m (a_k^+ - a_k^-)
 \frac{|\rho_k(x) - \rho_k(t)|^{m}}{m!}
 = C_{kk} \tau_k^m \frac{|\rho_k(x) - \rho_k(t)|^{m}}{m!},
\end{align}
which yields the desired relation~\eqref{eq:induction_hypothesis} for $N=m$ and $j=k$.

Let $m \in \bN_0$ and assume that the estimate~\eqref{eq:induction_hypothesis} is valid for $N = m$ and $j \in \onetor$. Let us prove it for $N = m+1$ and $j \ne k$. Observe that relation~\eqref{eq:gamjkxt.def} implies that
\begin{equation} \label{eq:rhok.gamjkxt}
 \rho_k(\gam_{jk}^{x,t}(u)) = \rho_j(u) - \rho_j(x) + \rho_k(t),
 \qquad x, t, u \in \bR.
\end{equation}
Let $j \ne k$ and $(x,t) \in \Omega_k$ be fixed. With account of identity~\eqref{eq:rhok.gamjkxt} and taking difference of~\eqref{eq:Rjkm.def} for $m+1$ and $m$, the induction hypothesis implies,
\begin{align}
\nonumber
 \abs{R^{(m+1)}_{jk}(x,t) - R^{(m)}_{jk}(x,t)}
 & \le \sum^r_{p=1} \abs{\int_{a_{jk}(x,t)}^x |Q_{jp}(u)| \cdot
 \abs{R_{pk}^{(m)}(u, \gam_{jk}^{x,t}(u)) -
 R_{pk}^{(m-1)}(u, \gam_{jk}^{x,t}(u))}\,du} \\
\nonumber
 & \le \tau_k^m \sum^r_{p=1} C_{pk} \|Q_{jp}\|_{\infty}
 \abs{\int_{a_{jk}(x,t)}^x\frac{\abs{\rho_k(u) -
 \rho_k(\gam_{jk}^{x,t}(u))}^m}{m!}\,du} \\
\label{eq:Rm+1-Rm.jk}
 & = C_{0jk} \tau_k^m \abs{\int_{a_{jk}(x,t)}^x\frac{\abs{\rho_k(u) -
 \rho_j(u) + \rho_j(x) - \rho_k(t)}^m}{m!}\,du},
\end{align}
where definitions~\eqref{eq:Cjk.wtqk.def}--\eqref{eq:tauk.def} imply
\begin{equation} \label{eq:C0jk.def}
 C_{0jk} := \sum^r_{p=1} C_{pk} \|Q_{jp}\|_{\infty}
 = \wt{q}_k \cdot \Bigl(q_k \|Q_{jk}\|_{\infty} +
 \sum^r_{\genfrac{}{}{0pt}{2}{p=1}{p \ne k}} \|Q_{jp}\|_{\infty}\Bigr),
 \qquad C_{0jk} \theta^{-1} \le C_{jk} \tau_k.
\end{equation}
Since $j \ne k$, then relations~\eqref{eq:rho.j-k.rho-1.j-k}--\eqref{eq:rho.Lip} imply that the function
$$
f(\cdot) := \rho_k(\cdot) - \rho_j(\cdot) + \rho_j(x) - \rho_k(t)
$$
is strictly monotonous and Lipshitz on $\bR$ (recall, that $j,k,x,t$ are fixed, hence we didn't add them into notation of $f$). Moreover, from definition~\eqref{eq:ajkdef} of $a_{jk}(x,t)$ if follows that $f(a_{jk}(x,t)) = 0$, while clearly $f(x) = \rho_k(x) - \rho_k(t)$. Hence, monotonicity of $f(\cdot)$ implies that
\begin{equation} \label{eq:fjkxt+-}
 \text{either}\quad f(u) \ge 0, \quad u \in [a_{jk}(x,t), x]
 \qquad\text{or}\qquad
 f(u) \le 0, \quad u \in [a_{jk}(x,t), x].
\end{equation}
Let $g := f^{-1}$ be the function inverse to $f$, which exists due to the above observations. It follows from the standard formula for derivate of the inverse function that
\begin{equation} \label{eq:g'}
 g'(v) = \frac{1}{\beta_k(f(v)) - \beta_j(f(v))},
 \qquad |g'(v)| \le \theta^{-1}, \qquad v \in \bR,
\end{equation}
where inequality is implied by conditions~\eqref{eq:beta-+r}--\eqref{eq:beta.k+eps<beta.k+1.r} and construction~\eqref{eq:betaj.ext}. Therefore, making a change of variable $u = g(v)$ in the integral~\eqref{eq:Rm+1-Rm.jk} we obtain
\begin{multline} \label{eq:int.fum}
 \abs{\int_{a_{jk}(x,t)}^x\frac{\abs{\rho_k(u) -
 \rho_j(u) + \rho_j(x) - \rho_k(t)}^m}{m!}\,du} \\
 = \abs{\int_{a_{jk}(x,t)}^x\frac{f(u)^m}{m!}\,du}
 = \abs{\int_0^{\rho_k(x) - \rho_k(t)}\frac{v^m}{m!} \cdot g'(v) \,dv}
 \le \theta^{-1} \frac{|\rho_k(x) - \rho_k(t)|^{m+1}}{(m+1)!}.
\end{multline}
Inserting~\eqref{eq:int.fum} into~\eqref{eq:Rm+1-Rm.jk} and taking into account estimate~\eqref{eq:C0jk.def} we arrive at
\begin{equation}
 \abs{R^{(m+1)}_{jk}(x,t) - R^{(m)}_{jk}(x,t)}
 \le C_{0jk} \tau_k^m \theta^{-1}
 \frac{|\rho_k(x) - \rho_k(t)|^{m+1}}{(m+1)!}
 \le C_{jk} \tau_k^{m+1} \frac{|\rho_k(x) - \rho_k(t)|^{m+1}}{(m+1)!},
\end{equation}
which proves~\eqref{eq:induction_hypothesis} for $N=m+1$ and $j \ne k$.

It is clear now, that the crucial estimate~\eqref{eq:induction_hypothesis} is proved. In turn, this estimate implies the
absolute and uniform convergence of the series~\eqref{eq:series_of_sucsev_approx} in $\Omega_k$ which ensures
the existence in $\Omega_k$ of the continuous solution $R_{jk} = \lim_{m\to\infty} R_{jk}^{(m)}$ to the integral
system~\eqref{eq:Rjkxt=wtQ+int.canon}. Moreover, inserting estimate~\eqref{eq:induction_hypothesis}
in~\eqref{eq:series_of_sucsev_approx} leads to the following estimate for the vector solution $(R_{jk})_{j=1}^r$:
\begin{equation} \label{eq:unif-m_estim_of_smooth_sol-n}
|R_{jk}(x,t)| \le C_{jk} \sum_{m=0}^{\infty} \tau_k^m \frac{|\rho_k(x) - \rho_k(t)|^m}{m!} \le \|\wt{Q}\|_{\infty} \wt{C}_{jk} \cdot e^{\tau_k \cdot |\rho_k(x) - \rho_k(t)|},
\qquad (x, t) \in \Omega_k,
\end{equation}
where $\wt{C}_{jk} = 1$ for $j \ne k$, $\wt{C}_{kk} = q_k \cdot (a_k^+ - a_k^-)$ and
$$
\|\wt{Q}\|_{\infty} = \max\{\|\wt{Q}_{jp}\|_{\infty} : j,p \in \onetor\} = \max\{\wt{q}_p : p \in \onetor\}.
$$

\textbf{(v)} At this final step, assuming $k \in \onetor$ to be fixed, we show that $R_{jk} \in \Lip_1(\Omega)$, $j \in \onetor$.
According to Lemma~\ref{lem:C.Lip}(iii) with account of inclusion~\eqref{eq:Rjkm.Lip} and the fact that $\Omega_k$ is a simply connected, closed and bounded set with a Lipshitz boundary, to show that $R_{jk} \in \Lip_1(\Omega_k)$ it is sufficient to show uniform boundedness of the derivatives $\Dx R_{jk}^{(m)}(x,t)$ and $\Dt R_{jk}^{(m)}(x,t)$ in $\Omega_k$.

To this end, let us obtain formulas for $D_2 R_{jk}^{(m)}(x,t) := \Dt R_{jk}^{(m)}(x,t)$ by differentiating formulas~\eqref{eq:Rjkm.def}--\eqref{eq:Rkkm.def}. Inclusions~\eqref{eq:all.Lip} and~\eqref{eq:Rjkm.Lip} allow us to apply standard rules of differentiation (for a.e.\ $(x,t) \in \Omega_k$). Applying the operator $\Dt$ to equations~\eqref{eq:Rjkm.def}--\eqref{eq:Rkkm.def} yields for $j \in \onetor$,
\begin{multline} \label{eq:Dt.Rjk}
 D_2 R_{jk}^{(m)}(x,t) = \Dt \wt{Q}_{jk}(a_{jk}(x,t))
 - \sum_{p=1}^r \int_{a_{jk}(x,t)}^x Q_{jp}(u)
 \left(D_2 R_{pk}^{(m-1)}(u, \gam_{jk}^{x,t}(u))\right) \Dt \gam_{jk}^{x,t}(u) \,du \\
 + \(\Dt a_{jk}(x,t)\)\sum^r_{p=1} Q_{jp}(a_{jk}(x,t)) \cdot
 R_{pk}^{(m-1)}\(a_{jk}(x,t),\gam_{jk}^{x,t}(a_{jk}(x,t))\), \quad j \ne k,
\end{multline}
\begin{multline} \label{eq:Dt.Rkk}
 D_2 R_{kk}^{(m)}(x,t) = - \sum_{p=1}^r \int_{a_{kk}(x,t)}^x Q_{kp}(u)
 \left(D_2 R_{pk}^{(m)}(u, \gam_{kk}^{x,t}(u))\right) \Dt \gam_{kk}^{x,t}(u) \,du \\
 + \(\Dt a_{kk}(x,t)\)\sum^r_{p=1} Q_{kp}(a_{kk}(x,t)) \cdot
 R_{pk}^{(m)}\(a_{kk}(x,t),\gam_{kk}^{x,t}(a_{kk}(x,t))\),
\end{multline}
for a.e.\ $(x,t) \in \Omega_k$. Equations~\eqref{eq:Dt.Rjk}--\eqref{eq:Dt.Rkk} have exact same form as~\eqref{eq:Rjkm.def}--\eqref{eq:Rkkm.def} with only three notable differences that do not prevent the application of successive approximation procedure used in the previous step:
\begin{itemize}
\item The ``initial data'' (out-of-integral term in~\eqref{eq:Dt.Rjk}) now belongs to $L^{\infty}(\Omega_k)$, which is implied by inclusions $a_{jk}, \wt{Q}_{jk} \circ a_{jk} \in \Lip_1(\Omega_k)$ and $Q_{jp} \in L^{\infty}(\bR)$. This only changes the smoothness of approximations $D_2 R_{jk}^{(m)}$ to $L^{\infty}(\Omega_k)$, but does not affect the proof in any way. The only difference is that the uniform limit of $D_2 R_{jk}^{(m)}$ will be also in $L^{\infty}(\Omega_k)$ instead of $C(\Omega_k)$;
\item The ``initial data'' now includes the term $R_{pk}^{(N)}$. When we estimate the difference $D_2 R_{jk}^{(m+1)} - D_2 R_{jk}^{(m)}$ this term will generate the difference $R_{jk}^{(N+1)} - R_{jk}^{(N)}$ (either with $N=m-1$ or $N=m$), for which we already have the key estimate~\eqref{eq:induction_hypothesis} and hence this can be handled properly during estimation;
\item The coefficient $Q_{kp}(u)$ of $R_{pk}(\ldots)$ in~\eqref{eq:Rjkxt=wtQ+int.canon} is replaced with $Q_{jp}(u) \Dt \gam_{jk}^{x,t}(u)$. It now depends on $x, t, u$, but inclusions $Q_{jp} \in L^{\infty}(\bR)$ and $\gam_{jk}^{\cdot,\cdot}(\cdot) \in \Lip_1(\bR^3)$ imply that it belongs to $L^{\infty}(\bR^3)$. Hence in the estimates for the difference $D_2 R_{jk}^{(m+1)} - D_2 R_{jk}^{(m)}$ similar to~\eqref{eq:Rkk.m+1-Rkkm},~\eqref{eq:Rm+1-Rm.jk} we can still estimate this coefficient from above as before.
\end{itemize}
There remarks prove the absolute and uniform convergence of the series
\begin{equation} \label{eq:sum.DtRjkm}
\sum_{m=1}^{\infty} D_2 \left[R_{jk}^{(m)} - R_{jk}^{(m-1)}\right] =
\sum_{m=1}^{\infty} \Dt \left[R_{jk}^{(m)} - R_{jk}^{(m-1)}\right]
\end{equation}
in $\Omega_k$, which implies in particular uniform boundedness of $\Dt R_{jk}^{(m)}(x,t)$ in $\Omega_k$ over $m \in \bN_0$.

Applying the operator $\Dx$ to equations~\eqref{eq:Rjkm.def}--\eqref{eq:Rkkm.def} yields similar formula for $\Dx R_{jk}^{(m)}(x,t)$ except in the integral we still have $D_2 R_{pk}^{(N)}(u, \gam_{jk}^{x,t}(u))$ term (either with $N=m-1$ or $N=m$). Hence uniform boundedness of $\Dx R_{jk}^{(m)}(x,t)$, $m \in \bN_0$, is implied by uniform boundedness of $\Dt R_{jk}^{(m)}(x,t)$, $m \in \bN_0$, and Lipshitz and boundedness properties of the involved functions $\wt{Q}_{jk}$, $Q_{jp}$, $a_{jk}$, $\gam_{jk}^{\cdot,\cdot}(\cdot)$.
Lemma~\ref{lem:C.Lip}(iii) now finishes the proof.
The proof of uniqueness is proved by applying the Gr\"onwall's lemma and is omitted.
\end{proof}
\begin{remark} \label{rem:Omegajk}
Let us prove properties~\eqref{eq:Omegajk.canon}--\eqref{eq:gamjk-} more formally.
First note that
\begin{equation} \label{eq:gamjkxtu<u}
 \gam_{jk}^{x,t}(u) \le u, \quad u \in [x, a_{jk}(x,t)],
 \quad 0 \le t \le x \le \ell.
\end{equation}
This follows from the definition~\eqref{eq:gamjkxt.def} of $\gam_{jk}^{x,t}(\cdot)$, relations $\gam_{jk}^{x,t}(x) = t \le x$ and $\gam_{jk}^{x,t}(a) = a$ for $a = a_{jk}(x,t)$, and monotonicity of functions $\rho_k$ and $\rho_k^{-1}$.
Further, since $\rho_j$ and $\rho_k$ are (absolutely) continuous and monotonous it follows that $\left\{\rho_j(x) - \rho_k(t) : \ \ 0 \le t \le x \le \ell \right\}$ is a smallest segment of the real line $\bR$ containing points $0$, $\rho_j(\ell)$ and $\rho_j(\ell) - \rho_k(\ell)$. I.e.
\begin{equation} \label{eq:set.rhojx-rhokt}
 \left\{\rho_j(x) - \rho_k(t) : \ 0 \le t \le x \le \ell \right\} =:
 [\rho_{jk}^-, \rho_{jk}^+],
\end{equation}
where
\begin{equation} \label{eq:ajk-+}
 \rho_{jk}^- := \min\{0, \ \rho_j(\ell), \ \rho_j(\ell) - \rho_k(\ell)\} \le 0,
 \qquad
 \rho_{jk}^+ := \max\{0, \ \rho_j(\ell), \ \rho_j(\ell) - \rho_k(\ell)\} \ge 0.
\end{equation}
Since $(\rho_j - \rho_k)^{-1}$ is continuous and monotonous, it follows that
the set of all values $a_{jk}(x,t)$ when $(x,t)$ runs through $\Omega$ is also a finite segment of $\bR$,
\begin{equation} \label{eq:set.yjk}
 \{a_{jk}(x,t) : \ 0 \le t \le x \le \ell\}
 = (\rho_j - \rho_k)^{-1}([\rho_{jk}^-, \rho_{jk}^+])
 =: [a_{jk}^-, a_{jk}^+] \supset [0,\ell],
\end{equation}
where the last inclusion follows from formulas~\eqref{eq:ajk-+} for $\rho_{jk}^{\pm}$. Namely, relations
$$
a_{jk}(0,0) = (\rho_j - \rho_k)^{-1}(0) = 0, \qquad
a_{jk}(\ell,0) = (\rho_j - \rho_k)^{-1}(\rho_j(\ell)), \qquad
a_{jk}(\ell,\ell) = \ell,
$$
imply inclusion $[a_{jk}^-, a_{jk}^+] \supset [0,\ell]$.
These formulas also imply one of the formula~\eqref{eq:gamjk-} for $a_{jk}^{\pm}$. Which of the cases we will have depends on the relation of $a_{jk}(\ell,0)$ to the numbers $0$ and $\ell$.

Observations~\eqref{eq:gamjkxtu<u} and~\eqref{eq:set.yjk} already imply boundedness of $\Omega_{jk}$ from ``three sides'',
\begin{equation} \label{eq:Omegak'.semi}
 \Omega_{jk} \subset \{(u,v) : u \in [a_{jk}^-, a_{jk}^+], v \le u\}.
\end{equation}
Let us fix $u \in [a_{jk}^-, a_{jk}^+]$ and find the intersection of $\Omega_k$ with the vertical line $\bfL_u := \{(u, v) : v \in \bR\}$.
Since $u \in [a_{jk}^-, a_{jk}^+]$, then by definition of $a_{jk}^-$ and $a_{jk}^+$ there exists $(x,t) \in \Omega$ such that $a_{jk}(x,t) = u$. Hence $\gam_{jk}^{x,t}(u) = u$ by definition of $a_{jk}(x,t)$, which implies that $(u,u) \in \Omega_{jk} \cup \bfL_u$. From general continuity and monotonicity reasoning it is clear that $\Omega_{jk} \cup \bfL_u$ is a finite segment of the form $[\gam_{jk}^-(u), u]$, where $\gam_{jk}^-(\cdot)$ satisfy one of the cases in~\eqref{eq:gamjk-}.
\end{remark}
Now we are ready to state the main result of this subsection which, in particular, states the similarity of the operators $\cL_{0}(Q)$ and $\cL_{0}(0)$.
\begin{theorem} \label{th:similarity.LQ.L0}
Let matrix functions $B(\cdot)$ and $Q(\cdot)$ satisfy conditions~\eqref{eq:Bx.block.def}--\eqref{eq:beta.k+eps<beta.k+1.r}.
In particular, we assume that $Q_{kk}=0$ for $j \in \onetor$.
Then the operators $\cL_0(Q)$ and $\cL_0(0) = B(x)^{-1}\otimes D_0$ are similar in
$L^p([0,\ell];\bC^{n\times n})$, $p \in [1,\infty]$.
Moreover, there exists a bounded on $L^p([0,\ell];\bC^n)$ triangular Volterra type operator $I + \cR$,
\begin{equation} \label{eq:interwine}
 (I + \cR)f = f(x) + \int^x_0 R(x,t) B(t) f(t)\,dt,
 \qquad f \in L^p([0,\ell];\bC^n),
\end{equation}
with a bounded inverse that intertwines the operators $\cL_0(Q)$ and $\cL_0(0)$, i.e.
\begin{equation} \label{eq:LQ.I+R=I+R.L0}
\cL_0(Q)(I + \cR)f = (I + \cR)\cL_0(0)f, \qquad f \in \dom\cL_{0}(0) = W^{1,2}_0([0,\ell]; \bC^n).
\end{equation}
Here $R(x,t) = (R_{jk}(x,t))_{j,k=1}^r$ is the block-matrix kernel
of the operator $\cR$ that meets the condition
\begin{equation} \label{eq:R.in.X}
 R \in \left(X_{1,0}(\Omega)\cap X_{\infty,0}(\Omega)\right)
 \otimes \bC^{n\times n},
\end{equation}
where the function spaces $X_{1,0}(\Omega)$, $X_{\infty,0}(\Omega)$ are defined in Subsection~\ref{subsec:X1inf} above.
\end{theorem}
\begin{proof}
First, let us extend matrix functions $B(\cdot)$ and $Q(\cdot)$ to be defined on $\bR$ as it was done in the step (ii) of the proof of Proposition~\ref{prop:similarity.LQ.L0}. Let us show that under the assumption $Q \in L^1(\bR;\bC^{n \times n})$ (see~\eqref{eq:Q.block.L1} and~\eqref{eq:Qx=0}) the integral system~\eqref{eq:Rjkxt=wtQ+int.canon} has a solution $R$ satisfying inclusion~\eqref{eq:R.in.X}. It can be treated as a generalized solution to the problem~\eqref{eq:Goursat}--\eqref{eq:Rkk.x.gam}.
For simplicity we restrict ourselves to the case $Q \in L^{\infty}(\bR;\bC^{n \times n})$ and $\wt{Q}_{jk} \in C(\bR)$, $j \ne k$ (see~\eqref{eq:Q.block.L1} and~\eqref{eq:wtQjk.Lip})
General case is treated similarly to that of our treatment of Theorem 2.5 from~\cite{LunMal16JMAA}.

To this end, we choose a sequence of smooth finite matrix functions $\wt{Q}_{m} = (\wt{Q}_{jk,m})^r_{j,k=1} \in \Lip_1(\bR; \bC^{n \times n})$, $m \in \bN$, where $Q_{jj,m} \equiv 0$, $j \in \in \onetor$, that approximate finite matrix function $\wt{Q} = (\wt{Q}_{jk})^r_{j,k=1} \in C(\bR; \bC^{n \times n_k})$. Following~\eqref{eq:Qjk.bR}--\eqref{eq:Qjj.bR} we set
\begin{align}
\label{eq:Qjkm.bR}
 Q_{jk,m}(x) & := (\beta_j(x) - \beta_k(x)) \cdot \wt{Q}_{jk,m}(x),
 \qquad x \in \bR, \qquad j \ne k, \\
\label{eq:Qjjm.bR}
 Q_{jj,m}(x) & := \wt{Q}_{jj,m}(x) := 0, \qquad x \in \bR, \qquad j \in \onetor.
\end{align}
It follows that there exists positive constants $C_q$ and $C_{\wt{q}}$ such that
\begin{equation} \label{eq:def-n_C_q_and_C_wt_q}
 \|\wt{Q}_m\|_{C(\bR)} \le C_{\wt{q}} \qquad \Longrightarrow \qquad
 \|Q_m\|_{L^{\infty}(\bR)} \le C_{q}, \qquad m \in \bN.
\end{equation}
Moreover
\begin{equation} \label{eq:Qjkm.to.Qjk}
 \|Q_{jk,m} - Q_{jk}\|_{\infty} \to 0,
 \quad\text{as}\quad m \to \infty, \qquad j, k \in \onetor.
\end{equation}

Let us fix $k \in \onetor$. In accordance with Proposition~\ref{prop:similarity.LQ.L0}, the system of equations~\eqref{eq:Rjkxt=wtQ+int.canon} with $Q_m$ in place of $Q$, has $\Lip_1$-smooth matrix block solution $(R_{jk,m})_{j=1}^r$, i.e.\ for $j \in \onetor$ and $(x,t) \in \Omega_k$ we have,
\begin{equation} \label{eq:estimate_for_R_{jk,m}}
R_{jk,m}(x,t) = \wt Q_{jk,m}\bigl(a_{jk}(x,t)\bigr)
 - \int_{a_{jk}(x,t)}^x \sum^r_{p=1}
 Q_{jp,m}(u)R_{pk,m}\(u, \gam_{jk}^{x,t}(u)\) \,du.
\end{equation}

To evaluate the difference $R_{jk,m} - R_{jk,s}$ we set
\begin{align} \label{eq:def_wt_R_{jk,m,s}_and_wt_Q}
\wh R_{jk,m,s}(x,t):= R_{jk,m}(x,t) - R_{jk,s}(x,t), \qquad \wh Q_{jk,m,s} := Q_{jk,m} - Q_{jk,s},
 \nonumber \\
 \wh Q_{m,s} := Q_{m} - Q_{s}, \quad \text{and} \qquad \wt Q_{jk,m,s} := \wt Q_{jk,m} - \wt Q_{jk,s}, \qquad m,s \in \bN.
\end{align}
Taking the difference of two equations~\eqref{eq:estimate_for_R_{jk,m}} with $m=m$ and $m=s$
and using the notations~\eqref{eq:def_wt_R_{jk,m,s}_and_wt_Q} one easily
rewrites it in the form
\begin{align} \label{eq:diff-ce_of_two_eq-s}
\wh R_{jk,m,s}(x,t) = \wt Q_{jk,m,s}\bigl(a_{jk}(x,t)\bigr) - \sum_{p=1}^r \int^x_{a_{jk}} \wh Q_{jp,m,s}(u) R_{pk,s}(u,\gam_{jk}^{x,t}(u)\bigr)d u \nonumber \\
- \sum_{p=1}^r \int^x_{a_{jk}}Q_{jp,m}(u) \wh R_{pk,m,s}(u,\gam_{jk}^{x,t}(u)\bigr)d u,
\end{align}
where for brevity we set $a_{jk} := a_{jk}(x,t)$.
Further, first we show that the family $\{R_{jk,m}\}_{m \in \bN}$ is uniformly bounded in $L^{\infty}$-norm. Indeed, one gets from
\eqref{eq:unif-m_estim_of_smooth_sol-n} with account of definitions~\eqref{eq:def-n_C_q_and_C_wt_q} that
\begin{equation} \label{eq:unif-m_estim_for_seq-ce_of_sol-n}
\|R_{jk,m}\|_{L^{\infty}(\Omega_k)} \le
C_{jk} \cdot \exp\(\tau_k \cdot |\rho_k(a_k^+) - \rho_k(a_k^-)|\)
= M_{jk}, \quad m \in \bN.
\end{equation}
It follows from~\eqref{eq:diff-ce_of_two_eq-s} with account of estimate~\eqref{eq:unif-m_estim_for_seq-ce_of_sol-n}
and notations~\eqref{eq:def-n_C_q_and_C_wt_q} that
\begin{align} \label{eq:estimate_of_diff-ce_of_two_eq-s}
|\wh R_{jk,m,s}(x,t)| = \|\wt Q_{jk,m} - \wt Q_{jk,s}\|_{L^{\infty}} + M_{jk} \sum_{p=1}^r \int^x_{a_{jk}} |Q_{jp,m}(u) - Q_{jp,s}(u)|\,du \nonumber \\
+ \sum_{p=1}^r \int^x_{a_{jk}}|Q_{jp,m}(u)|\cdot |\wh R_{pk,m,s}(u,\gam_{jk}^{x,t}(u)|\,du \nonumber \\
 \le \|\wt Q_{jk,m} - \wt Q_{jk,s}\|_{L^{\infty}} + M_{jk} r \| Q_{m} - Q_{s}\|_{L^{\infty}}
 + C_q \sum_{p=1}^r \int^x_{a_{jk}} |\wh R_{pk,m,s}(u,\gam_{jk}^{x,t}(u)|\,du.
\end{align}
Emphasize that all the constants in~\eqref{eq:estimate_of_diff-ce_of_two_eq-s} do not depend on $m,s \in \bN$.

Applying the method of successive approximation to system of equations~\eqref{eq:diff-ce_of_two_eq-s}
and repeating the reasoning from the step (iv) of the proof of Proposition~\ref{prop:similarity.LQ.L0} we arrive at the estimate similar to~\eqref{eq:unif-m_estim_of_smooth_sol-n}:
\begin{align} \label{eq:unif-m_estim_for_dif-ce}
|\wh R_{jk,m,s}(x,t)| \le \| \wt{Q}_{m} - \wt{Q}_{s}\|_{L^{\infty}} C_1 \exp( C_2 |\rho_k(x)-\rho_k(t)|),
\end{align}
with some $C_1, C_2 > 0$ that do not depend on $m,s,x,t,j,k$. (see~\eqref{eq:unif-m_estim_for_seq-ce_of_sol-n} and recall the definition of $C_{jk}$).
It follows that the sequence of solutions $R_{jk,m}$ to equations~\eqref{eq:estimate_for_R_{jk,m}}
is a Cauchy sequence in $C(\Omega_k)\otimes \bC^{n\times n}.$ Therefore for any pair $j,k \in \onetor$
there exists a uniform limit $R_{jk}(x,t) := \lim_{s\to\infty} R_{jk,s}(x,t)$ that meets the following
uniform in $(x,t) \in \Omega_k$ estimate
\begin{align} \label{eq:unif-m_estim_for_|R_{jk}-R_{jk,s}|}
|R_{jk}(x,t) - R_{jk,m}(x,t)| \le \| \wt{Q}_{m} - \wt{Q}\|_{L^{\infty}} C_1 \exp( C_2 |\rho_k(x)-\rho_k(t)|).
\end{align}
Moreover, due to this estimate and since $\|\wt{Q}_m - \wt{Q}\|_{\infty} \to 0$ and $\|Q_m - Q\|_{\infty} \to 0$ as $m \to \infty$, we can pass to the limit as $m \to \infty$ in equations~\eqref{eq:estimate_for_R_{jk,m}} to show that $R(x,t) = \{R_{jk}(x,t)\}_{j,k=1}^r$ is a matrix solution to
system of integral equations~\eqref{eq:Rjkxt=wtQ+int.canon} and define the Volterra operator $\cR$ with the matrix kernel
$R(x,t)$.

On the other hand, by Proposition~\ref{prop:similarity.LQ.L0}, since $Q_m \in \Lip_1([a_k^-,a_k^+]; \bC^{n\times n})$, the operator $I + \cR_m$ intertwines
the operators $\cL_0(Q_m)$ and $\cL_0(0)$, i.e.
equation~\eqref{eq:LQ.I+R=I+R.L0} holds
with $Q_m$ and $\cR_m$ in place of $Q$ and $\cR$, respectively,
\begin{equation} \label{eq:I+R_intertwines_L_0(Q_n)_and_L_0(0)}
\cL_0(Q_m)(I+ R_m)f = (I+ R_m)\cL_0(0)f, \qquad f \in \dom\cL_{0}(0) = \wt{W}^{1,2}_0([0,\ell]; \bC^n).
\end{equation}
Taking inverses we rewrite these equations in the form
\begin{equation} \label{eq:(I+R_n)^{-1}_intertwines_inverses}
(I+ R_m)^{-1}(\cL_0(Q_m))^{-1} = (\cL_0(0))^{-1}(I + R_m)^{-1}. \qquad
\end{equation}
Noting that $\cR$ is a Volterra operator, we can pass here to the limit as $m\to \infty$ to arrive
to the equation
\begin{equation} \label{eq:I+R.inv.interwines}
(I+ \cR)^{-1}(\cL_0(Q))^{-1} = (\cL_0(0))^{-1}(I+ \cR)^{-1}. \qquad
\end{equation}
It follows that $I+\cR$ maps $\dom \cL_0(0)$ onto $\dom \cL_0(Q)$. Therefore taking the inverses
in~\eqref{eq:I+R.inv.interwines}
we arrive at identity~\eqref{eq:LQ.I+R=I+R.L0}. Since $0 \in \rho(I+\cR)$,
this proves the similarity and completes the proof.

\end{proof}
\subsection{Transformation operators}
In this subsection we prove the existence of triangular transformation operators for equation~\eqref{eq:LQy.def.transform}. As in~\cite{Mal99} our proof is substantially relies on the similarity result, Theorem~\ref{th:similarity.LQ.L0}.

It is well known that the commutant $\{J\}'$ of the Volterra integration operator $J (J:\ f\to\int^x_0f(t)dt)$ on $L^2[0,\ell]$ consists of convolutions with distributions (see~\cite{Nik86} and~\cite{Mal98}). In particular,
a convolution operator $K: f\to k*f = \int_0^xk(x-t)f(t)dt$ with $k \in L^1[0,\ell]$
belongs to $\{J\}'$.

We complete this subsection by a simple lemma on commutant of the operator
$J \cB$ on $L^2[0,\ell]$ with $\cB: \ f \to b(t) f(t)$ being
 a multiplication operator. This result is a substantial ingredient in the proof of the existence of transformation
 operators in the next subsection (see Theorem~\ref{th:transform.oper}).

 Surprisingly, that partial differential equations technique is highly involved in
 a description of the commutant which is not so explicit as in the case of $b_0 \equiv \const$.
 \begin{lemma} \label{lem:commutant_lemma}
Let $b \in L^{\infty}[0,\ell]$ and let $b(\cdot)$ be either positive or negative for a.e.\ $x \in [0,\ell]$ and let
\begin{equation} \label{eq:Volter_Oper_P}
\cP:\ f \to \int^x_0 P(x,t)b(t)f(t)\,dt
\end{equation}     
be a Volterra operator with $P \in C(\Omega)$.
Then the operator $\cP$ commutes with the operator    
\begin{equation} \label{eq:oper-r_JB}
J \cB:\ f\to\int^x_0 b(t)f(t)dt
\end{equation}
if and only if the kernel $P(\cdot,\cdot)$ is given by
\begin{equation} \label{eq:P(x,t) = P(xi(x,t),0)}
P(x,t) = P(\xi(x,t),0).
\end{equation}
Here $\xi(x,t)$ is the implicit solution to the equation
\begin{equation} \label{eq:equation_for_xi(x,t)}
 \quad
 \rho(\xi(x,t)) - \rho(x) + \rho(t) = 0, \qquad \rho(x) := \int^x_0 b(s)\,ds.
\end{equation}
 \end{lemma}
 \begin{proof}
Changing the order of integrals one easily gets
\begin{equation}
P J B f = \int^x P (x,s)b (s)\,ds\int^s b (t)f(t)\,ds = \int^x \left(\int^x_t P (x,s)b (s)\,ds\right)b (t)f(t)\,dt,
\end{equation}
\begin{equation}
J B P f = \int^x b (s)\,ds\int^s P (s,t) b (t)f(t)\,dt = \int^x \left(\int^x_t b (s) P (s,t)\,ds\right)b (t)f(t)\,dt.
\end{equation}
Equating these relations one concludes that the commutation relation $[P,JB ]=0$ is equivalent  to
\begin{equation}
\int^x_t P (x,s)b (s)\,ds = \int^x_t b (s)P (s,t)\,ds.
\end{equation}
Assume first that $P \in C^1(\Omega)$. Then
differentiating this equation with respect to $x$ yields 
\begin{equation}
 P (x,x)b (x) + \int^x_t \Dx P (x,s)b (s)\,ds = b (x)P (x,t).
\end{equation}
In turn, applying the operator $\Dt$ to this equation  leads to 
 the first order partial differential equation
\begin{equation} \label{eq:(Dx+Dt)P=0}
 \Dx P(x,t)b(t) + b(x) D_t P(x,t) = 0.
\end{equation}

Conversely, if the kernel $P \in C^1(\Omega)$ and satisfies equation~\eqref{eq:(Dx+Dt)P=0} we obtain by reversing the reasonings that the operator $\cP$ of the form~\eqref{eq:Volter_Oper_P} commutes with $J \cB$, i.e.
$P \in \{J \cB\}'$.

Next we extend the function $b(\cdot)$ to the whole line $\bR$ preserving its $L^{\infty}$-norm and the
sign. In what follows we keep the notation $b(\cdot)$ for this extension.

It is easily seen that in the coordinates $\{\xi, \eta\}$ the characteristic of equation~\eqref{eq:(Dx+Dt)P=0}
passing through the point $(x,t)$, is given by
\begin{equation} \label{eq:char-cs_G_{j}}
\Gam(\xi,\eta): \quad \rho(\eta) = \rho(\xi) - \rho(x) + \rho(t), \qquad \rho(\xi) = \int^{\xi}_0 b(s)\,ds,
\end{equation}
where $\rho$ is also defined on the line.
The ``explicit'' form $\eta =\gam(\xi; x, t)$ of the characteristic is defined to be the unique solution to the equation
\begin{equation} \label{eq:eq-n_for_char-c_gamma(xi;x,t)}
\Gam(\xi,\gam(\xi;x, t) = \rho(\gam(\xi;x,t) - \rho(\xi) + \rho(x) - \rho(t) =0 \quad \text{and} \quad \gam(x;x,t) =t.
\end{equation}
Since $\rho$ is monotone on $\bR$, each characteristic is well defined on $\bR$.
$$
\eta(\xi) = \gam(\xi; x, t) := \rho^{-1} \circ (\rho(\xi) - \rho(x) + \rho(t)), \qquad \xi \in \bR.
$$

Next we denote by $M(\xi(x,t),0)$ the point of interaction of the characteristic
$\eta = \gam(\xi;x,t)$ with the axis $\eta =0$. Then in accordance with~\eqref{eq:eq-n_for_char-c_gamma(xi;x,t)}
 $\xi(x,t)$ satisfies the equation
\begin{equation} \label{eq:eq-n_for_xi(x,t)_New}
0= \rho(0) = \rho(\xi(x,t)) - \rho(x) + \rho(t),
\end{equation}
It follows that
\begin{equation} \label{eq:form-s_forD-x_xi,_D_t_xi}
\frac{\partial\xi(x,t)}{\partial x} = \frac{\rho'(x)}{\rho'\bigl(\xi(x,t)\bigr)} =
\frac{\beta(x)}{\beta\bigl(\xi(x,t)\bigr)} \qquad \text{and} \qquad \frac{\partial\xi(x,t)}{\partial t} =
- \frac{\beta(t)}{\beta\bigl(\xi(x,t)\bigr)}.
\end{equation}
In turn, using these relations imply that alongside $P(x,t)$ the function
\begin{equation} \label{eq:defin_wt_P(x,t)}
\wt P(x,t) := P(\xi(x, t),0)
\end{equation}
satisfies the equation~\eqref{eq:(Dx+Dt)P=0}. Besides,
 the identity $\xi(x, 0)=x$ yields
\begin{equation} \label{eq:init-al_cond_wt-P(x,0)=P(x,0)}
 \wt P(x, 0)= P(\xi(x, 0),0) = P(x,0) =:g(x).
\end{equation}
So, we have two solutions $\wt P(\cdot, \cdot)$ and $P(\cdot, \cdot)$
to the Cauchy problem~\eqref{eq:(Dx+Dt)P=0},~\eqref{eq:init-al_cond_wt-P(x,0)=P(x,0)}
in $\Omega$.
However, since the slope of the characteristic $\gam(\xi;x,t)$ is positive,
\begin{equation} \label{eq:slope_kapa_(xi,eta)}
\kappa( \xi_0,\eta_0) = \frac{d\gam(\xi;x,t)}{d\xi}\big|_{\xi= \xi_0} =
\frac{\beta( \xi_0)}{\beta(\eta_0)} = \frac{\beta( \xi_0)}{\beta(\gam_{jk}( \xi_0;x,t))} > 0,
\end{equation}
the Cauchy problem~\eqref{eq:(Dx+Dt)P=0},~\eqref{eq:init-al_cond_wt-P(x,0)=P(x,0)}
is not characteristic, and hence has the unique solution in $\Omega$, i.e.
$\wt P(x, t)= P(x,t)$.
 \end{proof}
\begin{remark}
 Note that as a byproduct we proved that if $P \in \{J \cB\}'$ and the kernel of $P$ is smooth, $P \in C^1(\Omega),$ then it is a solution to equation~\eqref{eq:(Dx+Dt)P=0}. If $P \in C(\Omega),$ then it
 is a generalized solution to equation~\eqref{eq:(Dx+Dt)P=0}.
 A complete description of the commutant $\{J \cB\}'$ will be published elsewhere.
\end{remark}
Now we are ready to establish our main result on transformation operators for the equation~\eqref{eq:LQy.def.transform}.
\begin{theorem} \label{th:transform.oper}
Let matrix functions $B(\cdot)$ and $Q(\cdot)$ satisfy conditions~\eqref{eq:Bx.block.def}--\eqref{eq:beta.k+eps<beta.k+1.r}.
In particular, we assume that $Q_{jj}=0$ for $j \in \onetor$. Further, let
\begin{align}
\label{eq:A=A1.Ar}
 & A = \begin{pmatrix} A_1 \\ \cdots \\ A_r \end{pmatrix} \in \bC^{n \times n_{\min}}, \qquad
 A_j \in \bC^{n_j \times n_{\min}}, \qquad j \in \onetor,
\end{align}
where $n_{\min} := \min\{n_1, \ldots, n_r\}$, and let all matrices $A_j$ be of maximal rank, i.e.\ $\rank(A_j) = n_{\min}$. $j \in \onetor$. Further, let
\begin{equation} \label{eq:YeA.def}
 Y_A(x,\l) = \begin{pmatrix} Y_1(x,\l) \\ \cdots \\ Y_r(x,\l) \end{pmatrix}
 \qquad \text{and} \qquad
 e_A(x,\l) := \begin{pmatrix} e^{i \l \rho_1(x)} A_1 \\ \cdots \\
 e^{i \l \rho_r(x)} A_r \end{pmatrix}
\end{equation}
are the $n \times n_{\min}$ block-matrix solutions to equations~\eqref{eq:LQy.def.transform} and~\eqref{eq:L0.def.transform}, respectively, satisfying the initial conditions
\begin{equation} \label{eq:YeA.init}
 Y_A(0,\l) = e_A(0,\l) = A.
\end{equation}
Then solution $Y_A(x,\l)$ admits a triangular representation
\begin{equation} \label{eq:trans_oper_repres-n}
 Y_A(x,\l)= (I + \cK_A) e_A(x,\l)
 = e_A(x, \l) + \int^x_0 K_A(x,t) B(t) e_A(t, \l) \,dt,
\end{equation}
where the block-matrix kernel $K_A = (K_{jk})_{j,k=1}^r$ in the operator $\cK_A$ satisfies
\begin{equation} \label{eq:KA.in.X}
 K_A \in \left(X_{1,0}(\Omega)\cap X_{\infty,0}(\Omega)\right)
 \otimes \bC^{n\times n_1}.
\end{equation}
Here domain $\Omega$ and function spaces $X_{1,0}(\Omega)$, $X_{\infty,0}(\Omega)$ are defined in Subsection~\ref{subsec:X1inf} above.
\end{theorem}
\begin{remark}
Note that Lemma~\ref{lem:trace} and inclusion $K_A \in X_{\infty,0}(\Omega) \otimes \bC^{n \times n_1}$ ensure that the traces $K(x,\cdot)$ are well-defined and summable for each $x \in [0,\ell]$. In particular, end trace $K(\ell,\cdot)$ is well defined. Hence formula~\eqref{eq:trans_oper_repres-n} is well-defined and valid for each $x \in [0,\ell]$. A more relaxed inclusion $K_A \in X_{\infty}(\Omega) \otimes \bC^{n\times n}$ would only yield this formula for a.e $x \in [0,\ell]$.
\end{remark}
\begin{proof}[Sketch of the proof]
Assume for definiteness that $n_{\min} = n_1$.

\textbf{(i)} At this step assuming the validity of representation~\eqref{eq:trans_oper_repres-n} with $K_A \in C^1(\Omega)$
we indicate the boundary value problem for the kernel $K_A(\cdot,\cdot)$.
Inserting representation~\eqref{eq:trans_oper_repres-n}
into equation~\eqref{eq:LQy.def.transform} we obtain
\begin{multline}
 Y'(x,\l) + Q(x)Y(x,\l) = e'_A(x,\l) + K_A(x,x) B(x) e_A(x,\l) \\
 + \int^x_0 \Dx K_A(x,t) B(t) e_A(t,\l)\,dt
 + Q(x)e_A(x,\l) + Q(x)\cdot \int^x_0 K_A(x,t) B(t) e_A(t,\l)\,dt \\
 = e'_A(x,\l) + \bigl(K_A(x,x) B(x) + Q(x)\bigr) e_A(x,\l) \\
 + \int_0^x \bigl(\Dx K_A(x,t) + Q(x) K_A(x,t)\bigr) B(t) e_A(t,\l)\,dt
\end{multline}
On the other hand, it follows from~\eqref{eq:trans_oper_repres-n} after integrating by parts that
\begin{align}
\nonumber
 i \l B(x)Y(x,\l) &= i \l B(x) e_A(x,\l)
 + i \l B(x) \int^x_0 K_A(x,t) B(t) e_A(t,\l)\,dt \\
\nonumber
 &= e'_A(x,\l) + B(x) \int^x_0 K_A(x,t) e_A'(t,\l)\,dt \\
\nonumber
 &= e'_A(x,\l) + B(x) K_A(x,x) e_A(x,\l) \\
 & \qquad \qquad
 - B(x) K_A(x,0) e_A(0,\l)
 - B(x) \int^x_0 D_t\bigl(K_A(x,t)\bigr) e_A(t,\l)\,dt.
\end{align}
Equating both sides of this identities we arrive at the following boundary value problem for the $n\times n$-matrix kernel $K_A(\cdot,\cdot)$:
\begin{equation} \label{eq:B.Dx+Dt.B}
D_x K_A(x,t) B(t) + B(x) D_t K_A(x,t) + Q(x)K_A(x,t)B(t) =0,
\end{equation}
\begin{equation} \label{eq:K(x,x)B-1(x)-B_1(x)K(x,x)=iBQ}
 B(x) K_A(x,x) - K_A(x,x) B(x) = Q(x),
\end{equation}
\begin{equation} \label{eq:K(x,0)B{-1}A=0}
 K_A(x,0) A = 0.
\end{equation}
Writing the kernel $K_A$ in the block-matrix form $K_A = \(K_{jk}\)_{j,k=1}^r$ with respect to the decomposition $\bC^n = \bC^{n_1} \oplus \ldots \oplus \bC^{n_r}$ and using the block-matrix form of $Q$ we rewrite the problem~\eqref{eq:B.Dx+Dt.B}--\eqref{eq:K(x,0)B{-1}A=0}
in the following form
\begin{align}
\label{eq:bk.Dx+bj.Dt}
 \beta_k(t) D_x K_{jk}(x,t) + \beta_j(x) D_t K_{jk}(x,t)
 & =-\sum^r_{p=1} \beta_k(t) Q_{jp}(x) K_{pk}(x,t),
 \quad j, k \in \onetor, \\
\label{eq:Kjk.xx}
 K_{jk}(x,x) & = \frac{Q_{jk}(x)}{\beta_j(x)- \beta_k(x)},
 \qquad j \ne k, \quad j, k \in \onetor, \\
\label{eq:Kjk.x0.Ak}
 \sum^r_{k=1} K_{jk}(x,0)A_k & = 0, \qquad j \in \onetor.
\end{align}
Emphasise that formula~\eqref{eq:Kjk.xx} has sense due to the conditions~\eqref{eq:beta-}--\eqref{eq:beta+}.

Conversely, reversing the reasonings one proves that any $C^1$-solution to the problem~\eqref{eq:bk.Dx+bj.Dt}--\eqref{eq:Kjk.x0.Ak} generates representation~\eqref{eq:trans_oper_repres-n}.
So, to prove the result it suffices to show the solvability of the problem~\eqref{eq:bk.Dx+bj.Dt}--\eqref{eq:Kjk.x0.Ak}.

\textbf{(ii)} At this step we prove the solvability.
To construct a solution $K_A$ to the problem~\eqref{eq:bk.Dx+bj.Dt}--\eqref{eq:Kjk.x0.Ak}
we use a solution $R(x,t)$ to the problem
\eqref{eq:bk.Dx+bj.Dt}--\eqref{eq:Kjk.xx},
constructed in Theorem~\ref{th:similarity.LQ.L0}. Besides, we introduce a Volterra operator
\begin{equation} \label{eq:oper-r_P=oplus_P_j}
P = \bigoplus_1^{r} P_j, \quad \text{where}\quad P_j:\ f_j\to\int^x_0 P_j(x,t)B_j(t)f_j(t)\,dt, \qquad f_j \in L^2\bigl([0,\ell];\bC^{n_j}\bigr).
\end{equation}
Moreover, we assume that $P_j$ has a smooth kernel $P_j \in C^1((\Omega);\bC^{n_j\times n_j})$ and
satisfies
$$
[P_j, J\otimes B_j] = [P_j, \left(J\beta_j(t)\right)\otimes I_{n_j}] = 0, \qquad j \in \onetor,
$$
i.e.\ $P_j \in \{J\otimes B_j\}'$.

Starting with the operator $I + \cR$ constructed in Theorem~\ref{th:similarity.LQ.L0}, we define
the operator $I + \cK_A$ as the product of two operators:
\begin{equation} \label{eq:I+K=(I+R)(I+P)}
I + \cK_A := (I + \cR)(I + \cP).
\end{equation}
In terms of the kernels of integral operators $\cK_A$, $\cR$, and $\cP$, equality~\eqref{eq:I+K=(I+R)(I+P)}
can be rewritten as
\begin{equation} \label{eq:K=R+P+R*P}
K_A(x,t) = R(x,t) + P(x,t) + \int^x_t R(x,s)B(s) P(s,t)\,ds.
\end{equation}
In fact, identity~\eqref{eq:I+K=(I+R)(I+P)} is equivalent to~\eqref{eq:K=R+P+R*P} after multiplying the last equality
by the factor $B(t)$ from the right. This factor is canceled in~\eqref{eq:K=R+P+R*P} because the matrix $B(t)$ is non-singular for every $t \in [0,\ell]$.

Since the operators $I + \cP$ and $\cL_0(0)$ commutes, the operator $I + \cK_A$ intertwines the operators $\cL_0(Q)$ and $\cL_0(0)$ alongside the operator $I + \cR$, i.e.\ identity~\eqref{eq:LQ.I+R=I+R.L0} holds with $I + \cK_A$ in place of $I + \cR$. Therefore Theorem~\ref{th:similarity.LQ.L0} (sufficiency) applies and ensures that the kernel $K_A(\cdot,\cdot)$ is also a solution to the problem~\eqref{eq:bk.Dx+bj.Dt}--\eqref{eq:Kjk.xx}. To complete the proof it suffices to find a kernel $P(\cdot,\cdot)$ in such a way that $K_A(\cdot,\cdot)$ meets the condition~\eqref{eq:Kjk.x0.Ak}. To this end we insert the right hand side of equality~\eqref{eq:K=R+P+R*P} in~\eqref{eq:Kjk.x0.Ak} and obtain
\begin{equation}
 \label{eq:P(x,0)A=0_in_blocks}
 P_{j}(x,0)A_j + \sum^r_{k=1} \left[ R_{jk}(x,0) + \int^x_0 R_{jk}(x,s)B_k(s) P_k(s,0)\,ds \right]A_k = 0, \qquad j \in \onetor.
\end{equation}
Rewriting this equality as
\begin{equation}
(I+\cR)
\begin{pmatrix}
P_1(x,0)A_1\\
\ldots \ldots \\
P_r(x,0)A_r
\end{pmatrix}
=
\begin{pmatrix}
\wt R_1(x,0)\\
\ldots \ldots \\
\wt R_r(x,0)
\end{pmatrix}\,,
\qquad \wt R_j(x,0):= - \sum^r_{k=1}R_{jk}(x,0)A_k
\end{equation}
we find the unique solution
\begin{equation} \label{eq:system_for_P_j(x,0)}
\begin{pmatrix}
P_1(x,0)A_1\\
\ldots \ldots \\
P_r(x,0)A_r
\end{pmatrix}
=
 (I+\cR)^{-1}
\begin{pmatrix}
\wt R_1(x,0)\\
\ldots \ldots \\
\wt R_r(x,0)
\end{pmatrix}\,
=: \begin{pmatrix}
\wt g_1(x)\\
\ldots \ldots \\
\wt g_r(x)
\end{pmatrix}\,.
\end{equation}
In turn, due to the condition $\rank(A_j) = n_{\min}=n_1$, meaning that
 each matrix $A_j$ is of the maximal rank, there exists (in general, non-unique)
matrix solution $\{g_j(x)\}_{j=1}^r := \{P_j(x,0)\}_{j=1}^r$ to the system~\eqref{eq:system_for_P_j(x,0)},
i.e.\ $g_jA_j = \wt g_j$, $j \in \onetor$.
A solution $\{P_j(x,0)\}_{j=1}^r$ is definitely unique whenever $n_1 = n_2 =\ldots = n_r$,
and hence $\det A_j \ne 0$.

Further, to find a system of matrix functions $\{P_j(x,t)\}_{j=1}^r$ we apply Lemma~\eqref{lem:commutant_lemma}.
In accordance with this lemma each $P_j(x,t)$ is a solution to the following Cauchy problem
\begin{align} \label{eq:Cauchy_pr-m_(D_x_+_D_t)P_j=0}
 D_x P_j(x,t)B_j(t) + B_j(x) D_t P_j(x,t) = 0, \\
 P_j(x,0) = g_j(x), \qquad j \in \onetor.
\end{align}
and is given by~\eqref{eq:P(x,t) = P(xi(x,t),0)} with $P_j(x,t)$ in place of $P(x,t)$, i.e.
$$
P_j(x,t) = P_j( \xi_j(x, t),0), \qquad P_j( \xi_j(x, 0),0) = P_j(x,0).
$$
Here $ \xi_j(x,t)$ is a solution to the equation
\begin{equation} \label{eq:eq-n_for_xi_j(x,t)}
 \quad
 \rho_j( \xi_j(x,t)) - \rho_j(x) + \rho_j(t) = 0, \qquad \rho_j(x) := \int^x_0 b_j(s)\,ds.
\end{equation}
This completes the proof.
\end{proof}
\begin{remark}
\textbf{(i)} For Dirac $2 \times 2$ system $(B = \diag(-1,1))$ with continuous $Q$ the triangular transformation operators have been constructed in~\cite[Ch.10.3]{LevSar88} and~\cite[Ch.1.2]{Mar77}. For $Q \in (L^1[0,1]; \bC^{2 \times 2})$ it is proved in~\cite{AlbHryMyk05} by an appropriate generalization of Marchenko's method.

(ii) Let $J: f\to \int_0^xf(t)dt$ denote the Volterra integration operator on $L^p[0,1]$. Note that the similarity of integral Volterra operators given by~\eqref{eq:cR.def} to the simplest Volterra operator of the form $B \otimes J$ acting in the spaces $L^p([0,1]; \bC^2)$ has been investigated in~\cite{Mal99, Rom08}. The technique of investigation of integral equations for the kernels of transformation operators in the spaces $X_{\infty,1}(\Omega)$ and $X_{1,1}(\Omega)$ goes back to the paper~\cite{Mal94}.
\end{remark}
\section{Fundamental matrix solution} \label{sec:fundament}
In this section we apply results of the previous section to obtain an important representation for the fundamental matrix solution of equation~\eqref{eq:LQ.def.intro} and its minors, involving Fourier transform of kernels from the transformation operators.
\subsection{Preliminaries}
\label{subsec:fundam.prelim}
For reader's convenience, recall the main equation~\eqref{eq:LQ.def.intro} and its version with zero potential:
\begin{align}
\label{eq:LQ.def.fund}
 & \cL(Q) y := -i B(x)^{-1} (y' + Q(x) y) = \l y,
 \qquad y = \col(y_1, \ldots, y_n), \quad x \in [0,\ell], \\
\label{eq:L0.def}
 & \cL_0 y := \cL(0) y := -i B(x)^{-1} y' = \l y,
 \qquad y = \col(y_1, \ldots, y_n), \quad x \in [0,\ell].
\end{align}
Let us recall definitions~\eqref{eq:Bx.def.intro}--\eqref{eq:Q=Qjk.def.intro} of matrix functions $B(\cdot)$ and $Q(\cdot)$
\begin{equation} \label{eq:Bx.def}
 B = \diag(\beta_1, \ldots, \beta_n), \qquad
 \beta_k \in L^1([0,\ell]; \bR \setminus \{0\}),
 \qquad k \in \oneton,
\end{equation}
is a self-adjoint invertible diagonal summable matrix function, and
\begin{equation} \label{eq:Q=Qjk.def}
 Q = (Q_{jk})_{j,k=1}^n, \qquad Q_{jk} \in L^1[0,\ell] := L^1([0,\ell]; \bC),
 \qquad j, k \in \oneton,
\end{equation}
is a summable (generally non-self-adjoint) potential matrix.

Throughout this section and many results in sections below we will assume conditions~\eqref{eq:beta-}--\eqref{eq:betak-betaj<-eps.intro} on entries of the matrix function $B(\cdot)$. Namely, we assume that
\begin{align}
\label{eq:betak.alpk.Linf}
 & \beta_k, 1/\beta_k \in L^{\infty}[0, \ell], \qquad
 s_k := \sign(\beta_k(\cdot)) \equiv \const \ne 0, \qquad
 k \in \oneton, \\
\label{eq:beta-+n}
 & \beta_1(x) \le \ldots \le \beta_{n_-}(x) < 0
 < \beta_{n_-+1}(x) \le \ldots \le \beta_n(x),
 \qquad x \in [0, \ell],
\end{align}
and there exists $\theta > 0$ such that for each $k \in \oneto{n-1}$
\begin{equation} \label{eq:betak-betaj<-eps}
 \text{either} \quad \beta_k \equiv \beta_{k+1}
 \quad \text{or} \quad \beta_k(x) + \theta < \beta_{k+1}(x), \quad x \in [0, \ell].
\end{equation}
Here $n_- \in \{0, 1, \ldots, n\}$ is the number of negative functions among $\beta_1, \ldots, \beta_n$. Let us also set $n_+ := n - n_-$. See Remark~\ref{rem:beta.vs.wtbeta} for some discussion about these conditions.

Throughout this section we will also assume the following ``zero block diagonality'' condition on entries of the matrix function $Q(\cdot)$,
\begin{equation} \label{eq:Qjk=0.bj=bk}
 Q_{jk} \equiv 0 \quad\text{whenever}\quad \beta_j \equiv \beta_k,
 \qquad j,k \in \oneton.
\end{equation}
In particular $Q_{jj} \equiv 0$, $j \in \oneton$.
\begin{remark} \label{rem:beta.vs.wtbeta}
\textbf{(i)} Note, that as opposed to the previous section, we work with notation~\eqref{eq:Bx.def.intro} for $B(x)$.
To avoid confusion, we rewrite block-matrix decomposition~\eqref{eq:Bx.block.def}--\eqref{eq:Bkx.def} as
\begin{equation} \label{eq:Bx.block.wt.def}
 B =: \diag(\wt{\beta}_1 I_{n_1}, \ldots, \wt{\beta}_r I_{n_r}),
 \qquad n_1 + \ldots + n_r = n.
\end{equation}
It is clear that
\begin{align}
\label{eq:b=b=wtb1}
 & \beta_1 \equiv \ldots \equiv \beta_{n_1} \equiv \wt{\beta}_1, \\
 & \beta_{n_1+1} \equiv \ldots \equiv \beta_{n_1+n_2} \equiv \wt{\beta}_2, \\
\nonumber
 & \qquad \qquad \ldots \\
\label{eq:b=b=wtbn}
 & \beta_{n-n_r+1} \equiv \ldots \equiv \beta_n \equiv \wt{\beta}_r.
\end{align}
These relations imply that conditions~\eqref{eq:betak.alpk.Linf}--\eqref{eq:betak-betaj<-eps} on functions $\beta_1, \ldots, \beta_n$ from representation~\eqref{eq:Bx.def} are equivalent to conditions~\eqref{eq:betak.alpk.Linf.r}--\eqref{eq:beta.k+eps<beta.k+1.r} on functions $\beta_1, \ldots, \beta_r$ from representation~\eqref{eq:Bx.block.def}--\eqref{eq:Bkx.def}.

Note also that relations~\eqref{eq:b=b=wtb1}--\eqref{eq:b=b=wtbn} and condition~\eqref{eq:betak-betaj<-eps} imply that condition~\eqref{eq:Qjk=0.bj=bk} on $Q$ means that $Q$ has zero block diagonal with respect to decomposition $\bC^n = \bC^{n_1} \oplus \ldots \oplus \bC^{n_r}$.

\textbf{(ii)} Sometimes it is useful to work with the matrix function $B(\cdot)$ without ordering its entries (see e.g.~\eqref{eq:Tim.B}). Let us reformulate conditions~\eqref{eq:beta-+n}--\eqref{eq:betak-betaj<-eps} for such general ``unordered'' case. Namely, it is easy to verify that equivalent form of these conditions is the following,
\begin{align}
\label{eq:theta.exists}
 & \text{there exists}\ \ \theta > 0 \ \ \text{such that for each}
 \quad j, k \in \oneton \ \ \text{the following condition holds}, \\
\label{eq:nujk=><0}
 & \text{either} \quad \nu_{jk} := \beta_j - \beta_k \equiv 0,
 \quad\text{or}\quad \nu_{jk}(x) > \theta, \quad x \in [0,\ell],
 \quad\text{or}\quad \nu_{jk}(x) < -\theta, \quad x \in [0,\ell].
\end{align}

\end{remark}
Further, we set
\begin{equation} \label{eq:rhok.bk.def}
 \rho_k(x) := \int_0^x \beta_k(t) dt \quad \text{and} \quad
 b_k := \rho_k(\ell) \in \bR \setminus \{0\}, \qquad k \in \oneton.
\end{equation}
Going forward for $u \le v$, notations $[u, v]$ and $[v, u]$ will mean the same segment of real line and will be used interchangeably. It follows from~\eqref{eq:betak.alpk.Linf} that
\begin{align}
\label{eq:rho1.rhon}
 & \rho_1(x) \le \ldots \le \rho_{n_-}(x) < 0
 < \rho_{n_-+1}(x) \le \ldots \le \rho_n(x), \qquad x \in [0, \ell], \\
\label{eq:b1.bn}
 & b_1 \le \ldots \le b_{n_-} < 0 < b_{n_-+1} \le \ldots \le b_n, \\
\label{eq:rhok.Lip}
 & \rho_k \in \Lip[0,\ell], \quad \rho_k^{-1} \in \Lip[0,b_k], \quad
 \rho_k, \rho_k^{-1} \ \ \text{are strictly monotonous},
 \quad k \in \oneton,
\end{align}
where $\rho_k^{-1}$ denotes the function inverse to $\rho_k$ on the segment $[0, \ell]$, $k \in \oneton$. Here $\Lip[0,\alp]$ is the class of functions $f : [0,\alp] \to \bC$ satisfying $|f(x)-f(y)| \le C |x-y|$, $x, y \in [0, \alp]$, for some $C > 0$. As per remark above, $[0, \alp] = [\alp, 0]$ if $\alp < 0$.

Next, we introduce the fundamental matrices $\Phi(\cdot,\l)$ and $\Phi_0(\cdot,\l)$ as the solutions to the equations $\cL(Q) \Phi = \l \Phi$ and $\cL(0) \Phi_0 = \l \Phi_0$, respectively, satisfying the initial condition $\Phi(0,\l) = \Phi_0(0,\l) = I_n$. Clearly,
\begin{align}
\label{eq:Phixl.def}
 \Phi(x,\l) & = \begin{pmatrix} \Phi_1(x,\l) & \ldots &
 \Phi_n(x,\l) \end{pmatrix}, \qquad
 \Phi_p(x,\l) = \col(\phi_{jp}(x,\l), \ldots, \phi_{np}(x,\l)), \\
\label{eq:Phi0k.def}
 \Phi_0(x,\l) & = \begin{pmatrix} \Phi^0_{1}(x,\l) & \ldots &
 \Phi^0_{n}(x,\l) \end{pmatrix}, \qquad
 \Phi^0_{k}(x,\l) = e^{i \l \rho_k(x)} \col\(\delta_{1k},\ldots,\delta_{nk}\).
\end{align}
Note in this connection, that the matrix equation $\cL(Q) \Phi = \l \Phi$ is equivalent to
\begin{equation} \label{eq:Phi'}
 \Phi'(x,\l) = (i \l B(x) - Q(x)) \Phi(x,\l), \qquad x \in [0, \ell],
 \quad \l \in \bC,
\end{equation}
where $i \l B(\cdot) - Q(\cdot)$ is a summable function on $[0, \ell]$ for each $\l \in \bC$. General theory of ODE implies the existence of global solution $\Phi(x,\l)$ on segment $[0,\ell]$, such that $\Phi(\cdot,\l) \in \AC([0,\ell]; \bC^{n \times n})$ for each $\l \in \bC$ and $\Phi(x,\cdot)$ is an entire function for each $x \in [0,\ell]$. Moreover, Liouville's formula (see~\eqref{eq:detPhi} below) implies that $\Phi(\cdot,\l)^{-1} \in \AC([0,\ell]; \bC^{n \times n})$ for each $\l \in \bC$.
\subsection{Key identities for fundamental matrix}
\label{subsec:fundam.fourier}
In the following proposition we relate the columns of fundamental matrices $\Phi(\cdot,\l)$ and $\Phi_0(\cdot,\l)$.
\begin{proposition} \label{prop:Phip=Phip0+int.sum}
Let matrix functions $B(\cdot)$ and $Q(\cdot)$ satisfy conditions~\eqref{eq:Bx.def}--\eqref{eq:Qjk=0.bj=bk} and let $p \in \oneton$. Then, there exist vector kernels
\begin{equation} \label{eq:Rpq.in.X}
 R_q^{[p]} \in \(X_{1,0}(\Omega) \cap X_{\infty,0}(\Omega)\)
 \otimes \bC^{n}, \qquad q \in \oneton,
\end{equation}
such that the following representation holds
\begin{equation} \label{eq:Phip=Phi0p+int.sum}
 \Phi_p(x,\l) = \Phi_p^0(x,\l) + \sum_{q=1}^n \int_0^x R_q^{[p]}(x,t)
 e^{i \l \rho_q(t)} \beta_q(t) \,dt, \qquad x \in [0, \ell], \quad \l \in \bC.
\end{equation}
\end{proposition}
\begin{proof}
It is clear that matrix functions $B(\cdot)$ and $Q(\cdot)$ satisfy assumptions of Theorem~~\ref{th:transform.oper}. For simplicity let's assume that $r=n$ and $n_1 = \ldots = n_r = 1$ in block-matrix decomposition~\eqref{eq:Bx.block.def}--\eqref{eq:Bkx.def}. Then notations~\eqref{eq:Bx.block.def}--\eqref{eq:Bkx.def} and~\eqref{eq:Bx.def} coincide. Let $\wt{A}$ be some invertible matrix with non-zero entries:
\begin{equation}
 \wt{A} = (a_{jk})_{j,k=1}^n, \qquad \det(\wt{A}) \ne 0, \qquad a_{jk} \ne 0,
 \quad j, k \in \oneton.
\end{equation}
E.g. one can set $a_{jk} := j^k$ to obtain invertible Vandermonde matrix with non-zero entries. Denote by $A^{[k]}$, the $k$-th column of $\wt{A}$:
\begin{equation} \label{eq:wtA=A1.An}
 \wt{A} = \begin{pmatrix} A^{[1]} & \ldots & A^{[n]} \end{pmatrix}, \qquad
 A^{[k]} = \col(a_{1k}, \ldots, a_{nk}) \quad k \in \oneton.
\end{equation}
It is clear, that for a given $k \in \oneton$, $n \times 1$ matrix $A = A^{[k]}$ satisfy conditions of Theorem~\ref{th:transform.oper}. Hence triangular representation~\eqref{eq:trans_oper_repres-n} takes place with some
\begin{equation} \label{eq:KAk.def}
 K_{A^{[k]}} =: K^{[k]} =: \bigl(K_{jp}^{[k]}\bigr)_{j,p=1}^n \in \(X_{1,0}(\Omega)
 \cap X_{\infty,0}(\Omega)\) \otimes \bC^{n \times n},
 \qquad k \in \oneton.
\end{equation}
Further, note that due to Cauchy uniqueness theorem,
\begin{equation} \label{eq:YAk.eAk}
 Y_{A^{[k]}}(x,\l) = \Phi(x,\l) A^{[k]}, \qquad
 e_{A^{[k]}}(x,\l) = \Phi_0(x,\l) A^{[k]}, \qquad k \in \oneton,
\end{equation}
where $Y_{A}(x,\l)$ and $e_{A}(x,\l)$ for $A = A^{[k]}$ are defined in~\eqref{eq:YeA.def}--\eqref{eq:YeA.init}. Inserting~\eqref{eq:YAk.eAk} into~\eqref{eq:trans_oper_repres-n} we arrive at
\begin{equation} \label{eq:Phixl.Ak}
 \Phi(x,\l) A^{[k]} = \Phi_0(x,\l) A^{[k]}
 + \int^x_0 K^{[k]}(x,t) B(t) \Phi_0(t,\l) A^{[k]} \,dt,
 \qquad k \in \oneton.
\end{equation}
Formulas~\eqref{eq:Phixl.Ak} and~\eqref{eq:wtA=A1.An} now imply
\begin{equation} \label{eq:Phixl.wtA}
 \Phi(x,\l) \wt{A} = \Phi_0(x,\l) \wt{A}
 + \int^x_0 \(K^{[k]}(x,t) B(t) \Phi_0(t,\l) A^{[k]}\)_{k=1}^n \,dt.
\end{equation}
Let $\wt{A}^{-1} = (\alp_{kp})_{k,p=1}^n$. Then with account of notations~\eqref{eq:KAk.def},~\eqref{eq:Bx.def},~\eqref{eq:Phi0k.def} and~\eqref{eq:wtA=A1.An} we have for $0 \le t \le x \le 1$:
\begin{multline} \label{eq:KAk.B.Phi0.Ak.wtA.inv}
 \(K^{[k]}(x,t) B(t) \Phi_0(t,\l) A^{[k]}\)_{k=1}^n \wt{A}^{-1}
 = \(\sum_{q=1}^n K_{jq}^{[k]}(x,t) \beta_q(t) e^{i \l \rho_q(t)}
 a_{qk}\)_{j,k=1}^n \cdot (\alp_{kp})_{k,p=1}^n \\
 = \(\sum_{k,q=1}^n K_{jq}^{[k]}(x,t) \beta_q(t)
 e^{i \l \rho_q(t)} a_{qk} \alp_{kp}\)_{j,p=1}^n
 = \(\sum_{q=1}^n R_{jq}^{[p]}(x,t) e^{i \l \rho_q(t)} \beta_q(t)\)_{j,p=1}^n,
\end{multline}
where we set
\begin{equation} \label{eq:Rpjk.def}
 R_{jq}^{[p]}(x,t) := \sum_{k=1}^n K_{jq}^{[k]}(x,t) a_{qk} \alp_{kp},
 \qquad p,j,q \in \oneton, \quad 0 \le t \le x \le 1.
\end{equation}
Multiplying~\eqref{eq:Phixl.wtA} by $\wt{A}^{-1}$ from the right with account of~\eqref{eq:KAk.B.Phi0.Ak.wtA.inv} and~\eqref{eq:Rpjk.def} and taking $p$-th column in the resulting matrix equation, we arrive at the desired formula~\eqref{eq:Phip=Phi0p+int.sum} by setting
\begin{equation} \label{eq:Rqp.def}
 R_q^{[p]} := \col\bigl(R_{1q}^{[p]}, \ldots, R_{nq}^{[p]}\bigr)_{j=1}^n.
\end{equation}
Desired inclusion~\eqref{eq:Rpq.in.X} follows from inclusion~\eqref{eq:KAk.def} and formula~\eqref{eq:Rpjk.def}.

The general case can be treated as follows. Start with the matrix $\wt{A}$ as above. For each column $A^{[k]}$ of $\wt{A}$ extend it be an $n \times n_{\min}$ matrix $A$ with blocks of maximal rank. After applying Theorem~\ref{th:transform.oper}, pick only the first column of the solutions $Y_A(x,\l)$ and $e_A(x,\l)$ as $Y_{A^{[k]}}(x,\l)$ and $e_{A^{[k]}}(x,\l)$, respectively.
\end{proof}
To study the integrals appearing in formula~\eqref{eq:Phip=Phi0p+int.sum}, we need the following generalization of Riemann-Lebesgue Lemma for space $X_{\infty,0}(\Omega)$.
\begin{lemma} \label{lem:RiemLebX}
Let $\beta \in L^{\infty}([0,\ell]; \bR)$ and $\beta$ do not change sign on $[0,\ell]$.
Set $\rho(x) := \int_0^x \beta(t) dt$. Let $R \in X_{\infty,0}(\Omega)$. Then for any $\delta>0$ there exists $R_{\delta}> 0$ such that
\begin{equation} \label{eq:intGxt.exp}
 \abs{\int_0^x R(x,t) e^{i \l \rho(t)} \beta(t) dt} <
 \delta \cdot \(\abs{e^{i \l \rho(x)}} + 1\),
 \qquad |\l| > R_{\delta}, \quad x \in [0,\ell].
\end{equation}
\end{lemma}
\begin{proof}
Let $\eps > 0$. By the definition of the space $X_{\infty,0}(\Omega)$, the inclusion $R \in X_{\infty,0}(\Omega)$ ensures that there exists $R_{\eps} \in C^1(\Omega)$ such that
\begin{equation}
 \|R - R_{\eps} \|_{X_{\infty}} = \esssup_{x \in [0,\ell]}
 \int_0^x \abs{R(x,t) - R_{\eps}(x,t)} dt < \eps.
\end{equation}
In particular, we get the following uniform estimate
\begin{equation} \label{eq:int.R-Re}
 \abs{\int_0^x \(R(x,t) - R_{\eps}(x,t)\) e^{i \l \rho(t)} \beta(t) dt}
 \le \eps \cdot \|\beta\|_{\infty} \cdot
 \max_{t \in [0, x]} \abs{e^{i \l \rho(t)}},
 \quad x \in [0,\ell], \quad \l \in \bC.
\end{equation}
Since $R_{\eps} \in C^1(\Omega)$, $\rho \in \AC[0,\ell]$ and $\rho' = \beta$,
integrating by parts we get for $x \in [0,\ell]$ and $\l \ne 0$,
\begin{multline} \label{eq:wtKjk.ebk<C/lam}
 \abs{\int_0^x R_{\eps}(x,t) e^{i \l \rho(t)} \beta(t) \,dt}
 = \abs{\int_0^x \frac{R_{\eps}(x,t)}{i \l} \,d\(e^{i \l \rho(t)}\)}
 = \frac{1}{|\l|} \abs{\int_0^x e^{i \l \rho(t)}
 \frac{\partial}{\partial t}R_{\eps}(x,t) \,dt} \\
 \le \frac{1}{|\l|} \cdot \max_{t \in [0, x]} \abs{e^{i \l \rho(t)}}
 \cdot \max_{t \in [0,x]} \abs{\frac{\partial}{\partial t}R_{\eps}(x,t)}
 \le \frac{\|R_{\eps}\|_{C^1(\Omega)}}{|\l|} \cdot \max_{t \in [0, x]} \abs{e^{i \l \rho(t)}}.
\end{multline}
Since function $\rho(\cdot)$ is real-valued and strictly monotonous on $[0,\ell]$, and $\rho(0)=0$, it follows that
\begin{equation} \label{eq:max.exp<sum}
 \max_{t \in [0, x]} \abs{e^{i \l \rho(t)}} =
 \max \{\abs{e^{i \l \rho(x)}}, 1\} < \abs{e^{i \l \rho(x)}} + 1.
\end{equation}
Setting
\begin{equation}
 \eps = \frac{\delta}{2 \|\beta\|_{\infty}}, \qquad
 R_{\delta} = \frac{2}{\delta} \|R_{\eps}\|_{C^1(\Omega)},
\end{equation}
and combining estimates~\eqref{eq:int.R-Re}--\eqref{eq:max.exp<sum} we arrive at the desired estimate~\eqref{eq:intGxt.exp}.
\end{proof}
Going forward, for $h \ge 0$ we will denote by $\Pi_h$ a horizontal strip of semi-width $h$ symmetrical with respect to the real line:
\begin{equation} \label{eq:Pih.def}
 \Pi_h := \{\l \in \bC : |\Im \l| \le h\}, \qquad \Pi_0 := \bR.
\end{equation}
\begin{corollary} \label{cor:Phip=Phip0+o}
Let matrix functions $B(\cdot)$ and $Q(\cdot)$ satisfy conditions~\eqref{eq:Bx.def}--\eqref{eq:Qjk=0.bj=bk}. Let $h \ge 0$ and $p \in \oneton$. Then the following uniform asymptotic formula holds:
\begin{equation} \label{eq:Phip=Phip0+o}
 \Phi_p(x,\l) = \Phi_p^0(x,\l) + o(1), \quad x \in [0,\ell], \quad\text{as}\quad
 \l \to \infty, \quad \l \in \Pi_h.
\end{equation}
\end{corollary}
\begin{proof}
Proposition~\ref{prop:Phip=Phip0+int.sum} implies representation~\eqref{eq:Phip=Phi0p+int.sum}. Due to condition~\eqref{eq:Rpq.in.X}, Lemma~\ref{lem:RiemLebX} implies the following uniform at $x \in [0, \ell]$ asymptotic behavior for integrals in~\eqref{eq:Phip=Phi0p+int.sum}:
\begin{equation} \label{eq:int.exp.rhop=o}
 \int_0^x R_q^{[p]}(x, t) e^{i \l \rho_p(t)} \beta_p(t) \,dt
 = o\(\abs{e^{i \l \rho_q(x)}} + 1\),
 \quad\text{as}\quad \l \to \infty, \quad \l \in \bC,
 \qquad q \in \oneton.
\end{equation}
Since functions $\rho_q(\cdot)$ are real-valued and bounded, these estimates imply that all integrals in~\eqref{eq:Phip=Phi0p+int.sum} are $o(1)$ uniformly at $x \in [0, \ell]$ as $\l \to \infty$ and $\l \in \Pi_h$, which finishes the proof.
\end{proof}
The following representation will be useful for studying characteristic determinant of the BVP~\eqref{eq:LQ.def.reg}--\eqref{eq:Uy=0}. Let us set
\begin{equation} \label{eq:rho-+def}
 \rho_1^-(x) = \min\{\rho_1(x), 0\}, \qquad
 \rho_n^+(x) = \max\{\rho_n(x), 0\}, \qquad x \in [0, \ell].
\end{equation}
It follow from~\eqref{eq:rho1.rhon} that
\begin{equation} \label{eq:rho-.rho+}
 \rho_1^-(x) \le 0 \le \rho_n^+(x), \qquad
 [0, \rho_q(x)] \subset [\rho_1^-(x), \rho_n^+(x)],
 \quad x \in [0, \ell], \quad q \in \oneton.
\end{equation}
Here $[0, \rho_q(x)]$ means $[\rho_q(x), 0]$ if $\rho_q(x) < 0$.
\begin{corollary} \label{cor:Phip=Phip0+int1n}
Let matrix functions $B(\cdot)$ and $Q(\cdot)$ satisfy conditions~\eqref{eq:Bx.def}--\eqref{eq:Qjk=0.bj=bk}. Let $p \in \oneton$. Then there exists a measurable vector kernel $\wt{R}_p$ defined on $\wt{\Omega} := \{(x,u) : x \in [0, \ell], u \in [\rho_1^-(x), \rho_n^+(x)]\}$ such that for each $x \in [0, \ell]$, a trace function $\wt{R}_p(x, \cdot)$ is well-defined, summable,
\begin{equation} \label{eq:sup.int.wtR}
 \sup_{x \in [0,\ell]} \int_{\rho_1^-(x)}^{\rho_n^+(x)}
 \norm{\wt{R}_p(x,u)}_{\bC^n} \,du < \infty,
\end{equation}
and the following representation holds
\begin{equation} \label{eq:Phip=Phi0p+int1n}
 \Phi_p(x,\l) = \Phi_p^0(x,\l) + \int_{\rho_1^-(x)}^{\rho_n^+(x)} \wt{R}_p(x,u)
 e^{i \l u} \,du, \qquad x \in [0, \ell], \quad \l \in \bC.
\end{equation}
\end{corollary}
\begin{proof}
Proposition~\ref{prop:Phip=Phip0+int.sum} implies representation~\eqref{eq:Phip=Phi0p+int.sum}. Let $q \in \oneton$ be fixed and consider the integral in~\eqref{eq:Phip=Phi0p+int.sum} that contains $R_q^{[p]}(x,t)$. Condition~\eqref{eq:rhok.Lip} allows us to make a change of variable $u = \rho_q(t)$ in this integral, which with account of~\eqref{eq:rho-.rho+} yields
\begin{equation} \label{eq:intRqp}
 \int_0^x R_q^{[p]}(x,t) e^{i \l \rho_q(t)} \beta_q(t) \,dt
 = \int_0^{\rho_q(x)} R_q^{[p]}(x, \rho_q^{-1}(u)) e^{i \l u} \,du
 = \int_{\rho_1^-(x)}^{\rho_n^+(x)} \wt{R}_q^{[p]}(x, u) e^{i \l u} \,du,
\end{equation}
where
\begin{equation} \label{eq:wtRqp.def}
 \wt{R}_q^{[p]}(x, u) = \begin{cases}
 R_q^{[p]}(x, \rho_q^{-1}(u)), & u \in [0, \rho_q(x)], \\
 0, & u \in [\rho_1^-(x), \rho_n^+(x)] \setminus [0, \rho_q(x)].
 \end{cases}
\end{equation}
Inserting~\eqref{eq:intRqp} into~\eqref{eq:Phip=Phi0p+int.sum} we arrive at the desired formula~\eqref{eq:Phip=Phi0p+int1n} with ${\wt{R}_p := \wt{R}_1^{[p]} + \ldots + \wt{R}_n^{[p]}}$.
\end{proof}
\begin{remark} \label{rem:RiemLeb.uniform}
Using Riemann-Lebesgue Lemma, one can prove that for each $x \in [0,\ell]$, the integral in the representation~\eqref{eq:Phip=Phi0p+int1n} tends to zero as $\l \to \infty$, $\l \in \Pi_h$ (without using property~\eqref{eq:sup.int.wtR}). Note, however, that to prove this convergence \emph{uniformly} at $x \in [0,\ell]$, property~\eqref{eq:sup.int.wtR} of the vector kernel $\wt{R}_p$ from representation~\eqref{eq:Phip=Phi0p+int1n} is not sufficient by itself. Which is why in Corollary~\ref{cor:Phip=Phip0+o} we used more suitable representation~\eqref{eq:Phip=Phi0p+int.sum} where vector kernels $R_q^{[p]}$ known to be approximated by functions from $C^1(\Omega)$.
\end{remark}
\subsection{Generalization of Liouville's formula}
\label{subsec:liouville}
A classical Liouville's formula applied to the fundamental matrix solution $\Phi(\cdot,\l)$ of equation~\eqref{eq:LQ.def.fund} yields
\begin{equation} \label{eq:detPhi'}
 \frac{d}{dx} \det\Phi(x,\l) = \tr(i \l B(x) - Q(x)) \cdot \det\Phi(x,\l),
 \qquad x \in [0,\ell], \quad \l \in \bC,
\end{equation}
which in turn implies
\begin{equation} \label{eq:detPhi}
 \det \Phi(x,\l) = \exp\(i \l \int_0^x \tr B(t) \,dt - \int_0^x \tr Q(t) \,dt\),
 \qquad x \in [0,\ell], \quad \l \in \bC.
\end{equation}
If matrix function $Q(\cdot)$ satisfies ``zero block diagonality'' condition~\eqref{eq:Qjk=0.bj=bk}, then formula~\eqref{eq:detPhi} simplifies,
\begin{equation} \label{eq:detPhi.Qjj=0}
 \det \Phi(x,\l) = \exp\(i \l \int_0^x \tr B(t) \,dt\)
 = \exp(i \l (\rho_1(x) + \ldots + \rho_n(x))),
 \qquad x \in [0,\ell], \quad \l \in \bC.
\end{equation}

Our goal is to obtain similar relation for minors of the fundamental matrix $\Phi(\cdot, \l)$. It appears, that the set of minors of a given size with a fixed set of columns considered as a vector function in certain $\bC^N$ satisfies equation similar to~\eqref{eq:LQ.def.fund}. After that, applying previous considerations allows us to obtain integral representation for these minors, which will be crucial to effectively study characteristic determinant of BVP~\eqref{eq:LQ.def.intro}--\eqref{eq:Uy=0.intro}.

Throughout this subsection we extensively use notation~\eqref{eq:cPm.def} for the set $\fP_m$ and notation~\eqref{eq:Aqp.def} for a minor $\cA[\fp,\fq]$. Additionally we set for $x \in [0,\ell]$ and $m \in \oneton$,
\begin{align}
\label{eq:beta.fp.def}
 \beta_{\fq}(x) & := \beta_{q_1}(x) + \ldots + \beta_{q_m}(x),
 \qquad \fq = (q_1, \ldots, q_m) \in \fP_m,
 \quad 1 \le q_1 < \ldots < q_m \le n, \\
\label{eq:rho.fp.def}
 \rho_{\fq}(x) & := \rho_{q_1}(x) + \ldots + \rho_{q_m}(x),
 \qquad \fq \in \fP_m. \\
\label{eq:fhm-}
 \tau_m^-(x) & := \min\{\rho_1(x) + \ldots + \rho_m(x), 0\}
 = \min\(\{\rho_{\fq}(x) : \fq \in \fP_m\} \cup \{0\}\), \\
\label{eq:fhm+}
 \tau_m^+(x) & := \max\{\rho_{n-m+1}(x) + \ldots + \rho_n(x), 0\}
 = \max\(\{\rho_{\fq}(x) : \fq \in \fP_m\} \cup \{0\}\).
\end{align}
\begin{proposition} \label{prop:minor.form}
Let matrix functions $B(\cdot)$ and $Q(\cdot)$ satisfy conditions~\eqref{eq:Bx.def}--\eqref{eq:Qjk=0.bj=bk}. Let $m \in \oneton$ and $\fq, \fp \in \fP_m$. Then there exists a measurable scalar kernel $R_{\fq,\fp}$ defined on $\Omega_m := \{(x,u) : x \in [0, \ell], u \in [\tau_m^-(x), \tau_m^+(x)]\}$ such that for each $x \in [0, \ell]$, a trace function $R_{\fq,\fp}(x, \cdot)$ is well-defined, summable,
\begin{equation} \label{eq:R.fq.fp.in.X}
 \sup_{x \in [0,\ell]} \int_{\tau_m^-(x)}^{\tau_m^+(x)}
 \abs{R_{\fq,\fp}(x,u)} \,du < \infty,
\end{equation}
and the following representation holds
\begin{equation} \label{eq:det.Phipq}
 \Phi(x,\l)[\fq,\fp]
 = \delta_{\fq,\fp} \exp\(i \l \rho_{\fq}(x)\) +
 \int_{\tau_m^-(x)}^{\tau_m^+(x)} R_{\fq,\fp}(x,u) e^{i \l u} du,
 \qquad x \in [0, \ell], \quad \l \in \bC.
\end{equation}
\end{proposition}
\begin{proof}
Let $\l \in \bC$ and $\fp = (p_1, \ldots, p_m) \in \fP_m$ be fixed for the entire proof. This allows us to set for brevity
\begin{equation} \label{eq:ffq.def}
 f_{\fq}(x) := \Phi(x,\l)[\fq,\fp],
 \qquad q \in \fP_m, \quad x \in [0,\ell].
\end{equation}
Further, set $N := N_m := \card \fP_m = \binom{n}{m}$ and order all elements of $\fP_m$ in some way, $\fP_m = \{\fq_1, \ldots, \fq_N\}$. Consider
\begin{align} \label{eq:fFfp.def}
 \fF_{\fp}(\cdot,\l)
 := \col\bigl(\Phi(\cdot,\l)[\fq,\fp]\bigr)_{\fq \in \fP_m}
 := \col\bigl(f_{\fq_1}(\cdot), \ldots, f_{\fq_N}(\cdot)\bigr)
\end{align}
as a vector function in $\bC^N$. First, note that
\begin{equation} \label{eq:fF.fp.0}
 \fF_{\fp}(0,\l) := \col\bigl(I_n[\fq,\fp]\bigr)_{\fq \in \fP_m}
 = \col\bigl(\delta_{\fq,\fp}\bigr)_{\fq \in \fP_m},
\end{equation}
since $\bigl(I_n[\fp,\fp] = 1$ and $\bigl(I_n[\fq,\fp] = 0$ if $\fq \ne \fp$.

Let us show that vector function $\fF_{\fp}(\cdot,\l)$ satisfy the following first order system of ODE:
\begin{equation} \label{eq:fF.system}
 \fF_p'(x,\l) = i \l \fB(x) \fF_p(x,\l) - \fQ(x) \fF_p(x,\l), \qquad
 \fB(x) := \diag \bigl(\beta_{\fq}(x)\bigr)_{\fq \in \fP_m},
 \qquad x \in [0, \ell],
\end{equation}
where $\fQ(x) = \bigl(\fQ_{\fq,\fr}(x)\bigr)_{\fq,\fr \in \fP_m} = \bigl(\fQ_{\fq_j,\fq_k}(x)\bigr)_{j,k=1}^N$ is summable $N \times N$ matrix function with zero block diagonal with respect to natural block decomposition of the matrix $\fB(x)$. Namely,
\begin{equation} \label{eq:fQ.diag=0}
 \fQ_{\fq,\fr} \equiv 0 \quad\text{whenever}\quad \beta_{\fq} \equiv \beta_{\fr}
 \quad\text{for}\quad \fq,\fr \in \fP_m.
\end{equation}
Emphasize, that each entry of the matrix function $\fB(\cdot)$ is a sum of some subset of the entries of the original matrix $B(x)$ \emph{with different indexes}.

To prove~\eqref{eq:fF.system}--\eqref{eq:fQ.diag=0}, note that the matrix equation~\eqref{eq:Phi'} has the following scalar form
\begin{equation} \label{eq:phiqp'}
 \phi_{q,p}'(x,\l) = i \l \beta_q(x) \phi_{q,p}(x,\l)
 - \sum_{s=1}^n Q_{q,s}(x) \phi_{s,p}(x, \l), \qquad q,p \in \oneton,
 \quad x \in [0,\ell].
\end{equation}
Let $\fq = (q_1, \ldots, q_m) \in \fP_m$ be fixed, $1 \le q_1 < \ldots < q_m \le n$. Using standard formula for the derivative of the determinant we get
\begin{align}
\nonumber
 f_{\fq}'(x)
 & = \frac{d}{dx} \det(\phi_{q_j,\,p_k}(x,\l))_{j,k=1}^m
 = \frac{d}{dx} \det \begin{pmatrix}
 \phi_{q_1,\,p_1}(x,\l) & \ldots & \phi_{q_1,\,p_m}(x,\l) \\
 \vdots & \ddots & \vdots \\
 \phi_{q_m,\,p_1}(x,\l) & \ldots & \phi_{q_m,\,p_m}(x,\l) \\
 \end{pmatrix} \\
\label{eq:f.f1'}
 & = \sum_{j=1}^m \det \begin{pmatrix}
 \phi_{q_1,\,p_1}(x,\l) & \ldots & \phi_{q_1,\,p_m}(x,\l) \\
 \vdots & \ldots & \vdots \\
 \phi_{q_{j-1},\,p_1}(x,\l) & \ldots & \phi_{q_{j-1},\,p_m}(x,\l) \\
 \phi_{q_j,\,p_1}'(x,\l) & \ldots & \phi_{q_j,\,p_m}'(x,\l) \\
 \phi_{q_{j+1},\,p_1}(x,\l) & \ldots & \phi_{q_{j+1},\,p_m}(x,\l) \\
 \vdots & \ldots & \vdots \\
 \phi_{q_m,\,p_1}(x,\l) & \ldots & \phi_{q_m,\,p_m}(x,\l) \\
 \end{pmatrix}
\end{align}
Formula~\eqref{eq:phiqp'} implies the following relations connecting rows of the minor $\Phi(x,\l)[\fq,\fp]$,
\begin{equation} \label{eq:phiqp'.row}
 \bigl(\phi_{q_j,\,p_k}'(x,\l)\bigr)_{k=1}^m
 = i \l \beta_{q_j}(x) \cdot \bigl(\phi_{q_j,\,p_k}(x,\l)\bigr)_{k=1}^m
 - \sum_{s=1}^n Q_{q_j,\,s}(x) \cdot \bigl(\phi_{s,\,p_k}(x, \l)\bigr)_{k=1}^m,
 \quad x \in [0,\ell].
\end{equation}
For $j \in \onetom$ and $s \in \oneton$ denote by $\fq(q_j \to s)$ a sequence one obtains from $\fq$ by replacing $j$-th element $q_j$ with $s$, i.e.
\begin{equation} \label{eq:fq.qj.s}
 \fq(q_j \to s) := (q_1, \ldots, q_{j-1}, s, q_{j+1}, \ldots, q_m).
\end{equation}
Not that $\fq(q_j \to s)$ is not necessarily an element of $\fP_m$, but notation $A[\fq(q_j \to s),\fp]$ is still valid. Note also that it is possible for $s$ to be equal to one of $q_k$, $k \ne j$. In this case minor $A[\fq(q_j \to s),\fp]$ has duplicate rows and is necessarily zero.

With account of notation~\eqref{eq:Aqp.def} for $A[\fq,\fp]$ and notation~\eqref{eq:fq.qj.s} for $\fq(q_j \to s)$, inserting~\eqref{eq:phiqp'.row} into~\eqref{eq:f.f1'} we arrive at
\begin{equation} \label{eq:f.fq'.sums}
 f_{\fq}'(x) = \sum_{j=1}^m i \l \beta_{q_j}(x) f_{\fq}(x)
 - \sum_{j=1}^m \sum_{s=1}^n Q_{q_j, s}(x) \cdot \Phi(x,\l)[\fq(q_j \to s),\fp],
 \qquad x \in [0,\ell].
\end{equation}
Let $j \in \onetom$ and $s \in \oneton$ be fixed in addition to $\fq$ we fixed above. It is clear that
\begin{equation} \label{eq:det.Phi.qjsp=0}
 \Phi(\cdot,\l)[\fq(q_j \to s),\fp] \equiv 0 \quad\text{whenever}\quad
 s = q_k \quad\text{for some}\quad k \ne j.
\end{equation}
And if it is not the case, ordering elements of $\fq(q_j \to s)$, we arrive at a sequence $\wt{\fq}(q_j \to s) \in \fP_m$. Hence, in this case
\begin{equation} \label{eq:det.Phi.qjsp.sort}
 \Phi(\cdot,\l)[\fq(q_j \to s),\fp]
 = \sigma(\fq,j,s) f_{\wt{\fq}(q_j \to s)}(\cdot),
\end{equation}
where $\sigma(\fq,j,s) = \pm 1$ is a signature of the permutation behind the ordering of the sequence $\fq(q_j \to s)$.

Further, note that $Q_{j,k} \equiv 0$ whenever $\beta_j \equiv \beta_k$. Hence $Q_{q_j,q_j} \equiv 0$, $j \in \onetom$, and so $s = q_j$ can be excluded from the summation in~\eqref{eq:f.fq'.sums}. With account of this observation, definition~\eqref{eq:beta.fp.def} of $\beta_{\fq}(\cdot)$ and relations~\eqref{eq:det.Phi.qjsp=0}--\eqref{eq:det.Phi.qjsp.sort}, we can rewrite~\eqref{eq:f.fq'.sums} as
\begin{equation} \label{eq:f.fq'.final}
 f_{\fq}'(x) = i \l \beta_{\fq}(x) f_{\fq}(x)
 - \sum_{j=1}^m \sum_{\genfrac{}{}{0pt}{2}{s=1}{s \notin \fq}}^n
 \sigma(\fq,j,s) Q_{q_j, s}(x) f_{\wt{\fq}(q_j \to s)}(x),
 \qquad x \in [0,\ell], \quad \fq \in \fP_m,
\end{equation}
which coincides with~\eqref{eq:fF.system} if we set
\begin{equation} \label{eq:fQ.fq.fr.def}
 \fQ_{\fq,\fr}(\cdot) = \begin{cases}
 \sigma(\fq,j,s) Q_{q_j, s}(\cdot), & \text{if} \ \ \fr = \wt{\fq}(q_j \to s)
 \ \ \text{for some} \ \ j \in \onetom
 \ \ \text{and} \ \ s \in \oneton \setminus \fq, \\
 0, & \text{otherwise}.
 \end{cases}
\end{equation}
Here and in~\eqref{eq:f.fq'.final}, for simplicity, we identified sequence $\fq$ with the corresponding set $\{q_j\}_{j=1}^m$.

Let us verify condition~\eqref{eq:fQ.diag=0}. It is clear, from~\eqref{eq:fQ.fq.fr.def} that we only need to consider the case when $\fr = \wt{\fq}(q_j \to s)$ for some $j \in \onetom$ and $s \in \oneton \setminus \fq$. By definition of $\fq(q_j \to s)$ and $\wt{\fq}(q_j \to s)$ it is clear that
\begin{equation}
 \beta_{\fr} = \beta_{q_1} + \ldots + \beta_{q_{j-1}} + s + \beta_{q_{j+1}}
 + \beta_{q_n} = \beta_{\fq} - \beta_{q_j} + \beta_s.
\end{equation}
Hence if $\beta_{\fr} \equiv \beta_{\fq}$, then $\beta_{q_j} \equiv \beta_s$, which implies that $Q_{q_j,\, s} \equiv 0$ by the corresponding condition~\eqref{eq:Qjk=0.bj=bk} on $Q$. Therefore, formula~\eqref{eq:fQ.fq.fr.def} implies that $\fQ_{\fq,\fr} \equiv 0$.

In conclusion, vector function $\fF_{\fp}(\cdot, \l)$ is a solution of the first order system of ODE~\eqref{eq:fF.system} that satisfies initial condition~\eqref{eq:fF.fp.0}. Moreover, potential matrix function $\fQ(\cdot)$ in~\eqref{eq:fF.system} is summable and satisfies ``zero block diagonality'' condition~\eqref{eq:fQ.diag=0}, while entries of the diagonal matrix function $\fB(\cdot)$ satisfy conditions~\eqref{eq:betak.alpk.Linf}--\eqref{eq:betak-betaj<-eps} if we rewrite them appropriately. Hence, all previous considerations of Sections~\ref{sec:transform.oper} and~\ref{sec:fundament} apply to solutions of system~\eqref{eq:fF.system}. In particular, Corollary~\ref{cor:Phip=Phip0+int1n} immediately yields formula~\eqref{eq:det.Phipq} if we compare definition~\eqref{eq:rho-+def} of $\rho_1^-(\cdot)$, $\rho_n^+(\cdot)$ and definition~\eqref{eq:rho.fp.def}--\eqref{eq:fhm+} of $\tau_m^{\pm}(\cdot)$. The proof is now complete.
\end{proof}
\begin{remark}
If $m=n$, then $\fP_m = \fP_n$ has exactly one element $\fp_0 := (1, \ldots, n)$. Moreover, $\Phi(x,\l)[\fq,\fp] = \det \Phi(x,\l)$ for $\fq = \fp = \fp_0$. Hence system~\eqref{eq:fF.system} turns into~\eqref{eq:detPhi'}. This shows that Proposition~\ref{prop:minor.form} contains Liouville's formula as a partial case.
\end{remark}
\section{Regular and strictly regular boundary conditions} \label{sec:regular.bc}
Results of the previous sections about solutions of the equation~\eqref{eq:LQ.def.intro} allow us to obtain many spectral properties of the corresponding BVP~\eqref{eq:LQ.def.intro}--\eqref{eq:Uy=0.intro}.
\subsection{General properties of BVP}
\label{subsec:bvp.props}
For the reader's convenience let us recall BVP~\eqref{eq:LQ.def.intro}--\eqref{eq:Uy=0.intro},
\begin{align}
\label{eq:LQ.def.reg}
 & \cL(Q) y := -i B(x)^{-1} (y' + Q(x) y) = \l y,
 \qquad y = \col(y_1, \ldots, y_n), \quad x \in [0,\ell], \\
\label{eq:Uy=0}
 & U(y) := C y(0) + D y(\ell) = 0, \quad\text{and}\quad \rank(C \ D) = n.
\end{align}
Note that the condition $\rank(C \ D) = n$ is equivalent to $\ker(CC^*+DD^*)=\{0\}$.

Emphasize that a pair of matrices $\{C, D\}$ in boundary conditions~\eqref{eq:Uy=0} is not unique. Indeed, two pairs $\{C, D\}$ and $\{\wh{C}, \wh{D}\}$ determine the same boundary conditions if and only if $\{\wh{C}, \wh{D}\} = \{XC, XD\}$ with some nonsingular $X \in \bC^{n\times n}$. In Lemma~\ref{lem:CD.canon} we present ``canonical'' form for matrices in boundary conditions~\eqref{eq:Uy=0} which is important in applications.

Let us introduce the Hilbert space $\fH$ as follows,
\begin{equation} \label{eq:fH.def}
 \fH := \oplus_{k=1}^n \fH_k, \qquad \fH_k := L^2_{|\beta_k|}[0,\ell],
 \quad k \in \oneton,
\end{equation}
i.e.\ for $f = \col(f_1, \ldots, f_n)$ and $g = \col(g_1, \ldots, g_n)$ we have
\begin{align}
\label{eq:fg.fH.def}
 & (f, g)_{\fH} := \int_0^{\ell} \angnorm{|B(x)| f(x), g(x)} dx =
 \sum_{k=1}^n (f_k, g_k)_{\fH_k}, \quad\text{where}\quad
 \angnorm{\cdot, \cdot} := \angnorm{\cdot, \cdot}_{\bC^n} \quad\text{and} \\
\label{eq:fHk.def}
 & (f_k, g_k)_{\fH_k} := \int_0^{\ell} f_k(x) \ol{g_k(x)} |\beta_k(x)| dx,
 \qquad f \in \fH_k \ \Leftrightarrow \
 \int_0^{\ell} |f_k(x)|^2 |\beta_k(x)| dx < \infty.
\end{align}
With BVP~\eqref{eq:LQ.def.reg}--\eqref{eq:Uy=0} one naturally associates Dirac-type operator $L_U(Q)$ in the Hilbert space $\fH$ as follows,
\begin{align}
\label{eq:LQU.def}
 & L_U(Q) y = \cL(Q) y = -i B^{-1} (y' + Q y),
 \qquad y \in \dom(L_U(Q)), \quad\text{where} \\
\label{eq:dom.LQU.def}
 & \dom(L_U(Q)) := \{y \in \AC([0,\ell]; \bC^n) \ : \ \cL(Q) y \in \fH, \ \
 \ C y(0) + D y(\ell) = 0 \}.
\end{align}
Alongside equation~\eqref{eq:LQ.def.reg} we consider the same equation but with $Q=0$,
\begin{equation} \label{eq:L0.def.reg}
 \cL_0 y := \cL(0) y := -i B(x)^{-1} y' = \l y,
 \qquad y = \col(y_1, \ldots, y_n), \quad x \in [0,\ell], \\
\end{equation}
and with same boundary conditions~\eqref{eq:Uy=0}. Similarly to $L_U(Q)$, we associate the unperturbed Dirac-type operator $L_{0,U} := L_U(0)$ in $\fH$ with BVP~\eqref{eq:L0.def.reg},~\eqref{eq:Uy=0}.

Recall that $b_k$ and $\rho_k(\cdot)$ are defined in~\eqref{eq:rhok.bk.def} via functions $\beta_k(\cdot)$, $k \in \oneton$. For most of the results in this section we will only assume the following relaxed conditions on functions $\beta_k(\cdot)$,
\begin{equation} \label{eq:betak.L1}
 \beta_k \in L^1([0,\ell], \bR),
 \qquad s_k := \sign(\beta_k(\cdot)) = \const \ne 0,
 \qquad k \in \oneton,
\end{equation}
which implies the following condition on functions $\rho_k(\cdot)$,
\begin{equation} \label{eq:rhok.AC}
 \rho_k \in \AC[0,\ell] \quad\text{and is strictly monotonous},
 \quad k \in \oneton.
\end{equation}
Since both matrix functions $B(\cdot)$ and $Q(\cdot)$ are summable, one can define fundamental matrix solutions $\Phi_Q(\cdot,\l) := \Phi(\cdot,\l)$ and $\Phi_0(\cdot,\l)$ of the equations~\eqref{eq:LQ.def.reg} and~\eqref{eq:L0.def.reg} via formulas~\eqref{eq:Phixl.def} and~\eqref{eq:Phi0k.def}, respectively.

Next we set
\begin{align}
\label{eq:A.def}
 A_Q(\l) & := A(\l) := C + D \Phi_Q(\ell, \l) =: (a_{jk}(\l))_{j,k=1}^n,
 \quad \text{where} \quad
 a_{jk}(\l) = c_{jk} + \sum_{p=1}^n d_{jp} \phi_{pk}(\l), \\
\label{eq:A0.def}
 A_0(\l) & := C + D \Phi_0(\ell, \l) =: (a_{jk}^0(\l))_{j,k=1}^n,
 \quad \text{where} \quad a_{jk}^0(\l) = c_{jk} + d_{jk} e^{i \l b_k},
\end{align}
where $\phi_{pk}(\l) := \phi_{pk}(\ell, \l)$ is the corresponding entry of the matrix $\Phi_Q(\ell, \l)$. Finally, we introduce the characteristic determinants of the problems~\eqref{eq:LQ.def.reg}--\eqref{eq:Uy=0} and~\eqref{eq:L0.def.reg},~\eqref{eq:Uy=0} by setting
\begin{equation} \label{eq:Delta.def}
 \Delta(\l) := \Delta_Q(\l) := \det(A_Q(\l)), \qquad \Delta_0(\l) := \det(A_0(\l)),
 \qquad \l \in \bC,
\end{equation}
respectively. The role of the characteristic determinant $\Delta(\cdot)$ in the spectral theory of BVP~\eqref{eq:LQ.def.reg}--\eqref{eq:Uy=0} becomes clear
from the following simple statement. To state it we denote by
\begin{align}
\label{eq:Aa.def}
 A^a(\l) & := A^a_Q(\l) =: (A_{jk}(\l))_{j,k=1}^n
 \ \ \ \text{the matrix adjugate to}\ \ \ A_Q(\l), \\
\label{eq:Aa0.def}
 A_0^a(\l) & =: (A_{jk}^0(\l))_{j,k=1}^n
 \ \ \ \text{the matrix adjugate to}\ \ \ A_0(\l).
\end{align}
With account of notations~\eqref{eq:cAa.A.fpk.fpj}--\eqref{eq:fpk.def} we see that for $j, k \in \oneton$,
\begin{align}
\label{eq:Aajk.def}
 A_{jk}(\l) & = A(\l)\{j,k\} = (-1)^{j+k} A(\l)[\fp_k, \fp_j], \\
\label{eq:Aajk0.def}
 A_{jk}^0(\l) & = A^0(\l)\{j,k\} = (-1)^{j+k} A^0(\l)[\fp_k, \fp_j].
\end{align}
\begin{lemma} \label{lem:eigen}
Let matrix functions $B(\cdot)$ and $Q(\cdot)$ satisfy conditions~\eqref{eq:Bx.def}--\eqref{eq:Q=Qjk.def}, i.e.\ $B, Q \in L^1([0,\ell]; \bC^{n \times n})$ and $B(x)$ is invertible for almost all $x$. Number $\l \in \bC$ is an eigenvalue of the operator $L_U(Q)$ given by~\eqref{eq:dom.LQU.def} if and only if $\Delta_Q(\l) = 0$. Moreover, the algebraic multiplicity $m_a(\l)$ of $\l$ coincides with the multiplicity of $\l$ as a root of the characteristic determinant $\Delta_Q(\cdot)$. In particular, $\dim\cR_{\l}(L_U(Q)) = 1$ if and only if $\Delta_Q(\l)=0$ and $\Delta_Q'(\l) \ne 0$. Moreover, in the later case, $\rank(A^a_Q(\l)) = 1$ and there exists $p \in \oneton$ such that
\begin{equation} \label{eq:Ypxl.def}
 y(x,\l) := Y_p(x,\l) := \sum_{k=1}^n A_{kp}(\l) \Phi_k(x,\l) \ne 0,
\end{equation}
is the (non-trivial) eigenvector of the operator $L_U(Q)$ corresponding to the eigenvalue $\l$.

In addition, if $Q=0$, then the following explicit formula holds,
\begin{equation} \label{eq:Yp0xl.def}
 y_0(x,\l) := Y_p^0(x,\l) := \sum_{k=1}^n A_{kp}^0(\l) \Phi_k^0(x,\l)
 = \col\(A_{1p}^0(\l) e^{i \l \rho_1(x)}, \ldots,
 A_{np}^0(\l) e^{i \l \rho_n(x)}\) \ne 0.
\end{equation}
\end{lemma}
\begin{proof}
The proof of general statement is similar to the proof of~\cite[Theorem 1.2, step (i)]{MalOri12}. Namely, one can show that if $\l$ is an $m$-multiple zero of the function $\Delta_Q(\cdot)$, then the system of functions
\begin{equation} \label{eq:DpYj}
 \left\{\left. \frac{\partial^k}{\partial \mu^k} Y_p(x,\mu) \right|_{\mu=\l}
 : \ \ k \in \{0,1,\ldots,m-1\},\ \ p \in \oneton \right\}
\end{equation}
spans the root subspace $\cR_{\l}(L_U(Q))$ of the operator $L_U(Q)$. Relation $\dim \cR_{\l}(L_U(Q)) = m$ can be proved similarly to how it was done for ordinary differential operators in~\cite{Nai69}.

It remains to consider the case $\dim\cR_{\l}(L_U(Q)) = 1$. Then $\l \in \bC$ is the eigenvalue of the problem~\eqref{eq:LQ.def.reg}--\eqref{eq:Uy=0} of geometric and algebraic multiplicity one, hence $\Delta_Q(\l) = 0$ and $\Delta_Q'(\l) \ne 0$. Jacobi's formula~\eqref{eq:jacobi.def} implies that
\begin{equation}
 \tr \( A_Q^a(\l) A_Q'(\l)\) = \Delta_Q'(\l) \ne 0,
\end{equation}
which in turn implies that $A_Q^a(\l) \ne 0$. Therefore, it follows from the identity
\begin{equation} \label{eq:adjug.ident}
A_Q(\l) A_Q^a(\l) = A_Q^a(\l) A_Q(\l) = \Delta_Q(\l) I_n = 0,
\end{equation}
that for a certain $j \in \oneton$, vector
\begin{equation} \label{eq:alp.def}
 \alp := \col(\alp_1, \ldots, \alp_n) := \col(A_{1j}(\l), \ldots, A_{nj}(\l))
\end{equation}
is non-zero and satisfies $A_Q(\l) \alp = 0$. The corresponding eigenvector of BVP~\eqref{eq:L0.def.reg},~\eqref{eq:Uy=0} is given by
\begin{equation} \label{eq:yx=alp.exp}
 y(x,\l) = \sum_{k=1}^n \alp_k \Phi_{k}(x,\l) \not\equiv 0,
 \qquad y(0,\l) = \alp \ne 0.
\end{equation}
Indeed, in accordance with definition~\eqref{eq:L0.def.reg}, $\cL_0 y(x,\l) = \l y(x,\l)$. Besides,
\begin{equation}
Cy(0,\l) + D y(\ell,\l) = (C\Phi_0(0, \l) + D \Phi_0(\ell, \l))\alp = A_Q(\l) \alp = 0.
\end{equation}
This proves~\eqref{eq:Ypxl.def}. Since $\l$ is a simple eigenvalue of $L_U(Q)$, then all eigenvectors are proportional to each other. Formulas~\eqref{eq:alp.def}--\eqref{eq:yx=alp.exp} for eigenvectors of $L_U(Q)$ imply that all columns of the matrix $A_Q^a(\l)$ are proportional, which means that $\rank(A_Q(\l)) = 1$ and finishes the proof.
\end{proof}
\begin{remark}
Note that each non-trivial $k$th column of the adjugate matrix~\eqref{eq:Aa.def} generates an eigenvector of the operator $L_U(Q)$ by formula~\eqref{eq:Ypxl.def} with $p$ replaced by $k$. Since $\rank(A_a(\l)) = 1$, the eigenvectors $Y_p(x,\l)$ and $Y_k(x,\l)$ are proportional. Therefore, we write $y(x,\l)$ in~\eqref{eq:Ypxl.def} instead of $y_p(x,\l)$ omitting the index $p$.
\end{remark}
The following observation trivially follows from Lemma~\ref{lem:eigen} and will be useful in the future.
\begin{lemma} \label{lem:eigen.canon}
Let matrix functions $B(\cdot)$ and $Q(\cdot)$ satisfy conditions~\eqref{eq:Bx.def}--\eqref{eq:Q=Qjk.def} and let $\l$ be an algebraically simple eigenvalue of the operator $L_U(Q)$ and let $f$ be any eigenvector of $L_U(Q)$ in $\fH$ corresponding to $\l$. Then, there exists $p = p_{\l} \in \oneton$ and $\gam_p \in \bC$, such that
\begin{equation}
f(\cdot) = \gam_p Y_p(\cdot, \l) = \gam_p \sum_{k=1}^n A_{kp}(\l) \Phi_k(\cdot, \l), \qquad |\gam_p| = \|f\|_{\fH} / \|Y_p(\cdot, \l)\|_{\fH}.
\end{equation}
Morever, this is valid for any $p \in \oneton$, for which $Y_p(\cdot, \l) \not\equiv 0$.
\end{lemma}
The following trivial properties of $\Phi_0(\cdot,\l)$, $A_0(\l)$, $A_0^a(\l)$, $\Delta_0(\l)$ and $Y_k^0(\cdot, \l)$ will be useful in the future. Recall, that $\Pi_h = \{\l \in \bC : |\Im \l| \le h\}$.
\begin{lemma} \label{lem:easy.upper.bound}
Let $h \ge 0$ and let $\beta_k \in L^1([0,\ell];\bR)$, $k \in \oneton$. There there exists a constant $M_h > 0$ that only depends on $h$, matrices $C$ and $D$, and values $\|\beta_k\|_1 := \|\beta_k\|_{L^1[0,\ell]}$, $k \in \oneton$, such that the following uniform inequalities hold
\begin{align}
\label{eq:e.rhokx<M.dlYk0<M}
 & \abs{e^{i \l \rho_k(x)}} \le M_h, \qquad
 \norm{\frac{d}{d\l}Y_k^0(x, \l)}_{\bC^n} \le M_h,
 \qquad \l \in \Pi_h, \quad x \in [0, \ell], \quad k \in \oneton, \\
\label{eq:A0.Delta0<M}
 & |a_{jk}^0(\l)| + |(a_{jk}^0)'(\l)| + |A_{jk}^0(\l)| + |\Delta_0(\l)| \le M_h,
 \qquad \l \in \Pi_h, \quad j, k \in \oneton. \\
\label{eq:Phij0.Yk0.norm}
 & \|\Phi_k^0(\cdot, \l)\|_{\fH} \le M_h, \qquad
 \|Y_k^0(\cdot, \l)\|_{\fH} \le M_h,
 \qquad \l \in \Pi_h, \quad k \in \oneton.
\end{align}
\end{lemma}
\begin{proof}
Let $\l \in \Pi_h$ be fixed for the entire proof. Since $\beta_k \in L^1([0,\ell];\bR)$, $k \in \oneton$, it follows that
\begin{equation} \label{eq:rhokx<}
 |\rho_k(x)| = \abs{\int_0^x \beta_k(t) dt} \le b_0, \qquad x \in [0, \ell], \quad k \in \oneton,
\end{equation}
where $b_0 := \max\{\|\beta_1\|_1, \ldots, \|\beta_n\|_1\}$. Since $\rho_k(\cdot)$ are real-valued functions, then
\begin{equation} \label{eq:exp.rhokx<}
 \abs{e^{i \l \rho_k(x)}} = e^{- \Im \l \cdot \rho_k(x)} \le e^{b_0 h},
 \qquad x \in [0, \ell], \quad k \in \oneton.
\end{equation}
Similarly, for $\l \in \Pi_h$, $j,k \in \oneton$, we have
\begin{equation} \label{eq:ajk.ajk'<}
 |a_{jk}^0(\l)| \le |c_{jk}| + |d_{jk}| e^{b_0 h}
 \le c_0 + d_0 e^{b_0 h} =: \gam_h,
 \qquad |(a_{jk}^0)'(\l)| \le |b_k d_{jk}| e^{b_0 h} \le b_0 d_0 e^{b_0 h},
\end{equation}
where
\begin{equation} \label{eq:c0.d0.def}
 c_0 := \max\{|c_{jk}| : j, k \in \oneton\}, \qquad
 d_0 := \max\{|d_{jk}| : j, k \in \oneton\}.
\end{equation}
From the definition of the adjugate matrix it follows that $(-1)^{j+k} A_{jk}^0(\l)$ is the determinant of some $(n-1) \times (n-1)$ submatrix of $A_0(\l)$. Moreover, $\Delta_0(\l) = \det A_0(\l)$. Hence~\eqref{eq:ajk.ajk'<} implies that
\begin{equation} \label{eq:Ajk<}
 |A_{jk}^0(\l)| \le (n-1)! \gam_h^{n-1}, \qquad
 |\Delta_0(\l)| \le n! \gam_h^{n}, \qquad
 j, k \in \oneton.
\end{equation}

With account of~\eqref{eq:Phi0k.def} and~\eqref{eq:rhok.bk.def} we have after making a change of variable $u = \rho_j(x)$,
\begin{align}
\label{eq:Phij0.norm}
 \|\Phi_j^0(\cdot, \l)\|_{\fH}^2 & =
 \int_0^{\ell} \abs{\exp\(i \l \rho_j(x)\)}^2 |\beta_j(x)| dx
 \le \max_{x \in [0,\ell]} \abs{\exp\(i \l \rho_j(x)\)}^2 \|\beta_j\|_1
 \le b_0 e^{2 b_0 h}.
\end{align}
Combining formula~\eqref{eq:Ypxl.def} with estimates~\eqref{eq:Ajk<} and~\eqref{eq:Phij0.norm} we arrive at
\begin{multline}
 \|Y_k^0(\cdot, \l)\|_{\fH} =
 \norm{\sum_{j=1}^n A_{jk}^0(\l) \Phi_j^0(\cdot, \l)}_{\fH} \le
 \sum_{j=1}^n \abs{A_{jk}^0(\l)} \cdot \|\Phi_j^0(\cdot, \l)\|_{\fH} \\
 \le n! \gam_h^{n-1} b_0^{1/2} e^{b_0 h},
 \qquad k \in \oneton.
\end{multline}
Derivative $\frac{d}{d\l}Y_k^0(x, \l)$ can be estimated similarly. Since
$$
\frac{d}{d\l} \Phi_j^0(x, \l) = i \rho_j(x) e^{i \l \rho_j(x)} \col(\delta_{1j}, \ldots, \delta_{nj}), \qquad x \in [0,\ell], \quad \l \in \bC,
\quad j \in \oneton,
$$
it is clear, that each entry of the vector function $\frac{d}{d\l}Y_k^0(x, \l)$ is a polynomial in $e^{i b_j \l}, e^{i \rho_j(x) \l}, b_j, \rho_j(x)$, $c_{jp}, d_{jp}$. This observation and estimates~\eqref{eq:rhokx<}--\eqref{eq:exp.rhokx<} imply that
$$
 \norm{\frac{d}{d\l}Y_k^0(x, \l)}_{\bC^n} \le \wt{M}_h
 \qquad x \in [0,\ell], \quad \l \in \Pi_h,
 \quad k \in \oneton,
$$
for some $\wt{M}_h > 0$. Thus, setting
\begin{equation}
 M_h := \max\{e^{b_0 h}, \ \ \gam_h + b_0 d_0 e^{b_0 h} + (n-1)! \gam_h^{n-1} + n! \gam_h^n, \ \ n! \gam_h^{n-1} b_0^{1/2} e^{b_0 h}, \ \ \wt{M}_h\},
\end{equation}
and combining all the estimates established above we arrive at~\eqref{eq:e.rhokx<M.dlYk0<M}--\eqref{eq:Phij0.Yk0.norm}.
\end{proof}
\subsection{Regular boundary conditions} \label{subsec:regular}
Considerations of this and the next subsection are performed only in terms of number $b_1, \ldots, b_n$ given by~\eqref{eq:rhok.bk.def} without their connection to original functions $\beta_1, \ldots, \beta_n$. Hence, the only condition we need to impose in this and the next subsection is condition~\eqref{eq:b1.bn}, i.e.\ that numbers $b_1, \ldots, b_n$ are ordered and non-zero. Let us also set
\begin{equation} \label{eq:b-+.def}
 b_- := b_1 + \ldots + b_{n_-} \le 0 \qquad\text{and}\qquad
 b_{+} := b_{n_-+1} + \ldots + b_n \ge 0.
\end{equation}
Note that if $n_-=0$ then $b_- = 0$ and if $n_-=n$ then $b_+=0$.

Let us recall the definition of regular boundary conditions from the introduction. Note that considerations below are valid without canonical ordering~\eqref{eq:b1.bn}. To this end, let $\cP_n$ be the set of diagonal idempotent $n \times n$ matrices:
\begin{equation} \label{eq:cP.def}
 \cP_n := \{P = \diag(p_1, \ldots, p_n) :
 \ p_k \in \{0, 1\}, \ k \in \oneton\}.
\end{equation}
For any $P \in \cP_n$ we put
\begin{equation} \label{eq:TP.CD.def}
 J_P := J_P(C,D) := \det(T_P(C,D)), \qquad T_P(C,D) := C (I_n - P) + D P.
\end{equation}
Finally, we set
\begin{equation} \label{eq:P.pm.def}
 P_{\pm} := \diag(p_1^{\pm}, \ldots, p_n^{\pm}). \qquad
 p_k^+ = \begin{cases} 1, & b_k>0, \\ 0, & b_k<0, \end{cases},
 \qquad
 p_k^- = \begin{cases} 0, &b_k>0, \\ 1, & b_k<0, \end{cases},
 \qquad k \in \oneton.
\end{equation}
Clearly $P_+ + P_- = I_n$ and $P_+$ (resp. $P_-$) is the projector onto the positive (resp. negative) part of the spectrum of the signature matrix $S = \sign(B(\cdot)) \equiv \const$.
\begin{definition} \label{def:regular}
Boundary conditions~\eqref{eq:Uy=0} for equation~\eqref{eq:LQ.def.reg} are called \textbf{regular} if
\begin{equation} \label{eq:regular.def}
 J_{P_+}(C, D) = \det(C P_- + D P_+) \ne 0
 \quad\text{and}\quad
 J_{P_-}(C, D) = \det(C P_+ + D P_-) \ne 0.
\end{equation}
\end{definition}
Let us obtain some general properties of the characteristic determinant $\Delta_0(\cdot)$.
\begin{lemma}
With account of the notations $\cP_n$, $J_P(C,D)$, and $b_k := \rho_k(\ell) = \int_0^{\ell} \beta_k(x) dx$, the characteristic determinant $\Delta_0$ admits a representation
\begin{equation} \label{eq:Delta0.sum}
 \Delta_0(\l) = \sum_{P \in \cP_n} J_P(C,D) e^{i \l b_P}, \qquad
 \quad b_P := \sum_{k=1}^n p_k b_k.
\end{equation}
\end{lemma}
\begin{proof}
Denoting by $c_k$ and $d_k$, $k \in \oneton$, the columns of the matrices $C$ and $D$, respectively, we write them in the form $C = (c_1 \ \ldots \ c_n)$ and $D = (d_1 \ \ldots \ d_n)$.
In accordance with~\eqref{eq:Phi0k.def}, $\Phi_0(\ell,\l) = \diag(e^{i \l b_1}, \ldots, e^{i \l b_n})$. Hence
$$
A_0(\l) = C + D \Phi_0(\ell,\l) =
\(c_1 + e^{i \l b_1} d_1 \ \ \ldots \ \ c_n + e^{i \l b_n} d_n\).
$$
Formula~\eqref{eq:Delta0.sum} easily follows from this representation and the general formula for the determinant of the sum of two matrices as the sum of determinants of all $2^n$ matrices, where for each such matrix we choose either $c_k$ or $e^{i \l b_k} d_k$ as the $k$-th column.
\end{proof}
To obtain further properties of the characteristic determinant $\Delta_0(\cdot)$, we need to recall some definitions.
\begin{definition}[\cite{LunMal16JMAA,Katsn71}] \label{def:incompressible}
The sequence $\fM$ is called \textbf{incompressible} if for some $d \in \bN$ every rectangle $[t-1,t+1] \times \bR \subset \bC$ contains at most $d$ entries of the sequence, i.e.
\begin{equation} \label{eq:card.incomp}
 \card\{m \in \bZ : |\Re \mu_m - t| \le 1 \} \le d, \quad t \in \bR.
\end{equation}
To emphasize parameter $d$ we will sometimes call $\fM$ an incompressible
sequence of density $d$.
\end{definition}
\begin{definition}[\cite{Lev61}] \label{def:sinetype}
An entire function $F(\cdot)$ of exponential type is said to be of \textbf{sine-type} if

\textbf{(i)} all zeros of $F(\cdot)$ lie in the strip $\Pi_h$
for some $h \ge 0$, and

\textbf{(ii)} there exists $C_1, C_2 > 0$ and $h_0 > h$ such
that
\begin{equation} \label{eq:C1<|Fz|<C2}
 0 < C_1 \le |F(x + ih_0)| \le C_2 <\infty, \quad x \in \bR.
\end{equation}
\end{definition}
This definition is borrowed from~\cite{Lev61} (see
also~\cite{Katsn71}). It differs from that contained
in~\cite{Lev96}. Namely, it is assumed in~\cite{Lev96} that the
sequence of zeros of $F(\cdot)$ is \emph{separated} and the
indicator function $h_F(\cdot)$ of $F(\cdot)$,
\begin{equation} \label{eq:hF.phi.def}
 h_F(\phi) := \varlimsup_{r \to +\infty}
 \frac{\ln\left|F\left(r e^{i \phi}\right)\right|}{r},
 \quad \phi \in (-\pi,\pi],
\end{equation}
satisfies the condition $h_F(\pi/2) = h_F(-\pi/2)$. The latter
is imposed for convenience and can easily be achieved with
multiplying $F(\cdot)$ by a function $e^{i\gam z}$ with
an appropriate $\gam \in \bR.$
\begin{lemma} \label{lem:Delta0.prop}
Let boundary conditions be regular~\eqref{eq:Uy=0}. Then the following statements hold:

\textbf{(i)} The characteristic determinant $\Delta_0(\cdot)$ is a sine-type function with $h_{\Delta_0}(\pi/2) = -b_-$ and $h_{\Delta_0}(-\pi/2) = b_+$. In particular, $\Delta_0(\cdot)$ has infinitely many zeros
\begin{equation} \label{eq:Lam0.def}
 \L_0 := \{\l_m^0\}_{m \in \bZ}
\end{equation}
counting multiplicity and $\L_0 \subset \Pi_h$ for some $h>0$.

\textbf{(ii)} The sequence $\L_0$ is incompressible.

\textbf{(iii)} For any $\eps > 0$ the determinant $\Delta_0(\cdot)$ admits the following estimate from below
\begin{equation} \label{eq:Delta0>C.exp}
 |\Delta_0(\l)| > C_{\eps} (e^{-\Im \l \cdot b_-} + e^{-\Im \l \cdot b_+})
 > C_{\eps}, \qquad \l \in \bC \setminus \bigcup_{m \in \bZ} \bD_{\eps}(\l_m^0),
\end{equation}
with some $C_{\eps} > 0$, where numbers $b_{\pm}$ are given by~\eqref{eq:b-+.def}.

\textbf{(iv)} The sequence $\L_0$ can be ordered in such a way that the following asymptotical formula holds
\begin{equation} \label{eq:lam0.n=an+o1}
 \l_m^0 = \frac{2 \pi m}{b_+ - b_-} (1 + o(1)) \quad\text{as}\quad m \to \infty.
\end{equation}
\end{lemma}
\begin{proof}
\textbf{(i-iii)} It follows from~\eqref{eq:Delta0.sum} that
\begin{equation} \label{eq:Delta0.sum2}
\Delta_0(\l) = \sum_{k=1}^N \gam_k e^{i \l \sigma_k}, \qquad \gam_k
:= \sum_{\genfrac{}{}{0pt}{2}{P \in \cP_n}{b_P = \sigma_k}} J_P(C,D), \quad k \in \onetoN,
\end{equation}
where $\sigma_1 < \ldots < \sigma_N$, $N \le 2^n$, are all distinct values in the set
$$
\{b_P : P \in \cP_n\} = \left\{\sum_{k \in S} b_k : S \subset \oneton\right\}
=: \{\sigma_1, \ldots, \sigma_N\},
$$
with $b_P$ defined by~\eqref{eq:Delta0.sum} for $P \in \cP_n$.

Taking into account definition~\eqref{eq:P.pm.def} of $P_-$, definition~\eqref{eq:Delta0.sum} of $b_P$ and definition~\eqref{eq:b-+.def} of $b_-$, we have
\begin{equation} \label{eq:sigma1}
 \sigma_1 = \min_{P \in \cP_n} b_P = \sum_{b_k < 0} b_k = b_{P-},
 \quad\text{hence}\quad \sigma_1 = b_1 + \ldots + b_{n_-} = b_-.
\end{equation}
It is also clear that $b_P > b_{P_-}$ whenever $P \ne P_-$, $P \in \cP_n$. Hence $\gam_1 = J_{P_-}(C,D) \ne 0$, since boundary conditions are regular (see~\eqref{eq:regular.def}). Similarly
\begin{equation} \label{eq:sigman}
 \sigma_N = \max_{P \in \cP_n} b_P = \sum_{b_k > 0} b_k = b_{P+} = b_+
 \quad\text{and}\quad \gam_N = J_{P_+}(C,D) \ne 0.
\end{equation}

Thus, formula~\eqref{eq:Delta0.sum2} for $\Delta_0(\cdot)$ turns into,
\begin{equation} \label{eq:Delta0.sum+-}
 \Delta_0(\l) = J_{P_-}(C,D) e^{i \l b_-} + J_{P_+}(C,D) e^{i \l b_+} + \sum_{k=2}^{N-1} \gam_k e^{i \l \sigma_k}, \qquad \l \in \bC.
\end{equation}
This immediately implies that $\Delta_0(\l) \not\equiv 0$. Moreover, if $\Delta_0(\cdot)$ has no zeros, then canonical factorization for entire functions of exponential type implies that $\Delta_0(\l) = e^{\alp + \beta \l}$, $\l \in \bC$, for some $\alp, \beta \in \bC$. This contradicts representation~\eqref{eq:Delta0.sum+-}. Hence $\Delta_0(\cdot)$ has zeros. In turn, since $\Delta_0(\cdot) \not\equiv 0$, has zeros and bounded on the real line, the canonical factorization of entire functions of exponential type implies that the set of zeros of $\Delta_0$ is countable.
It is clear from~\eqref{eq:Delta0.sum+-} that function $f(\l) := \exp(-i \l \frac{b_- + b_+}{2}) \Delta_0(\l)$ satisfy the following uniform estimate for some $h \ge 0$,
\begin{equation} \label{eq:f.sin.type}
 \tau_h^{-1} e^{\sigma |\Im \l|} \le |f(\l)| \le \tau_h e^{\sigma |\Im \l|} > 0,
 \qquad |\Im \l| > h,
\end{equation}
with some $\tau_h > 1$ that does not depend on $\l$. Here $\sigma := \frac{b_+ - b_-}{2} > 0$. It is clear from estimate~\eqref{eq:f.sin.type}
that $f(\cdot)$ is the sine-type function of exponential type $\sigma$, with $h_f(\pm \pi/2) = \sigma$. The desired properties of zeros as well as estimate~\eqref{eq:Delta0>C.exp} are now immediate from~\cite[Lemmas 3 and 4]{Katsn71}.

\textbf{(iv)} The proof is the same as in~\cite[Proposition 4.6(iv)]{LunMal16JMAA}.
\end{proof}
\begin{remark}
\textbf{(i)} Lemma~\ref{lem:Delta0.prop} remains valid if characteristic determinant $\Delta_0(\cdot)$ is not identically zero and has zeros. Indeed, it follows that at least two coefficients in~\eqref{eq:Delta0.sum2} are non-zero, i.e.
\begin{equation} \label{eq:Delta0.sum3}
 \Delta_0(\l) = \sum_{k=N_1}^{N_2} \gam_k e^{i \l \sigma_k},
 \quad\text{where}\quad \gam_{N_1} \gam_{N_2} \ne 0
 \quad\text{and}\quad 1 \le N_1 < N_2 \le N,
\end{equation}
which implies that it is a sine-type function with all the properties listed in Lemma~\ref{lem:Delta0.prop} if we replace $b_-$ with $\sigma_{N_1}$ and $b_+$ with $\sigma_{N_2}$. In particular, $h_{\Delta_0}(\pi/2) = -\sigma_{N_1}$ and $h_{\Delta_0}(-\pi/2) = \sigma_{N_2}$.

\textbf{(ii)} Note, that since $\sigma_1 = b_- \le 0$ and $\sigma_N = b_+ \ge 0$, identity~\eqref{eq:Delta0.sum2} implies that the indicator diagram of the entire function $\Delta_0(\cdot)$ is always contained in the vertical segment $[-i b_+, -i b_-]$ and coincides with it if and only if boundary conditions are regular. In other words, boundary conditions are regular if and only if the determinant $\Delta_0(\cdot)$ is of maximal possible growth in both half-planes $\bC_{\pm}$.

\textbf{(iii)} Let us clarify the previous remark when either $b_- = 0$ or $b_+ = 0$, where $b_{\pm}$ is given by~\eqref{eq:sigma1}--\eqref{eq:sigman}. In other words, entries of the matrix $B(x)$ are either all positive or all negative. In this case, regularity condition~\eqref{eq:regular.def} turns into $\det (CD) \ne 0$, since either $P_- = 0$ or, respectively, $P_+ =0$, where $P_{\pm}$ is given by~\eqref{eq:P.pm.def}. Since $b_- + b_+ = b_1 + \ldots + b_n$, it is clear that in both cases $\Delta_0(\cdot)$ is the entire function of exponential type $|b_1 + \ldots + b_n|$, growing in $\bC_+$, resp. $\bC_-$, and bounded from above and below in $\bC_-$, resp. $\bC_+$.
\end{remark}
Finally, we reduce regular boundary conditions~\eqref{eq:Uy=0} to a certain equivalent canonical form which is much simpler and convenient to work with.
\begin{lemma} \label{lem:CD.canon}
Let boundary conditions~\eqref{eq:Uy=0} be regular and assume that equations in~\eqref{eq:L0.def.reg} and boundary conditions~\eqref{eq:Uy=0} are reordered to make sure canonical ordering~\eqref{eq:b1.bn}, i.e.\ for some $n_- \in \{0, 1, \ldots, n\}$,
\begin{equation} \label{eq:bk<0>0}
 b_1 \le \ldots \le b_{n_-} < 0 < b_{n_- + 1} \le \ldots \le b_n.
\end{equation}
Then a pair of matrices $\{C, D\}$ determined by the linear form $Uy=0$ in~\eqref{eq:Uy=0} can be chosen to admit the following triangular block-matrix representation with respect to the orthogonal decomposition
$\bC^n = \bC^{n_-} \oplus \bC^{n_+}$, where $n_+ = n - n_-$:
\begin{equation} \label{eq:CD.canon}
 C = \begin{pmatrix} I_{n_-} & C_{12} \\ \bO & C_{22} \end{pmatrix}, \qquad
 D = \begin{pmatrix} D_{11} & \bO \\ D_{21} & I_{n_+} \end{pmatrix},
\end{equation}
for some
matrices $C_{12}, C_{22}, D_{11}, D_{21}$.
Here, in the case $n_-=0$, $n_+=n$ the canonical form is $\{C=I_n, D\}$ with any arbitrary invertible $D$, while in the case $n_-=n$, $n_+=0$ the canonical form is $\{C, D=I_n\}$ with arbitrary invertible $C$.
\end{lemma}
\begin{proof}
Definition~\eqref{eq:P.pm.def} and relation~\eqref{eq:bk<0>0} imply that matrices $P_{\pm}$ admit the following block-matrix representation with respect to the orthogonal decomposition $\bC^n = \bC^{n_-} \oplus \bC^{n_+}$:
\begin{equation}
 P_- = \begin{pmatrix} I_{n_-} & \bO \\ \bO & \bO\end{pmatrix},
 \qquad
 P_+ = \begin{pmatrix} \bO & \bO \\ \bO & I_{n_+}\end{pmatrix}.
\end{equation}
Assume that boundary conditions in~\eqref{eq:Uy=0} are given by a pair $\{\wh{C}, \wh{D}\}$, i.e.
$Uy = \wh{C}y(0) + \wh{D}y(\ell)=0$. Consider their block-matrix representation with respect to the orthogonal decomposition $\bC^n = \bC^{n_-} \oplus \bC^{n_+}:$
\begin{equation} \label{eq:widehat_CD.block}
\wh{C} = \begin{pmatrix} \wh{C}_{11} & \wh{C}_{12} \\
\wh{C}_{21} & \wh{C}_{22}
\end{pmatrix},
\qquad
 \wh{D} = \begin{pmatrix} \wh{D}_{11} & \wh{D}_{12} \\
 \wh{D}_{21} & \wh{D}_{22} \end{pmatrix}.
\end{equation}
Definition~\ref{def:regular} of regularity (see~\eqref{eq:regular.def}) implies thatt $J_{P_+}(\wh{C}, \wh{D}) \cdot J_{P_-}(\wh{C}, \wh{D}) \ne 0$. In particular, one has
\begin{equation}
 T_{P_+}(\wh{C},\wh{D}) = \begin{pmatrix}
 \wh{C}_{11} & \wh{D}_{12} \\ \wh{C}_{21} & \wh{D}_{22} \end{pmatrix}
 = \wh{C} P_- + \wh{D} P_+
 \quad\text{and}\quad
 \det(T_{P_+}(\wh{C}, \wh{D})) = J_{P_+}(\wh{C}, \wh{D}) \ne 0.
\end{equation}
Hence, multiplying the equation $\wh{C}y(0) + \wh{D}y(\ell)=0$ by $T_{P_+}(\wh{C},\wh{D})^{-1}$ from the left we arrive at the equivalent equation with new matrices~\eqref{eq:CD.canon} instead of~\eqref{eq:widehat_CD.block}.
\end{proof}
\begin{remark} \label{rem:canon}
The proof remains valid for non-regular boundary conditions provided that $J_P(C,D) \ne 0$ for some $P \in \cP_n$, after a proper reordering of equations in~\eqref{eq:L0.def.reg} and boundary conditions~\eqref{eq:Uy=0}.
\end{remark}
\subsection{Strictly regular boundary conditions} \label{subsec:strict.regular}
Let us introduce a notion of strictly regular boundary conditions.
\begin{definition} \label{def:strict.regular}
\textbf{(i)} A sequence $\fM := \{\mu_m\}_{m \in \bZ}$ of complex numbers is called \textbf{separated} if for some $\delta > 0$,
\begin{equation} \label{separ_cond}
|\mu_j - \mu_k| > 2 \delta \quad \text{whenever}\quad j \ne k.
\end{equation}
In particular, all entries of a separated sequence are distinct.

\textbf{(ii)} The sequence $\fM$ is called \textbf{asymptotically separated} if for some $m_0 \in \bN$ the subsequence $\fM({m_0}) := \{\mu_m\}_{|m| > m_0}$ is separated.

\textbf{(iii)} Boundary conditions~\eqref{eq:Uy=0} are called \textbf{strictly regular}, if they are regular and the sequence of zeros $\L_0 = \{\l_m^0\}_{m \in \bZ}$ of the characteristic determinant $\Delta_0(\cdot)$ is asymptotically separated.
In particular, there is $m_0$ such that zeros $\{\l_m^0\}_{|m| > m_0}$
are algebraically (hence geometrically) simple.
\end{definition}
See the next subsection for concrete examples of strictly regular boundary conditions.
In this subsection we obtain certain estimates from below involving $\Delta_0'(\cdot)$ and the corresponding eigenvectors assuming boundary conditions to be strictly regular.
\begin{lemma} \label{lem:Delta'>}
Let boundary conditions~\eqref{eq:Uy=0} be strictly regular and let $\L_0 = \{\l_m^0\}_{m \in \bZ}$ be the sequence of eigenvalues of the operator $L_{0,U}$, counting multiplicity. Then there exist $\delta, C_0 > 0$, not dependent on $m$, and such that with $m_0$ from Definition~\ref{def:strict.regular}(iii) the following estimate from below holds
\begin{equation} \label{eq:Delta'>}
 |\Delta_0'(\l)| \ge C_0, \qquad \l \in \bD_{\delta}(\l_m^0), \quad |m| > m_0.
\end{equation}
\end{lemma}
\begin{proof}
Since $\Delta(\cdot)$ is a sine-type function with asymptotically separated zeros, then in accordance with~\cite[Lemmas 5]{Katsn71} and~\cite[Lecture 22]{Lev96}
$$
|\Delta_0'(\l_m^0)| \ge C_0', \quad |m| > m_0,
$$
for some $C_0' > 0$. Lemma~\ref{lem:Delta0.prop} implies inclusion $\l_m^0 \in \Pi_h$, $m \in \bZ$. It follows from~\eqref{eq:Delta0.sum} (see also Lemma~\ref{lem:easy.upper.bound}) that for some $C_0'' > 0$
$$
|\Delta_0''(\l)| \le C_0'', \qquad |\Im \l| \le h+1. \qquad
$$
Hence Taylor expansion and inclusion $\l_m^0 \in \Pi_h$ yield
\begin{equation} \label{eq:Delta'.Taylor}
|\Delta_0'(\l)| \ge |\Delta_0'(\l_m^0)| - \int_{\l_m^0}^{\l} |\Delta_0''(z)| |dz| \ge
C_0' - |\l - \l_m^0| C_0'', \qquad |\l - \l_m^0| < 1, \quad |m| > m_0.
\end{equation}
Setting $\delta = \min\{1, \frac{C_0'}{2C_0''}\}$ and $C_0 = C_0' - \delta C_0'' \ge C_0'/2$ one derives that
for $|\l - \l_m^0| < \delta$ inequality~\eqref{eq:Delta'.Taylor} implies~\eqref{eq:Delta'>}.
\end{proof}
\begin{lemma} \label{lem:D0jk>}
Let boundary conditions~\eqref{eq:Uy=0} be strictly regular. Then there exist $\delta>0$ and $C_1 > 0$ such that with $m_0$ from Definition~\ref{def:strict.regular}(iii) the following estimate holds
\begin{equation} \label{eq:sum.adjug}
 \sum_{j,k=1}^n |A_{jk}^0(\l)| \ge C_1, \qquad \l \in \bD_{\delta}(\l_m^0),
 \quad |m| > m_0.
\end{equation}
In particular, for any $m$ satisfying $|m|>m_0$, there exist $j,k \in \oneton$ that depend on $m$ and such that
\begin{equation} \label{eq:D0jk>}
 |A_{jk}^0(\l)| \ge C_2 ( = C_1/n^2),
 \qquad \l \in \bD_{\delta}(\l_m^0).
\end{equation}
\end{lemma}
\begin{proof}
Since boundary conditions are strictly regular, Lemma~\ref{lem:Delta'>} implies existence of constants $\delta, C_0 > 0$ such that~\eqref{eq:Delta'>} holds. Further, in accordance with Jacobi's formula~\eqref{eq:jacobi.def}
\begin{equation} \label{eq:jacobi}
 \Delta_0'(\l) =
 \sum_{j,k=1}^n A_{jk}^0(\l) (a_{kj}^0)'(\l).
\end{equation}
Combining estimate~\eqref{eq:A0.Delta0<M} on $|(a_{kj}^0)'(\l)|$ from above and estimate~\eqref{eq:Delta'>} on $|\Delta_0'(\l)|$ from below with identity~\eqref{eq:jacobi} yields
\begin{equation}
 C_0 \le |\Delta_0'(\l)| \le \sum_{j,k=1}^n |A_{jk}^0(\l)|
 \cdot |(a_{kj}^0)'(\l)|
 \le M_h \sum_{j,k=1}^n |A_{jk}^0(\l)|, \qquad \l \in \bD_{\delta}(\l_m^0).
\end{equation}
This implies~\eqref{eq:sum.adjug} with $C_1 = C_0/M_h$.
\end{proof}
\begin{proposition} \label{prop:eigenvec0}
Let entries of the matrix function $B(\cdot)$ satisfy condition~\eqref{eq:betak.L1}. Let boundary conditions~\eqref{eq:Uy=0} be strictly regular and let $\L_0 = \{\l_m^0\}_{m \in \bZ}$ be the sequence of eigenvalues of the operator $L_{0,U} = L_0(U)$, counting multiplicity. Then for each $|m| > m_0$ there exists $p = p_m \in \oneton$, such that the vector function $\wt{f}_m^0(\cdot) := Y_{p_m}^0(\cdot, \l_m^0)$ given by~\eqref{eq:Yp0xl.def} is a non-trivial eigenvector of the operator $L_{0,U}$ corresponding to its simple eigenvalue $\l_m^0$. Moreover, the following uniform estimate holds,
\begin{equation} \label{eq:Yp0.norm}
 C_3 \le \|\wt{f}_m^0\|_{\fH} = \|Y_{p_m}^0(\cdot, \l_m^0)\|_{\fH} \le C_4, \qquad |m| > m_0,
\end{equation}
where $C_4 > C_3 > 0$ do not depend on $m$.
\end{proposition}
\begin{proof}
Lemma~\ref{lem:Delta0.prop} implies that $\l_m^0 \in \Pi_h$, $m \in \bZ$, for some $h \ge 0$. Hence estimate~\eqref{eq:Phij0.Yk0.norm} from Lemma~\ref{lem:easy.upper.bound} trivially implies desired estimate from above with any choice of $p = p_m$,
\begin{equation}
 \|Y_p^0(\cdot, \l_m^0)\|_{\fH} \le C_4 = M_h,
 \qquad m \in \bZ, \quad p \in \oneton.
\end{equation}

By Lemma~\ref{lem:D0jk>}, there exist indices $p = p_m \in \oneton$ and $q = q_m \in \oneton$, and a constant $C_2>0$ such that estimate~\eqref{eq:D0jk>} holds, i.e.\ $|A_{qp}^0(\l_m^0)| \ge C_2$, $|m| > m_0$. Emphasize, that although $p$ and $q$ depend on $m$, the constant $C_2$ in the above estimate does not.
This estimate, definition~\eqref{eq:Yp0xl.def} of $Y_p^0(x,\l)$ and orthogonality in $\fH$ of the vector function $\Phi_q^0(\cdot,\l)$ to other $\Phi_k^0(\cdot,\l)$ imply that
\begin{equation} \label{eq:Yp0>}
 \|Y_p^0(\cdot,\l_m^0)\|_{\fH} \ge |A_{qp}^0(\l_m^0)| \cdot
 \|\Phi_q^0(\cdot, \l_m^0)\|_{\fH}
 \ge C_2 \sqrt{\int_0^{\ell} \abs{e^{i \l_m^0 \rho_q(x)}}^2 |\beta_q(x)| \,dx},
 \qquad |m| > m_0.
\end{equation}
Let us estimate the integral in~\eqref{eq:Yp0>}. Recall that $\beta_q(\cdot)$ does not change sign on $[0,\ell]$. Hence $|\beta_q(x)| = s_q \beta_q(x)$, $x \in [0,\ell]$, where $s_q = \sign(\beta_q(\cdot))$. Making a change of variable $t = \rho_q(x)$ (and so $\beta_q(x) dx = dt$), we have
\begin{equation} \label{eq:Phij0.norm.low}
 \int_0^{\ell} \abs{\exp\(i \l \rho_q(x)\)}^2 |\beta_q(x)| dx
 = \abs{\int_0^{b_q} \abs{\exp\(i \l t\)}^2 dt}
 \ge \min\{1, e^{-2 \Im \l \cdot b_q} \} \ge e^{-2 |b_q| h},
 \qquad \l \in \Pi_h.
\end{equation}
Inserting~\eqref{eq:Phij0.norm.low} with $\l = \l_m^0 \in \Pi_h$ into~\eqref{eq:Yp0>} we arrive at the estimate from below in~\eqref{eq:Yp0.norm} with with some $C_3 > 0$ that does not depend on $q$ and $m$. Thus, vector function $\wt{f}_m^0(\cdot) = Y_p^0(\cdot,\l_m^0)$ is non-zero. Lemma~\ref{lem:eigen} implies that it is a non-trivial eigenvector of the operator $L_{0,U}$ corresponding to its simple eigenvalue $\l_m^0$, which finishes the proof.
\end{proof}
\subsection{Examples of strictly regular boundary conditions}
\label{subsec:examples.strict}
In the next remark we outline known cases of strictly regular boundary conditions for $n=2$ established in~\cite{LunMal16JMAA}.
\begin{remark} \label{rem:cond.examples}
If $n=2$ and $b_1 < 0 < b_2$, canonical form (see~\eqref{eq:CD.canon}) of regular boundary conditions~\eqref{eq:Uy=0} is
\begin{equation} \label{eq:Uy=0.canon.n=2}
\begin{cases}
 \wh{U}_{1}(y) = y_1(0) + b y_2(0) + a y_1(1) = 0, \\
 \wh{U}_{2}(y) = d y_2(0) + c y_1(1) + y_2(1) = 0,
\end{cases}
\end{equation}
with some $a,b,c,d \in \bC$, such that $ad \ne bc$, while the characteristic determinant $\Delta_0(\cdot)$ takes the following form,
\begin{equation} \label{eq:Delta0.new}
 \Delta_0(\l) = d + a e^{i (b_1+b_2) \l} + (ad-bc) e^{i b_1 \l}
 + e^{i b_2 \l}, \\
\end{equation}
Let us list some types of \emph{strictly regular} boundary conditions of the form~\eqref{eq:Uy=0.canon.n=2}. In all of these cases except 4b the set of zeros
of $\Delta_0$ is a union of finite number of arithmetic progressions.

\begin{enumerate}

\item Regular boundary conditions~\eqref{eq:Uy=0.canon.n=2} for Dirac operator ($-b_1 = b_2 = 1$) are
strictly regular if and only if $(a-d)^2 \ne -4bc$.

\item Separated boundary conditions ($a=d=0$, $bc \ne 0$) are always strictly regular.

\item Let $b_1 / b_2 \in \bQ$, i.e.\ $b_1 = -m_1 b_0$, $b_2 = m_2 b_0$, where $m_1, m_2 \in \bN$, $b_0 > 0$ and $\gcd(m_1,m_2)=1$.
Since $ad \ne bc$, $\Delta_0(\cdot)$ is a polynomial at $e^{i b_0 \l}$ of degree $m_1 + m_2$. Hence, boundary conditions~\eqref{eq:Uy=0.canon.n=2} are strictly regular if and only if this polynomial does not have multiple roots. Let us list some cases with explicit conditions.

\begin{enumerate}

\item~\cite[Lemma 5.3]{LunMal16JMAA} Let $ad \ne 0$ and $bc=0$. Then
boundary conditions~\eqref{eq:Uy=0.canon.n=2} are strictly regular if and only if
\begin{equation} \label{eq:bc=0.crit.rat}
 b_1 \ln |d| + b_2 \ln |a| \ne 0 \quad\text{or}\quad
 m_1 \arg(-d) - m_2 \arg(-a) \notin 2 \pi \bZ.
\end{equation}

\item In particular, antiperiodic boundary conditions ($a=d=1$, $b=c=0$) are strictly regular if
and only if $m_1 - m_2$ is odd. Note that these boundary conditions are not strictly regular in
the case of a Dirac system.

\item~\cite[Proposition 5.6]{LunMal16JMAA} Let $a=0$, $bc \ne 0$. Then
boundary conditions~\eqref{eq:Uy=0.canon.n=2} are strictly regular if and only if
\begin{equation} \label{eq:a=0.crit.rat}
 m_1^{m_1} m_2^{m_2} (-d)^{m_1 + m_2} \ne (m_1 + m_2)^{m_1 + m_2} (-b c)^{m_2}.
\end{equation}

\end{enumerate}

\item Let $\alp := -b_1 / b_2 \notin \bQ$. Then the problem of strict regularity of boundary conditions is
generally much more complicated. Let us list some known cases:

\begin{enumerate}

\item~\cite[Lemma 5.3]{LunMal16JMAA} Let $ad \ne 0$ and $bc=0$. Then
boundary conditions~\eqref{eq:Uy=0.canon.n=2} are strictly regular if and only if
\begin{equation} \label{eq:bc=0.crit.irrat}
 b_1 \ln |d| + b_2 \ln |a| \ne 0.
\end{equation}

\item~\cite[Proposition 5.6]{LunMal16JMAA} Let $a=0$ and $bc, d \in \bR
\setminus \{0\}$. Then boundary conditions~\eqref{eq:Uy=0.canon.n=2} are strictly regular if and only if
\begin{equation} \label{eq:a=0.crit}
 d \ne -(\alp+1)\(|bc| \alp^{-\alp}\)^{\frac{1}{\alp+1}}.
\end{equation}

\end{enumerate}

\end{enumerate}
\end{remark}
Let us extend results listed in this remark to the case of arbitrary $n$. The next result establishes criterion of strict regularity of boundary conditions of periodic type \emph{for any $n \in \bN$}. To this end, for any pair $x, y \in \bR \setminus \{0\}$ of real numbers with $x/y \in \bQ$ denote by $\gcd(x, y)$ their greatest common divisor, i.e.\ the largest number $b>0$ such that $x/b$ and $y/b$ are integers.
\begin{lemma} \label{lem:periodic.strict}
Let boundary conditions~\eqref{eq:Uy=0} be of the form
\begin{equation} \label{eq:yell=C.y0}
 y_k(\ell) = c_k y_k(0), \qquad c_k \ne 0, \qquad k \in \oneton,
\end{equation}
i.e.\ $U(y) = C y(0) + D y(\ell) = 0$, where $C = \diag(c_1, \ldots, c_n)$ is an invertible diagonal matrix and $D = -I_n$. Then boundary conditions~\eqref{eq:yell=C.y0} are regular.

\textbf{(i)} Let $\L_0 = \{\l_m^0\}_{m \in \bZ}$ be the
sequence of zeros of the characteristic determinant
$\Delta_0(\cdot)$ and assume it is ordered in such a way that
$\Re \l_m^0 \le \Re \l_{m+1}^0$, $m \in \bZ$. Then there
exists a sequence of integers $\{m_k\}_{k \in \bZ}$, such that
\begin{equation} \label{eq:Lambda0=union}
 m_k < m_{k+1} \le m_k + n, \quad \Re \l_{m_k}^0 - \Re \l_{m_k-1}^0 \ge \eps, \quad k \in \bZ,
\end{equation}
where $\eps := \frac{2 \pi}{b_{\max} n} > 0$ and $b_{\max} := \max\{|b_1|, \ldots, |b_n|\}$.

\textbf{(ii)} Let numbers $\left\{\frac{\ln |c_k|}{b_k}\right\}_{k=1}^n$ be distinct, i.e.\
\begin{equation} \label{eq:bj.ln.ck.ne.bk.ln.cj}
 b_j \ln |c_k| \ne b_k \ln |c_j|, \qquad j \ne k.
\end{equation}
Then boundary conditions~\eqref{eq:yell=C.y0} are strictly regular.

\textbf{(iii)} More precisely, boundary conditions~\eqref{eq:yell=C.y0} are strictly regular if and only if for all $j \ne k$ the following condition holds,
\begin{equation} \label{eq:strict.crit.per}
 \text{either}\quad b_j \ln |c_k| \ne b_k \ln |c_j|
 \quad\text{or}\quad \(
 \frac{b_j}{b_k} \in \bQ
 \quad\text{and}\quad
 \frac{b_j \arg(c_k) - b_k \arg(c_j)}{2 \pi \gcd(b_j, b_k)} \not \in \bZ\).
\end{equation}

\textbf{(iv)} \emph{Periodic} boundary conditions $(c_1 = \ldots = c_n = 1)$ are always non-strictly regular. \emph{Antiperiodic} boundary conditions $(c_1 = \ldots = c_n = -1)$ are strictly regular if and only if there exists $b_0 > 0$, odd integers $M_1, \ldots, M_n$ and distinct non-negative integers $a_1, \ldots, a_n$ such that $b_k = 2^{a_k} M_k b_0$, $k \in \oneton$. In other words, numbers $b_1, \ldots, b_n$ can be ordered in such a way that the following representation holds
\begin{equation}
 b_k = 2^{a_k} (2 u_k + 1) b_0, \qquad a_k, u_k \in \bZ, \quad k \in \oneton,
 \qquad 0 \le a_1 < a_2 < \ldots < a_n.
\end{equation}
In particular, if $b_k = 2^k$, $k \in \oneton$, then antiperiodic boundary conditions are strictly regular.
\end{lemma}
\begin{proof}
Since matrices $C$ and $D$ are invertible and diagonal, it follows that $T_P(C,D)$ given by~\eqref{eq:TP.CD.def} is also invertible diagonal matrix for every $P \in \cP_n$. This implies regularity of boundary conditions. Further, it is clear that
$$
C + D \Phi_0(\ell, \l) = \diag(c_1 - e^{i \l b_1}, \ldots, c_n - e^{i \l b_n}).
$$
Hence the characteristic determinant $\Delta_0(\cdot)$ defined in~\eqref{eq:Delta.def} becomes
\begin{equation} \label{eq:Delta0.per}
 \Delta_0(\l) = \det(C + D \Phi_0(\ell, \l))
 = (c_1 - e^{i \l b_1}) \times \ldots \times (c_n - e^{i \l b_n}).
\end{equation}
Let $\L^{\per}_k = \{\l^{\per}_{k,m}\}_{m \in \bZ}$, $k \in \oneton$, be the sequences of zeros of the $k$-th factor in this product. Clearly,
\begin{equation} \label{eq:lkm.per.proof}
 \l^{\per}_{k,m} = \frac{-i \ln c_k + 2 \pi m}{b_k}
 = \frac{\arg(c_k) + 2 \pi m}{b_k} - i \frac{\ln|c_k|}{b_k},
 \qquad m \in \bZ, \quad k \in \oneton.
\end{equation}
Thus, each sequence $\L^{\per}_k$, $k \in \oneton$, is algebraically simple and constitutes an arithmetic progression that lies on the line parallel to the real axis.

\textbf{(i)} Let $m \in \bZ$ and consider $n+1$ consecutive eigenvalues $\l_m^0, \l_{m+1}^0, \ldots, \l_{m+n}^0$ (ordered by their real parts). By pigeonhole principle, we can find two different indexes $p, q \in \{0,1,\ldots,n\}$, such that eigenvalues $\l_{m+p}^0$ and $\l_{m+q}^0$ belong to the same arithmetic progression $\L^{\per}_k$. Without loss of generality, we can assume that they are consecutive elements of this progression,
$$
\l_{m+p}^0 = \l^{\per}_{k,u}, \qquad \l_{m+q}^0 = \l^{\per}_{k,u + s_k},
$$
for some $0 \le p < q \le n$, $k \in \oneton$ and $u = u_{m+p,k}$. Here $s_k = \sign(b_k)$, i.e.\ $u + s_k = u+1$ if $b_k > 0$ and $u + s_k = u-1$ if $b_k<0$. Then formula~\eqref{eq:lkm.per.proof} implies that
$$
 \Re\l_{m+q}^0 - \Re\l_{m+p}^0
 = \Re\l^{\per}_{k,u + s_k} - \Re\l^{\per}_{k,u}
 = \frac{\arg(c_k) + 2 \pi (u+s_k)}{b_k} - \frac{\arg(c_k) + 2 \pi u}{b_k}
 = \frac{2 \pi}{|b_k|}.
$$
Hence there exists $r = r_m \in \{p, p+1, \ldots, q-1\}$ such that
$$
\Re\l_{m+r_m+1}^0 - \Re\l_{m+r_m}^0 \ge \frac{2 \pi}{|b_k|\cdot(q-p)} \ge \frac{2 \pi}{b_{\max} n} = \eps.
$$
Now we can choose the desired sequence $\{m_k\}_{k \in \bZ}$ as a sequence one obtains from the set ${\{m+r_m+1\}_{m \in \bZ}}$ after ordering it and removing repetitions.

\textbf{(ii)} It follows from~\eqref{eq:lkm.per.proof} and~\eqref{eq:bj.ln.ck.ne.bk.ln.cj} that
\begin{equation} \label{eq:l1n-l2m>=}
 |\l^{\per}_{j,p} - \l^{\per}_{k,m}| \ge |\Im \l^{\per}_{j,p} - \Im \l^{\per}_{k,m}|
 = \left| \frac{\ln|c_j|}{b_j} - \frac{\ln|c_k|}{b_k} \right| =: \eps_{j,k} > 0,
 \qquad m, p \in \bZ.
\end{equation}
It follows from~\eqref{eq:lkm.per.proof} and~\eqref{eq:l1n-l2m>=} that the sequence of zeros of $\Delta_0(\cdot)$ is separated. Namely, separation parameter $\delta$ can be chosen as any number less than $\frac12 \min\{\eps_{j,k} : 1 \le j < k \le n\} > 0$. Hence boundary conditions~\eqref{eq:yell=C.y0} are strictly regular.

\textbf{(iii)} It is clear that boundary conditions~\eqref{eq:yell=C.y0} are strictly regular if and only if for each $j \ne k$ arithmetic progressions $\L^{\per}_j$ and $\L^{\per}_k$ are asymptotically separated. Thus, we need to show that for each $j, k \in \oneton$ such that $j \ne k$, arithmetic progressions $\L^{\per}_j$ and $\L^{\per}_k$ are asymptotically separated if and only if condition~\eqref{eq:strict.crit.per} is satisfied. To this end, let $j, k \in \oneton$, $j \ne k$, be fixed.

First assume that $\alp := \alp_{jk} := b_j/b_k \not \in \bQ$. Then condition~\eqref{eq:strict.crit.per} is equivalent to~\eqref{eq:bj.ln.ck.ne.bk.ln.cj}. Part (ii) of the lemma implies that if condition~\eqref{eq:bj.ln.ck.ne.bk.ln.cj} is satisfied, then arithmetic progressions $\L^{\per}_j$ and $\L^{\per}_k$ are separated.
Now let condition~\eqref{eq:bj.ln.ck.ne.bk.ln.cj} be violated. In this case
$$
 \Im \l^{\per}_{j,p} = \Im \l^{\per}_{k,m} = -\frac{\ln|c_j|}{b_j}
 = -\frac{\ln|c_k|}{b_k}, \qquad p,m \in \bZ,
$$
i.e.\ the progressions $\L^{\per}_j$ and $\L^{\per}_k$ lie on the same line parallel to the real axis. Hence for each $p,m \in \bZ$ we have,
\begin{align}
\label{eq:l1n-l2m=}
 & |\l^{\per}_{j,p} - \l^{\per}_{k,m}| = \abs{\frac{\arg(c_j)}{b_j} - \frac{\arg(c_k)}{b_k}
 + 2 \pi\(\frac{p}{b_j} - \frac{m}{b_k}\)}
 = \frac{2 \pi}{b_j} \cdot | r + p - \alp m|, \\
 & \text{where}\quad r := r_{jk} := \frac{\arg(c_j) - \alp \arg(c_k)}{2 \pi} \in \bR.
\end{align}
Since $\alp$ is irrational, the Kronecker theorem ensures that for any $\eps > 0$ and $M > 0$ there exist $p,m \in \bZ$ such that $|p|, |m| > M$ and $|r + p - \alp m| < \eps$. This means that arithmetic progressions $\L^{\per}_j$ and $\L^{\per}_k$ are not asymptotically separated and finishes the proof when $b_j/b_k \notin \bQ$.

Now let $\alp = b_j/b_k \in \bQ$. As before, part (ii) of the lemma implies that if condition~\eqref{eq:bj.ln.ck.ne.bk.ln.cj} is satisfied, then arithmetic progressions $\L^{\per}_j$ and $\L^{\per}_k$ are separated. Assuming that condition~\eqref{eq:bj.ln.ck.ne.bk.ln.cj} is violated let us express a criterion for arithmetic progressions $\L^{\per}_j$ and $\L^{\per}_k$ to be asymptotically separated. As in the previous case, arithmetic progressions $\L^{\per}_j$ and $\L^{\per}_k$ lie on the same line parallel to the real axis and condition~\eqref{eq:l1n-l2m=} holds. Since $\alp$ is rational, the union of these progressions is asymptotically separated if and only if they have no common entries. Due to~\eqref{eq:l1n-l2m=} this is equivalent to the fact that Diophantine equation $p - \alp m = -r$ does not have integer solutions $p,m$. It is well-known that such equation has solutions if and only if $r / \gcd(\alp, 1) \in \bZ$. Since $r = \frac{\arg(c_j) - \alp \arg(c_k)}{2 \pi}$ and $\alp = \frac{b_j}{b_k}$, condition $\frac{r}{\gcd(\alp, 1)} \in \bZ$ is equivalent to $\frac{b_j \arg(c_k) - b_k \arg(c_j)}{2 \pi \gcd(b_j, b_k)} \in \bZ$. Comparing this with~\eqref{eq:strict.crit.per}, we see that the proof is now complete.

\textbf{(iv)} If $c_1 = \ldots = c_n =1$, then $\ln|c_k| = 0$ and $\arg c_k = 0$, $k \in \oneton$. Hence condition~\eqref{eq:strict.crit.per} is violated. Which implies that periodic boundary conditions are not strictly regular.

If $c_1 = \ldots = c_n -1$, then $\ln|c_k| = 0$ and $\arg c_k = \pi$, $k \in \oneton$. Hence condition~\eqref{eq:strict.crit.per} turns into
\begin{equation} \label{eq:b1+b2=2k+1}
 \frac{b_j}{b_k} \in \bQ \qquad\text{and}\qquad \frac{b_j - b_k}{\gcd(b_j, b_k)}
 \quad\text{is odd}, \qquad j \ne k.
\end{equation}
Let us simplify this condition. Assuming all ratios $b_j/b_k$ to be rational, we can choose a ``base'' $b_0 > 0$, such that $b_k = N_k b_0$, $N_k \in \bZ$, $k \in \oneton$. Let $N_k = 2^{a_k} M_k$, $k \in \oneton$, where $a_k \in \{0, 1, 2, \ldots\}$ and $M_k$ is odd. Clearly each non-zero integer has unique representation of this form. Let $j \ne k$. Without loss of generality we can assume that $a_j \le a_k$. Then, with the above representation in mind, we have
$$
b_j - b_k = 2^{a_j} b_0 (M_j - 2^{a_k-a_j} M_k), \qquad
\gcd(b_j, b_k) = 2^{a_j} b_0 \gcd(M_j, M_k).
$$
Since $M_j$ and $M_k$ are odd, it is clear, that the ratio $(b_j - b_k) / \gcd(b_j, b_k)$ is odd if and only if $a_j \ne a_k$, which finishes the proof.
\end{proof}
\begin{example}
In some cases we can select the ``blocks'' in inequality~\eqref{eq:Lambda0=union} and the number $\eps$ more explicitly. For instance, if $c_1, \ldots, c_n > 0$, and $b_j/b_k \in \bQ$, $j,k \in \oneton$, we show that $\eps$ in~\eqref{eq:Lambda0=union} can be chosen as
$$
 \eps := \min\left\{\frac{2 \pi \cdot \gcd(b_j, b_k)}{|b_j b_k|} :
 \ \ j, k \in \oneton, \ \ j \ne k \right\}.
$$
Indeed, let $j \ne k$ be fixed. In this case $\arg(c_j) = \arg(c_k) = 0$ and $b_j = N_j b_{jk}$, $b_k = N_k b_{jk}$, where $b_{jk} = \gcd(b_j, b_k) > 0$, $N_j, N_k \in \bZ$. Therefore, formula~\eqref{eq:lkm.per.proof} implies that
$$
\Re \l^{\per}_{j,p} = \frac{2 \pi p}{b_{jk} N_j}, \qquad
\Re \l^{\per}_{k,m} = \frac{2 \pi m}{b_{jk} N_k}, \qquad p, m \in \bZ.
$$
Hence for pairs of eigenvalues $\l^{\per}_{j,p}$, $\l^{\per}_{k,m}$, satisfying $\Re \l^{\per}_{j,p} \ne \Re \l^{\per}_{k,m}$ one derives
\begin{equation} \label{eq:Rel1n-Rel2m}
 |\Re \l^{\per}_{j,p} - \Re \l^{\per}_{k,m}|
 = \frac{2\pi \cdot |N_k p - N_j m|}{b_{jk} |N_j N_k|}
 \ge \frac{2\pi}{b_{jk} |N_j N_k|}
 = \frac{2\pi\cdot \gcd(b_j, b_k)}{|b_j b_k|} \ge \eps > 0.
\end{equation}
Therefore putting the eigenvalues $\l^{\per}_{j,p}$, $\l^{\per}_{k,m}$ with $\Re \l^{\per}_{j,p} = \Re \l^{\per}_{k,m}$ in one block we conclude from~\eqref{eq:Rel1n-Rel2m} that estimate~\eqref{eq:Lambda0=union} holds with the desired $\eps$.

Note also that if there are infinitely many pairs of eigenvalues
with equal real parts one cannot replace the second inequality
in~\eqref{eq:Lambda0=union} by $n_{k+1} \le n_k + 1$ for
$n_k$ big enough. As a consequence of this fact, the system of
root vectors of the BVP~\eqref{eq:LQ.def.reg}--\eqref{eq:Uy=0},
forms a Riesz basis only with parentheses (cf. Corollary
\ref{cor_sizes_of_Riesz_blocks}).
\end{example}
The next result establishes similar criterion for a certain subclass of separated boundary conditions.
\begin{lemma} \label{lem:separ.regular}
Let $n=2N$, $N \in \bN$, and let numbers $b_1, \ldots, b_n$
satisfy the following condition,
\begin{equation} \label{eq:b.2k-1<0<b.2k}
 b_1 < 0 < b_2, \quad b_3 < 0 < b_4, \quad \ldots, \quad b_{n-1} < 0 < b_n.
\end{equation}
Further, let boundary conditions~\eqref{eq:Uy=0} be of the form
\begin{equation} \label{eq:separ.cond}
 c_{2k-1} y_{2k-1}(0) + c_{2k} y_{2k}(0) = 0, \qquad
 d_{2k-1} y_{2k-1}(1) + d_{2k} y_{2k}(1) = 0, \qquad k \in \onetoN,
\end{equation}
where $c_j, d_j \ne 0$, $j \in \oneton$, i.e.\ $U(y) = C y(0) + D y(\ell) = 0$, where
\begin{align}
 C &= \diag\(
 \begin{pmatrix} c_1 & c_2 \\ 0 & 0 \end{pmatrix},
 \begin{pmatrix} c_3 & c_4 \\ 0 & 0 \end{pmatrix},
 \ldots,
 \begin{pmatrix} c_{n-1} & c_{n} \\ 0 & 0 \end{pmatrix}\), \\
 D &= \diag\(
 \begin{pmatrix} 0 & 0 \\ d_1 & d_2 \end{pmatrix},
 \begin{pmatrix} 0 & 0 \\ d_3 & d_4 \end{pmatrix},
 \ldots,
 \begin{pmatrix} 0 & 0 \\ d_{n-1} & d_{n} \end{pmatrix}\).
\end{align}
Then boundary conditions~\eqref{eq:separ.cond} are regular. Set
\begin{equation} \label{eq:sigma.tau.def}
 \sigma_k := b_{2k} - b_{2k-1} > 0, \qquad
 \tau_k := \frac{c_{2k} d_{2k-1}}{c_{2k-1} d_{2k}} \ne 0,
 \qquad k \in \onetoN.
\end{equation}

\textbf{(i)} Let $\L_0 = \{\l_m^0\}_{m \in \bZ}$ be the
sequence of zeros of the characteristic determinant
$\Delta_0(\cdot)$ and assume it is ordered in such a way that
$\Re \l_m^0 \le \Re \l_{m+1}^0$, $m \in \bZ$. Then there
exists a sequence of integers $\{m_k\}_{k \in \bZ}$, such that
\begin{equation} \label{eq:Lambda0=union.sep}
 m_k < m_{k+1} \le m_k + n/2, \quad \Re \l_{m_k}^0 - \Re \l_{m_k-1}^0 \ge \eps, \quad k \in \bZ,
\end{equation}
where
$$
\eps := \frac{4 \pi}{\sigma_{\max} n} > 0 \quad\text{and}\quad
\sigma_{\max} := \max\{\sigma_1, \ldots, \sigma_N\} =
\max\{b_2-b_1, b_4-b_3, \ldots, b_n-b_{n-1}\}.
$$

\textbf{(ii)} Let numbers $\left\{\frac{\ln |\tau_k|}{\sigma_k}\right\}_{k=1}^N$ be distinct, i.e.\
\begin{equation} \label{eq:sigmaj.ln.tauk}
 \sigma_j \ln |\tau_k| \ne \sigma_k \ln |\tau_j|, \qquad j \ne k,
 \quad j, k \in \onetoN.
\end{equation}
Then boundary conditions~\eqref{eq:separ.cond} are strictly regular. In particular, this is always the case if $n=2$.

\textbf{(iii)} More precisely, boundary conditions~\eqref{eq:separ.cond} are strictly regular if and only if for all $j \ne k$ the following condition holds,
\begin{equation} \label{eq:strict.crit.separ}
 \text{either}\quad \sigma_j \ln |\tau_k| \ne \sigma_k \ln |\tau_j|
 \quad\text{or}\quad \(
 \frac{\sigma_j}{\sigma_k} \in \bQ
 \quad\text{and}\quad
 \frac{\sigma_j \arg(\tau_k) - \sigma_k \arg(\tau_j)}{2 \pi \gcd(\sigma_j, \sigma_k)} \not \in \bZ\).
\end{equation}
\end{lemma}
\begin{proof}
Condition~\eqref{eq:b.2k-1<0<b.2k} implies that
$$
P_+ = \diag(0,1,0,1,\ldots,0,1), \qquad
P_- = \diag(1,0,1,0,\ldots,1,0).
$$
where ``projectors'' $P_{\pm}$ are defined in~\eqref{eq:P.pm.def}. Hence for determinant $J_{P_-}(C,D)$ given by~\eqref{eq:TP.CD.def} we have,
$$
J_{P_-}(C,D) = \det\(\diag\(
 \begin{pmatrix} 0 & c_2 \\ d_1 & 0 \end{pmatrix},
 \begin{pmatrix} 0 & c_4 \\ d_3 & 0 \end{pmatrix},
 \ldots,
 \begin{pmatrix} 0 & c_{n} \\ d_{n-1} & 0 \end{pmatrix}\)\) \ne 0,
$$
since numbers $c_j, d_j$, $j \in \oneton$ are non-zero. Similarly $J_{P_+}(C,D) \ne 0$. This implies regularity of boundary conditions. Further, it is clear that
$$
C + D \Phi_0(\ell, \l) = \diag\(
 \begin{pmatrix} c_1 & c_2 \\ d_1 e^{i \l b_1} & d_2 e^{i \l b_2} \end{pmatrix},
 \begin{pmatrix} c_3 & c_4 \\ d_3 e^{i \l b_3} & d_4 e^{i \l b_4} \end{pmatrix},
 \ldots,
 \begin{pmatrix} c_{n-1} & c_n \\
 d_{n-1} e^{i \l b_{n-1}} & d_n e^{i \l b_n} \end{pmatrix}\).
$$
Hence the characteristic determinant $\Delta_0(\cdot)$ defined in~\eqref{eq:Delta.def} becomes
\begin{equation} \label{eq:Delta0.separ}
 \Delta_0(\l) = \det(C + D \Phi_0(\ell, \l))
 = \prod_{k=1}^N \bigl(c_{2k-1} d_{2k} e^{i \l b_{2k}}
 - c_{2k} d_{2k-1} e^{i \l b_{2k-1}}\bigr).
\end{equation}
Let $\L^{\separ}_k = \{\l^{\separ}_{k,m}\}_{m \in \bZ}$, $k \in \oneton$, be the sequences of zeros of the $k$-th factor
in this product. Clearly,
\begin{equation} \label{eq:lkm.separ}
 \l^{\separ}_{k,m} = \frac{-i \ln \tau_k + 2 \pi m}{\sigma_k}
 = \frac{\arg(\tau_k) + 2 \pi m}{\sigma_k}
 - i \frac{\ln|\tau_k|}{\sigma_k},
 \qquad m \in \bZ, \quad k \in \onetoN,
\end{equation}
where $\sigma_k$ and $\tau_k$ are given by~\eqref{eq:sigma.tau.def}.
Thus, each sequence $\L_k^{\separ}$, $k \in \onetoN$, is algebraically simple and constitutes an arithmetic progression that lies on the line parallel to the real axis. From here the proof if finished the same way as in Lemma~\ref{lem:periodic.strict}.
\end{proof}
\section{Asymptotic behavior of eigenvalues and eigenvectors}
\label{sec:asymp.eigen}
\subsection{The key identity for characteristic determinant}
\label{subsec:determinant}
Here we present the key formula relating the characteristic determinants $\Delta_Q(\cdot)$ and $\Delta_0(\cdot)$.
\begin{proposition} \label{prop:DeltaQ=Delta0+int}
Let matrix functions $B(\cdot)$ and $Q(\cdot)$ satisfy conditions~\eqref{eq:Bx.def}--\eqref{eq:Qjk=0.bj=bk} and let $\Delta_Q(\l)$ and $\Delta_0(\l)$ be the characteristic determinants of BVP~\eqref{eq:LQ.def.reg}--\eqref{eq:Uy=0} and BVP~\eqref{eq:L0.def.reg},~\eqref{eq:Uy=0}, respectively, given by~\eqref{eq:Delta.def}. Then there exists function $g \in L^1[b_-, b_+]$, where $b_{\pm}$ are defined in~\eqref{eq:b-+.def}, such that the following identity holds
\begin{equation} \label{eq:DeltaQ=Delta0+int}
 \Delta_Q(\l) = \Delta_0(\l) + \int_{b_-}^{b_+} g(u) e^{i \l u} \,du, \qquad
 \l \in \bC.
\end{equation}
In addition, entries of adjugate matrices $A_Q^a(\l)$ and $A_0^a(\l)$, given by~\eqref{eq:Aa.def}--\eqref{eq:Aa0.def}, are connected via
\begin{equation} \label{eq:AjkQ=Ajk0+int}
 A_{kp}(\l) = A_{kp}^0(\l) + \int_{b_-}^{b_+} g_{kp}(u) e^{i \l u} \,du, \qquad \l \in \bC, \quad k,p \in \oneton,
\end{equation}
for some $g_{kp} \in L^1[b_-, b_+]$, $k,p \in \oneton$.
\end{proposition}
\begin{proof}
Let $\l \in \bC$ be fixed. Since $\Delta_Q(\l) = \det(C + D \Phi_Q(\ell, \l))$, applying formula~\eqref{eq:det.A+B.C} one gets,
\begin{equation} \label{eq:DeltaQ.minors}
 \Delta_Q(\l) = \det(C) + \sum_{m=1}^n \sum_{\fq, \fp, \fr \in \fP_m}
 (-1)^{\sigma(\fp) + \sigma(\fq)} C[\wh{\fq}, \wh{\fp}]
 \cdot D[\fq,\fr] \cdot \Phi_Q(\ell, \l)[\fr,\fp],
\end{equation}
where notations $\fP_m$, $\cA[\fp,\fq]$ and $\wh{\fp}$ were introduced in Subsection~\ref{subsec:det.sum.prod}. To transform~\eqref{eq:DeltaQ.minors}, we will apply formula~\eqref{eq:det.Phipq} for $\Phi_Q(\ell, \l)[\fr,\fp]$. To this end, observe that diagonal structure of the matrix function $\Phi_0(\cdot,\l)$, ${\Phi_0(\cdot,\l) = \diag(e^{i \l \rho_1(\cdot)}, \ldots, e^{i \l \rho_n(\cdot)})}$, implies
\begin{equation} \label{eq:det.Phi0pq}
 \Phi_0(\ell,\l)[\fr,\fp] = \delta_{\fr,\fp} \exp\(i \l \rho_{\fr}(\ell)\),
 \qquad \fr, \fp \in \fP_m.
\end{equation}
where $\rho_{\fr}(\ell) = \rho_{r_1}(\ell) + \ldots + \rho_{r_m}(\ell)$, $\fr = (r_1, \ldots, r_m)$, and was defined in~\eqref{eq:rho.fp.def}. Thus, setting $x = \ell$ in~\eqref{eq:det.Phipq} and taking into account formula~\eqref{eq:det.Phi0pq}, we get
\begin{equation} \label{eq:det.Phipq.ell}
 \Phi_Q(\ell,\l)[\fr,\fp] = \Phi_0(\ell,\l)[\fr,\fp] +
 \int_{\tau_m^-(\ell)}^{\tau_m^+(\ell)} R_{\fr,\fp}(\ell,u) e^{i \l u} du,
 \qquad \fr, \fp \in \fP_m.
\end{equation}
With account of notations~\eqref{eq:fhm-}--\eqref{eq:fhm+} and notation $b_k = \rho_k(\ell)$, we get
\begin{align} \label{eq:fhm-+.ell}
 \tau_m^-(\ell) = \min\{b_1 + \ldots + b_m, 0\}, \qquad
 \tau_m^+(\ell) = \max\{b_{n-m+1} + \ldots + b_n, 0\},
 \qquad m \in \oneton.
\end{align}
It is clear from the canonical ordering $b_1 \le \ldots \le b_{n_-} < 0 < b_{n_-+1} \le \ldots \le b_n$, definition~\eqref{eq:b-+.def} of $b_{\pm}$ and identities~\eqref{eq:fhm-+.ell} that
\begin{equation}
 [\tau_m^-(\ell), \tau_m^+(\ell)] \subset [b_-, b_+].
\end{equation}
Hence by setting
\begin{equation}
 g_{\fr,\fp}(u) := \begin{cases}
 R_{\fr,\fp}(\ell,u), & u \in [\tau_m^-(\ell), \tau_m^+(\ell)], \\
 0, & u \in [b_-, b_+] \setminus [\tau_m^-(\ell), \tau_m^+(\ell)],
 \end{cases}
 \quad \fr, \fp \in \fP_m.
\end{equation}
we can further transform~\eqref{eq:det.Phipq.ell} into
\begin{equation} \label{eq:det.Phipq.ell.final}
 \Phi_Q(\ell,\l)[\fr,\fp] = \Phi_0(\ell,\l)[\fr,\fp] +
 \int_{b_-}^{b_+} g_{\fr,\fp}(u) e^{i \l u} du,
 \qquad \fr, \fp \in \fP_m.
\end{equation}
Note also, that condition~\eqref{eq:R.fq.fp.in.X} implies inclusion $g_{\fr,\fp} \in L^1[b_-, b_+]$, $\fr, \fp \in \fP_m$.

Inserting~\eqref{eq:det.Phipq.ell.final} into~\eqref{eq:DeltaQ.minors} and using~\eqref{eq:DeltaQ.minors} for $Q=0$
we get
\begin{align}
\nonumber
 \Delta_Q(\l) & = \det(C) + \sum_{m=1}^n \sum_{\fq, \fp, \fr \in \fP_m}
 (-1)^{\sigma(\fp) + \sigma(\fq)} C[\wh{\fq}, \wh{\fp}]
 D[\fq,\fr] \cdot \(\Phi_0(\ell,\l)[\fr,\fp] +
 \int_{b_-}^{b_+} g_{\fr,\fp}(u) e^{i \l u} du\) \\
\label{eq:DeltaQ.minors.final}
 & = \Delta_0(\l) + \int_{b_-}^{b_+} g(u) e^{i \l u} du,
\end{align}
where
\begin{equation}
 g(u) := \sum_{\fq, \fp, \fr \in \fP_m}
 (-1)^{\sigma(\fp) + \sigma(\fq)} C[\wh{\fq}, \wh{\fp}]
 \cdot D[\fq,\fr] \cdot g_{\fr,\fp} (u).
\end{equation}
Since $g_{\fr,\fp} \in L^1[b_-, b_+]$, $\fr, \fp \in \fP_m$, it is clear that $g \in L^1[b_-, b_+]$, which finishes the proof of formula~\eqref{eq:DeltaQ=Delta0+int}.

Formula~\eqref{eq:AjkQ=Ajk0+int} can be obtained the same way as above by using Lemma~\ref{lem:det.A+B.C.jk} instead of Lemma~\ref{lem:det.A+B.C} and taking into account notations~\eqref{eq:Aajk.def}.
\end{proof}
\begin{remark}
Note that if $n_- = 0$, and so $b_-=0$ and $b_+ = b_1 + \ldots + b_n$, then integration limits in~\eqref{eq:AjkQ=Ajk0+int} can be reduced from $[0, b_+]$ to $[0, b_+ - b_1]$. Similar effect happens if $n_- = n$.
\end{remark}
For completeness of exposition, let us obtain similar Fourier transform related representation for vector functions $Y_p(\cdot,\l)$ defined in~\eqref{eq:Ypxl.def}. Study of these vector functions is motivated by Lemma~\ref{lem:eigen} and their appearance as eigenvectors of the operator $L_U(Q)$. To this end, let us set
\begin{align}
\label{eq:wt.rho-.def}
 \wh{\rho}_-(x) &:= \rho_1^-(x) + b_- = \min\{\rho_1(x), 0\}
 + b_1 + \ldots + b_{n_-}, \qquad x \in [0,\ell], \\
\label{eq:wt.rho+.def}
 \wh{\rho}_+(x) &:= \rho_n^+(x) + b_+ = \min\{\rho_n(x), 0\}
 + b_{n_- + 1} + \ldots + b_n, \qquad x \in [0,\ell].
\end{align}
where $\rho_1^-(x)$, $\rho_n^+(x)$ are defined in~\eqref{eq:rho-+def} and $b_{\pm}$ are defined in~\eqref{eq:b-+.def}.
\begin{proposition} \label{prop:eigenvec}
Let matrix functions $B(\cdot)$ and $Q(\cdot)$ satisfy conditions~\eqref{eq:Bx.def}--\eqref{eq:Qjk=0.bj=bk} and let $p \in \oneton$. Then there exists a measurable vector kernel $G_p$ defined on
$$
 \wh{\Omega} := \{(x,u) : x \in [0, \ell],
 u \in [\wh{\rho}_-(x), \wh{\rho}_+(x)]\}
$$
such that for each $x \in [0, \ell]$, a trace function $G_p(x, \cdot)$ is well-defined, summable,
\begin{equation} \label{eq:sup.int.Gp}
 \sup_{x \in [0,\ell]} \int_{\wh{\rho}_-(x)}^{\wh{\rho}_+(x)}
 \norm{G_p(x,u)}_{\bC^n} \,du < \infty,
\end{equation}
and the following representation holds
\begin{equation} \label{eq:Yp=Y0p+int}
 Y_p(x,\l) = Y_p^0(x,\l) + \int_{\wh{\rho}_-(x)}^{\wh{\rho}_+(x)} G_p(x,u)
 e^{i \l u} \,du, \qquad x \in [0, \ell], \quad \l \in \bC,
\end{equation}
where $Y_p(x,\l)$ and $Y_p^0(x,\l)$ are defined in~\eqref{eq:Ypxl.def}--\eqref{eq:Yp0xl.def}.
\end{proposition}
\begin{proof}
Let $p \in \oneton$, $x \in [0,\ell]$ and $\l \in \bC$ be fixed for the entire proof. Inserting formula~\eqref{eq:Phip=Phi0p+int1n} for $\Phi_p(x,\l)$ and formula~\eqref{eq:AjkQ=Ajk0+int} into formula~\eqref{eq:Ypxl.def} for $Y_p(x,\l)$ and taking into account formula~\eqref{eq:Yp0xl.def} for $Y_p^0(x,\l)$ we get
\begin{align}
\nonumber
 Y_p(x,\l) &:= \sum_{k=1}^n A_{kp}(\l) \Phi_k(x,\l) \\
\nonumber
 & = \sum_{k=1}^n \(A_{kp}^0(\l) + \int_{b_-}^{b_+} e^{i \l u} g_{kp}(u) \,du\)
 \(\Phi_k^0(x,\l)
 + \int_{\rho_1^-(x)}^{\rho_n^+(x)} e^{i \l u} \wt{R}_k(x,u) \,du\) \\
\nonumber
 & = Y_p^0(x,\l)
 + \sum_{k=1}^n \int_{\rho_1^-(x)}^{\rho_n^+(x)}
 e^{i \l u} A_{kp}^0(\l) \wt{R}_k(x,u) \,du
 + \sum_{k=1}^n \int_{b_-}^{b_+} e^{i \l u} g_{kp}(u) \Phi_k^0(x,\l) \,du \\
\label{eq:Yp.xl=3sum.int}
 & \qquad \qquad + \sum_{k=1}^n \int_{b_-}^{b_+} e^{i \l u} g_{kp}(u) \,du
 \int_{\rho_1^-(x)}^{\rho_n^+(x)} e^{i \l u} \wt{R}_k(x,u) \,du.
\end{align}

Let us analyze each sum in r.h.s.\ of~\eqref{eq:Yp.xl=3sum.int}.

\textbf{1st sum.} Note, that $A_{kp}^0(\l)$ has a formula similar to~\eqref{eq:Delta0.sum},
\begin{equation}
 A_{kp}^0(\l) = \sum_{P \in \cP_n} \gam_{kp}^{[P]} \cdot e^{i \l b_P},
 \qquad k,p \in \oneton, \quad
\end{equation}
with some coefficients $\gam_{kp}^{[P]}$ that only depend on matrices $C$ and $D$ from boundary conditions. Recall that $b_P = p_1 b_1 + \ldots + p_n b_n$, where $P = \diag(p_1, \ldots, p_n) = P^2$.
Hence the first sum in r.h.s.\ of~\eqref{eq:Yp.xl=3sum.int} turns into
\begin{align}
\nonumber
 \sum_{k=1}^n \int_{\rho_1^-(x)}^{\rho_n^+(x)} e^{i \l u} A_{kp}^0(\l)
 \wt{R}_k(x,u) \,du
 & = \sum_{k=1}^n \sum_{P \in \cP_n} \gam_{kp}^{[P]} \int_{\rho_1^-(x)}^{\rho_n^+(x)}
 e^{i \l (u + b_P)} \wt{R}_k(x,u) \,du \\
\nonumber
 & = \sum_{k=1}^n \sum_{P \in \cP_n} \gam_{kp}^{[P]} \int_{\rho_1^-(x) + b_P}^{\rho_n^+(x) + b_P}
 e^{i \l v} \wt{R}_k(x,v - b_P) \,dv \\
\label{eq:int.Akp0.wtRk}
 & = \int_{\rho_1^-(x) + b_-}^{\rho_n^+(x) + b_+} e^{i \l v} G_{1p}(x,v) \,dv,
\end{align}
with some $G_{1p}$ that satisfy condition~\eqref{eq:sup.int.Gp}. Here we used the fact that $b_- = \min\{b_P : P \in \cP_n\}$ and $b_+ = \max\{b_P : P \in \cP_n\}$.

\textbf{2nd sum.} Due to explicit formula
$$
 \Phi_k^0(x,\l) = e^{i \l \rho_k(x)} \col(\delta_{1k}, \ldots, \delta_{nk}),
 \qquad k \in \oneton,
$$
the second sum in r.h.s.\ of~\eqref{eq:Yp.xl=3sum.int} turns into
\begin{align}
\nonumber
 \sum_{k=1}^n \int_{b_-}^{b_+} e^{i \l u} g_{kp}(u) \Phi_k^0(x,\l) \,du
 & = \sum_{k=1}^n \int_{b_-}^{b_+} e^{i \l (u + \rho_k(x)} g_{kp}(u) \,du
 \cdot \col(\delta_{1k}, \ldots, \delta_{nk}) \\
\nonumber
 & = \sum_{k=1}^n \int_{\rho_k(x) + b_-}^{\rho_k(x) + b_+} e^{i \l v}
 g_{kp}(v - \rho_k(x)) \,dv \cdot \col(\delta_{1k}, \ldots, \delta_{nk}) \\
\label{eq:int.gkp.Phik0}
 & = \int_{\rho_1^-(x) + b_-}^{\rho_n^+(x) + b_+} e^{i \l v} G_{2p}(x,v) \,dv,
\end{align}
with some $G_{2p}$ that satisfy condition~\eqref{eq:sup.int.Gp}. Here, we used the fact that $\rho_1^-(x) \le \rho_k(x) \le \rho_n^+(x)$.

\textbf{3rd sum.} Finally, for the third sum in r.h.s.\ of~\eqref{eq:Yp.xl=3sum.int} we have after doing a change of variable $t = u+v$ and changing order of integration
\begin{align}
\nonumber
 \sum_{k=1}^n \int_{b_-}^{b_+} e^{i \l u} g_{kp}(u) \,du &
 \int_{\rho_1^-(x)}^{\rho_n^+(x)} e^{i \l u} \wt{R}_k(x,u) \,du \\
\nonumber
 & = \sum_{k=1}^n \int_{\rho_1^-(x)}^{\rho_n^+(x)}
 \(\int_{b_-}^{b_+} e^{i \l (u+v)} g_{kp}(u) \wt{R}_k(x,v) \,du\) \,dv \\
\nonumber
 & = \sum_{k=1}^n \int_{\rho_1^-(x)}^{\rho_n^+(x)}
 \(\int_{b_- + v}^{b_+ + v} e^{i \l t} g_{kp}(t-v) \wt{R}_k(x,v) \,dt\) \,dv \\
\nonumber
 & = \sum_{k=1}^n \int_{\rho_1^-(x) + b_-}^{\rho_n^+(x) + b_+}
 e^{i \l t} \(\int_{\rho_1^-(x)}^{\rho_n^+(x)} g_{kp}(t-v) \wt{R}_k(x,v) \,dv\) \,dt \\
 & = \int_{\rho_1^-(x) + b_-}^{\rho_n^+(x) + b_+}
 e^{i \l t} G_{3p}(x,t) \,dt,
\end{align}
with some $G_{3p}$ that satisfy condition~\eqref{eq:sup.int.Gp}.

Inserting formulas we obtained for each sum above into~\eqref{eq:Yp.xl=3sum.int}, we arrive at~\eqref{eq:Phip=Phi0p+int1n} with $G_p := G_{1p} + G_{2p} + G_{3p}$.
\end{proof}
\subsection{Asymptotic behavior of eigenvalues}
\label{subsec:asymp.eigenvalues}
To effectively estimate integral term in representations~\eqref{eq:DeltaQ=Delta0+int} and~\eqref{eq:AjkQ=Ajk0+int} we need the following simple generalization of Riemann-Lebesgue Lemma.
\begin{lemma}[cf. Lemma 3.5 in~\cite{LunMal16JMAA}] \label{lem:RiemLeb}
Let $a_- \le 0 \le a_+$ and let $f \in L^1[a_-,a_+]$. Then for any $\delta > 0$ there exists $R_{\delta} > 0$ such that the following estimate holds,
\begin{equation}
 \abs{\int_{a_-}^{a_+} f(u) e^{i \l u} \,du} < \delta \cdot
 \(e^{-\Im \l \cdot a_-} + e^{-\Im \l \cdot a_+}\),
 \qquad |\l| > R_{\delta}, \quad \l \in \bC.
\end{equation}
\end{lemma}
The following result generalizes Lemma~\ref{lem:Delta0.prop} and establishes certain important properties of the characteristic determinant $\Delta_Q(\cdot)$ as entire function of exponential type.
\begin{proposition} \label{prop:sine.type}
Let matrix functions $B(\cdot)$ and $Q(\cdot)$ satisfy conditions~\eqref{eq:Bx.def}--\eqref{eq:Qjk=0.bj=bk}, let boundary conditions~\eqref{eq:Uy=0} be regular and let $\Delta(\cdot) = \Delta_Q(\cdot)$ be the characteristic determinant of the problem~\eqref{eq:LQ.def.reg}--\eqref{eq:Uy=0} given by~\eqref{eq:Delta.def}. Then the following statements hold:

\textbf{(i)} The characteristic determinant $\Delta(\cdot)$ is a
sine-type function with $h_{\Delta}(\pi/2) = -b_-$ and
$h_{\Delta}(-\pi/2) = b_+$. In particular, $\Delta(\cdot)$ has
infinitely many zeros
\begin{equation} \label{eq:Lam.def}
 \L := \{\l_m\}_{m \in \bZ}
\end{equation}
counting multiplicity and $\L \subset \Pi_h$ for some $h>0$.

\textbf{(ii)} The sequence $\L$ is incompressible.

\textbf{(iii)} For any $\eps > 0$ the determinant
$\Delta(\cdot)$ admits the following estimate from below
\begin{equation} \label{eq:Delta>=}
 |\Delta(\l)| > C_{\eps}(e^{-\Im \l \cdot b_-} + e^{-\Im \l \cdot b_+}) > C_{\eps},
 \qquad \l \in \bC \setminus \bigcup_{m \in \bZ} \bD_{\eps}(\l_m),
\end{equation}
with some $C_{\eps} > 0$.

\textbf{(iv)} The sequence $\L$ can be ordered in such a
way that the following asymptotical formula holds
\begin{equation} \label{eq:lam.n=an+o1}
 \l_m = \frac{2 \pi m}{b_+ - b_-} + o(m) \quad\text{as}\quad m \to \infty.
\end{equation}
\end{proposition}
\begin{proof}
The proof is the same as in~\cite[Proposition 4.6]{LunMal16JMAA} with only a few minor differences. For reader's convenience we show the full proof here.

\textbf{(i)} Let $\Delta_0(\cdot)$ be the characteristic
determinant of the problem~\eqref{eq:LQ.def.reg}--\eqref{eq:Uy=0} with
$Q=0$. It easily follows from~\eqref{eq:Delta0.sum+-} that
$\Delta_0(\cdot)$ admits a representation
\begin{equation} \label{4.12}
\Delta_0(\l) = \int_{b_-}^{b_+}e^{i \l u} d\sigma_0(u), \qquad \l \in \bC,
\end{equation}
with a piecewise constant function $\sigma_0(\cdot)$ having up to $2^n$ jump-points $\left\{\sum_{k \in S} b_k : S \subset \oneton\right\}$. Regularity of boundary conditions and formula~\eqref{eq:Delta0.sum+-} imply that
\begin{equation} \label{4.13}
 \sigma_0(b_-+0) - \sigma_0(b_-) = J_{P_-}(C,D) \ne 0
 \quad\text{and}\quad
 \sigma_0(b_+) - \sigma_0(b_+ -0) = J_{P_+}(C,D) \ne 0.
\end{equation}
Proposition~\ref{prop:DeltaQ=Delta0+int} implies representation~\eqref{eq:DeltaQ=Delta0+int} with certain $g \in L^1[b_-, b_+]$.
Let us set
\begin{equation} \label{4.15}
 \sigma(u) = \sigma_0(u) + \int_{b_-}^u g(s) ds, \qquad u \in [b_-, b_+].
\end{equation}
Combining these notations with formulas~\eqref{eq:DeltaQ=Delta0+int}
and~\eqref{4.12} we arrive at the following representation for
the characteristic determinant
\begin{equation} \label{4.12New}
 \Delta(\l) = \int_{b_-}^{b_+} e^{i \l u} d\sigma(u), \qquad \l \in \bC,
\end{equation}
It follows from~\eqref{4.15} and~\eqref{4.13} that
\begin{equation} \label{4.17}
 \sigma(b_-+0) - \sigma(b_-) = J_{P_-}(C,D) \ne 0 \quad \text{and}\quad \sigma(b_+) -
\sigma(b_+ -0) = J_{P_+}(C,D) \ne 0.
\end{equation}
Due to the property~\eqref{4.17} representation~\eqref{4.12New}
ensures that $\Delta(\cdot)$ is a sine-type function with
$h_{\Delta_0}(\pi/2) = -b_-$ and $h_{\Delta_0}(-\pi/2) = b_+$
(see~\cite{Lev96}). Moreover, statement \textbf{(i)} is also
implied by the representation~\eqref{4.12New} (see~\cite[Chapter
1.4.3]{Leon76}).

\textbf{(ii)} and \textbf{(iii)}. These statements coincide
with the corresponding statements of~\cite[Lemmas 3 and
4]{Katsn71} for sine-type functions (see also~\cite[Lemma
22.1]{Lev96} in connection with part \textbf{(iii)}).

\textbf{(iv)} The proof is the same as in~\cite[Proposition 4.6(iv)]{LunMal16JMAA}.
\end{proof}
Based on Lemma~\ref{lem:Delta0.prop} the characteristic determinant $\Delta_0(\cdot)$ given by~\eqref{eq:Delta.def} has the same properties provided that boundary conditions~\eqref{eq:Uy=0} are regular. Recall, that $\L_0 = \{\l_m^0\}_{m \in \bZ}$ is the sequence of its zeros counting multiplicity. Let us order the sequence $\L_0$ in a (possibly non-unique) way such that $\Re \l_m^0 \le \Re \l_{m+1}^0$, $m \in \bZ$. The following result establishes a key asymptotic formula for zeros of $\Delta_Q(\cdot)$ (eigenvalues of the operator $L_U(Q)$).
\begin{theorem} \label{th:ln=ln0+o}
Let matrix functions $B(\cdot)$ and $Q(\cdot)$ satisfy conditions~\eqref{eq:Bx.def}--\eqref{eq:Qjk=0.bj=bk}, in particular
\begin{equation} \label{eq:Qjk=0.bj=bk.eigen}
 Q_{jk} \equiv 0 \quad\text{whenever}\quad \beta_j \equiv \beta_k,
 \qquad j,k \in \oneton.
\end{equation}
Let boundary conditions~\eqref{eq:Uy=0} be regular and let $\L_0 = \{\l_m^0\}_{m \in \bZ}$ be the sequence of eigenvalues (counting multiplicity) of the unperturbed operator $L_U(0)$ (sequence of zeros of the characteristic determinant $\Delta_0(\cdot)$). Then operator $L_U(Q)$ has a discrete spectrum and the sequence $\L = \{\l_m\}_{m \in \bZ}$ of its eigenvalues (counting multiplicity), which is the sequence of zeros the characteristic determinants $\Delta_Q(\l)$ of BVP~\eqref{eq:LQ.def.reg}--\eqref{eq:Uy=0} given by~\eqref{eq:Delta.def}, can be ordered in such a way that the following asymptotic formula holds
\begin{equation} \label{eq:lm=lm0+o}
 \l_m = \l_m^0 + o(1) \quad\text{as}\quad m \to \infty, \quad m \in \bZ.
\end{equation}
\end{theorem}
\begin{proof}
Let $\eps \in (0, 1)$. By Proposition~\ref{prop:sine.type}(iii) there exists $C_{\eps}> 0$ such that the estimate~\eqref{eq:Delta>=} for $\Delta(\cdot)$ holds, On the other hand, it follows from Lemma~\ref{lem:RiemLeb} with $\delta = C_{\eps}$ and Proposition~\ref{prop:DeltaQ=Delta0+int} that
\begin{equation} \label{eq:Delta-Delta0}
 |\Delta(\l) - \Delta_0(\l)|
 = \abs{\int_{b_-}^{b_+} g(u) e^{i \l u} \,du}
 < C_{\eps} (e^{-\Im \l \cdot b_-} + e^{-\Im \l \cdot b_+}),
 \qquad |\l| \ge \ R_{\eps},
\end{equation}
with certain $R_{\eps} > 0$. Combining estimate~\eqref{eq:Delta>=} with~\eqref{eq:Delta-Delta0} yields
\begin{align}
\label{eq:Rouche_estimate}
 &|\Delta(\l) - \Delta_0(\l)| < |\Delta(\l)|, \qquad \l \not \in \wt{\Omega}_{\eps}, \\
\label{eq:wtOmega.eps}
 &\wt{\Omega}_{\eps} := \bD_{R_{\eps}}(0) \cup \Omega_{\eps},
 \qquad \Omega_{\eps} := \bigcup\limits_{m \in \bZ} \bD_{\eps}(\l_m).
\end{align}

The proof is finished the same way as the proof of~\cite[Proposition 4.7]{LunMal16JMAA} by using~\cite[Lemma 4.3]{LunMal16JMAA} (certain geometric property of incompressible sequences) and Rouch\'e theorem.
\end{proof}
\subsection{Asymptotic behavior of root vectors}
\label{subsec:asymp.eigenvectors}
Based on Lemma~\ref{lem:eigen}, for a given $p \in \oneton$, vector function $Y_p(\cdot, \l)$ given by~\eqref{eq:Ypxl.def} is the eigenvector of the operator $L_U(Q)$ corresponding to the eigenvalue $\l$, provided that this function is not zero. To obtain asymptotic behavior of such eigenvectors we first need the following asymptotic result.
\begin{lemma} \label{lem:Yp=Yp0+o.Pih}
Let matrix functions $B(\cdot)$ and $Q(\cdot)$ satisfy conditions~\eqref{eq:Bx.def}--\eqref{eq:Qjk=0.bj=bk}. Let $h \ge 0$ and $p \in \oneton$. Then the following uniform asymptotic formula holds:
\begin{equation} \label{eq:Yp=Yp0+o.Pih}
 Y_p(x,\l) = Y_p^0(x,\l) + o(1), \quad x \in [0,\ell], \quad\text{as}\quad
 \l \to \infty, \quad \l \in \Pi_h.
\end{equation}
\end{lemma}
\begin{proof}
Let $p \in \oneton$ be fixed for the entire proof. Recall, that by definition,
\begin{equation} \label{eq:Ypxl}
 Y_p(x,\l) = \sum_{k=1}^n A_{kp}(\l) \Phi_k(x,\l),
 \qquad Y_p^0(x,\l) = \sum_{k=1}^n A_{kp}^0(\l) \Phi_k^0(x,\l),
 \qquad x \in [0,\ell], \quad \l \in \bC.
\end{equation}
Hence
\begin{equation} \label{eq:Ypxl-Yp0xl}
 Y_p(x,\l) - Y_p^0(x,\l)
 = \sum_{k=1}^n A_{kp}(\l) (\Phi_k(x,\l) - \Phi_k^0(x,\l))
 + \sum_{k=1}^n (A_{kp}(\l) - A_{kp}^0(\l)) \Phi_k^0(x,\l),
\end{equation}
for $x \in [0,\ell]$ and $\l \in \bC$.

Let $\delta > 0$. It follows from Lemma~\ref{lem:RiemLeb} and formula~\eqref{eq:AjkQ=Ajk0+int} (see Proposition~\ref{prop:DeltaQ=Delta0+int}) that for given $k \in \oneton$ we have
\begin{equation} \label{eq:Akp-Akp0}
 |A_{kp}(\l) - A_{kp}^0(\l)|
 = \abs{\int_{b_-}^{b_+} g_{kp}(u) e^{i \l u} \,du}
 < \delta \cdot (e^{-\Im \l \cdot b_-} + e^{-\Im \l \cdot b_+}),
 \qquad |\l| > R_{\delta},
\end{equation}
for some $R_{\delta} > 0$. This implies that
\begin{equation} \label{eq:Akp-Akp0<eps}
 |A_{kp}(\l) - A_{kp}^0(\l)| < \delta \cdot (e^{-h b_-} + e^{h b_+}) =: \eps,
 \qquad |\l| > R_{\delta}, \quad \l \in \Pi_h, \quad k \in \oneton.
\end{equation}
It follows from asymptotic formula~\eqref{eq:Phip=Phip0+o} for $\Phi_k(x,\l)$ that
\begin{equation} \label{eq:Phik-Phik0<eps}
 \norm{\Phi_k(x,\l) - \Phi_k^0(x,\l)}_{\bC^n} < \eps,
 \qquad x \in [0,\ell]
 \qquad |\l| > R_{\eps}', \quad \l \in \Pi_h, \quad k \in \oneton,
\end{equation}
for some $R_{\eps}' \ge R_{\delta}$.
It follows from Lemma~\ref{lem:easy.upper.bound} that
\begin{equation} \label{eq:Akp.Phik}
 |A_{kp}^0(\l)| \le M, \qquad \norm{\Phi_k^0(x,\l)}_{\bC^n} \le M,
 \qquad x \in [0,\ell], \quad \l \in \Pi_h, \quad k \in \oneton,
\end{equation}
for some $M > 0$. Combining~\eqref{eq:Akp.Phik} and~\eqref{eq:Akp-Akp0<eps} we see that
\begin{equation} \label{eq:Akp<M+eps}
 |A_{kp}(\l)| \le M + \eps, \qquad |\l| > R_{\delta},
 \quad \l \in \Pi_h, \quad k \in \oneton.
\end{equation}
Inserting all of the above estimates into~\eqref{eq:Ypxl-Yp0xl} we arrive at
\begin{align}
\nonumber
 \norm{Y_p(x,\l) - Y_p^0(x,\l)}_{\bC^n}
 & \le \sum_{k=1}^n
 |A_{kp}(\l)| \cdot \norm{\Phi_k(x,\l) - \Phi_k^0(x,\l))}_{\bC^n} \\
\nonumber
 & \qquad \qquad + \sum_{k=1}^n |A_{kp}(\l) - A_{kp}^0(\l)| \cdot \norm{\Phi_k^0(x,\l)}_{\bC^n} \\
\nonumber
 & \le \sum_{k=1}^n ((M + \eps) \eps + \eps M) \\
\label{eq:Yp-Yp0<eps}
 &= 2 n \eps (2 M + \eps),
 \qquad x \in [0,\ell], \quad |\l| > R_{\eps}', \quad \l \in \Pi_h.
\end{align}
Since $\delta > 0$ can be chosen arbitrary small and $\eps = \delta \cdot (e^{-h b_-} + e^{h b_+})$, then estimate~\eqref{eq:Yp-Yp0<eps} implies desired uniform asymptotic relation~\eqref{eq:Yp=Yp0+o.Pih}.
\end{proof}
\begin{remark}
It might be tempting to use Proposition~\ref{prop:eigenvec} to prove Lemma~\ref{lem:Yp=Yp0+o.Pih}, but, unfortunately, as explained in Remark~\ref{rem:RiemLeb.uniform}, the properties of the vector kernel $G_p$ from representation~\eqref{eq:Yp=Y0p+int} are not sufficient to prove relation~\eqref{eq:Yp=Yp0+o.Pih} \emph{uniformly} at $x \in [0,\ell]$.
\end{remark}
Now we are ready to state asymptotic result for root vectors of the operator $L_U(Q)$. Going forward we will call sequence $\{\phi_m\}_{m \in \bZ}$ of vectors in $\fH$ \emph{normalized} if $\|\phi_m\|_{\fH} = 1$, $m \in \bZ$.
\begin{theorem} \label{th:eigenvec}
Let matrix functions $B(\cdot)$ and $Q(\cdot)$ satisfy conditions~\eqref{eq:Bx.def}--\eqref{eq:Qjk=0.bj=bk}, let boundary conditions~\eqref{eq:Uy=0} be regular and let $\L = \{\l_m\}_{m \in \bZ}$ and $\L_0 = \{\l_m^0\}_{m \in \bZ}$ be the sequences of zeros of characteristic determinants $\Delta_Q(\cdot)$ and $\Delta_0(\cdot)$, respectively, satisfying asymptotic formula~\eqref{eq:lm=lm0+o}.

\textbf{(i)} Let $p \in \oneton$. Then the following asymptotic formula holds uniformly at $x \in [0,\ell]$
\begin{equation} \label{eq:Yp=Yp0+o}
 Y_p(x,\l_m) = Y_p^0(x,\l_m^0) + o(1) \quad\text{as}\quad m \to \infty, \quad m \in \bZ,
\end{equation}
where vector functions $Y_p(\cdot,\l)$ and $Y_p^0(\cdot,\l)$ are given by~\eqref{eq:Ypxl.def}--\eqref{eq:Yp0xl.def}. In particular,
\begin{equation} \label{eq:|Yp-Yp0|}
 \norm{Y_p(\cdot,\l_m) - Y_p^0(\cdot,\l_m^0)}_{\fH} \to 0 \quad\text{as}\quad m \to \infty, \quad m \in \bZ.
\end{equation}

\textbf{(ii)} Let in addition boundary conditions~\eqref{eq:Uy=0} be strictly regular, then one can choose normalized system of root vector $\{f_m\}_{m \in \bZ}$ of the operator $L_U(Q)$ and normalized system of root vector $\{f_m^0\}_{m \in \bZ}$ of the operator $L_U(0)$ such that the following relation holds
\begin{align} \label{eq:fm-fm0.to.0}
 \|f_m - f_m^0\|_{\infty} := \|f_m - f_m^0\|_{C([0,\ell]; \bC^n)} \to 0 \quad\text{as}\quad m \to \infty, \quad m \in \bZ.
\end{align}
Moreover, for sufficiently large $|m|$, eigenvalues $\l_m$ and $\l_m^0$ are simple, and the corresponding eigenvectors $f_m$ and $f_m^0$ admit the following representation,
\begin{equation} \label{eq:fm.fm0.th}
 f_m(\cdot) = \alp_m Y_{p_m}(\cdot,\l_m), \qquad
 f_m^0(\cdot) = \alp_m^0 Y_{p_m}^0(\cdot,\l_m^0),
\end{equation}
for some $\alp_m, \alp_m^0 \in \bC \setminus \{0\}$ and $p_m \in \oneton$.
\end{theorem}
\begin{proof}
\textbf{(i)} Let $\delta > 0$. By Proposition~\ref{prop:sine.type}, $\l_m, \l_m^0 \in \Pi_h$, $m \in \bZ$, for some $h \ge 0$. It is also clear, that $\l_m \to \infty$ and $\l_m^0 \to \infty$ as $m \to \infty$. Hence Lemma~\ref{lem:Yp=Yp0+o.Pih} implies that
\begin{equation} \label{eq:|Yp-Yp0|<delta}
 \norm{Y_p(\cdot,\l_m) - Y_p^0(\cdot,\l_m)}_{\infty} < \delta,
 \qquad |m| \ge \wt{m}_{\delta},
\end{equation}
for some $\wt{m}_{\delta} \in \bN$. Since boundary conditions~\eqref{eq:Uy=0} are regular, then by Theorem~\ref{th:ln=ln0+o} there exists $m_{\delta} \ge \wt{m}_{\delta}$ such that
\begin{equation} \label{eq:lm-lm0<delta}
 |\l_m - \l_m^0| < \delta, \qquad |m| > m_{\delta}.
\end{equation}
It follows from~\eqref{eq:e.rhokx<M.dlYk0<M} and~\eqref{eq:lm-lm0<delta} that
\begin{align} \label{eq:|Yp0-Yp00|<dif+int}
 \norm{Y_p^0(\cdot,\l_m) - Y_p^0(\cdot,\l_m^0)}_{\infty} \le \sup_{\l \in \Pi_h}
 \norm{\frac{d}{d\l} Y_p^0(\cdot,\l)}_{\infty} |\l_m - \l_m^0|
 & \le M_h \cdot \delta,
 \qquad |m| > m_{\delta}.
\end{align}
Since $\delta > 0$ can be chosen arbitrarily small, estimates~\eqref{eq:|Yp-Yp0|<delta},~\eqref{eq:|Yp0-Yp00|<dif+int} imply desired relation~\eqref{eq:Yp=Yp0+o}.

Finally, note that for any $f = \col(f_1, \ldots, f_n) \in C([0,\ell]; \bC^n)$ we have,
\begin{equation} \label{eq:f.fH}
 \|f\|_{\fH}^2 = \sum_{k=1}^n \int_0^{\ell} |f_k(x)|^2 |\beta_k(x)| dx
 \le \|f\|_{\infty}^2 \cdot (|b_1| + \ldots + |b_n|)
 = \|f\|_{\infty}^2 \cdot (b_+ - b_-).
\end{equation}
Relation~\eqref{eq:|Yp-Yp0|} is now implied by~\eqref{eq:Yp=Yp0+o} and~\eqref{eq:f.fH}.

\textbf{(ii)} Strict regularity of boundary conditions~\eqref{eq:Uy=0} and asymptotic formula~\eqref{eq:lm=lm0+o} imply that for some $m_1 \ge m_0$, eigenvalues $\l_m$ and $\l_m^0$, $|m| > m_1$, of operators $L_U(Q)$ and $L_U(0)$ are algebraically and geometrically simple. By Lemma~\ref{lem:D0jk>}, there exist indices $p = p_m \in \oneton$ and $q = q_m \in \oneton$, and a constant $C_2>0$ such that estimate~\eqref{eq:D0jk>} holds, i.e.\ $|A_{qp}^0(\l_m^0)| \ge C_2$, $|m| > m_0$. Emphasize, that although $p$ and $q$ depend on $m$, the constant $C_2$ in the above estimate does not. By the proof of Proposition~\ref{prop:eigenvec0}, this choice of $p=p_m$ guarantees uniform estimates~\eqref{eq:Yp0.norm},
\begin{equation} \label{eq:Yp0.norm.minim}
 C_3 \le \|Y_p^0(\cdot, \l_m^0)\|_{\fH} \le C_4, \qquad |m| > m_0,
\end{equation}
where $C_4 > C_3 > 0$ do not depend on $m$. Combining this estimate with relation~\eqref{eq:|Yp-Yp0|} imply that for some $m_2 \ge m_1$, we have
\begin{equation} \label{eq:Yp.norm}
 C_3/2 \le \|Y_p(\cdot,\l_m)\|_{\fH} \ge 2 C_4, \qquad |m| > m_2.
\end{equation}
Note also that trivial estimates~\eqref{eq:e.rhokx<M.dlYk0<M}--\eqref{eq:A0.Delta0<M} provide uniform estimate on $\|Y_p^0(\cdot, \l)\|_{\infty}$, $\l \in \Pi_h$. Hence,
\begin{equation} \label{eq:Yp0.normC}
 \|Y_p^0(\cdot,\l_m^0)\|_{\infty}
 \le C_5, \qquad m \in \bZ,
\end{equation}
for some $C_5 > 0$ that does not depend on $m$ and $p=p_m$.

Since vector functions $Y_p(\cdot,\l_m)$ and $Y_p^0(\cdot,\l_m^0)$ are non-zero for $|m| > m_2$, Lemma~\ref{lem:eigen} implies that they are proper eigenvectors of the operators $L_U(Q)$ and $L_U(0)$ corresponding to simple eigenvalues $\l_m$ and $\l_m^0$, respectively. Let us normalize them, by setting
\begin{equation} \label{eq:fm.fm0.def}
 f_m(\cdot) := \frac{Y_p(\cdot, \l_m)}{\|Y_p(\cdot, \l_m)\|_{\fH}}, \qquad
 f_m^0(\cdot) := \frac{Y_p^0(\cdot, \l_m^0)}{\|Y_p^0(\cdot, \l_m^0)\|_{\fH}},
 \qquad |m| > m_2.
\end{equation}
For any vector functions $u, v \in C([0,\ell]; \bC^n)$ we have
\begin{equation}
 \norm{\frac{u}{\|u\|_{\fH}} - \frac{v}{\|v\|_{\fH}}}_{\infty}
 \le \frac{\|u - v\|_{\infty}}{\|u\|_{\fH}}
 + \frac{\|v\|_{\infty} \bigabs{\|u\|_{\fH} - \|v\|_{\fH}}}
 {\|u\|_{\fH} \|v\|_{\fH}}
 \le \frac{\|v\|_{\fH} \|u - v\|_{\infty}
 +\|v\|_{\infty} \|u - v\|_{\fH}}
 {\|u\|_{\fH} \|v\|_{\fH}}.
\end{equation}
Setting $u(\cdot) = Y_p(\cdot, \l_m)$ and $v(\cdot) = Y_p^0(\cdot, \l_m^0)$ in this inequality and combining it with relations~\eqref{eq:Yp=Yp0+o}--\eqref{eq:|Yp-Yp0|}, estimates~\eqref{eq:Yp0.norm.minim}--\eqref{eq:Yp.norm} from below on $\|u\|_{\fH} = \|Y_p(\cdot, \l_m)\|_{\fH}$ and $\|v\|_{\fH} = \|Y_p^0(\cdot, \l_m^0)\|_{\fH}$ and estimate~\eqref{eq:Yp0.normC} on $\|v\|_{\infty} = \|Y_p^0(\cdot, \l_m^0)\|_{\infty}$ from above, we arrive at the desired relation~\eqref{eq:fm-fm0.to.0} for sequences $\{f_m\}_{|m| > m_2}$ and $\{f_m^0\}_{|m| > m_2}$.
Extending these sequences by arbitrary chosen normalized root vectors of the operators $L_U(Q)$ and $L_U(0)$ corresponding to eigenvalues $\l_m$ and $\l_m^0$ for $|m| \le m_2$, we arrive at the normalized system of root vectors $\{f_m\}_{m \in \bZ}$ and $\{f_m^0\}_{m \in \bZ}$, satisfying desired relation~\eqref{eq:fm-fm0.to.0}. Formula~\eqref{eq:fm.fm0.def} trivially implies~\eqref{eq:fm.fm0.th}.
\end{proof}
\begin{remark}
In a very recent paper~\cite{Rzep21} L. Rzepnicki obtained sharp asymptotic formulas for deviations $\l_n - \l_n^0 = \delta_n + \rho_n$ in the case of Dirichlet BVP for Dirac system, i.e.\ system~\eqref{eq:LQ.def.reg} with $B(\cdot) \equiv \diag(-1,1)$, with $Q \in \LL{p}$, $1 \le p < 2$. Namely, $\delta_n$ is explicitly expressed via Fourier coefficients and Fourier transforms of $Q_{12}$ and $Q_{21}$, while $\{\rho_n\}_{n \in \bZ} \in \ell^{p'/2}(\bZ)$.
Similar result was obtained for eigenvectors. For Sturm-Liouville operators with singular potentials, A.~Gomilko and L.~Rzepnicki obtained similar results in another recent paper~\cite{GomRze20}.
\end{remark}
\subsection{The case of matrix function \texorpdfstring{$Q$}{Q} with non-trivial block diagonal}
\label{subsec:general.Q}
Our main results on asymptotic behavior of eigenvalues and eigenvectors, Theorems~\ref{th:ln=ln0+o} and~\ref{th:eigenvec}, assume that $Q$ satisfies ``zero block diagonality'' condition~\eqref{eq:Qjk=0.bj=bk}. Let us formulate them without this condition by reducing
general case to ``zero block diagonal'' case using special gauge transform.

To this end, recall that the matrix function $B(\cdot) = B(\cdot)^*$ has block-diagonal form~\eqref{eq:Bx.block.wt.def},
\begin{equation} \label{eq:Bx.block.wt}
 B = \diag(\wt{\beta}_1 I_{n_1}, \ldots, \wt{\beta}_r I_{n_r}), \qquad
 \wt{\beta}_k \in L^1([0,\ell]; \bR \setminus \{0\}),
 \qquad k \in \onetor,
\end{equation}
where $n_1 + \ldots + n_r = n$.
Matrix-function $Q(\cdot)$ has related block-matrix decomposition,
\begin{equation} \label{eq:Q=wtQjk}
 Q =: (\cQ_{jk})_{j,k=1}^r, \qquad
 \cQ_{jk} \in L^1([0,\ell]; \bC^{n_j \times n_k}), \qquad j, k \in \onetor.
\end{equation}
Let $Q_{\diag}(\cdot)$ be the block diagonal of matrix function $Q(\cdot)$,
\begin{equation} \label{eq:Qdiag.def}
 Q_{\diag} := \diag(\cQ_{11}, \ldots, \cQ_{rr}),
\end{equation}
and let $W(\cdot)$ be the ${n \times n}$-matrix solution of the Cauchy problem
\begin{equation} \label{eq:W'+QW}
 W'(x) + Q_{\diag}(x) W(x) = 0, \quad x \in [0,\ell],
 \qquad W(0) = I_n.
\end{equation}
Since $Q_{\diag}$ is summable and has block-diagonal form~\eqref{eq:Qdiag.def}, it is clear that
\begin{equation} \label{eq:Wx=diag}
 W(x) = \diag(W_{11}(x), \ldots, W_{rr}(x)),
 \qquad W_{kk}(x) \in \bC^{n_k \times n_k},
 \quad k \in \onetor, \quad x \in [0,\ell],
\end{equation}
and
\begin{equation} \label{eq:WAC.WB=BW}
 W, W^{-1} \in \AC([0,\ell],\bC^{n \times n}), \qquad
 W(x) B(x) = B(x) W(x), \quad x \in [0,\ell].
\end{equation}
Let us also define operator $\cW : \fH \to \fH$ such that
\begin{equation} \label{eq:cWTh.def}
 (\cW y)(x) = W(x) y(x), \qquad y \in \fH.
\end{equation}
Inclusions $W, W^{-1} \in \AC([0,\ell],\bC^{n \times n})$ imply that $\cW$ is bounded in $\fH$ and have a bounded inverse.
\begin{lemma} \label{lem:gauge}
Let $Q \in L^1([0,\ell]; \bC^{n \times n})$ and let matrix functions $B(\cdot)$, $Q_{\diag}(\cdot)$ and $W(\cdot)$ be given by~\eqref{eq:Bx.block.wt},~\eqref{eq:Qdiag.def},~\eqref{eq:Wx=diag} and satisfy conditions above. Let also operator $\cW$ be given by~\eqref{eq:cWTh.def}. Then the following statements hold:

\textbf{(i)} Operator $\cW$ (gauge transform) transforms operator $L_U(Q)$ to the operator $L_{\wt{U}}(\wt{Q})$ with the same matrix function $B(\cdot)$,
\begin{equation} \label{eq:LUQ.simil}
 L_{\wt{U}}(\wt{Q}) = \cW^{-1} L_U(Q) \cW,
\end{equation}
where
\begin{equation} \label{eq:wtUy=0}
 \wt{U}(y):= C y(0) + D W(\ell) y(\ell) = 0,
 \quad\text{and}\quad \wt{Q} := W^{-1} (Q - Q_{\diag}) W.
\end{equation}

\textbf{(ii)} Matrix $\wt{Q}$ has zero block diagonal with respect to the decomposition $\bC^n = \bC^{n_1} \oplus \ldots \oplus \bC^{n_r}$.

\textbf{(iii)} Characteristic determinants corresponding to the operators $L_{\wt{U}}(\wt{Q})$ and $L_U(Q)$ coincide.

\textbf{(iv)} Boundary conditions $\wt{U}(y)=0$ and $U(y)=0$ are regular only simultaneously.
\end{lemma}
\begin{proof}
\textbf{(i)} Let $y \in \AC([0,\ell]; \bC^n)$. It follows from~\eqref{eq:W'+QW} and~\eqref{eq:WAC.WB=BW} that
\begin{multline} \label{eq:cLQ.WThy}
 \cL(Q) \cdot \cW y = -i B^{-1} (W y' + W' y + Q W y)
 = -i B^{-1} (W y' - Q_{\diag} W y + Q W y) \\
 = W \bigl(-i B^{-1} (y' + W^{-1}(Q - Q_{\diag})W y)\bigr)
 = \cW \cdot \cL(\wt{Q}) y.
\end{multline}
Since $W(0) = I_n$, it is also clear that if $U(\cW y) = \wt{U}(y)$. Hence~\eqref{eq:cLQ.WThy} implies~\eqref{eq:LUQ.simil}.

\textbf{(ii)} It is clear that $Q - Q_{\diag}$ has zero block diagonal with respect to the decomposition $\bC^n = \bC^{n_1} \oplus \ldots \oplus \bC^{n_r}$. Block-diagonal form~\eqref{eq:Wx=diag} of $W$, implies that $\wt{Q} = W^{-1} (Q - Q_{\diag}) W$ also has zero block diagonal.

\textbf{(iii)} It follows from~\eqref{eq:cLQ.WThy} that $\wt{\Phi} = W^{-1} \Phi_Q$, where $\wt{\Phi}$ is a fundamental solution of equation~\eqref{eq:LQ.def.reg} with $\wt{Q}$ in place of $Q$. With account of this, we have for the characteristic determinant $\wt{\Delta}(\cdot)$ corresponding to the operator $L_{\wt{U}}(\wt{Q})$,
\begin{equation} \label{eq:wtDelta=Delta}
 \wt{\Delta}(\l) := \det(C + DW(\ell) \wt{\Phi}(\ell,\l))
 = \det(C + D \Phi(\ell,\l)) = \Delta(\l), \qquad \l \in \bC,
\end{equation}
which implies desired equality of characteristic determinants.

\textbf{(iv)} Recall that regularity of boundary conditions $U(y) = Cy(0)+Dy(\ell)=0$ means condition~\eqref{eq:regular.def}, i.e.
\begin{equation} \label{eq:regular.gauge}
 J_{P_\pm}(C, D) = \det(C P_{\mp} + D P_{\pm}) \ne 0,
\end{equation}
where ``projectors'' $P_{\pm}$ are defined in~\eqref{eq:P.pm.def}. Block-diagonal structure~\eqref{eq:Wx=diag} of the matrix $W(x)$ and definition~\eqref{eq:P.pm.def} of ``projectors'' $P_{\pm}$ imply that
$$
 P_{\pm} W(\ell) P_{\pm} = W(\ell) P_{\pm}, \qquad
 P_{\mp} W(\ell) P_{\pm} = 0.
$$
It is also clear that $P_{\pm} P_{\pm} = P_{\pm}$ and $P_{\pm} P_{\mp} = 0$. Hence
\begin{multline} \label{eq:JPpmCDW}
 J_{P_{\pm}}(C, D W(\ell))
 = \det(C P_{\mp} + D W(\ell) P_{\pm}) \\
 = \det(C P_{\mp} + D P_{\pm}) \det(P_{\mp} + W(\ell) P_{\pm})
 = J_{P_{\pm}}(C, D) \prod_{b_k > 0} \det W_{kk}(\ell).
\end{multline}
It is clear that $\prod_{b_k > 0} \det W_{kk}(\ell) \ne 0$. Hence $J_{P_{\pm}}(C, D W(\ell)) \ne 0 \Leftrightarrow J_{P_{\pm}}(C, D)$, and definition of regularity~\eqref{eq:regular.gauge} implies that the new boundary conditions $\wt{U}(y) = C y(0) + D W(\ell) y(\ell)$ are regular if and only if original boundary conditions $U(y) = C y(0) + D y(\ell)$ are regular.
\end{proof}
\begin{remark} \label{rem:similarity}
Note, that similarity of the operators $L_{\wt{U}}(\wt{Q})$ and $L_U(Q)$ implies that both operators have the same spectrum (counting multiplicity). Moreover, $y$ is a root vector of the operator $L_U(Q)$ corresponding to the eigenvalue $\l$ if and only if $\cW y$ is a root vector of the operator $L_{\wt{U}}(\wt{Q})$ corresponding to the eigenvalue $\l$. Since operator $\cW$ is bounded in $\fH$ and has a bounded inverse, then systems of root vectors of the operators $L_{\wt{U}}(\wt{Q})$ and $L_U(Q)$ have many spectral properties only simultaneously: completeness, minimality, uniform minimality, Riesz basis property (see corresponding definitions in future sections).
\end{remark}
Note that gauge transform changes boundary conditions. Even though characteristic determinant $\Delta(\cdot)$ and regularity of boundary conditions is preserved under this transform, the unperturbed operator $L_U(0)$ changes to $L_{\wt{U}}(0)$ and they in general have different eigenvalues. This observation motivates the following definition.
\begin{definition} \label{def:bvp.strict}
Let $B, Q \in L^1([0,\ell]; \bC^{n \times n})$ and $B(x)$ is invertible for almost all $x$. Let matrix function $W(\cdot)$ be constructed from block diagonal of $Q$ using~\eqref{eq:W'+QW}--\eqref{eq:Wx=diag}. BVP~\eqref{eq:LQ.def.reg}--\eqref{eq:Uy=0} is called \textbf{strictly regular} if modified boundary conditions $\wt{U}(y) = C y(0) + D W(\ell) y(\ell) = 0$ are strictly regular.
\end{definition}
\begin{remark}
Note that strict regularity of BVP~\eqref{eq:LQ.def.reg}--\eqref{eq:Uy=0} is only expressed in terms of matrices $C$, $D$ and $Q_{\diag}$ and numbers $b_1, \ldots, b_n$. If $Q_{\diag} \equiv 0$ then strict regularity of BVP~\eqref{eq:LQ.def.reg}--\eqref{eq:Uy=0} simply means strict regularity of original boundary conditions~\eqref{eq:Uy=0}.
\end{remark}
Lemma~\ref{lem:gauge}(iv) implies that BVP~\eqref{eq:LQ.def.reg}--\eqref{eq:Uy=0} is strictly regular if and only if boundary conditions~\eqref{eq:Uy=0} are regular, and modified characteristic determinant
\begin{equation} \label{eq:wtDelta0}
 \wt{\Delta}_0(\cdot) := \det(C + D W(\ell) \Phi_0(\ell, \cdot))
\end{equation}
has countable asymptotically separated sequence of zeros.

Note also, that if boundary conditions~\eqref{eq:Uy=0} are regular, then Lemma~\ref{lem:gauge}(iv) and Lemma~\ref{lem:Delta0.prop} imply that $\wt{\Delta}_0(\cdot)$ has countable sequence of zeros satisfying all the properties from Lemma~\ref{lem:Delta0.prop}.

Now we are ready to formulate our main results on asymptotic behavior of eigenvalues and eigenvectors, Theorems~\ref{th:ln=ln0+o} and~\ref{th:eigenvec}, for arbitrary summable $Q$.
\begin{theorem} \label{th:eigen.gen}
Let matrix function $B(\cdot)$ given by~\eqref{eq:Bx.def} satisfy conditions~\eqref{eq:betak.alpk.Linf}--\eqref{eq:betak-betaj<-eps} and let $Q \in L^1([0,\ell]; \bC^{n \times n})$.
Let matrix function $W(\cdot)$ be constructed from the block diagonal $Q_{\diag}$ of $Q$ using~\eqref{eq:W'+QW}--\eqref{eq:Wx=diag}. Let boundary conditions~\eqref{eq:Uy=0} be regular and let $\wt{\L}_0 = \{\wt{\l}_m^0\}_{m \in \bZ}$ be the sequence of zeros (counting multiplicity) of the modified characteristic determinant $\wt{\Delta}_0(\cdot)$ given by~\eqref{eq:wtDelta0}.

Then operator $L_U(Q)$ has a countable sequence of eigenvalues $\L := \{\l_m\}_{m \in \bZ}$ counting multiplicity. The sequence $\L$ is incompressible (see Definition~\ref{def:incompressible}) and lies in the strip $\Pi_h = \{\l \in \bC : |\Im \l| \le h\}$ for some $h \ge 0$. In addition, the sequences $\wt{\L}_0$ and $\L$ can be ordered in such a way that the following asymptotical formulas hold
\begin{equation} \label{eq:lm=lm0+o.gen}
 \l_m = \wt{\l}_m^0 + o(1) = \frac{2 \pi m}{b_+ - b_-} + o(m) \quad\text{as}\quad m \to \infty,
\end{equation}
where $b_{\pm}$ are defined in~\eqref{eq:b-+.def}. Moreover, if BVP~\eqref{eq:LQ.def.reg}--\eqref{eq:Uy=0} is strictly regular according to Definition~\ref{def:bvp.strict} (i.e.\ the sequence $\wt{\L}_0$ is asymptotically separated) then the sequence $\L$ is asymptotically separated.
\end{theorem}
\begin{proof}
Applying gauge transform from Lemma~\ref{lem:gauge}, we transform operator $L_U(Q)$ to the operator $L_{\wt{U}}(\wt{Q})$ with the same matrix function $B(\cdot)$, and $\wt{U}$ and $\wt{Q}$ given by~\eqref{eq:wtUy=0} with $\wt{Q}$ satisfying ``zero block diagonality'' condition~\eqref{eq:Qjk=0.bj=bk}. Moreover, based on Lemma~\ref{lem:gauge}(iii) this transform preserves the characteristic determinant and thus preserves the spectrum. Applying Proposition~\ref{prop:sine.type} and Theorem~\ref{th:ln=ln0+o} to the operator $L_{\wt{U}}(\wt{Q})$ we arrive at the desired relation~\eqref{eq:lm=lm0+o.gen} and all the desired properties of the sequence $\L$.
\end{proof}
Reformulation of Theorem~\ref{th:eigenvec} on asymptotic behavior of eigenvectors in the case of general matrix function $Q(\cdot)$ is cumbersome and is omitted.
\section{Completeness property} \label{sec:compl}
Let us recall definition of completeness in a Hilbert space $\cH$.
\begin{definition}
Let $\cH$ be a separable Hilbert space. A sequence $\{\phi_m\}_{m \in \bZ}$ of vectors in $\cH$ is called \textbf{complete in $\cH$} if closure of its span coincides with $\cH$. Equivalently, a sequence $\{\phi_m\}_{m \in \bZ}$ is complete in $\cH$ if and only if the following implication holds for every $f \in \cH$,
\begin{equation} \label{eq:phim.f=0}
 (\phi_m, f) = 0, \quad m \in \bZ \qquad \Rightarrow \qquad f = 0.
\end{equation}
\end{definition}
Completeness property in $\cH := L^2([0,\ell]; \bC^n)$ of the system of root vectors of the operator $L_U(Q)$ with so called weakly regular boundary conditions in the case of constant (not necessarily self-adjoint) matrix $B(x) \equiv B = \const$ and summable potential matrix $Q$ was established in~\cite{MalOri12} using certain generalization of Birkhoff theorem on asymptotic behavior of solutions of system~\eqref{eq:LQ.def.reg} in special ``narrowed'' sectors of $\bC$. We need to extend this asymptotic result to the case of \emph{non-constant $B(x)$}. To this end let use introduce special ``narrowed'' sectors $S_{\eps}^{\pm} \subset \bC_{\pm}$,
\begin{align}
 S_{\eps}^{+} & := \{\l : \eps < \arg \l < \pi - \eps\} \subset \bC_+,
 \qquad \eps > 0, \\
 S_{\eps}^{-} & := \{\l : -\pi + \eps < \arg \l < - \eps\} \subset \bC_-,
 \qquad \eps > 0, \\
 S_{\eps,R}^{\pm} & := \{\l \in S_{\eps}^{\pm}: |\l| > R\} \subset \bC_{\pm},
 \qquad \eps, R > 0,
\end{align}
\begin{proposition}[cf. Proposition 2.2 in~\cite{MalOri12}] \label{prop:BirkSys}
Let matrix functions $B(\cdot)$ and $Q(\cdot)$ satisfy conditions~\eqref{eq:Bx.def}--\eqref{eq:Qjk=0.bj=bk}. Let $\eps > 0$ be sufficiently small. Then for a sufficiently large $R$,
equation~\eqref{eq:LQ.def.reg} has fundamental matrix solutions $Y^{\pm}(x,\l)$,
\begin{equation}
 Y^{\pm} = \begin{pmatrix} Y_1^{\pm} & \ldots & Y_n^{\pm} \end{pmatrix},
 \quad Y_k^{\pm} = \col(y_{1k}^{\pm}, \ldots, y_{nk}^{\pm}),
 \quad k \in \oneton,
\end{equation}
which are analytic with respect to $\l \in S_{\eps,R}^{\pm}$ and have
the following asymptotic behavior uniformly in $x \in [0,\ell]$,
\begin{equation} \label{eq:yjk=(1+o).exp}
 y_{jk}^{\pm}(x,\l) = (\delta_{jk} + o(1)) e^{i \l \rho_k(x)},
 \quad\text{as}\quad \l \to \infty,\ \l \in S_{\eps,R}^{\pm},
 \quad j,k \in \oneton,
\end{equation}
where $\delta_{jk}$ is a Kronecker symbol.
\end{proposition}
\begin{proof}
It is clear that the matrix equation $\cL(Q) Y = \l Y$, $Y = (y_{jk})_{j,k=1}^n$ is equivalent to $$Y'(x,\l) = (i \l B(x) - Q(x)) Y(x,\l)$$ and has the following scalar form
\begin{equation} \label{eq:yjk'}
 y_{jk}'(x,\l) = i \l \beta_j(x) y_{jk}(x,\l)
 - \sum_{s=1}^n Q_{js}(x) y_{sk}(x, \l).
\end{equation}
This formula and formulas below are assumed to be valid for all $j,k \in \oneton$, $x \in [0,\ell]$ and $\l \in \bC$, unless stated otherwise. Relations~\eqref{eq:yjk'} can be rewritten as follows,
\begin{equation} \label{eq:ddx.zjk}
 \frac{d}{dx}\(e^{-i \l \rho_j(x)} y_{jk}(x, \l)\)
 = - e^{-i \l \rho_j(x)}\sum_{s=1}^n Q_{js}(x) y_{sk}(x, \l).
\end{equation}
We will look for solution $Y^{\pm}(x,\l)$ as the solution of~\eqref{eq:ddx.zjk} satisfying mixed initial conditions,
\begin{align}
\label{eq:yjkajk=0}
 y_{jk}^{\pm}(a_{jk}^{\pm}, \l) &= \delta_{jk},
\end{align}
where $a_{kk}^{\pm} = 0$ and $a_{jk}^{\pm}$, $j \ne k$, is either $0$ or $\ell$ and will be chosen later.

Integrating~\eqref{eq:ddx.zjk} with account of~\eqref{eq:yjkajk=0} we arrive at
\begin{equation} \label{eq:system-int}
 y_{jk}^{\pm}(x,\l) = \delta_{jk} e^{i \l \rho_j(x)}
 -\int_{a_{jk}^{\pm}}^x e^{i \l (\rho_j(x) - \rho_j(t))}
 \sum_{s=1}^n Q_{js}(t) y_{sk}^{\pm}(t,\l)dt.
\end{equation}
Setting $z_{jk}(x,\l) := e^{-i \l \rho_k(x)} y_{jk}(x, \l)$, we can rewrite~\eqref{eq:system-int} as follows,
\begin{equation} \label{eq:system-int.z}
 z_{jk}^{\pm}(x,\l) = \delta_{jk}
 - \int_{a_{jk}^{\pm}}^x
 e^{i \l \((\rho_j-\rho_k)(x) - (\rho_j-\rho_k)(t)\)}
 \sum_{s=1}^n Q_{js}(t) z_{sk}^{\pm}(t,\l)dt.
\end{equation}

For definiteness consider the case $\l \in S_{\eps}^+$. It is clear from definition of $S_{\eps}^+$ that
\begin{equation} \label{eq:Iml.S+}
 \Im \l \ge \delta |\l|, \qquad \l \in S_{\eps}^+,
\end{equation}
with some $\delta = \delta_{\eps}$ that does not depend on $\l$.

Recall that $\rho_k(x) = \int_0^x \beta_k(t) dt$, where functions $\beta_k(\cdot)$ satisfy uniform separation conditions~\eqref{eq:betak.alpk.Linf}--\eqref{eq:betak-betaj<-eps}.
Let $j > k$ and assume that $\beta_j \not\equiv \beta_k$. Condition~\eqref{eq:betak-betaj<-eps} implies that $$\beta_j(u) - \beta_k(u) > \theta, \qquad u \in [0,\ell].$$ Hence
\begin{multline} \label{eq:abs.exp.t<x}
 \abs{e^{i \l \((\rho_j-\rho_k)(x) - (\rho_j-\rho_k)(t)\)}}
 = \exp\(-\Im \l \int_t^x (\beta_j(u) - \beta_k(u)) du \) \\
 \le \exp\Bigl(-\delta |\l| \cdot \theta |x-t|\Bigr),
 \qquad \beta_j > \beta_k,
 \quad 0 \le t \le x \le \ell,
 \quad \l \in S_{\eps}^+.
\end{multline}
Similarly, if $j < k$ and $\beta_j \not\equiv \beta_k$ then the same estimate is valid for $0 \le x \le t \le \ell$,
\begin{multline} \label{eq:abs.exp.t>x}
 \abs{e^{i \l \((\rho_j-\rho_k)(x) - (\rho_j-\rho_k)(t)\)}}
 \le \exp\Bigl(-\delta |\l| \cdot \theta |x-t|\Bigr),
 \qquad \beta_j < \beta_k,
 \quad 0 \le x \le t \le \ell,
 \quad \l \in S_{\eps}^+.
\end{multline}

With estimates~\eqref{eq:abs.exp.t<x}--\eqref{eq:abs.exp.t>x} in mind, we can now set
\begin{equation} \label{eq:ajk.def}
 a_{jk}^{+} := \begin{cases}
 0, \ \ \text{if}\ \ \ \beta_j(u) \ge \beta_k(u), \ \ u \in [0,\ell], \\
 1, \ \ \text{if}\ \ \ \beta_j(u) < \beta_k(u), \ \ u \in [0,\ell].
 \end{cases}
\end{equation}
In particular, $a_{jk}^{+}=0$ if $\beta_j \equiv \beta_k$, which agrees with relation $a_{kk}^{+} = 0$, we set earlier.

Now if $\beta_j \not\equiv \beta_k$, estimates~\eqref{eq:abs.exp.t<x}--\eqref{eq:abs.exp.t>x} imply that for given $t \ne x$ the exponential function in the integral~\eqref{eq:system-int.z} can be arbitrarily small for $\l \in S_{\eps,R}^+$ and sufficiently large $R$.
If $\beta_j \equiv \beta_k$, then $\rho_k \equiv \rho_j$ and exponential function disappears. In this case we need to insert expressions for such $z_{jk}^{+}(x, \l)$ terms into other equations in~\eqref{eq:system-int.z} to obtain a system only on functions $z_{jk}^+(x, \l)$ with $\beta_j \not\equiv \beta_k$.
The proof is now finished the same way as in~\cite[Proposition 2.2]{MalOri12} by using Banach fixed point theorem and following~\cite[Lemma II.4.4.1]{Nai69}.
\end{proof}
Following~\cite{LunMal14IEOT} we will first establish general completeness property provided that the trace of characteristic determinant $\Delta_Q(\cdot)$ on certain three rays has a certain asymptotic behavior, and then show that this is the case for regular boundary conditions~\eqref{eq:Uy=0}.
\begin{proposition} \label{prop:compl}
Let matrix function $B(\cdot)$ given by~\eqref{eq:Bx.def} satisfy conditions~\eqref{eq:betak.alpk.Linf}--\eqref{eq:betak-betaj<-eps} and let $Q \in L^1([0,\ell]; \bC^{n \times n})$.
Let $\Delta_Q(\l)$ be the characteristic determinants of BVP~\eqref{eq:LQ.def.reg}--\eqref{eq:Uy=0} given by~\eqref{eq:Delta.def}.
Assume that there exist $C,M > 0$, $s \in \bZ_+ := \{0, 1, 2, \ldots\}$ and $z_1, z_2, z_3 \not \in \bR$ satisfying the following conditions:

\item $(i)$ the origin is the interior point of the triangle $\triangle_{z_1z_2z_3}$;

\item $(ii)$ the following estimate holds
\begin{equation} \label{eq:Delta>e/lam}
 |\Delta(\l)| \ge \frac{C}{|\l|^s}
 \(e^{-\Im \l \cdot b_-} + e^{-\Im \l \cdot b_+}\),
 \qquad |\l| > M, \quad \arg\l=\arg z_k, \quad k \in \{1,2,3\}.
\end{equation}
Then operator $L_U(Q)$ has discrete spectrum and the system of root vectors of the operator $L_U(Q)$ is complete in $\fH$.
\end{proposition}
\begin{proof}[Sketch of the proof]
\textbf{(i)} The proof will be divided into multiple steps following the proofs of~\cite[Theorem 1.2]{MalOri12} and~\cite[Theorem 3.2]{LunMal14IEOT}.

\textbf{Step 1.} Applying gauge transform from Lemma~\ref{lem:gauge} we transform operator $L_U(Q)$ to the operator $L_{\wt{U}}(\wt{Q})$ with the same matrix function $B(\cdot)$, and $\wt{U}$ and $\wt{Q}$ given by~\eqref{eq:wtUy=0} with $\wt{Q} = W^{-1} (Q - Q_{\diag}) W$ satisfying ``zero block diagonality'' condition~\eqref{eq:Qjk=0.bj=bk}. Moreover, based on Lemma~\ref{lem:gauge}(iii) this transform preserves the characteristic determinant. Hence characteristic determinant of the new BVP still satisfies condition~\eqref{eq:Delta>e/lam}. Hence, without loss of generality we can assume that original $Q$ satisfies ``zero block diagonality'' condition.

\textbf{Step 2.}
It is clear that $\Phi_Q(x,\cdot)$ is an entire function of exponential type for each $x \in [0,\ell]$. Hence $\Delta(\cdot) = \Delta_Q(\cdot) = \det(C + D \Phi_Q(\ell, \l))$ is an entire function of exponential type.
In turn, condition~\eqref{eq:Delta>e/lam} and canonical factorization for entire functions of exponential type imply that
$\Delta(\cdot)$ has a countable set of zeros of finite multiplicities.
Lemma~\ref{lem:eigen} now implies that the operator $L_U(Q)$ has discrete spectrum. Let $\{\mu_k\}_{k \in \bN}$ be the set of (distinct) eigenvalues of the operator $L_U(Q)$, $\mu_k \ne \mu_j$, $k \ne j$, and let $m_k \in \bN$ be the algebraic multiplicity of the eigenvalue $\mu_k$, $k \in \bN$. We used notation $\mu_k$ to avoid confusion with notation $\{\l_m\}_{m \in \bZ}$ used in other sections.

\textbf{Step 3.} Based on the proof of Lemma~\ref{lem:eigen} the root subspace $\cR_{\mu_k}(L_U(Q))$ of the operator $L_U(Q)$ is of the following form,
\begin{equation} \label{eq:cR.muk}
 \cR_{\mu_k}(L_U(Q)) = \Span\left\{\left. \frac{\partial^p}{\partial \mu^p} Y_j(x,\mu) \right|_{\mu=\mu_k}
 : \ \ p \in \{0,1,\ldots,m_k-1\},\ \ j \in \oneton \right\},
 \qquad k \in \bN,
\end{equation}
where vector functions $Y_j(\cdot,\cdot)$ are introduced in~\eqref{eq:Ypxl.def}.

Let $f \in \fH$ be a vector orthogonal to the system of root vectors of the operator $L_U(Q)$. Next, we will follow~\cite[p.~87--88]{LunMal14IEOT} (see also step (iii) of the proof of~\cite[Theorem 1.2]{MalOri12}).

Consider the entire functions
\begin{equation} \label{eq:Fj(lam).def}
 F_j(\l):=(Y_j(\cdot,\l),f(\cdot))_{\fH}, \qquad j \in \oneton.
\end{equation}
Since $f$ is orthogonal to $\cR_{\mu_k}(L_U(Q))$, $k \in \bN$, it follows from~\eqref{eq:cR.muk} that each $\mu_k ( \in \sigma(L_U(Q)))$ is a zero of $F_j(\cdot)$ of multiplicity at least $m_k$, i.e.
\begin{equation} \label{eq:Fjpl=0}
 F_j^{(p)}(\mu_k) =0, \qquad
 p \in \{0,1,\dots,m_k-1\}, \quad j \in \oneton, \quad k \in \bN.
\end{equation}
Lemma~\ref{lem:eigen} implies the multiplicity of $\mu_k$ as a root of the characteristic determinant $\Delta_Q(\cdot)$ equals to $m_k$. Thus, the ratio
\begin{equation} \label{eq:Gj(lam).def}
 G_j(\l):=\frac{F_j(\l)}{\Delta(\l)}, \qquad j \in \oneton,
\end{equation}
is an entire function. Moreover, since functions
$Y_j(x,\cdot)$ and $\Delta(\cdot)$ are entire functions of
exponential type then so are $G_1(\cdot), \ldots, G_n(\cdot)$. Denote
\begin{align}
\label{eq:Gl.def}
 G(\l) & := \begin{pmatrix} G_1(\l) & \ldots & G_n(\l) \end{pmatrix},
 \qquad \l \in \bC\\
\label{eq:Yl.def}
 Y(\cdot, \l) & := \begin{pmatrix} Y_1(\cdot, \l) & \ldots &
 Y_n(\cdot, \l)\end{pmatrix} = \Phi(\cdot, \l) A^a(\l), \qquad \l \in \bC,
\end{align}
where $A^a(\l)$ is matrix adjugate to $A(\l) = C + D \Phi(\ell,\l)$ and is defined in~\eqref{eq:Aa.def}. It follows from~\eqref{eq:Fj(lam).def} and~\eqref{eq:Gj(lam).def}--\eqref{eq:Yl.def}
that
\begin{equation} \label{eq:int.f*U=Delta.G}
 \int_0^1 f^*(x) \Phi(x,\l) A^a(\l) dx = \int_0^1 f^*(x) Y(x,\l) dx
 = \Delta(\l) G(\l), \quad \l \in \bC,
\end{equation}
where $f^*(x) := \begin{pmatrix}\ol{f_1(x)} & \ldots &
\ol{f_n(x)}\end{pmatrix} = \ol{f(x)}^T$.

Multiplying~\eqref{eq:int.f*U=Delta.G} by the matrix $A(\l)$ from the right we get in view of~\eqref{eq:adjug.ident}
\begin{equation}
 \Delta(\l) \int_0^1 f^*(x) \Phi(x,\l) dx
 = \Delta(\l) G(\l) A(\l), \quad \l \in \bC,
\end{equation}
or equivalently
\begin{equation}
 \int_0^1 f^*(x) \Phi(x,\l) dx
 = G(\l) A(\l), \quad \l \not \in \sigma(L_U(Q)).
\end{equation}
Now the continuity of the integral in the last equality with
respect to $\l$, the discreteness of the set $\sigma(L_U(Q))$
and definition of $A(\l)$ yield the following relation
\begin{equation} \label{eq:int.f*Phi=G.A_Phi}
 \int_0^1 f^*(x) \Phi(x,\l) dx
 = G(\l) (C + D \Phi(1,\l)), \quad \l \in \bC.
\end{equation}

\textbf{Step 4.} Let us prove that functions $G_1(\cdot), \ldots, G_n(\cdot)$ are polynomials in $\l$ by estimating their growth. To this end we consider solutions $Y^{\pm}(x,\l)$ satisfying asymptotic behavior~\eqref{eq:yjk=(1+o).exp}. Following the proof of~\cite[Theorem 1.2]{MalOri12} and~\cite[Theorem 3.2]{LunMal14IEOT} we can derive that
\begin{equation} \label{eq:Gj.via.Ypm}
 \Delta_Q^{\pm}(\l) G_j(\l) = \(U^{\pm}_j(\cdot, \l), f(\cdot)\)_{\fH}
 =: F_j^{\pm}(\l), \qquad \l \in S_{\eps,R}^{\pm},
\end{equation}
where
\begin{align}
 A_Q^{\pm}(\l) &:= C Y^{\pm}(0, \l) + D Y^{\pm}(\ell, \l), \\
 \Delta_Q^{\pm}(\l) &:= \det(A_Q^{\pm}(\l)), \\
 Y_{jk}^{a,\pm}(\l) &:= A_Q^{\pm}(\l)\{j,k\}
 \quad\text{are entries of the corresponding adjugate matrix}, \\
 U^{\pm}_j(x, \l) &:= \sum_{k=1}^n Y_{jk}^{a,\pm}(\l) Y_k^{\pm}(x,\l),
 \qquad j \in \oneton.
\end{align}
Following the proof of~\cite[Theorem 1.2]{MalOri12} we can show using asymptotic behavior~\eqref{eq:yjk=(1+o).exp} that
$$
F_j^{\pm}(\l) = o\(e^{-\Im \l \cdot b_-} + e^{-\Im \l \cdot b_+}\),
\quad\text{as}\quad \l \to \infty, \quad \l \in S_{\eps,R}^{\pm}.
$$
Inserting this estimate and the estimate~\eqref{eq:Delta>e/lam} into~\eqref{eq:Gj.via.Ypm} implies that
\begin{equation} \label{eq:Gj<=lam^m}
G_j(\l) = o(|\l|^s), \quad\text{as}\quad \l \to \infty,
\quad \l \in \Gam_k, \quad k \in \{1,2,3\},
\end{equation}
where $\Gam_k := \{\l \in \bC : \arg \l = \arg z_k\}$, $k \in \{1,2,3\}$.
Since zero is the interior point of the triangle $\triangle_{z_1
z_2 z_3}$, the rays $\Gam_1, \Gam_2, \Gam_3$ divide the
complex plane into three closed sectors $\Omega_1, \Omega_2,
\Omega_3$ of opening less than $\pi$. Fix $k \in \{1,2,3\}$ and
apply the Phragm\'{e}n-Lindel\"{o}f theorem~\cite[Theorem
6.1]{Lev96} to the function $G_j(\l)$ considered in the
sector $\Omega_k$. Using~\eqref{eq:Gj<=lam^m} we get
\begin{equation}
 \left|G_j(\l)\right| \le C_j |\l|^s, \quad \l \in \Omega_k,
\end{equation}
for some $C_j > 0$, and hence
\begin{equation}
 \left|G_j(\l)\right| \le C_j |\l|^s, \quad \l \in \bC.
\end{equation}
By the Liouville theorem~\cite[Theorem 1.1]{Lev96},
$G_j(\l)$ is a polynomial of degree not exceeding $s$.

\textbf{Step 5.} Following~\cite[p. 89-90]{LunMal14IEOT} we can prove that $G_j(\cdot) \equiv 0$, $j \in \oneton$. The proof there works for non-constant matrix $B(\cdot)$ without any changes. Now it follows from~\eqref{eq:int.f*Phi=G.A_Phi} that
\begin{equation} \label{eq:Phi.f=0}
 \int_0^1 \bigl\langle \Phi_j(x,\l) , f(x) \bigr\rangle \,dx \equiv 0,
 \qquad \l \in \bC, \quad j \in \oneton.
\end{equation}
Following~\cite[Theorem 1.2, step (vi)]{MalOri12} we can show that the vector function $f$ satisfying~\eqref{eq:Phi.f=0} is zero. Again, the proof there works for non-constant matrix $B(\cdot)$ without any changes. This means that the system of root vectors of the operator $L_U(Q)$ is complete, which finishes the proof.
\end{proof}
\begin{theorem} \label{th:compl}
Let invertible diagonal matrix function $B(\cdot) = B(\cdot)^*$ satisfy relaxed condition~\eqref{eq:betak.L1}, i.e.\ $B \in L^1([0,\ell]; \bR^{n \times n})$ and every its entry does not change sign on $[0,\ell]$. Let boundary conditions of the boundary value problem~\eqref{eq:L0.def.reg},~\eqref{eq:Uy=0} be regular.

\textbf{(i)} Then the system of root vectors of the operator $L_U(0)$ is complete in $\fH$.

\textbf{(ii)} Let in addition matrix function $B(\cdot)$ satisfy uniform separation conditions~\eqref{eq:betak.alpk.Linf}--\eqref{eq:betak-betaj<-eps} and let $Q \in L^1([0,\ell]; \bC^{n \times n})$. Then~\eqref{eq:Delta>e/lam} holds with $s=0$ and the system of root vectors of the operator $L_U(Q)$ is complete in $\fH$.
\end{theorem}
\begin{proof}
\textbf{(i)} Since boundary conditions are regular and $Q=0$, then estimate~\eqref{eq:Delta>e/lam} with $s=0$ follows from Lemma~\ref{lem:Delta0.prop}. Namely, we need to combine estimate~\eqref{eq:Delta0>C.exp} and the fact that zeros of $\Delta_0(\cdot)$ lie in the strip $\Pi_h$. Hence Proposition~\ref{prop:compl} finishes the proof.

Note that if $Q=0$ then special solutions $Y^{\pm}(x,\l)$ of system $\cL(0) Y = \l Y$ satisfying asymptotic formula~\eqref{eq:yjk=(1+o).exp} always exist whenever $B \in L^1([0,\ell])$. Namely, $Y^{\pm}(x,\l) = \Phi^0(x,\l)$. Hence all steps of the proof of Proposition~\ref{prop:compl} remain valid under relaxed condition~\eqref{eq:betak.L1} on $B(\cdot)$.

\textbf{(ii)} Applying gauge transform from Lemma~\ref{lem:gauge} we transform operator $L_U(Q)$ to the operator $L_{\wt{U}}(\wt{Q})$ with the same matrix function $B(\cdot)$, and $\wt{U}$ and $\wt{Q}$ given by~\eqref{eq:wtUy=0} with $\wt{Q}$ satisfying ``zero block diagonality'' condition~\eqref{eq:Qjk=0.bj=bk}. Moreover, based on Lemma~\ref{lem:gauge}(iii-iv) this transform preserves the characteristic determinant and regularity of boundary conditions.
Since new boundary conditions are regular and new characteristic determinant is the same, then estimate~\eqref{eq:Delta>e/lam} with $s=0$ follows from Proposition~\ref{prop:sine.type} (it is applicable, since $Q$ now satisfies ``zero block diagonality'' condition~\eqref{eq:Qjk=0.bj=bk}). Namely, we need to combine estimate~\eqref{eq:Delta>=} and the fact that zeros of $\Delta_Q(\cdot)$ lie in the strip $\Pi_h$. Hence Proposition~\ref{prop:compl} finishes the proof.
\end{proof}
\section{Adjoint operator} \label{sec:adjoint}
\subsection{General properties of the adjoint operator}
\label{subsec:adjoint.general}
In our approach to Riesz basis property one needs to work with the biorthogonal system to the system of root vectors of the operator $L_U(Q)$. It is well-known that the properly chosen system of root vectors of the operator $(L_U(Q))^*$ plays this role.
As a first step, we show that, as in the case of $B(x) = B = \const$ (see~\cite{MalOri12}), $(L_U(Q))^* = L_{U_*}(Q_*)$, i.e.\ the adjoint operator is also associated to the problem~\eqref{eq:LQ.def.reg}--\eqref{eq:Uy=0}, but with another pair of matrices $C,D$ and another potential $Q_*$.
\begin{lemma} \label{lem:adjoint}
Let diagonal matrix function $B(\cdot)$ satisfies relaxed conditions~\eqref{eq:betak.L1}, let $Q \in L^1([0,\ell]; \bC^{n \times n})$, and let $L_U(Q)$ be the operator associated in $\fH$ with BVP~\eqref{eq:LQ.def.reg}--\eqref{eq:Uy=0}.

\textbf{(i)} Let $S = \sign(B(\cdot))$ be the signature matrix defined in~\eqref{eq:S.def}. Set $Q_* := -S Q^* S$. Then there exists a pair of matrices $\{C_*, D_*\} \subset \bC^{n \times n}$ with $\rank(C_* \ D_*) = n$ and such that the adjoint operator $(L_U(Q))^*$
is associated in $\fH$ with the boundary value problem
\begin{align} \label{eq:L0*.def}
 & \cL(Q_*) y = -i B(x)^{-1} (y' + Q_*(x) y),
 \qquad y = \col(y_1, \ldots, y_n), \quad x \in [0,\ell], \\
\label{eq:U*y=0}
 & U_*(y) := C_* y(0) + D_* y(\ell) = 0, \qquad\text{and}\qquad
 \rank(C_* \ D_*) = n.
\end{align}
In other words, $L_U^*(Q) := (L_U(Q))^* = L_{U_*}(Q_*)$, where operator $L_{U_*}(Q_*)$ is generated by the differential expression $\cL(Q_*)$ on the domain
\begin{equation} \label{eq:dom_op-r_L0*_U}
\dom((L_U(Q))^*) = \{y \in \AC([0,\ell]; \bC^n) : \ \ \cL(Q_*) y \in \fH, \ \ C_* y(0) + D_* y(\ell) = 0 \}.
\end{equation}

\textbf{(ii)} The boundary conditions $U(y) = C y(0) + D y(\ell) = 0 $ and $U_*(y) = C_* y(0) + D_* y(\ell) = 0$ are regular only simultaneously.
\end{lemma}
\begin{proof}
It follows from definition~\eqref{eq:S.def} of the signature matrix $S$ that
\begin{equation} \label{eq:S=S*}
S = S^* = S^{-1}, \quad S^2 = I_n, \qquad |B(x)| = S B(x) = B(x) S.
\end{equation}
First, we note that for any $f,g \in L^{\infty}([0,\ell]; \bC^n)$ we have
\begin{multline} \label{eq:Qf.g}
 \int_0^{\ell} \angnorm{|B(x)| \cdot (-i B(x)^{-1}) Q(x) f(x), g(x)} dx
 = \int_0^{\ell} \angnorm{-i S Q(x) f(x), g(x)} dx \\
 = \int_0^{\ell} \angnorm{f(x), i Q^*(x) S g(x)} dx
 = \int_0^{\ell} \angnorm{|B(x)| f(x), -i B(x)^{-1} Q_*(x) g(x)} dx
\end{multline}
Further, integrating by parts and taking~\eqref{eq:S=S*} into account we obtain for $f, g \in \AC([0,\ell]; \bC^n)$:
\begin{multline} \label{eq:f'.g}
 \int_0^{\ell} \angnorm{|B(x)| \cdot (-i B(x)^{-1}) f'(x),
 g(x)} dx = \int_0^{\ell} \angnorm{-i S f'(x), g(x)} dx \\
 = -i \bigl(\angnorm{S f(\ell), g(\ell)} - \angnorm{S f(0), g(0)}\bigr)
 + \int_0^{\ell} \angnorm{|B(x)| f(x), -i B(x)^{-1} g'(x)} dx.
\end{multline}
Adding~\eqref{eq:Qf.g} and~\eqref{eq:f'.g} we arrive at
\begin{multline}
 (\cL(Q) f, g)_{\fH} = -i \bigl(\angnorm{S f(\ell), g(\ell)}
 - \angnorm{S f(0), g(0)}\bigr) + (f, \cL(Q_*) g)_{\fH},
 \qquad f, g \in \AC([0,\ell]; \bC^n).
\end{multline}
Since $\AC([0,\ell]; \bC^n)$ is dense in $\dom(L_U(Q))$, this identity implies $L_U^*(Q) g = L_{U_*}(Q_*) g$ and $g \in \dom(L_U^*(Q))$ if and only if
\begin{equation} \label{eq:domL*}
 \angnorm{S f(\ell), g(\ell)} = \angnorm{S f(0), g(0)},
 \qquad f \in \dom(L_U(Q)).
\end{equation}
This leads to existence of the desired matrices $C_*$ and $D_*$. Namely, put $\cS = \diag(S, -S)$ and equip the space $\bC^n \oplus \bC^n$ with the bilinear form
$$
w(u, v) = \angnorm{\cS u, v} = \angnorm{S u_1, v_1} - \angnorm{S u_2, v_2}, \qquad u = \col(x,t), \quad v = \col(x',t').
$$
Then condition~\eqref{eq:domL*} means that the subspace $\Ker(C_* \ D_*)$ is the right $w$-orthogonal complement to $\Ker(C \ D)$ in $\bC^n \oplus \bC^n$.

(ii) This statement was proved in~\cite[Corollary 3.3]{MalOri12} in the case of constant matrix $B(x) = B = \const$. The proof remains the same in the case of non-constant matrix $B(x)$.
\end{proof}
Next, assuming boundary conditions~\eqref{eq:Uy=0} to be regular, we find an explicit form of the matrices $C_*$ and $D_*$. The proof substantially relies on the canonical form~\eqref{eq:CD.canon} of regular boundary conditions.
\begin{lemma} \label{lem:C*D*.canon}
Let the pair of matrices $\{C, D\}$ from regular boundary conditions $Uy=0$ be of the canonical form~\eqref{eq:CD.canon}.
Then the matrices $C_*$ and $D_*$ from boundary conditions~\eqref{eq:U*y=0} of the adjoint operator $L_{0,U}^*$
admit the following triangular block-matrix representation:
\begin{equation} \label{eq:C*D*.canon}
C_* = \begin{pmatrix} D_{11}^* & \bO \\
 C_{12}^* & I_{n_+} \end{pmatrix}
\qquad
D_*=\begin{pmatrix} I_{n_-} & D_{21}^* \\
\bO & C_{22}^* \end{pmatrix}.
\end{equation}
\end{lemma}
\begin{proof}
Recall that $S = \diag(-I_{n_-}, I_{n_+})$ is a canonical form of the signature matrix of $B(x)$, where $n_- \in \{0, 1, \ldots, n\}$ and $n_+ = n - n_-$.

\textbf{(i)} First assume that $n_- = 0$, and so $n_+ = n$. Then $D = S = I_n$ and $f(\ell) = - C f(0)$ for $f \in \dom(L_{0,U})$. Hence~\eqref{eq:domL*} turns into
$$
\angnorm{-C f(0), g(\ell)} = \angnorm{f(0), g(0)}, \qquad f \in \dom(L_{0,U}), \quad g \in \dom(L_{0,U}^*),
$$
or
$$
\angnorm{u, g(0) + C^* g(\ell)} = 0, \qquad u \in \bC^n, \quad g \in \dom(L_{0,U}^*).
$$
Hence $g \in \dom{L_{0,U}^*}$ if and only if $g(0) + C^* g(\ell) = 0$,
which leads to the desired formulas for $C_*$ and $D_*$. The case $n_+ = 0$ is treated similarly.

\textbf{(ii)} Now assume that $n_-, n_+ > 0$ in the representation~\eqref{eq:CD.canon}.
For any vector function $f=f(x)$ let $f =: \col(f_-, f_+)$, where $f_{\pm}(x) \in \bC^{n_{\pm}}$, $x \in [0,\ell]$, be its decomposition with respect to the decomposition $\bC^n = \bC^{n_-} \oplus \bC^{n_+}$. With account of this notation we have for $f \in \dom(L_{0,U})$:
\begin{align}
\nonumber
 0 = C f(0) + D f(\ell)
 &= \begin{pmatrix} I_{n_-} & C_{12} \\ \bO & C_{22} \end{pmatrix}
 \binom{f_-(0)}{f_+(0)}
 + \begin{pmatrix} D_{11} & \bO \\ D_{21} & I_{n_+} \end{pmatrix}
 \binom{f_-(\ell)}{f_+(\ell)} \\
\nonumber
 &= \binom{f_-(0) + C_{12} f_+(0) + D_{11} f_-(\ell)}
 {C_{22} f_+(0) + D_{21} f_-(\ell) + f_+(\ell)} \\
\label{eq:f0l=Mfl0}
 &= \binom{f_-(0)}{f_+(\ell)} +
 \begin{pmatrix} D_{11} & C_{12} \\ D_{21} & C_{22} \end{pmatrix}
 \binom{f_-(\ell)}{f_+(0)} = f_{0,\ell} + M f_{\ell,0},
\end{align}
where
\begin{equation} \label{eq:fl0.def}
f_{0,\ell} := \binom{f_-(0)}{f_+(\ell)}, \quad
f_{\ell,0} := \binom{f_-(\ell)}{f_+(0)}, \quad f \in C([0,\ell]; \bC^n),
\qquad M :=
\begin{pmatrix} D_{11} & C_{12} \\ D_{21} & C_{22} \end{pmatrix}.
\end{equation}
Recall that $S = \diag(-I_{n_-}, I_{n_+})$, $f = \col(f_-, f_+)$ and $g = \col(g_-, g_+)$ are decompositions of the matrix $S$ and the vector functions $f$ and $g$ with respect to the decomposition $\bC^n = \bC^{n_-} \oplus \bC^{n_+}$. Taking into account notation~\eqref{eq:fl0.def} and explicit formula~\eqref{eq:f0l=Mfl0} for the domain $\dom(L_{0,U})$, we have from~\eqref{eq:domL*}:
\begin{align}
\nonumber
 0 & = \angnorm{S f(0), g(0)} - \angnorm{S f(\ell), g(\ell)} \\
\nonumber
 & = \angnorm{\binom{-f_-(0)}{f_+(0)}, \binom{g_-(0)}{g_+(0)}}
 - \angnorm{\binom{-f_-(\ell)}{f_+(\ell)}, \binom{g_-(\ell)}{g_+(\ell)}} \\
\nonumber
 & = \angnorm{\binom{f_-(\ell)}{f_+(0)}, \binom{g_-(\ell)}{g_+(0)}}
 - \angnorm{\binom{f_-(0)}{f_+(\ell)}, \binom{g_-(0)}{g_+(\ell)}} \\
\nonumber
 & = \angnorm{f_{\ell,0}, g_{\ell,0}} - \angnorm{f_{0,\ell}, g_{0,\ell}}
 = \angnorm{f_{\ell,0}, g_{\ell,0}} + \angnorm{M f_{\ell,0}, g_{0,\ell}} \\
 & = \angnorm{f_{\ell,0}, g_{\ell,0} + M^* g_{0,\ell}}, \qquad
 f \in \dom(L_{0,U}), \quad g \in \dom(L_{0,U}^*).
\end{align}
This implies that $g \in \dom(L_{0,U}^*)$ if and only if
\begin{align}
\nonumber
 0 = g_{\ell,0} + M^* g_{0,\ell}
 & = \binom{g_-(\ell)}{g_+(0)}
 + \begin{pmatrix} D_{11}^* & D_{21}^* \\ C_{12}^* & C_{22}^* \end{pmatrix}
 \binom{g_-(0)}{g_+(\ell)} \\
\nonumber
 & = \binom{g_-(\ell) + D_{11}^* g_-(0) + D_{21}^* g_+(\ell)}
 {g_+(0) + C_{12}^* g_-(0) + C_{22}^* g_+(\ell)} \\
 & =
 \begin{pmatrix} D_{11}^* & \bO \\ C_{12}^* & I_{n_+} \end{pmatrix}
 \binom{g_-(0)}{g_+(0)} +
 \begin{pmatrix} I_{n_-} & D_{21}^* \\ \bO & C_{22}^* \end{pmatrix}
 \binom{g_-(\ell)}{g_+(\ell)} \\
\label{eq:gl0.g0l}
 & = C_* g(0) + D_* g(\ell),
\end{align}
where $C_*$ and $D_*$ are given by~\eqref{eq:C*D*.canon}. This finishes the proof.
\end{proof}

\begin{remark}
\textbf{(i)} If either $n_- = 0$ or $n_- = n$, the canonical form~\eqref{eq:C*D*.canon} can be simplified. Indeed, if $n_- = n$, in which case $n_+ = 0$ and $D = S = I_n$, then $C_*=I_n$ and $D_*=C^*$. And if $n_- = 0$, in which case $n_+ = n$ and $C = -S = I_n$, then $C_*=D^*$ and $D_*=I_n$.

\textbf{(ii)} Let $D_{11}$ be invertible. It is clear that $C_*$ is also invertible. Set $X := C_*^{-1} D_*$. It is interesting to mention that the matrices $C_*^{-1}$ and $D_*$ serve the triangular factorization of the matrix $X$.
\end{remark}
Next, we indicate an explicit form of eigenvectors of the operators $L_U^*(Q)$ and $L_U^*(0)$ corresponding to their simple eigenvalues. To this end let
$$
 \Phi_*(x, \mu) := \Phi_{Q_*}(x, \mu) =: \begin{pmatrix}
 \Phi_{*1}(x, \mu) & \ldots & \Phi_{*n}(x, \mu) \end{pmatrix}
 =: \(\phi_{*jk}(x,\mu)\)_{j,k=1}^n
$$
be a fundamental matrix solution of the system $\cL(Q_*) Y = \mu Y$ that corresponds to the adjoint operator $L_U^*(Q)$ according to Lemma~\ref{lem:adjoint}. Recall, that $\Phi_0(x, \mu)$ is a fundamental matrix solution of $\cL_0^* Y = \cL_0 Y = \mu Y$. Set
\begin{align}
\label{eq:Delta*}
 \Delta_{*}(\mu) & := \det(A_{*}(\mu)), \qquad
 A_{*}(\mu) := C_* + D_* \Phi_*(\ell, \mu) =: (a_{*jk}(\mu))_{j,k=1}^n, \\
\label{eq:Delta0*}
 \Delta_{0*}(\mu) & := \det(A_{0*}(\mu)), \qquad
 A_{0*}(\mu) := C_* + D_* \Phi_0(\ell, \mu) =: (a_{*jk}^0(\mu))_{j,k=1}^n.
\end{align}
Recall that $A_*^a(\mu)$ ($A_{0*}^a(\mu)$) denotes the adjugate matrix of $A_{*}(\mu)$ (resp. $A_{0*}(\mu)$). Set
$$
A^a_{*}(\mu) =: (A_{*jk}(\mu))_{j,k=1}^n, \qquad
A^a_{0*}(\mu) =: (A_{*jk}^0(\mu))_{j,k=1}^n.
$$
Lemma~\ref{lem:eigen} implies that for any simple eigenvalue $\mu$ of
$L_{U}^*(Q) = L_{U_*}(Q_*)$ (simple zero of $\Delta_*(\cdot)$) there exists $q \in \oneton$ such that the vector function
\begin{equation} \label{eq:yj.eigen*}
 y_{*}(x,\mu) := Y_{*q}(x,\mu) := \sum_{k=1}^n A_{*kq}(\mu) \Phi_{*k}(x, \mu)
\end{equation}
(see~\eqref{eq:Ypxl.def}) is a non-trivial eigenvector of $L_{U}^*(Q) = L_{U_*}(Q_*)$ corresponding to its simple eigenvalue $\mu$. Similarly,
\begin{equation} \label{eq:yj.eigen0*}
 y_{0*}(x,\mu_0) := Y_{*q}^0(x,\mu_0) := \col\(A_{*1q}^0(\mu_0)
 e^{i \mu_0 \rho_1(x)}, \ldots, A_{*nq}^0(\mu_0) e^{i \mu_0 \rho_n(x)}\)
\end{equation}
is a non-trivial eigenvector of $L_{0,U}^* (= L_{0,U_*})$ corresponding to its simple eigenvalue $\mu_0$.

Let us formulate Lemma~\ref{lem:eigen.canon} for adjoint operator $L_U^*(Q)$.
\begin{lemma} \label{lem:eigen.canon*}
Let $\l$ be an algebraically simple eigenvalue of the operator $L_U(Q)$. Then $\ol{\l}$ is an algebraically simple eigenvalue of the operator $L_U^*(Q)$. Let $g$ be any eigenvector of $L_U^*(Q)$ in $\fH$ corresponding to $\ol{\l}$. Then, there exists $q = q_{\ol{\l}} \in \oneton$ and $\gam_{*q} \in \bC$, such that
\begin{equation}
g(\cdot) = \gam_{*q} Y_{*q}(\cdot, \ol{\l}) = \gam_{*q} \sum_{k=1}^n A_{*kq}(\ol{\l}) \Phi_k(\cdot, \ol{\l}), \qquad |\gam_{*q}| = \|g\|_{\fH} / \|Y_{*q}(\cdot, \ol{\l})\|_{\fH}.
\end{equation}
Morever, this is valid for any $q \in \oneton$, for which $Y_{*q}(\cdot, \ol{\l}) \not\equiv 0$.
\end{lemma}
\subsection{Key identity for scalar product of eigenvectors}
\label{subsec:adjoint.identity}
The following result plays a crucial role in proving the Riesz basis property. To state it we set
\begin{align}
\label{eq:dlam.def}
 \cE(\l) &:= \det(\Phi_0(\ell, -\l)) = \exp(-i(b_1+\ldots+b_n)\l). \\
\label{eq:bk+-.def}
 b_k^- &:= \min\{b_k, 0\} \le 0, \qquad b_k^+ := \max\{b_k, 0\} \ge 0,
 \qquad k \in \oneton.
\end{align}
\begin{proposition} \label{prop:y0p.y0*q}
Let diagonal matrix function $B(\cdot)$ satisfies relaxed conditions~\eqref{eq:betak.L1} and let regular boundary conditions~\eqref{eq:Uy=0} be of canonical form~\eqref{eq:CD.canon}. Further, let $\l \in \bC$ be a simple zero of $\Delta_0(\cdot)$ and let $y_0(\cdot, \l)$ and $y_{0*}(\cdot, \ol{\l})$ be the corresponding eigenvectors
given by~\eqref{eq:Ypxl.def} and~\eqref{eq:yj.eigen0*}, respectively.
Then the following identity holds
\begin{equation} \label{eq:y0p.y0*q.ident}
 \(y_0(\cdot, \l), y_{0*}(\cdot, \ol\l)\)_{\fH}
 = \(Y_p^0(\cdot, \l), Y_{*q}^0(\cdot, \ol\l)\)_{\fH}
 = -i \cE(\l) \exp\(i b_q^- \l\) \cdot A_{qp}^0(\l) \Delta_0'(\l),
\end{equation}
where $A_{qp}^0(\l) \ne 0$.

\end{proposition}
\begin{proof}
First we note that combining formulas for the eigenvectors~\eqref{eq:Ypxl.def} and~\eqref{eq:yj.eigen0*} of the operators $L_{0,U}$ and $L_{0,U_*}$ corresponding to the eigenvalues $\l$ and $\ol \l$, respectively, yields
\begin{multline} \label{eq:y0p.y0*q.raw}
 \(Y_p^0(\cdot, \l), Y_{*q}^0(\cdot, \ol\l)\)_{\fH}
 = \int_0^\ell \sum_{k=1}^n A_{kp}^0(\l) e^{i \l \rho_k(x)}
 \ol{A_{*kq}^0(\ol{\l}) e^{i \ol{\l} \rho_k(x)}} |\beta_k(x)| dx \\
 = \sum_{k=1}^n \ol{A_{*kq}^0(\ol{\l})} \cdot |b_k| \cdot
 A_{kp}^0(\l), \qquad p, q \in \oneton, \quad \l \in \bC.
\end{multline}
Our purpose is to transform identity~\eqref{eq:y0p.y0*q.raw} into~\eqref{eq:y0p.y0*q.ident}.
To this end we divide the proof in three steps.

\textbf{(i)} First we consider the case $n_-=0$, i.e.\ $D = S = I_n$.
Let us start with some general facts valid without the assumption $\Delta_0(\l)=0$. Since $D = I_n$, one has $A_0(\l) = C + \Phi_0(\ell, \l)$ and
\begin{equation} \label{eq:A0'}
 A_0'(\l) =
 \frac{d}{d\l}\Phi_0(\ell, \l)
 = \diag\(i b_1 e^{i \l b_1},\ldots, i b_n e^{i \l b_n}\), \qquad \l \in \bC.
\end{equation}

Further, Lemma~\ref{lem:C*D*.canon} implies that $C_* = I_n$ and $D_* = C^*$. Since $[\Phi_0(\ell, \ol{\l})]^{-1} = \Phi_0(\ell, -\ol{\l}) = [\Phi_0(\ell, \l)]^*$ then
\begin{multline} \label{eq:A_{0*}=SA_0*_Phi_0_S}
A_{0*}(\ol{\l}) = I_n +C^* \Phi_0(\ell, \ol{\l})
= (\Phi_0(\ell, -\ol{\l}) + C^*) \Phi_0(\ell, \ol{\l}) \\
= [C + \Phi_0(\ell, \l)]^* \Phi_0(\ell, \ol{\l})
= [A_0(\l)]^* \Phi_0(\ell, \ol{\l}).
\end{multline}
It follows from~\eqref{eq:dlam.def} that
\begin{equation} \label{eq:det.Phi0}
 \det(\Phi_0(\ell, \ol{\l})) = \exp(i(b_1+\ldots+b_n)\ol{\l}) = \ol{\cE(\l)}.
\end{equation}
Combining this identity with~\eqref{eq:A_{0*}=SA_0*_Phi_0_S} and taking into account properties~\eqref{eq:Aa.adjont} and~\eqref{eq:adjug.prod} of adjugate matrices and identity $[\Phi_0(\ell, \ol{\l})]^{-1} = \Phi_0(\ell, -\ol{\l})$, imply
\begin{equation} \label{eq:Delta*A*}
 \Delta_{0*}(\ol{\l}) = \ol{\cE(\l) \cdot \Delta_0(\l)}, \qquad
 A^a_{0*}(\ol{\l}) = \ol{\cE(\l)} \cdot \Phi_0(\ell, -\ol{\l})
 \cdot [A_0^a(\l)]^*.
\end{equation}
Taking adjoint in~\eqref{eq:Delta*A*} we arrive at
\begin{equation} \label{eq:Djk*=Dkj}
 [A^a_{0*}(\ol{\l})]^* = \cE(\l) A_0^a(\l) \Phi_0(\ell, \l),
 \quad\text{i.e.}\quad
 \ol{A_{*kj}^0(\ol{\l})} = \cE(\l) e^{i \l b_k} A_{jk}^0(\l),
\end{equation}
for $j, k \in \oneton$. Note that since $n_- = 0$ then $b_k > 0$, $k \in \oneton$. Taking this into account and inserting~\eqref{eq:Djk*=Dkj} into~\eqref{eq:y0p.y0*q.raw}, we get
\begin{equation} \label{eq:y0p.y0*q}
 \(Y_p^0(\cdot, \l), Y_{*q}^0(\cdot, \ol\l)\)_{\fH}
 = \cE(\l) \sum_{k=1}^n A_{qk}^0(\l) b_k e^{i \l b_k} A_{kp}^0(\l),
 \qquad p, q \in \oneton.
\end{equation}
Since $\rank(A_0^a(\l)) = 1$, one has
\begin{equation} \label{eq:A_{qk}^0A_{kp}^0 = A_{qp}^0 A_{kk}^0}
 0 = \det \begin{pmatrix} A_{qk}^0(\l) & A_{kk}^0(\l) \\ A_{qp}^0(\l) & A_{kp}^0(\l)
 \end{pmatrix} = A_{qk}^0(\l) A_{kp}^0(\l) - A_{qp}^0(\l) A_{kk}^0(\l),
 \qquad k, p, q \in \oneton.
\end{equation}
Inserting this identity into~\eqref{eq:y0p.y0*q}, we derive
\begin{equation} \label{eq:main_iden-ty_with_D=I}
 \(Y_p^0(\cdot, \l), Y_{*q}^0(\cdot, \ol\l)\)_{\fH} =
 \cE(\l) \sum_{k=1}^n b_k e^{i \l b_k} A_{qk}^0(\l) A_{kp}^0(\l) \\
 = \cE(\l) A_{qp}^0(\l) \sum_{k=1}^n b_k e^{i \l b_k} A_{kk}^0(\l).
\end{equation}
Inserting~\eqref{eq:A0'} in formula~\eqref{eq:jacobi} implies
\begin{equation} \label{eq:jacobi2}
 \Delta_0'(\l) = \tr\(A_0^a(\l) A_0'(\l)\) =
 i \sum_{k=1}^n A_{kk}^0(\l) b_k e^{i \l b_k} =
 i \sum_{k=1}^n b_k e^{i \l b_k} A_{kk}^0(\l).
\end{equation}
Inserting~\eqref{eq:jacobi2} into~\eqref{eq:main_iden-ty_with_D=I} and taking into account that $b_q^- = 0$ when $n_- = 0$, we arrive at~\eqref{eq:y0p.y0*q.ident}, which finishes the proof in the case $n_-=0$.

\textbf{(ii)} In this step we consider the general case $n_-, n_+ > 0$. Our first goal is to obtain formula for $\Delta_0'(\l)$ similar to~\eqref{eq:jacobi2}. Since $n_- > 0$, the matrix $A_0^a(\l)$ has more complicated structure. Applying Jacobi formula to it directly will lead to more complicated unusable formula for $\Delta_0'(\l)$. Hence, we need to do some preparations first. With account of definition~\eqref{eq:bk+-.def}, let us decompose the diagonal matrix function $\Phi_0(\ell, \l) = \diag\(e^{i \l b_1},\ldots, e^{i \l b_n}\)$ with respect to the orthogonal decomposition $\bC^n = \bC^{n_-} \oplus \bC^{n_+}$,
\begin{align}
 \Phi_0(\ell, \l) & =: \diag(\Phi^0_{-}(\ell, \l), \Phi^0_{+}(\ell, \l)),
 \qquad
 \Phi^0_{\pm}(\ell, \l) \in \bC^{n_{\pm} \times n_{\pm}}, \quad \l \in \bC, \\
\label{eq:wtPhi1}
 \wh{\Phi}^0_-(\ell, \l) & := \diag(\Phi^0_{-}(\ell, \l), I_{n_+}) =
 \diag\bigl(e^{i \l b_k^-}\bigr)_{k=1}^n, \\
\label{eq:wtPhi2}
 \wh{\Phi}^0_+(\ell, \l) & := \diag(I_{n_-}, \Phi^0_{+}(\ell, \l)) =
 \diag\bigl(e^{i \l b_k^+}\bigr)_{k=1}^n.
\end{align}
Using formulas~\eqref{eq:wtPhi1}--\eqref{eq:wtPhi2} and (canonical) triangular block-matrix representation~\eqref{eq:CD.canon} of matrices $C$ and $D$ from the boundary conditions $Uy=0$, we obtain
\begin{align}
 A_0(\l) & = C + D \Phi_0(\ell, \l) \nonumber \\
 & = \begin{pmatrix} I_{n_-} & C_{12} \\ 0 & C_{22} \end{pmatrix} +
 \begin{pmatrix} D_{11} & 0 \\ D_{21} & I_{n_+} \end{pmatrix}
 \begin{pmatrix} \Phi^0_{-}(\ell, \l) & 0 \\
 0 & \Phi^0_{+}(\ell, \l) \end{pmatrix} \nonumber \\
 & =
 \begin{pmatrix} I_{n_-} + D_{11} \Phi^0_{-}(\ell, \l) & C_{12} \\
 D_{21} \Phi^0_{-}(\ell, \l) & C_{22} + \Phi^0_{+}(\ell, \l)
 \end{pmatrix} \nonumber \\
 & = \begin{pmatrix} \Phi^0_{-}(\ell, -\l) + D_{11} & C_{12} \\
 D_{21} & C_{22} + \Phi^0_{+}(\ell, \l) \end{pmatrix}
 \begin{pmatrix} \Phi^0_{-}(\ell, \l) & 0 \\ 0 & I_{n_+}
 \end{pmatrix} \nonumber \\
\label{eq:A0.lam=wtA.wtPh.0-}
 & = \wh{A}_0(\l) \wh{\Phi}^0_-(\ell, \l), \qquad
\end{align}
where
\begin{equation} \label{eq:def_matrix_wt_A_0(lam)}
 \wh{A}_0(\l) := \begin{pmatrix} \Phi^0_{-}(\ell, -\l) + D_{11} & C_{12} \\
 D_{21} & C_{22} + \Phi^0_{+}(\ell, \l) \end{pmatrix}
 = A_0(\l) \wh{\Phi}^0_-(\ell, -\l).
\end{equation}
It is clear that the derivative of the matrix $\wh{A}_0(\l)$ is diagonal, which makes it more suitable for applying Jacobi's formula. To this end we note that formulas~\eqref{eq:wtPhi1},~\eqref{eq:wtPhi2} and~\eqref{eq:bk+-.def} imply
\begin{equation} \label{eq:Phi1-Phi2}
 \diag\(\Phi^0_{-}(\ell, -\l), \Phi^0_{+}(\ell, \l)\) =
 [\wh{\Phi}^0_-(\ell, \l)]^{-1} \wh{\Phi}^0_+(\ell, \l)=
 \diag\bigl( e^{-i \l b_k^-} \cdot e^{i \l b_k^+} \bigr)_{k=1}^n =
 \diag\( e^{i \l |b_k|} \)_{k=1}^n,
\end{equation}
where we used trivial identity $|b_k| = -b_k^- + b_k^+$, $k \in \oneton$. Let also
\begin{equation}
 e_-(\l) := \det(\Phi^0_{-}(\ell, -\l)) = e^{-i \l b_-}, \quad
 \l \in \bC, \qquad b_- := b_1 + \ldots + b_{n_-} < 0.
\end{equation}
Definition~\eqref{eq:def_matrix_wt_A_0(lam)}, properties~\eqref{eq:ABa} and~\eqref{eq:Aa.inv} of adjugate matrices, and relation~\eqref{eq:Phi1-Phi2} imply that
\begin{align}
 & \wh{\Delta}_0(\l) := \det(\wh{A}_0(\l)) =
 \det(A_0(\l)) \det(\Phi^0_{-}(\ell, -\l)) =
 e_-(\l) \Delta_0(\l), \\
 & \wh{A}_0^a(\l) = [A_0(\l) \wh{\Phi}^0_-(\ell, -\l)]^a
 = e_-(\l) \wh{\Phi}^0_-(\ell, \l) A_0^a(\l)
 = \bigl(e_-(\l) e^{i \l b_j^-} A_{jk}^0(\l)\bigr)_{j,k=1}^n, \label{eq:wtA_0_a} \\
 & \wh{A}_0'(\l)
 = \frac{d}{d\l} \diag\(\Phi^0_{-}(\ell, -\l), \Phi^0_{+}(\ell, \l)\)
 = \diag\( i |b_k| e^{i \l |b_k|} \)_{k=1}^n. \label{eq:wtA_0'}
\end{align}
Now Jacobi's formula~\eqref{eq:jacobi.def} applies to $\wh{\Delta}_0(\cdot) = e_-(\cdot){\Delta}_0(\cdot)$ and gives with account of~\eqref{eq:wtA_0_a} and~\eqref{eq:wtA_0'}
\begin{multline} \label{eq:wtD'}
 e_-(\l) \Delta'_0(\l) + e_-'(\l) \Delta_0(\l) = \wh{\Delta}'_0(\l)
 = \tr\bigl(\wh{A}_0^a(\l) \wh{A}_0'(\l)\bigr) \\
 = \sum_{k=1}^n e_-(\l) e^{i \l b_k^-} A_{kk}^0(\l)
 \cdot i |b_k| e^{i \l |b_k|}
 = i e_-(\l) \sum_{k=1}^n |b_k| e^{i \l b_k^+} A_{kk}^0(\l).
\end{multline}
Given that $e_-'(\l) = -i b_- e_-(\l)$, it follows from~\eqref{eq:wtD'} that
\begin{equation} \label{eq:Delta'0}
 \Delta'_0(\l) = i b_- \Delta_0(\l) + i \sum_{k=1}^n
 |b_k| e^{i \l b_k^+} A_{kk}^0(\l), \qquad \l \in \bC.
\end{equation}

Now let us find an explicit form of the adjugate matrix $A^a_{0*}(\ol{\l})$.
Since
\begin{equation}
 [\Phi^0_{\pm}(\ell, \ol{\l})]^{-1} = \Phi^0_{\pm}(\ell, -\ol{\l})
 = [\Phi^0_{\pm}(\ell, \l)]^*,
\end{equation}
then
\begin{align*}
 A_{0*}(\ol{\l}) & = C_* + D_* \Phi_0(\ell, \ol{\l}) \\
 & = \begin{pmatrix} D_{11}^* & 0 \\ C_{12}^* & I_{n_+} \end{pmatrix} +
 \begin{pmatrix} I_{n_-} & D_{21}^* \\ 0 & C_{22}^* \end{pmatrix}
 \begin{pmatrix} \Phi^0_{-}(\ell, \ol{\l}) & 0 \\
 0 & \Phi^0_{+}(\ell, \ol{\l}) \end{pmatrix} \\
 & = \begin{pmatrix}
 D_{11}^* + \Phi^0_{-}(\ell, \ol{\l}) &
 D_{21}^* \Phi^0_{+}(\ell, \ol{\l}) \\
 C_{12}^* & I_{n_+} + C_{22}^* \Phi^0_{+}(\ell, \ol{\l}) \\
 \end{pmatrix} \\
 & =
 \begin{pmatrix} \Phi^0_{-}(\ell, \ol{\l}) & 0 \\ 0 & I_{n_+} \end{pmatrix}
 \begin{pmatrix}
 [\Phi^0_{-}(\ell, \l)]^* D_{11}^* + I_{n_-} &
 [\Phi^0_{-}(\ell, \l)]^* D_{21}^* \\
 C_{12}^* &
 [\Phi^0_{+}(\ell, \l)]^* + C_{22}^* \\
 \end{pmatrix}
 \begin{pmatrix} I_{n_-} & 0 \\ 0 & \Phi^0_{+}(\ell, \ol{\l}) \end{pmatrix} \\
 & = \wh{\Phi}^0_-(\ell, \ol{\l}) \cdot [A_0(\l)]^* \cdot
 \wh{\Phi}^0_+(\ell, \ol{\l}).
\end{align*}
Further, it follows from~\eqref{eq:wtPhi1},~\eqref{eq:wtPhi2} and~\eqref{eq:dlam.def} that
\begin{equation}
 \det\bigl(\wh{\Phi}^0_-(\ell, \ol{\l}) \cdot \wh{\Phi}^0_+(\ell, \ol{\l})\bigr) =
 \det(\Phi_0(\ell, \ol{\l})) = \exp(i (b_1 + \ldots + b_n) \ol{\l}) = \ol{\cE(\l)}.
\end{equation}
Hence due to~\eqref{eq:det.Phi0} and properties~\eqref{eq:Aa.adjont} and~\eqref{eq:adjug.prod} of adjugate matrix, we have
\begin{align}
\nonumber
 \Delta_{0*}(\ol{\l}) &= \det(A_{0*}(\ol{\l}))
 = \det\(\wh{\Phi}^0_-(\ell, \ol{\l}) \cdot \wh{\Phi}^0_+(\ell, \ol{\l})\)
 \cdot \ol{\det(A_0(\l))} = \ol{\cE(\l) \cdot \Delta_0(\l)}, \\
\label{eq:Delta*A*n1n2}
 A^a_{0*}(\l) &= \ol{\cE(\l)} \cdot \wh{\Phi}^0_+(\ell, -\ol{\l})
 \cdot [A_0^a(\l)]^* \cdot \wh{\Phi}^0_-(\ell, -\ol{\l}).
\end{align}
Taking adjoint we get
\begin{equation} \label{eq:wtA**}
 [A^a_{0*}(\ol{\l})]^* = \cE(\l) \cdot \wh{\Phi}^0_-(\ell, \l)
 \cdot A_0^a(\l) \cdot \wh{\Phi}^0_+(\ell, \l).
\end{equation}
Inserting~\eqref{eq:wtPhi1}--\eqref{eq:wtPhi2} into~\eqref{eq:wtA**} we arrive at the key identity
\begin{equation} \label{eq:Djk*=Dkj.n1n2}
 \ol{A_{*kj}^0(\ol{\l})} = \cE(\l) e^{i \l b_j^-} e^{i \l b_k^+} A_{jk}^0(\l),
 \qquad j,k \in \oneton, \quad \l \in \bC.
\end{equation}
Inserting~\eqref{eq:Djk*=Dkj.n1n2} into~\eqref{eq:y0p.y0*q.raw} and taking into account formula~\eqref{eq:Delta'0} for $\Delta'_0(\cdot)$, identity~\eqref{eq:A_{qk}^0A_{kp}^0 = A_{qp}^0 A_{kk}^0}, one deduces
\begin{align}
\nonumber
 \(Y_p^0(\cdot, \l), Y_{*q}^0(\cdot, \ol\l)\)_{\fH}
 &= \cE(\l) e^{i \l b_q^-}
 \sum_{k=1}^n A_{qk}^0(\l) |b_k| e^{i \l b_k^+} A_{kp}^0(\l) \\
\nonumber
 &= \cE(\l) e^{i \l b_q^-} A_{qp}^0(\l)
 \sum_{k=1}^n |b_k| e^{i \l b_k^+} A_{kk}^0(\l) \\
\label{eq:y0p.y0*q.n1n2}
 &= -i \cE(\l) e^{i \l b_q^-} A_{qp}^0(\l)
 \left(\Delta'_0(\l) - i b_- \Delta_0(\l)\right),
 \quad p, q \in \oneton.
\end{align}
Since $\l$ is a simple root of $\Delta_0(\cdot)$ then~\eqref{eq:y0p.y0*q.n1n2} implies~\eqref{eq:y0p.y0*q.ident}.

The case $n_+=0$ can be treated similarly.

\textbf{(iii)} Assuming $\l$ to be a simple eigenvalue of the operator $L_{0,U}$, let us show that $A_{qp}^0(\l) \ne 0$ whenever, $Y_p^0(\cdot, \l) \not\equiv 0$ and $Y_{*q}^0(\cdot, \ol\l) \not\equiv 0$. According to Lemma~\ref{lem:eigen}, we have
\begin{equation}
 \Delta_0(\l) = 0, \qquad \Delta_0'(\l) \ne 0, \qquad \rank(A_0^a(\l)) = 1.
\end{equation}
Since $A_0^a(\l)$ is non-zero matrix, there exist $j, k \in \oneton$ such that $A_{kj}^0(\l) \ne 0$. Formula~\eqref{eq:y0p.y0*q.ident} implies that
\begin{equation} \label{eq:Yj.Y*k.ne.0}
 \(Y_j^0(\cdot, \l), Y_{*k}^0(\cdot, \ol\l)\)_{\fH}
 = -i \cE(\l) \exp\(i b_k^- \l\) \cdot A_{kj}^0(\l) \Delta_0'(\l) \ne 0,
\end{equation}
since $A_{jk}^0(\l) \ne 0$ and $\Delta_0'(\l) \ne 0$.

Now let $p, q \in \oneton$ be such that $Y_p^0(\cdot, \l) \not\equiv 0$ and $Y_{*q}^0(\cdot, \ol\l) \not\equiv 0$. Since $\rank(A_0^a(\l)) = 1$, then $Y_p^0(\cdot, \l)$ is proportional to $Y_j^0(\cdot, \l)$. Formulas~\eqref{eq:Djk*=Dkj} and~\eqref{eq:Delta*A*n1n2} imply that $\rank(A_{*0}^a(\ol{\l})) = 1$. Hence $Y_{*q}^0(\cdot, \ol{\l})$ is proportional to $Y_{*k}^0(\cdot, \ol{\l})$. Therefore,
\begin{equation} \label{eq:Yp0.Y*q0.proport}
 Y_p^0(\cdot, \l) = \alp_{pj} Y_j^0(\cdot, \l), \qquad
 Y_{*q}^0(\cdot, \ol{\l}) = \alp_{*qk} Y_{*k}^0(\cdot, \ol{\l}), \qquad
 \alp_{pj}, \alp_{*qk} \in \bC \setminus \{0\}.
\end{equation}
Combining formulas~\eqref{eq:y0p.y0*q.ident},~\eqref{eq:Yj.Y*k.ne.0} and~\eqref{eq:Yp0.Y*q0.proport} we arrive at
\begin{align}
 -i \cE(\l) \exp\(i b_q^- \l\) \cdot A_{qp}^0(\l) \Delta_0'(\l)
 = \(Y_p^0(\cdot, \l), Y_{*q}^0(\cdot, \ol\l)\)_{\fH}
 = \alp_{pj} \ol{\alp_{*qk}} \(Y_j^0(\cdot, \l),
 Y_{*k}^0(\cdot, \ol\l)\)_{\fH} \ne 0,
\end{align}
which implies $A_{qp}^0(\l) \ne 0$ and finishes the proof.
\end{proof}
\begin{remark}
The proof remains valid for non-regular boundary conditions provided that $J_P(C,D) \ne 0$ for some $P \in \cP_n$. Indeed, we can use alternative canonical form of boundary conditions~\eqref{eq:Uy=0} outlined in Remark~\ref{rem:canon}.
\end{remark}
\section{Uniform minimality and Riesz basis property}
\label{sec:uniform.minim.riesz}
\subsection{Uniform minimality} \label{subsec:uniform.minim}
Here we apply results of the previous section to show an important property of the system of root vectors of the operator $L_U(Q)$: uniform minimality.
\begin{definition} \label{def:unif.minimality}
A sequence $\{\phi_m\}_{m \in \bZ}$ in a Banach space $X$ is called minimal if
\begin{equation} \label{eq:minim-ty}
\phi_m\notin \Span\{\phi_k:k \ne m\} \qquad \text{for any}\quad m \in \bZ.
\end{equation}
It is called uniformly minimal if
\begin{equation} \label{eq:unif.minim}
\inf_{m \in \bZ}{\dist}\bigl(\|\phi_m\|^{-1}\phi_m, {\Span}\{\phi_k:\ k \ne m\}\bigr)>0.
\end{equation}
\end{definition}
The following statement is well known.
\begin{lemma} \label{lem:criter.unif.minimality}
(i) The sequence $\{\phi_m\}_{m \in \bZ}\subset X$ is minimal if and only if there exists a biorthogonal system $\{\phi_{*m}\}_{m \in \bZ}\subset X^*$, i.e.\ a system satisfying $(f_k,\phi_{*m}) = \delta_{km}$, $k,m \in \bZ$.

(ii) The sequence $\{\phi_m\}_{m \in \bZ}\subset X$ is uniformly minimal if and only if it admits a biorthogonal system $\{\phi_{*m}\}_{m \in \bZ}\subset X^*$ satisfying
\begin{equation} \label{eq:criter_unif_minim-ty}
\sup_{m \in \bZ}\|\phi_m\|\cdot\|\phi_{*m}\|<\infty.
\end{equation}
\end{lemma}
Now we are ready to prove uniform minimality of the strictly regular BVP~\eqref{eq:LQ.def.reg}--\eqref{eq:Uy=0}. If $Q=0$, this can be done under relaxed conditions on the matrix function $B(\cdot)$.
\begin{proposition} \label{prop:uniform.minim}
Let self-adjoint invertible diagonal matrix function $B(\cdot)$ satisfy relaxed conditions~\eqref{eq:betak.L1}. Namely, $B \in L^1([0,\ell]; \bR^{n \times n})$ and every its entry does not change sign on $[0,\ell]$. Let boundary conditions of the boundary value problem~\eqref{eq:L0.def.reg},~\eqref{eq:Uy=0} be strictly regular. Then any system of root vectors of the operator $L_{0,U} = L_U(0)$ is uniformly minimal in $\fH$.
\end{proposition}
\begin{proof}
Since boundary conditions are strictly regular then $\Delta_0(\cdot)$ has a countable asymptotically separated sequence of zeros $\L_0 := \{\l_m^0\}_{m \in \bZ}$ (counting multiplicity), satisfying~\eqref{separ_cond}
with certain $\delta, m_0 > 0$ and lying in the strip $\Pi_h$. Clearly, $\L_0$ is a sequence of eigenvalues of $L_{0,U}$ (counting multiplicity) and $\ol{\L_0} := \{\ol{\l_m^0}\}_{m \in \bZ}$ is a sequence of eigenvalues of $L_{0,U}^*$ (counting multiplicity). Moreover, each eigenvalue of $L_{0,U}$ has finite multiplicity.

Let $\cF_0 := \{f_m^0\}_{m \in \bZ}$ be any system of root vectors of the operator $L_{0,U}$. Since operator $L_{0,U}$ has discrete spectrum and each its eigenvalue has finite multiplicity, we can choose system of root vectors $\cF_{*0} := \{f_{*m}^0\}_{m \in \bZ}$ of the operator $L_{0,U}^*$ in such a way that $\cF_0$ and $\cF_{*0}$ are biorthogonal systems. i.e.\ $(f_j^0, f_{*k}^0) = \delta_{jk}$, $j, k \in \bZ$. This, implies minimality of the system $\cF_0$. In accordance with Lemma~\ref{lem:criter.unif.minimality}(ii), to prove uniform minimality it is sufficient to show that
\begin{equation} \label{eq:sup.fm0.gm0}
 \sup_{|m| > m_0} \|f_m^0\|_{\fH} \cdot \|f_{*m}^0\|_{\fH} < \infty,
\end{equation}
where $m_0$ is from Definition~\ref{def:strict.regular}(iii) of strict regularity.

Let $|m| > m_0$. Then $\l_m^0$ and $\ol{\l_m^0}$ are algebraically simple eigenvalues of the operators $L_{0,U}$ and $L_{0,U}^*$, respectively. Moreover, by Lemma~\ref{lem:D0jk>}, there exist indices $p = p_m \in \oneton$ and $q = q_m \in \oneton$, and a constant $C_2>0$ such that estimate~\eqref{eq:D0jk>} holds, i.e.\ $|A_{qp}^0(\l_m^0)| \ge C_2$, $|m| > m_0$. Emphasize, that although $p$ and $q$ depend on $m$, the constant $C_2$ in the above estimate does not. Starting with this $A_{qp}^0(\l_m^0) (\not=0)$ we define the vector functions
\begin{equation} \label{eq:wtfgm0.def}
\wt{f}_m^0(\cdot) := Y_p^0(\cdot, \l_m^0) \quad\text{and}\quad
\wt{f}_{*m}^0(\cdot) := Y_{*q}^0(\cdot, \ol{\l_m^0})
\end{equation}
by formulas~\eqref{eq:Yp0xl.def} and~\eqref{eq:yj.eigen0*}, respectively.

Since boundary conditions~\eqref{eq:Uy=0} are regular then by Lemmas~\ref{lem:CD.canon} and~\ref{lem:C*D*.canon} we can assume boundary conditions of the operators $L_{0,U}$ and $L_{0,U}^* = L_{0,U_*}$ to be of canonical forms~\eqref{eq:CD.canon} and~\eqref{eq:C*D*.canon}, respectively. Hence Proposition~\ref{prop:y0p.y0*q} can be used. Besides, in accordance with Lemma~\ref{lem:Delta'>}, $|\Delta_0'(\l_m^0)| \ge C_0$, $|m| > m_0$, where a constant $C_0>0$ is independent on $m$. Combining Proposition~\ref{prop:y0p.y0*q} with this estimate and the above estimate on $|A_{qp}^0(\l_m^0)|$ yields
\begin{align}
\nonumber
 \bigabs{\bigl(\wt{f}_m^0, \wt{f}_{*m}^0 \bigr)_{\fH}}
 & = \bigabs{\bigl(Y_p^0(\cdot, \l_m^0),
 Y_{*q}^0(\cdot, \ol{\l_m^0})\bigr)_{\fH}} \\
\nonumber
 & = \bigabs{\cE(\l_m^0) \exp\(i b_q^- \l_m^0\)} \cdot
 \bigabs{A_{qp}^0(\l_m^0) \cdot \Delta_0'(\l_m^0)} \\
\label{eq:y0p.y0*q>}
 & \ge \bigabs{\cE(\l_m^0) \exp\(i b_q^- \l_m^0\)} C_2 C_0 \ge C_3,
\end{align}
with some $C_3 > 0$ that does not depend on $m$. Inequality~\eqref{eq:y0p.y0*q>} in particular implies that both vector functions $\wt{f}_m^0$ and $\wt{f}_{*m}^0$ are non-zero.

Applying Lemma~\ref{lem:easy.upper.bound} to the ``adjoint'' BVP~\eqref{eq:L0*.def}--\eqref{eq:U*y=0}, we arrive at the estimate
\begin{equation} \label{eq:Y*q0.norm}
 \norm{Y_{*q}^0(\cdot,\ol{\l})}_{\fH} \le M_{*h}, \qquad \l \in \Pi_h,
\end{equation}
with some different constant $M_{*h}$. Inclusion $\l_m^0 \in \Pi_h$ and estimates~\eqref{eq:Phij0.Yk0.norm} and~\eqref{eq:Y*q0.norm} imply that
\begin{equation} \label{eq:wtfgm0.norm}
 \|\wt{f}_m^0\|_{\fH} = \|Y_p^0(\cdot, \l_m^0)\|_{\fH} \le M_h, \qquad
 \|\wt{f}_{*m}^0\|_{\fH} = \|Y_{*q}^0(\cdot, \ol{\l_m^0})\|_{\fH} \le M_{*h},
 \qquad |m| > m_0.
\end{equation}

Since $\l_m^0$ is algebraically simple eigenvalue of the operator $L_{0,U}$, and $\wt{f}_m^0 \ne 0$ and $\wt{f}_{*m}^0 \ne 0$, Lemmas~\ref{lem:eigen.canon} and~\ref{lem:eigen.canon*} ensure that
\begin{equation} \label{eq:fm0.gm0.def}
 f_m^0(\cdot) = \gam_m^0 \wt{f}_m^0(\cdot) = \gam_m^0 Y_p^0(\cdot, \l_m^0), \qquad
 f_{*m}^0(\cdot) = \gam_{*m}^0 \wt{f}_{*m}^0(\cdot)
 = \gam_{*m}^0 Y_{*q}^0(\cdot, \ol{\l_m^0}), \qquad |m| > m_1,
\end{equation}
with some $\gam_m^0, \gam_{*m}^0 \in \bC \setminus \{0\}$. Since vector systems $\cF_0$ and $\cF_{*0}$ are biorthogonal, it follows that
\begin{equation} \label{eq:fm0.gm0}
 1 = (f_m^0, f_{*m}^0)_{\fH} = \gam_m^0 \ol{\gam_{*m}^0} \cdot (\wt{f}_m^0, \wt{f}_{*m}^0)_{\fH}.
\end{equation}
Combining estimates~\eqref{eq:wtfgm0.norm} with estimate~\eqref{eq:y0p.y0*q>}, relations~\eqref{eq:fm0.gm0.def} and equality~\eqref{eq:fm0.gm0} yields
\begin{equation} \label{eq:fm.gm<C}
 \|f_m^0\|_{\fH} \cdot \|f_{*m}^0\|_{\fH} =
 |\gam_m^0 \gam_{*m}^0| \cdot \|\wt{f}_m^0\|_{\fH} \cdot \|\wt{f}_{*m}^0\|_{\fH}
 \le \frac{M_h M_{*h}}{\bigabs{(\wt{f}_m^0, \wt{f}_{*m}^0)_{\fH}}}
 \le \frac{M_h M_{*h}}{C_3} =: C_4, \qquad |m| > m_0.
\end{equation}
Since $C_4$ does not depend on $m$, estimate~\eqref{eq:fm.gm<C} implies estimate~\eqref{eq:sup.fm0.gm0}, which completes the proof.
\end{proof}
In general case we need more strict conditions on the matrix function $B(\cdot)$
\begin{theorem} \label{th:uniform.minim}
Let matrix function $B(\cdot)$ given by~\eqref{eq:Bx.def} satisfy conditions~\eqref{eq:betak.alpk.Linf}--\eqref{eq:betak-betaj<-eps}, let $Q \in L^1([0,\ell]; \bC^{n \times n})$ and let BVP~\eqref{eq:LQ.def.reg}--\eqref{eq:Uy=0} be strictly regular according to Definition~\ref{def:bvp.strict}. Then any system of root vectors of the operator $L_U(Q)$ is uniformly minimal in $\fH$.
\end{theorem}
\begin{proof}
As usual, applying gauge transform from Lemma~\ref{lem:gauge} we can reduce general case to the case of $Q$ satisfying ``zero block diagonality'' condition~\eqref{eq:Qjk=0.bj=bk}. Since operators $L_U(Q)$ and $L_{\wt{U}}(\wt{Q})$ are similar this transform preserves uniform minimality as explained in Remark~\ref{rem:similarity}. Hence, without loss of generality we can assume that original $Q$ satisfies ``zero block diagonality'' condition~\eqref{eq:Qjk=0.bj=bk}.

Due to assumptions on matrix functions $B(\cdot)$ and $Q(\cdot)$,
Proposition~\ref{prop:sine.type} and Theorem~\ref{th:ln=ln0+o} imply that characteristic determinant $\Delta_Q(\cdot)$ has a countable sequence of eigenvalues $\L := \{\l_m\}_{m \in \bZ}$ (counting multiplicity) with asymptotic behavior~\eqref{eq:lm=lm0+o} and lying in the strip $\Pi_h$ (we can assume that both $\L_0$ and $\L$ lie in the same strip $\Pi_h$ by increasing $h$ if needed). By Lemma~\ref{lem:eigen}, $\L$ is a sequence of eigenvalues of $L_U(Q)$ (counting multiplicity) and $\ol{\L} := \{\ol{\l_m}\}_{m \in \bZ}$ is a sequence of eigenvalues of $L_U^*(Q)$ (counting multiplicity). Moreover, each eigenvalue of $L_U(Q)$ has finite multiplicity.
Combining asymptotic formula~\eqref{eq:lm=lm0+o} and separation condition~\eqref{separ_cond} on $\L_0$ imply that for some $m_1 \ge m_0$ we have
\begin{equation} \label{eq:lk-lj}
 |\l_j - \l_k| > \delta, \qquad j \ne k, \quad |j|, |k| > m_1.
\end{equation}

Let $\cF := \{f_m\}_{m \in \bZ}$ be any system of root vectors of the operator $L_U(Q)$. Since operator $L_U(Q)$ has discrete spectrum and each its eigenvalue has finite multiplicity,
we can choose system of root vectors $\cF_* := \{f_{*m}\}_{m \in \bZ}$ of the operator $L_U^*(Q)$ in such a way that $\cF$ and $\cF_*$ are biorthogonal systems. i.e.\ $(f_j, f_{*k}) = \delta_{jk}$, $j, k \in \bZ$. This, implies minimality of the system $\cF$. In accordance with Lemma~\ref{lem:criter.unif.minimality}(ii), to prove uniform minimality it is sufficient to show that
\begin{equation} \label{eq:sup.fm.gm}
 \sup_{|m| > m_2} \|f_m\|_{\fH} \cdot \|f_{*m}\|_{\fH} < \infty,
\end{equation}
for some $m_2 \ge m_1$ that we will choose later. Here $m_1 \ge m_0$ is from the separation condition~\eqref{eq:lk-lj} on $\l_m$, while $m_0$ is from separation condition~\eqref{separ_cond} on $\l_m^0$.

Let $|m| > m_1 \ge m_0$. Then $\l_m$ and $\ol{\l_m}$ are algebraically simple eigenvalues of the operators $L_U(Q)$ and $L_U^*(Q)$, respectively. Let $p = p_m \in \oneton$ and $q = q_m \in \oneton$ be indices chosen in the part (i) of the proof for which estimate~\eqref{eq:y0p.y0*q>} holds. Following part (i) we define similar vector functions
\begin{equation} \label{eq:wtfgm.def}
\wt{f}_m(\cdot) := Y_p(\cdot, \l_m) \quad\text{and}\quad
\wt{f}_{*m}(\cdot) := Y_{*q}(\cdot, \ol{\l_m})
\end{equation}
by formulas~\eqref{eq:Ypxl.def} and~\eqref{eq:yj.eigen*}, respectively. Set
\begin{align}
 \wt{F}_m(\cdot) := \wt{f}_m(\cdot) - \wt{f}_m^0(\cdot)
 = Y_p(\cdot, \l_m) - Y_p^0(\cdot, \l_m^0), \qquad |m| \ge m_1, \\
 \wt{F}_{*m}(\cdot) := \wt{f}_{*m}(\cdot) - \wt{f}_{*m}^0(\cdot)
 = Y_{*q}(\cdot, \ol{\l_m}) - Y_{*q}^0(\cdot, \ol{\l_m^0}),
 \qquad |m| \ge m_1.
\end{align}
It follows from Schwarz inequality that
\begin{equation} \label{eq:wtfm.wtgm}
 \bigabs{\bigl(\wt{f}_m, \wt{f}_{*m} \bigr)_{\fH}}
 \ge \bigabs{\bigl(\wt{f}_m^0, \wt{f}_{*m}^0 \bigr)_{\fH}}
 - \|\wt{f}_m^0\|_{\fH} \|\wt{F}_{*m}\|_{\fH}
 - \|\wt{F}_m\|_{\fH} \|\wt{f}_{*m}^0\|_{\fH}
 - \|\wt{F}_m\|_{\fH} \|\wt{F}_{*m}\|_{\fH}.
\end{equation}
It follows from Theorem~\ref{th:eigenvec}, that $\|\wt{F}_m\|_{\fH} \to 0$ and $\|\wt{F}_{*m}\|_{\fH} \to 0$ as $|m| \to \infty$.
Combining this observation with estimates~\eqref{eq:wtfm.wtgm},~\eqref{eq:y0p.y0*q>}, and estimates~\eqref{eq:Phij0.Yk0.norm} and~\eqref{eq:Y*q0.norm} on $\|\wt{f}_m^0\|_{\fH} = \|Y_p^0(\cdot, \l_m^0)\|_{\fH}$ and $\|\wt{f}_{*m}^0\|_{\fH} = \|Y_{*q}^0(\cdot, \ol{\l_m^0)}\|_{\fH}$, we see that
\begin{align}
\label{eq:wtfm.wtgm>C3}
 & \bigabs{\bigl(\wt{f}_m, \wt{f}_{*m} \bigr)_{\fH}} \ge C_3/2,
 \qquad |m| > m_2, \\
\label{eq:wtfm<.wtgm<}
 & \|\wt{f}_m\|_{\fH} \le 2 M_h, \qquad \|\wt{f}_{*m}\|_{\fH} \le 2 M_h
 , \qquad |m| > m_2,
\end{align}
for some $m_2 \ge m_1$. Inequality~\eqref{eq:wtfm.wtgm>C3} implies that both vector functions $\wt{f}_m$ and $\wt{f}_{*m}$ are non-zero.
Since $\l_m$ is algebraically simple eigenvalue of the operator $L_U(Q)$, and $\wt{f}_m \ne 0$ and $\wt{f}_{*m} \ne 0$, Lemmas~\ref{lem:eigen.canon} and~\ref{lem:eigen.canon*} ensure that
\begin{equation} \label{eq:fm.gm.def}
 f_m(\cdot) = \gam_m \wt{f}_m(\cdot) = \gam_m Y_p(\cdot, \l_m), \qquad
 f_{*m}(\cdot) = \gam_{*m} \wt{f}_{*m}(\cdot)
 = \gam_{*m} Y_{*q}(\cdot, \ol{\l_m}), \qquad |m| > m_2,
\end{equation}
with some $\gam_m, \gam_{*m} \ne 0$. The proof of estimate~\eqref{eq:sup.fm.gm} is finished the same way as in part (i) by using estimates~\eqref{eq:wtfm.wtgm>C3}--\eqref{eq:wtfm<.wtgm<}.
\end{proof}
\subsection{Riesz basis property} \label{subsec:riesz}
First, let us recall some definitions.
\begin{definition}
\textbf{(i)} A sequence $\{\phi_m\}_{m \in \bZ}$ of vectors in $\fH$ is called a \textbf{Riesz basis} if it admits a representation $\phi_m = T e_m$, $m \in \bZ$, where $\{e_m\}_{m \in \bZ}$ is an orthonormal basis in $\fH$ and $T : \fH \to \fH$ is a bounded operator with a bounded inverse.

\textbf{(ii)} A sequence $\{\phi_m\}_{m \in \bZ}$ of vectors in $\fH$ is called \textbf{Besselian} if
\begin{equation} \label{eq:bessel}
\sum_{m \in \bZ} \abs{(f, \phi_m)_{\fH}}^2 < \infty,
\qquad f \in \fH.
\end{equation}
\end{definition}
\begin{remark} \label{rem:gelfand}
In accordance with closed graph theorem, inequality~\eqref{eq:bessel} is equivalent to
\begin{equation} \label{eq:bessel2}
\sum_{m \in \bZ} \abs{(f, \phi_m)_{\fH}}^2 \le \gam^2 \|f\|_{\fH}^2,
\qquad f \in \fH,
\end{equation}
where $\gam > 0$ does not depend on $f$. Putting in~\eqref{eq:bessel2} $f = \phi_m$ implies
$\|\phi_m\|_{\fH} \le \gam$, $m \in \bZ$.
\end{remark}
Our investigation of the Riesz basis property of the system of root vectors of the operator $L_U(Q)$ heavily relies on the following well-known Bari criterion.
\begin{theorem} \label{th:Bari.crit}\cite[Theorem\ VI.2.1]{GohKre65}
Let $\fH$ be a separable Hilbert space. The vectors system $\{\phi_m\}_{m \in \bZ} \subset \fH$ forms a Riesz basis in $\fH$ if and only if it is complete and Besselian in $\fH$, and there exists a biorthogonal system $\{\phi_{*m}\}_{m \in \bZ}$ that is also complete and Besselian.
\end{theorem}
First, we establish a result that implies Besselian property for eigenvectors of the unperturbed operator $L_U(0)$.
\begin{lemma} \label{lem:bessel}
Let entries of $B(\cdot)$ satisfy condition~\eqref{eq:betak.alpk.Linf}. Let $\{\mu_m\}_{m \in \bZ}$ be an incompressible sequence lying in the strip $\Pi_h$ (see Definition~\ref{def:incompressible}). Then the following statements hold:

\textbf{(i)} For any $k \in \oneton$ the sequence $\{\Phi^0_{k}(\cdot, \mu_m)\}_{m \in \bZ}$ defined in~\eqref{eq:Phi0k.def} is Besselian in $\fH$.

\textbf{(ii)} For any $p \in \oneton$ the sequence $\{Y_p^0(\cdot, \mu_m)\}_{m \in \bZ}$ defined in~\eqref{eq:Yp0xl.def} is Besselian in $\fH$.
\end{lemma}
\begin{proof}
\textbf{(i)} Let $k \in \oneton$ be fixed and let
$$
f = \col(f_1, \ldots, f_n) \in \fH, \qquad\text{i.e.}\qquad
f_j \in L^2_{|\beta_j|}[0,\ell], \quad j \in \oneton.
$$
It follows from definition~\eqref{eq:Phi0k.def} of $\Phi_k^0(\cdot,\l)$ that
\begin{equation} \label{eq:f.Phik0.step1}
 (f, \Phi^0_{k}(\cdot, \mu_m))_{\fH}
 = \int_0^{\ell} f_k(x) \cdot \ol{e^{i \mu_m \rho_k(x)}} |\beta_k(x)|\,dx.
\end{equation}

Condition~\eqref{eq:betak.alpk.Linf} implies conditions~\eqref{eq:rhok.Lip} on $\rho_k$ and the inverse function $\rho_k^{-1}$.
Let $x_k := |\rho_k^{-1}| \in \Lip[0,\ell_k]$, where $\ell_k := |b_k|$. Since $\beta_k$ and $\rho_k$ do not change sign on the segment $[0, \ell]$, then $x_k$ is the inverse function of $|\rho_k|$.
Consider the function $f_k \circ x_k$ defined on $[0,\ell_k]$. Since $1/\beta_k \in L^{\infty}[0,\ell]$, it follows from definition of $L^2_{|\beta_k|}[0,\ell]$ that $f_k \in L^2[0,\ell]$.
Since $x_k \in \Lip[0,\ell]$ and strictly monotonous, and $f_k \in L^2[0,\ell]$ it follows that $f_k \circ x_k \in L^2[0, \ell_k] \subset L^1[0, \ell_k]$.
Finally, recall that $s_k = \sign(\beta_k(\cdot)) = \sign(\rho_k(\cdot)) = \const$.

Taking observations of the previous paragraph into account and making a change of variable $x = x_k(u)$ in~\eqref{eq:f.Phik0.step1} (and so $u = |\rho_k(x)| = s_k \rho_k(x)$ and $du = |\beta_k(x)| dx$), we get
\begin{equation} \label{eq:f.Phik0}
 (f, \Phi^0_{k}(\cdot, \mu_m))_{\fH}
 = \int_0^{\ell_k} f_k(x_k(u)) \cdot \ol{e^{i s_k \mu_m u}} \,du.
\end{equation}
Since sequence $\{s_k \mu_m\}_{m \in \bZ}$ is incompressible, then~\cite[Lemma~2.2]{Katsn71} implies that the sequence $\{e^{i s_k \mu_m u}\}_{m \in \bZ}$ is Besselian in $L^2[0,\ell_k]$ (see also the proof of Lemma 6.4 in~\cite{LunMal16JMAA}). With account of this observation, it follows from~\eqref{eq:f.Phik0} and inclusion $f_k \circ x_k \in L^2[0,\ell_k]$, that the sequence $\{\Phi^0_{k}(\cdot, \mu_m)\}_{m \in \bZ}$ is Besselian in $\fH$.

\textbf{(ii)} Combining formula~\eqref{eq:Yp0xl.def} with
Schwarz inequality and estimate~\eqref{eq:A0.Delta0<M} (applicable since $\mu_m \in \Pi_h$), we arrive at
\begin{multline} \label{eq:f.Yp0}
 \bigabs{\(f, Y_p^0(\cdot, \mu_m)\)_{\fH}}^2
 = \Bigl|\sum_{k=1}^n A_{kp}^0(\mu_m) \cdot
 \(f, \Phi_k^0(\cdot, \mu_m)\)_{\fH}\Bigr|^2
 \le \sum_{k=1}^n |A_{kp}^0(\mu_m)|^2 \cdot
 \sum_{k=1}^n \bigabs{\(f, \Phi_k^0(\cdot, \mu_m)\)_{\fH}}^2 \\
 \le n M_h^2 \sum_{k=1}^n \bigabs{\(f, \Phi_k^0(\cdot, \mu_m)\)_{\fH}}^2,
 \qquad m \in \bZ, \quad p \in \oneton.
\end{multline}
Estimate~\eqref{eq:f.Yp0} and part (i) of Lemma now finish the proof.
\end{proof}
Now, using integral representation~\eqref{eq:Phip=Phi0p+int.sum} we can extend the previous result to vector functions $\Phi_k(\cdot,\l)$ and $Y_p(\cdot,\l)$, which will imply Besselian property for eigenvectors of the operator $L_U(Q)$.
\begin{proposition} \label{prop:bessel}
Let matrix functions $B(\cdot)$ and $Q(\cdot)$ satisfy conditions~\eqref{eq:Bx.def}--\eqref{eq:Qjk=0.bj=bk}. Let $\{\mu_m\}_{m \in \bZ}$ be an incompressible sequence lying in the strip $\Pi_h$.
Then the following statements hold:

\textbf{(i)} For any $k \in \oneton$ the sequence $\{\Phi_k(\cdot, \mu_m)\}_{m \in \bZ}$ defined in~\eqref{eq:Phixl.def} is Besselian in $\fH$.

\textbf{(ii)} For any $p \in \oneton$ the sequence $\{Y_p(\cdot, \mu_m)\}_{m \in \bZ}$ defined in~\eqref{eq:Ypxl.def} is Besselian in $\fH$.
\end{proposition}
\begin{proof}
\textbf{(i)} Let $k \in \oneton$ be fixed and let
$$
f = \col(f_1, \ldots, f_n) \in \fH, \qquad\text{i.e.}\qquad
f_j \in L^2_{|\beta_j|}[0,\ell], \quad j \in \oneton.
$$
By Proposition~\ref{prop:Phip=Phip0+int.sum}, representation~\eqref{eq:Phip=Phi0p+int.sum} holds with $p=k$, where vector kernel $R_q^{[k]}$ is defined via~\eqref{eq:Rpjk.def}--\eqref{eq:Rqp.def} and satisfies inclusion~\eqref{eq:Rpq.in.X}.
Moreover, by definition of signature matrix $S$ we have,
$$
\beta_q(x) = s_q |\beta_q(x)|, \qquad q \in \oneton.
$$
Taking these observations and formula~\eqref{eq:f.Phik0} into account, we get by changing order of integration
\begin{align}
\nonumber
 (f, \Phi_k(\cdot, \l))_{\fH}
 & = (f, \Phi_k^0(x,\l))_{\fH} + \sum_{j,q=1}^n \int_0^{\ell} f_j(x)
 \(\int_0^x \ol{R_{jq}^{[k]}(x,t) e^{i \l \rho_q(t)}} \beta_q(t) \,dt\)
 \cdot |\beta_j(x)| \,dx \\
\nonumber
 & = (f, \Phi_k^0(x,\l))_{\fH} + \sum_{j,q=1}^n \int_0^{\ell}
 \(\int_t^{\ell} f_j(x) \ol{R_{jq}^{[k]}(x,t)} |\beta_j(x)| \,dx\)
 \ol{e^{i \l \rho_q(t)}} \beta_q(t) \,dt \\
\label{eq:f.Phik}
 & = \int_0^{\ell} f_k(t) \cdot \ol{e^{i \l \rho_k(t)}} |\beta_k(t)|\,dx + \sum_{j,q=1}^n \int_0^{\ell}
 f_{kjq}(t) \cdot \ol{e^{i \l \rho_q(t)}} |\beta_q(t)| \,dt,
\end{align}
where
\begin{equation} \label{eq:fkjq.def}
 f_{kjq}(t) := \int_t^{\ell} \ol{R_{jq}^{[k]}(x,t)}
 \cdot s_q f_j(x) |\beta_j(x)| \,dx,
 \qquad t \in [0,\ell], \quad j,q \in \oneton.
\end{equation}
Since $\beta_j \in L^{\infty}[0,\ell]$ and $f_j \in L^2_{|\beta_j|}[0,\ell]$, it follows that
\begin{equation} \label{eq:sq.fj.betaj.in.L2}
 f_j, \ s_q f_j |\beta_j| \in L^2[0,\ell], \qquad j,q \in \oneton.
\end{equation}
Consider an operator $\cR_{kjq}$ generated by the kernel $R_{jq}^{[k]}$ by formula~\eqref{eq:cR.def} in $L^2[0,\ell]$. Lemma~\ref{lem.Volterra.operGeneral} and inclusion~\eqref{eq:Rpq.in.X} imply that operator $\cR_{kjq}$ is bounded in $L^2[0,\ell]$ (and even Volterra operator). It is clear that operator $\cR_{kjq}$ and its adjoint operator $\cR_{kjq}^*$ are of the form
\begin{align}
\label{eq:cRkjq}
 (\cR_{kjq} f)(x) &= \int_0^x R_{jq}^{[k]}(x,t) f(t) dt,
 \qquad f \in L^2[0,\ell], \\
\label{eq:cRkjq*}
 (\cR_{kjq}^* g)(t) &= \int_t^{\ell} \ol{R_{jq}^{[k]}(x,t)} g(x) dx,
 \qquad g \in L^2[0,\ell].
\end{align}
Formulas~\eqref{eq:fkjq.def},~\eqref{eq:cRkjq*}, inclusion~\eqref{eq:sq.fj.betaj.in.L2} and boundedness of the operator $\cR_{kjq}^*$ imply that ${f_{kjq} \in L^2[0,\ell]}$. Recall, that $k \in \oneton$ is fixed. Let us set
\begin{equation} \label{eq:fq.fkq.def}
 F_q := F_{kq} \col(\delta_{1q}, \ldots, \delta_{nq}), \qquad
 F_{kq} := \delta_{kq} f_k + \sum_{j=1}^n f_{kjq} \in L^2[0,\ell],
 \qquad q \in \oneton.
\end{equation}
Since $\beta_q \in L^{\infty}[0,\ell]$, $q \in \oneton$, it is clear that $F_q \in \fH$. With account of notation~\eqref{eq:fq.fkq.def} and formula~\eqref{eq:f.Phik0.step1}, we get by setting $\l = \mu_m$ in~\eqref{eq:f.Phik},
\begin{equation}
 (f, \Phi_k(\cdot, \mu_m))_{\fH} = \sum_{q=1}^n
 \int_0^{\ell} F_{kq}(t) \cdot \ol{e^{i \mu_m \rho_q(t)}} |\beta_q(t)| \,dt
 = \sum_{q=1}^n (F_q, \Phi_q(\cdot, \mu_m))_{\fH},
 \qquad m \in \bZ.
\end{equation}
Besselian property of the sequence $\{\Phi_k(\cdot, \mu_m)\}_{m \in \bZ}$ is now implied by Besselian property of the sequences $\{\Phi_q^0(\cdot, \mu_m)\}_{m \in \bZ}$, $q \in \oneton$, established in Lemma~\ref{lem:bessel}(i), and inclusion $F_q \in \fH$, $q \in \oneton$, which finishes the proof.

\textbf{(ii)} Let $p \in \oneton$ be fixed. It follows from representation~\eqref{eq:AjkQ=Ajk0+int} for $A_{kp}(\l)$, inclusion $g_{kp} \in L^1[b_-,b_+]$ and estimate~\eqref{eq:A0.Delta0<M} on $A_{kp}^0(\l)$ that for all $k \in \oneton$ the following estimate holds,
\begin{equation} \label{eq:Akpl<}
 |A_{kp}(\l)| \le |A_{kp}^0(\l)|
 + \abs{\int_{b_-}^{b_+} g_{kp}(u) e^{i \l u} \,du}
 \le M_h + \|g_{kp}\|_1 (e^{-b_- h} + e^{b_+ h}) := M_{h,kp},
 \qquad |\l| \le \Pi_h.
\end{equation}
As in the proof of Lemma~\ref{lem:bessel}(ii), combining formula~\eqref{eq:Ypxl.def} with
Schwarz inequality and estimate~\eqref{eq:Akpl<} (applicable since $\mu_m \in \Pi_h$, $m \in \bZ$), we arrive at
\begin{equation} \label{eq:f.Yp}
 \bigabs{\(f, Y_p(\cdot, \mu_m)\)_{\fH}}^2
 = \Bigl|\sum_{k=1}^n A_{kp}(\mu_m) \cdot
 \(f, \Phi_k(\cdot, \mu_m)\)_{\fH}\Bigr|^2
 \le \sum_{k=1}^n M_{h,kp}^2 \cdot
 \sum_{k=1}^n \bigabs{\(f, \Phi_k(\cdot, \mu_m)\)_{\fH}}^2,
 \quad m \in \bZ.
\end{equation}
Estimate~\eqref{eq:f.Yp} and part (i) of Proposition finish the proof.
\end{proof}
\begin{proposition} \label{prop:Riesz_basis}
Let entries of matrix function $B(\cdot)$ satisfy condition~\eqref{eq:betak.alpk.Linf}, i.e.
\begin{equation} \label{eq:betak.alpk.Linf.basis0}
 \beta_k, \ 1/\beta_k \ \in L^{\infty}([0, \ell]; \bR), \qquad
 \sign(\beta_k(\cdot)) \equiv \const, \qquad
 k \in \oneton,
\end{equation}
and let boundary conditions~\eqref{eq:Uy=0} be strictly regular. Then any normalized system of root vectors of the operator $L_{0,U} = L_U(0)$ forms a Riesz basis in $\fH$.
\end{proposition}
\begin{proof}
As in the proof of Proposition~\ref{prop:uniform.minim}, operator $L_{0,U}$ has countable asymptotically separated sequence of eigenvalues $\L_0 := \{\l_m^0\}_{m \in \bZ}$. Let $\cF_0 := \{f_m^0\}_{m \in \bZ}$ be some normalized system of root vectors of the operator $L_{0,U}$, where $f_m^0$ is a root vector corresponding to $\l_m^0$, $\|f_m^0\|_{\fH} = 1$, $m \in \bZ$. As in the proof of Proposition~\ref{prop:uniform.minim}, we can choose system of root vectors $\cF_{*0} := \{f_{*m}^0\}_{m \in \bZ}$ of the operator $L_{0,U}^*$ in such a way that $\cF_0$ and $\cF_{*0}$ are biorthogonal systems.

Since boundary conditions~\eqref{eq:Uy=0} are regular then by Lemma~\ref{lem:adjoint} boundary conditions~\eqref{eq:U*y=0} of $L_{0,U}^*$ are also regular. Hence Theorem~\ref{th:compl}(i) ensures the completeness property of both systems $\cF_0$ and $\cF_{*0}$.

Let $|m| > m_0$, where $m_0$ is from Definition~\ref{def:strict.regular}(iii) of strict regularity. Based on the proof of Proposition~\ref{prop:uniform.minim}, all of relations~\eqref{eq:wtfgm0.def}--\eqref{eq:fm0.gm0} hold. Since $\|f_m^0\| = 1$, it is clear from~\eqref{eq:fm0.gm0.def}--\eqref{eq:fm0.gm0} that
\begin{equation} \label{eq:gam.m0.gam.*m0}
 |\gam_m^0| = \frac{1}{\|\wt{f}_m^0\|_{\fH}}, \qquad
 |\gam_{*m}^0| = \frac{\|\wt{f}_m^0\|_{\fH}}{\bigabs{(\wt{f}_m^0, \wt{f}_{*m}^0)_{\fH}}}.
\end{equation}
Combining estimates~\eqref{eq:y0p.y0*q>} and~\eqref{eq:wtfgm0.norm} with the Schwartz inequality yields
\begin{equation}
C_3 \le |(\wt f_m^0, \wt f_{*m}^0)| \le \|\wt f_m^0\| \cdot \|\wt f_{*m}^0\| \le M_h \|\wt f_m^0\| \le M_h^2.
\end{equation}
Inserting this estimate into~\eqref{eq:gam.m0.gam.*m0} implies
\begin{equation} \label{eq:gam.m0.gam.*m0<}
|\gam_m^0| \le C_5, \qquad |\gam_{*m}^0| \le C_5, \qquad C_5 := M_h / C_3.
\end{equation}
With account of definitions~\eqref{eq:wtfgm0.def} and assumption~\eqref{eq:betak.alpk.Linf.basis0}, Lemma~\ref{lem:bessel}(ii) implies that the sequences $\{\wt{f}_m^0\}_{|m|>m_0}$ and $\{\wt f_{*m}^0\}_{|m|>m_0}$ are Besselian in $\fH$. Since $f_m^0 = \gam_m^0 \wt{f}_m^0$, $f_{*m}^0 = \gam_{*m}^0 \wt{f}_{*m}^0$, $|m| > m_0$, inequality~\eqref{eq:gam.m0.gam.*m0<} implies that the sequences $\{f_m^0\}_{|m|>m_0}$ and $\{f_{*m}^0\}_{|m|>m_0}$ are also Besselian. And hence so are $\cF_0 = \{f_m^0\}_{m \in \bZ}$ and $\cF_{*0} = \{f_{*m}^0\}_{m \in \bZ}$. Theorem~\ref{th:Bari.crit} now finishes the proof.
\end{proof}
\begin{theorem} \label{th:Riesz_basis}
Let matrix function $B(\cdot) = \diag(\beta_1, \ldots, \beta_n)$\
satisfy conditions~\eqref{eq:betak.alpk.Linf}--\eqref{eq:betak-betaj<-eps}, i.e.\ for some $n_- \in \{0, 1, \ldots, n\}$,
\begin{align}
\label{eq:betak.alpk.Linf.basis}
 & \beta_k, 1/\beta_k \in L^{\infty}[0, \ell], \qquad
 s_k := \sign(\beta_k(\cdot)) \equiv \const \ne 0, \qquad
 k \in \oneton, \\
\label{eq:beta-+n.basis}
 & \beta_1(x) \le \ldots \le \beta_{n_-}(x) < 0
 < \beta_{n_-+1}(x) \le \ldots \le \beta_n(x),
 \qquad x \in [0, \ell],
\end{align}
and there exists $\theta > 0$ such that for each $k \in \oneto{n-1}$
\begin{equation} \label{eq:betak-betaj<-eps.basis}
 \text{either} \quad \beta_k \equiv \beta_{k+1}
 \quad \text{or} \quad \beta_k(x) + \theta < \beta_{k+1}(x), \quad x \in [0, \ell].
\end{equation}
Further, let $Q \in L^1([0,\ell]; \bC^{n \times n})$ and let BVP~\eqref{eq:LQ.def.reg}--\eqref{eq:Uy=0} be strictly regular according to Definition~\ref{def:bvp.strict}. Then any normalized system of root vectors of the operator $L_U(Q)$ forms a Riesz basis in $\fH$.
\end{theorem}
\begin{proof}
As in the proof of Theorem~\ref{th:uniform.minim}, applying gauge transform from Lemma~\ref{lem:gauge}, we can assumed that $Q$ satisfy ``zero block diagonality'' condition~\eqref{eq:Qjk=0.bj=bk}.
Hence, as in the proof of Theorem~\ref{th:uniform.minim}, operator $L_U(Q)$ has countable asymptotically separated sequence of eigenvalues $\L := \{\l_m\}_{m \in \bZ}$.

Let $\cF := \{f_m\}_{m \in \bZ}$ be some normalized system of root vectors of the operator $L_U(Q)$, where $f_m$ is a root vector corresponding to $\l_m$, $\|f_m\|_{\fH} = 1$, $m \in \bZ$. As in the proof of Theorem~\ref{th:uniform.minim}, we can choose system of root vectors $\cF_* := \{f_{*m}\}_{m \in \bZ}$ of the operator $L_U^*(Q)$ in such a way that $\cF$ and $\cF_*$ are biorthogonal systems.

Since boundary conditions~\eqref{eq:Uy=0} are regular then by Lemma~\ref{lem:adjoint} boundary conditions~\eqref{eq:U*y=0} of $L_U^*(Q)$ are also regular. Hence assumptions~\eqref{eq:betak.alpk.Linf}--\eqref{eq:betak-betaj<-eps},~\eqref{eq:Qjk=0.bj=bk} and Theorem~\ref{th:compl}(ii) ensures the completeness property of both systems $\cF$ and $\cF_*$.

Let $|m| > m_2$, where $m_2$ was chosen in the proof of Theorem~\ref{th:uniform.minim} to satisfy relations~\eqref{eq:wtfm.wtgm>C3}--\eqref{eq:wtfm<.wtgm<}.
Based on the proof of Theorem~\ref{th:uniform.minim}, all of relations~\eqref{eq:wtfgm.def}--\eqref{eq:fm.gm.def} hold. With account of definitions~\eqref{eq:wtfgm.def}, Proposition~\ref{prop:bessel}(ii) implies that the sequences $\{\wt{f}_m\}_{|m|>m_2}$ and $\{\wt f_{*m}\}_{|m|>m_2}$ are Besselian in $\fH$. From this point the proof is finished the same way as in part (i) by using key estimates~\eqref{eq:wtfm.wtgm>C3}--\eqref{eq:wtfm<.wtgm<}.
\end{proof}
\begin{remark} \label{rem:riesz.basis.2x2}
Let us provide brief history of results on Riesz basis property for BVP~\eqref{eq:LQ.def.reg}--\eqref{eq:Uy=0} with constant $2 \times 2$ matrix $B(\cdot) \equiv B = \diag(b_1, b_2) = B^*$.

\textbf{(i)} The Riesz basis property for $2\times 2$ Dirac type operators
$L_U(Q)$ and separated boundary conditions was established earlier than for the operators with general regular boundary conditions. Namely, this property was proved firstly in~\cite{TroYam01,TroYam02} by I.~Trooshin and M.~Yamamoto for $B = \diag(-1,1)$ and $Q \in C^1([0,1]; \bC^{2 \times 2})$. Later, P.~Djakov and B.~Mityagin in~\cite{DjaMit10}, and A.G.~Baskakov, A.V.~Derbushev and A.O.~Shcherbakov in~\cite{Bask11} relaxed smoothness assumption on potential to $Q \in L^2([0,1]; \bC^{2 \times 2})$. Independently, in~\cite{HasOri09} S.~Hassi and L.~Oridoroga established this property for Dirac type operator with $B = \diag(b_1,b_2) = B^*$ and $Q \in C^1([0,1]; \bC^{2 \times 2})$.

\textbf{(ii)} The Bari-Markus property of the Riesz projectors of unperturbed and perturbed $2\times 2$ Dirac operators with separated, periodic and antiperiodic boundary conditions was established in~\cite{Mit04,DjaMit10} and reproved by another method in~\cite{Bask11}. In~\cite{DjaMit12UncDir} similar results have been obtained for general regular boundary conditions.

\textbf{(iii)} The most complete result on the Riesz basis property for $2\times 2$ Dirac and Dirac-type systems with $Q \in L^1$ and strictly regular boundary conditions was obtained independently by different methods and at the same time by A.M.~Savchuk and A.A.~Shkalikov~\cite{SavShk14} and by the authors~\cite{LunMal14Dokl, LunMal16JMAA}. The case of regular boundary conditions and $Q \in L^1$ is treated in~\cite{SavShk14} for the first time. Other proofs were obtained later in~\cite{SavSad15DAN},~\cite{SavSad15} (see also their recent survey~\cite{SavSad20} and references therein).

\textbf{(iv)} Periodic and antiperiodic (necessarily non-strictly regular) BVP for $2 \times 2$ Dirac equation have attracted certain attention during the last decade. In~\cite[Theorem~13]{DjaMit12TrigDir},~\cite[Theorem~19]{DjaMit12Crit} and~\cite{DjaMit13CritDir}, P. Djakov and B. Mityagin established a \emph{criterion} for the system of root vectors to contain a Riesz basis for periodic (resp., antiperiodic) $2 \times 2$ Dirac operator in terms of the Fourier coefficients of $Q$ as well as in terms of periodic (resp., antiperiodic) and Dirichlet spectra. A.~Makin~\cite{Mak20,Mak21} established Riesz basis property for periodic $2 \times 2$ Dirac operator under certain explicit algebraic assumptions on a potential matrix. See also recent survey~\cite{DjaMit20UMNper} by P. Djakov and B. Mityagin, survey~\cite{Mak21Ito} by A.S. Makin, and the references therein.
\end{remark}
\begin{remark} \label{rem:riesz.basis.nxn}
In this remark we go over known results on Riesz basis property for BVP~\eqref{eq:LQ.def.reg}--\eqref{eq:Uy=0} with constant $n \times n$ matrix $B(\cdot) \equiv B = B^*$, when $n > 2$.

\textbf{(i)} In~\cite{MykPuy13} the results of~\cite{DjaMit10} regarding the Bari-Markus property in $L^2([0,1]; \bC^2)$ were extended to the case of the Dirichlet BVP for $2m \times 2m$ Dirac equation with $Q \in L^2([0,1]; \bC^{2m\times 2m})$.

\textbf{(ii)} To the best of our knowledge the first result on the Riesz basis property for BVP~\eqref{eq:LQ.def.reg}--\eqref{eq:Uy=0} generated by general $n \times n$ system~\eqref{eq:LQ.def.reg} with $B(x) = B = \diag(b_1, \ldots, b_n) \in \bC^{n\times n} \ne B^*$ and bounded $Q \in L^{\infty}([0,1]; \bC^{n \times n})$
was obtained by the authors in~\cite{LunMal15JST}. Treated boundary conditions form rather broad class that covers, in particular, periodic, antiperiodic, and regular separated (not necessarily self-adjoint) boundary conditions.

\textbf{(iii)} In~\cite{KurAbd18, KurGad20}, Bessel and Riesz basis properties on abstract level were established, i.e.\ the operator $L_U(Q)$ was studied without explicit boundary conditions.
\end{remark}
\subsection{Riesz basis property with parentheses}
\label{subsec:riesz.paren}
First, let us recall a corresponding definition.
\begin{definition} \label{def:basis}
(i) A sequence of subspaces $\{\cH_m\}_{m=1}^{\infty}$ is called a \textbf{Riesz basis of subspaces} in a separable Hilbert space $\cH$ if there exists a complete sequence of mutually orthogonal subspaces $\{\cH'_m\}_{m=1}^{\infty}$ and a bounded operator $T$ in $\cH$ with a bounded inverse such that $\cH_m = T \cH'_m$, $m \in \bN$.

(ii) A sequence $\{\phi_m\}_{m=1}^{\infty}$ of vectors in $\cH$ is called a \textbf{Riesz basis with parentheses} if each its finite subsequence is linearly independent, and there exists an increasing sequence $\{m_k\}_{k=0}^{\infty} \subset \bN$ such that $m_0=1$ and the sequence $\cH_k := \Span\{\phi_j\}_{j=m_{k-1}}^{m_k-1}$, forms a Riesz basis of subspaces in $\cH$. Subspaces $\cH_k$ are called blocks.
\end{definition}
In~\cite{Shk79}, A.A.~Shkalikov established Riesz basis property with parentheses for BVP for ODE of $n$-th order with coefficients $q_2, \ldots, q_n \in L^1$ and regular boundary conditions $V(y)=0$. Denote corresponding operator as $\cL(V,q)$. The idea was to present the operator $\cL(V,q)$ as a bounded perturbation of the operator similar to $\cL(\wt{V},\wt{q})$, where boundary conditions $\wt{V}(y)=0$ are already strictly regular and the operator is known to have a Riesz basis property (without parentheses). Then abstract result of Katsnel'son-Markus-Matsaev (see~\cite[Theorem 3.1]{Katsn67} and also~\cite{Agran99},~\cite{MarMats84},~\cite[Theorem 6.12]{Markus88}) implies desired Riesz basis property with parentheses for the original operator $\cL(V,q)$.

This idea later was used in~\cite{SavShk14} to establish Riesz basis property with parentheses for $2 \times 2$ Dirac operator with regular boundary conditions, i.e.\ operator $L_U(Q)$ with $B(\cdot) \equiv \diag(-1,1)$. Later we also used this idea in~\cite{LunMal16JMAA}, to establish this result for $2 \times 2$ Dirac type operator, i.e.\ when $B(\cdot) \equiv \diag(b_1, b_2)$, $b_1 < 0 < b_2$.

In this subsection we use the same idea, and, following~\cite[Section 7]{LunMal16JMAA}, show that the system of root vectors of BVP~\eqref{eq:LQ.def.reg}--\eqref{eq:Uy=0} with regular boundary conditions forms a Riesz basis with parentheses under assumptions~\eqref{eq:Bx.def}--\eqref{eq:betak-betaj<-eps}.

As in~\cite[Section 7]{LunMal16JMAA} we start with certain properties of zeros of exponential polynomials. The $n \times n$ case is more difficult and requires more advanced properties. We start with the following simple lemma.
\begin{lemma} \label{lem:f.g.in.S}
Let $f \not\equiv 0$ and $g \not\equiv 0$ be entire functions.

\textbf{(i)} Let $g$ have simple zeros (possibly empty set), i.e.\ $|g(z)|+|g'(z)|>0$, $z \in \bC$. Then $f + w g$ has simple zeros for all but countable number of values of $w \in \bC$.

\textbf{(ii)} Let $f$, $f'$, $g$, $g'$ be sine-type functions with separated (possible empty) sets of zeros. Then for all $w$ outside of a certain strip $\Pi_h$, $h = h_{f,g}$, function $f + w g$ is a sine-type function with separated (possible empty) set of zeros.
\end{lemma}
\begin{proof}
\textbf{(i)} If $f$ and $g$ are proportional, i.e.\ $f \equiv \alp g$ for some $\alp \in \bC$, it is clear that $f + w g$ has simple zeros for all $w \ne -\alp$.
Let $f$ and $g$ be non-proportional. Since $f$ and $g$ are non-zero entire functions, it follows that each of them has at most countable set of zeros. Hence, ratio $f/g$ is a non-constant analytic function with at most countable set of poles. This implies that $(f/g)'$ is non-zero analytic function with at most countable set of poles, which implies that $f g'-f' g$ is non-zero entire function and has at most countable set of zeros $\{\mu_m\}_{m=1}^N$, $0 \le N \le \infty$ (it can be empty as well).

Let $w \in \bC$ be such that some $z = z_w$ is a multiple zero of $f+wg$. Then
$$
f(z)+wg(z)=f'(z)+wg'(z)=0.
$$
This implies that $f(z) g'(z) - f'(z) g(z) = 0$, i.e.\ $z$ is a zero of the entire function $f g'-f' g$. Hence $z = z_w = \mu_m$ for some $m = m_w \in \onetoN$. Since function $g$ has simple zeros it follows that $|g(z)|+|g'(z)|>0$. Hence
$$
\text{either} \quad w = -f(z)/g(z) = -f(\mu_m)/g(\mu_m)
\quad\text{or}\quad w = -f'(z)/g'(z) = -f'(\mu_m)/g'(\mu_m).
$$
This implies that the set of values $w$ for which function $f+wg$ has multiple zeros is at most countable and finishes the proof.

\textbf{(ii)} Since $f$ is a sine-type function it follows from estimate on $f$ from below outside of zeros and estimate on $f'$ near zeros that $|f(z)| + |f'(z)| > \eps$, $z \in \bZ$, for some $\eps > 0$ and the same is valid for $g$. It also follows that zeros of $f$ and $f'$ combined are separated. I.e. zeros of $f f'$ and $g g'$ are separated. These properties are sufficient to show the desired property of function $f+wg$ for sufficiently large $|\Im w|$.
As an example, let $f(z) = e^{i b z} g(z)$ for some $b > 0$. Then
$f(z) + w g(z) = (e^{i b z} + w) g(z)$. So we need to find $w$ for each arithmetic progression
$
\left\{\frac{-i \ln w + 2 \pi m}{b}\right\}_{m \in \bZ}
$
is separated from zeros of $g(\cdot)$. It is clear that this is true for $w$ with sufficiently large $|\Im w|$, which finishes the proof.
\end{proof}
Let $S$ be the set of entire functions with simple zeros, i.e.\ $f \in S$ iff $|f(z)|+|f'(z)|>0$, $z \in \bC$. Denote $\cS := S \cup \{0\}$. The following property is a trivial consequence of Lemma~\ref{lem:f.g.in.S}.
\begin{corollary} \label{cor:f.g.in.S}
Let $f, g \in \cS$. Then $f + w g \in \cS$ for all but countable number of values of $w \in \bC$.
\end{corollary}
\begin{proof}
If either $f$ or $g$ is zero the statement is trivial and $f + w g \in \cS$ for all $w \in \bC$. Otherwise the statement follows from Lemma~\ref{lem:f.g.in.S}.
\end{proof}
\begin{lemma} \label{lem:deltagen.lw.zeros}
Let $(f_P)_{P \in \cP_n}$, be a sequence of $2^n$ functions from the class $\cS$ indexed with diagonal idempotent matrices of size $n$ (see~\eqref{eq:cP.def} for the definition of the class $\cP_n$). Then there exists a diagonal matrix $W = \diag(w_1, \ldots, w_n)$ with non-zero entries such that
\begin{equation} \label{eq:sum.prod.wj.fJ}
 \sum_{P \in \cP_n} \det(P W) \cdot f_P \in \cS.
\end{equation}
For example, for $n=2$, this function looks like $f_{\diag(0,0)} + w_1 f_{\diag(1,0)} + w_2 f_{\diag(0,1)} + w_1 w_2 f_{\diag(1,1)}$.
\end{lemma}
\begin{proof}
Let us prove this via induction by $n$. For $n=1$ the statement trivially follows from Corollary~\ref{cor:f.g.in.S}. Assume the statement is valid for $n = m-1 \in \bN$ and consider $n = m$, i.e.\ we have a sequence $(f_P)_{P \in \cP_m}$ of $2^m$ functions from the class $\cS$. Set $\wt{W} := \diag(w_1, \ldots, w_{m-1})$ and for each $\wt{P} = \diag(p_1, \ldots, p_{m-1}) \in \cP_{m-1}$ denote,
$$
\wt{P}_0 := \diag(p_1, \ldots, p_{m-1}, 0) \in \cP_m, \qquad
\wt{P}_1 := \diag(p_1, \ldots, p_{m-1}, 1) \in \cP_m.
$$
We can transform the sum in~\eqref{eq:sum.prod.wj.fJ} the following way,
\begin{equation} \label{eq:sum.prod.wj.fJ.ind}
 \sum_{P \in \cP_m} \det(P W) \cdot f_P
 = \sum_{\wt{P} \in \cP_{m-1}} \det(\wt{P} \cdot \wt{W}) \cdot
 \(f_{\wt{P}_0} + w_m f_{\wt{P}_1}\)
\end{equation}
By Corollary~\ref{cor:f.g.in.S}, for each $P \in \cP_{m-1}$ the function $f_{\wt{P}_0} + w_m f_{\wt{P}_1} \in \cS$ for all but countable number of values of $w_m \in \bC$. Hence we can choose a single $w_m \ne 0$ that ``serves'' all $P$, i.e.\ $f_{\wt{P}_0} + w_m f_{\wt{P}_1} \in \cS$ for all $P \in \cP_{m-1}$. Now induction hypothesis applied to the r.h.s.\ of~\eqref{eq:sum.prod.wj.fJ.ind} implies existence of the desired non-zero $w_1, \ldots, w_{m-1}$, which finishes the proof.
\end{proof}
\begin{proposition} \label{prop:Delta0.to.strict}
Let $b_1, \ldots, b_n$ given by~\eqref{eq:b1.bn} are non-zero and let boundary conditions~\eqref{eq:Uy=0} be regular. Then there exists $W_{\ell} = \diag(w_1, \ldots, w_n) \in \bC^{n \times n}$ with non-zero entries such that the sequence of zeros of entire function
$$
\Delta_{0,W_{\ell}}(\cdot) := \det(C + D W_{\ell} \Phi_0(\ell,\cdot))
$$
is separated, where $\Phi_0(\cdot,\l)$ is given by~\eqref{eq:Phi0k.def}.
\end{proposition}
\begin{proof}
Throughout the proof we will heavily use notations and formulas from Subsection~\ref{subsec:regular}: $b_{\pm}$, $P \in \cP_n$, $b_P$, $P_{\pm}$, $J_P(C,D)$, etc.
First note that
\begin{equation} \label{eq:JPCDW}
 J_P(C, D W_{\ell}) = \det(P W_{\ell}) \cdot J_P(C,D), \qquad P \in \cP_n.
\end{equation}
The proof is very similar to the proof of~\eqref{eq:JPpmCDW}. It follows from~\eqref{eq:Delta0.sum} and~\eqref{eq:JPCDW} that
\begin{equation} \label{eq:Delta0W.sum}
 \Delta_{0,W_{\ell}}(\l)
 = \sum_{P \in \cP_n} J_P(C,D W_{\ell}) e^{i \l b_P}
 = \sum_{P \in \cP_n} \det(P W_{\ell}) f_P(\l),
 \quad f_P(\l) := J_P(C,D) e^{i \l b_P}, \quad \l \in \bC.
\end{equation}

\textbf{(i)} Let us consider the case when $b_k = m_k b_0$, $k \in \oneton$, for some $b_0 > 0$ and $m_1, \ldots, m_n \in \bZ$. Similar to notations $b_{\pm}$ and $b_P$ we denote
$$
m_{\pm} := b_{\pm} / b_0, \quad\text{e.g.}\quad m_- := m_1 + \ldots + m_{n_-},
$$
and
$$
m_P := b_P / b_0 = \sum_{k=1}^n p_k m_k \quad\text{for}\quad
P = \diag(p_1, \ldots, p_n) \in \cP_n.
$$
Let $W_{\ell}$ be any invertible diagonal matrix. It follows from~\eqref{eq:Delta0.sum} and~\eqref{eq:JPCDW} that
\begin{equation} \label{eq:Delta0W.sum=FW}
 \Delta_{0,W_{\ell}}(\l)
 = \sum_{P \in \cP_n} \det(P W_{\ell}) J_P(C,D) (e^{i \l b_0})^{m_P}
 = (e^{i \l b_0})^{m_-} F_{W_{\ell}}(e^{i \l b_0}),
\end{equation}
where $F_{W_{\ell}}(\cdot)$ is some polynomial (since $m_P \ge m_-$, $P \in \cP_n$, are integers). Since boundary conditions are regular it follows that $J_{P_{\pm}}(C,D) \ne 0$. Hence $F_{W_{\ell}}(0) = J_{P_-}(C,D) \ne 0$ and $\deg F_{W_{\ell}} = m_+ - m_- =: N$. Therefore $F_{W_{\ell}}$ has exactly $N$ non-zero roots $z_1, \ldots, z_N$ (that depend on $W_{\ell}$, $z_k = z_{k,W_{\ell}}$, $k \in \onetoN$). It is clear that the sequence of zeros of $\Delta_{0,W_{\ell}}(\cdot)$ is of the form
$$
\left\{\frac{-i \ln z_k + 2 \pi m}{b_0}\right\}_{m \in \bZ, k \in \onetoN}.
$$
Hence $\Delta_{0,W_{\ell}}(\cdot)$ has separated zeros if and only if roots are distinct $z_1, \ldots, z_N$. This also means that if $\Delta_{0,W_{\ell}}(\cdot)$ has simple zeros then they are separated.

Going back to~\eqref{eq:Delta0W.sum} we see that $f_P \in \cS$, $P \in \cP_n$. Indeed, if $J_P(C,D) = 0$, then $f_P \equiv 0 \in \cS$. Otherwise $f_P$ has no zeros (and thus has simple zeros by definition). Thus, Lemma~\ref{lem:deltagen.lw.zeros} implies existence of invertible $W = W_{\ell} = \diag(w_1, \ldots, w_n)$, for which~\eqref{eq:sum.prod.wj.fJ} holds. In view of~\eqref{eq:Delta0W.sum} this implies that $\Delta_{0,W_{\ell}}(\cdot)$ has simple zeros (since it is non-zero function, due to observations above about polynomial $F_{W_{\ell}}(\cdot)$). Another observation above, implies that $\Delta_{0,W_{\ell}}(\cdot)$ has separated zeros and finishes the proof.

\textbf{(ii)} The proof is similar by using Lemma~\ref{lem:f.g.in.S}. Namely, this lemma implies variation of Lemma~\eqref{lem:deltagen.lw.zeros} for sine-type functions with separated zeros. Using this result we can finish the proof by following steps above.
\end{proof}
To state the next result we recall that $m_a(\mu_0)$ and
$m_g(\mu_0)$ denote the algebraic and geometric multiplicities of
$\mu_0$, respectively. Moreover, if $\mu_0$ is an isolated
eigenvalue, then $m_a(\mu_0)$ equals to the dimension of the
Riesz projection.
We need the following known abstract result (see
e.g.~\cite{SavShk14,LunMal16JMAA}) that follows from Katsnel'son-Markus-Matsaev
theorem with $p=1$ (see~\cite[Theorem 3.1]{Katsn67} and
also~\cite{Agran99},~\cite{MarMats84},~\cite[Theorem
6.12]{Markus88}).
\begin{proposition}[Proposition 7.3 in~\cite{LunMal16JMAA}] \label{prop:Riesz.basis.abstract}
Let $L$ be an operator with compact resolvent in a separable
Hilbert space $\cH$ and let $\{\mu_m\}_{m \in \bZ}$ be the
sequence of its distinct eigenvalues. Assume that $m_a(\mu_m) <
\infty$ for $m \in \bN$ and that $A$ has finitely many
associative vectors, i.e.\ there exists $n_0 \in \bN$ such that
$m_a(\mu_m) = m_g(\mu_m)$ for $|m| > n_0.$ Further, assume that
\begin{equation} \label{eq:ln>cn}
 |\mu_m| \ge C |m|, \quad |\Im \mu_m| \le \tau, \qquad m \in \bZ,
\end{equation}
for some $C, \tau > 0$. Finally, let the system of root vectors
of the operator $L$ forms a Riesz basis in $\fH$. Then for any
bounded operator $T$ in $\fH$ the system of root vectors of the
perturbed operator $A=L+T$ forms a Riesz basis with parentheses
in $\fH$.
\end{proposition}
Finally, we are ready to prove the main result of this subsection concerning Riesz basis property with parentheses.
\begin{theorem} \label{th:riesz.paren}
Let matrix function $B(\cdot)$ given by~\eqref{eq:Bx.def} satisfy conditions~\eqref{eq:betak.alpk.Linf}--\eqref{eq:betak-betaj<-eps} and let $Q \in L^1([0,\ell]; \bC^{n \times n})$. Let boundary conditions~\eqref{eq:Uy=0} be regular
Then any normalized system of root vectors of the operator $L_U(Q)$ forms a Riesz basis with parentheses in $\fH$. Moreover, block sizes are uniformly bounded by $2^n - 1$ and each block consists of root vectors corresponding to the eigenvalues with pairwise close real parts.
\end{theorem}
\begin{proof}
By Lemma~\ref{lem:gauge}(iv), the regularity of boundary conditions is preserved under the gauge transformation used in Lemma~\ref{lem:gauge}. Therefore one can assume that $Q$ is off-diagonal. Now let us consider a perturbation of the operator $L(Q)$ by a constant diagonal potential matrix $Q_0 = \diag(q_1, \ldots, q_n)$, $q_1, \ldots, q_n \in \bC$.
Applying Lemma~\ref{lem:gauge} again, but this time to the operator $L_U(Q+Q_0)$, and noting that $(Q+Q_0)_{\diag} = Q_0$, we see that the operator $L_U(Q+Q_0)$ is similar to the operator $L_{\wt{U}}(\wt{Q})$ with off-diagonal $\wt{Q}$ and with boundary conditions
\begin{equation} \label{eq:wtU.pert}
 \wt{U}(y) = C y(0) + \wt{D} y(\ell)=0,\quad \text{where}\quad
 \wt{D} = D \cdot W_{\ell},
\end{equation}
and
\begin{equation} \label{eq:Well.def}
 W_{\ell} = \diag(w_1, \ldots, w_n),
 \qquad w_k = e^{-i b_k q_k}, \quad k \in \oneton.
\end{equation}

By Proposition~\ref{prop:Delta0.to.strict}, we can choose $w_1, \ldots, w_n \in \bC \setminus \{0\}$ such that the boundary conditions~\eqref{eq:wtU.pert} are strictly regular. In turn, setting, $q_k = i b_k^{-1} \ln w_k$, $k \in \oneton$, provides us with the desired perturbation $Q_0$. Namely, solution $W(\cdot)$ of the equation~\eqref{eq:W'+QW} with $Q_0$ in place of $Q_{\diag}$, satisfies condition $W(\ell) = W_{\ell}$, where $W_{\ell}$ is given by~\eqref{eq:Well.def}.

Therefore, by Theorem~\ref{th:ln=ln0+o} and Definition~\ref{def:strict.regular}(iii), the eigenvalues of $L(Q+Q_0)$ are of finite multiplicity, asymptotically simple and separated. In particular, the operator $L(Q+Q_0)$ has only finitely many associated vectors. Moreover, according to Theorem~\ref{th:Riesz_basis} the root vectors system of the operator $L(Q+Q_0)$ forms a Riesz basis in $\fH$.

To verify that the operator $L(Q+Q_0)$ satisfies conditions of Proposition~\ref{prop:Riesz.basis.abstract} it suffices to note that inequalities~\eqref{eq:ln>cn} are implied by Proposition~\ref{prop:sine.type}(iv). Thus, the operator $L(Q+Q_0)$ meets the conditions of Proposition~\ref{prop:Riesz.basis.abstract}, and hence the root vectors system of the original operator $L(Q) = L(Q+Q_0) - Q_0$ forms a Riesz basis with parentheses, since operator $T : y \to Q_0 y$ is clearly bounded in $\fH$.

Further, note that since $\Delta_0(\cdot)$ is exponential polynomial with at most $2^n$ terms, we can find a number $w > 0$ such that any rectangle $[x, x + w] \times [-h, h]$, $x \in \bR$, has no more than $2^n-1$ zeros of  $\Delta_0(\cdot)$ (see~\cite[pp. 95--97]{BirLan23} where this property of exponential polynomial was proved for $n=2$). The same is valid for $\Delta_Q(\cdot)$ due to asymptotic formula $\l_m = \l_m^0 + o(1)$ as $m \to \infty$. Recall that this formula is valid in the case of any regular boundary conditions. Hence we can separate sequence $\{\l_m\}_{m \in \bZ}$ (ordered in increasing order of real parts) into blocks
$$
\{\l_m\}_{m=m_{k}}^{m_{k+1}-1}, \quad\text{where}\quad m_{k} < m_{k+1} < m_{k} + 2^n \quad\text{and}\quad
\Re \l_{m_k - 1} + \eps \le \Re \l_{m_k}, \qquad k \in \bZ,
$$
with $\eps := \frac{w}{\max\{2^n-2,1\}}$. From the proof of Katsnel'son-Markus-Matsaev theorem with $p=1$ (see~\cite[Theorem 3.1]{Katsn67} and also~\cite{Agran99},~\cite{MarMats84},~\cite[Theorem 6.12]{Markus88}), it follows that each block consists of root vectors corresponding to the eigenvalues with pairwise close real parts, which finishes the proof.
\end{proof}
Under certain additional assumptions on boundary conditions we can clarify the sizes of blocks in Theorem~\ref{th:riesz.paren} even more. We restrict ourselves to a special type of boundary conditions including periodic and antiperiodic conditions.
\begin{corollary} \label{cor_sizes_of_Riesz_blocks}
Assume the conditions of Theorem~\ref{th:riesz.paren} and let $b_1 < \ldots < b_n$. Assume also that boundary conditions~\eqref{eq:Uy=0} are of the form $C y(0) - y(\ell) = 0$, where $C = \diag(c_1, \ldots, c_n)$ is invertible. Then the blocks of Riesz basis decomposition with parentheses stated in
Theorem~\ref{th:riesz.paren} have sizes at most $n$ and correspond to eigenvalues of the operator $L_U(Q)$ with mutually close real parts. In particular, this is true for periodic and antiperiodic boundary conditions.
\end{corollary}
\begin{proof}
Applying the gauge transform from Lemma~\ref{lem:gauge} we can assume that $Q$ is off-diagonal matrix, $Q_{11} \equiv \ldots \equiv Q_{nn} \equiv 0$. Since $b_1 < \ldots < b_n$ it is clear that the new matrix $D$ is of the same diagonal form as $C$. Hence multiplying boundary conditions by $-D^{-1}$ from the left we can assume that again $D = -I_n$. Let $\L := \{\l_n\}_{n \in \bZ}$ be the sequence of eigenvalues of $A := L_U(Q)$ counting multiplicity, ordered in such a way that $\Re \l_n \le \Re \l_{n+1}$, $n \in \bZ$. Since $Q$ is off-diagonal, one derives by combining Theorem~\ref{th:ln=ln0+o} with Lemma~\ref{lem:periodic.strict}(i) that there exists $\eps > 0$ and a sequence of integers $\{m_k\}_{k \in \bZ}$ such that
\begin{equation} \label{eq:Lambda=union}
 m_k < m_{k+1} \le m_k + n, \quad \Re \l_{m_k} - \Re \l_{m_k-1} \ge \eps, \quad k \in \bZ.
\end{equation}

Let $N(t)$ be the number of eigenvalues (counting multiplicity)
of the operator $A$ belonging to the vertical strip $\{\l :
|\Re \l| \le t\}$. Setting $t_k := 2^{-1}\Re(\l_{m_k-1} +
\l_{m_k})$ we obtain from~\eqref{eq:Lambda=union} that there
exists $C > 0$ such that
\begin{equation} \label{eq:sup.nt-ntk}
 \sup_{t >0, t \ne t_k} \frac{N(t) - N(t_k)}{t-t_k} < C, \quad k \in \bZ.
\end{equation}
Let $f_n$ be the root vector of $A$ corresponding to $\l_n$, $n
 \in \bZ$. Then it follows from the proof of Theorem 3.1
in~\cite{Katsn67} (see also Theorem 1.3 in~\cite{Katsn67}) that
under the condition~\eqref{eq:sup.nt-ntk} the subspaces $\fH_k
:= \Span\{f_j\}_{j=m_k}^{m_{k+1}-1}$ constitute a Riesz basis
of subspaces in $\fH$. To complete the proof it
remains to note that due to~\eqref{eq:Lambda=union} the sizes
$\dim \fH_k$ of blocks $\fH_k$ do not exceed $n$. Indeed, $\dim \fH_k =
m_{k+1} - m_k \le n$, $k \in \bZ$.
\end{proof}
\begin{remark} \label{rem:block.size}
\textbf{(i)} Corollary~\ref{cor_sizes_of_Riesz_blocks} is valid under the following slightly more general assumption,
\begin{equation}
 \forall j,k \in \oneton: \quad
 j \ne k \quad\text{and}\quad b_j = b_k \quad\Rightarrow\quad Q_{jk} \equiv 0.
\end{equation}
This is evident from the fact that for such $Q$ its block diagonal $Q_{\diag}$ with respect to the decomposition $\bC^n = \bC^{n_1} \oplus \ldots \oplus \bC^{n_r}$ is actually a regular diagonal matrix, $Q_{\diag} = \diag(Q_{11}, \ldots, Q_{nn})$. Hence, the solution $W(\cdot)$ of the Cauchy problem~\eqref{eq:W'+QW} is a diagonal matrix function. This in turn imply, that the new matrix $D = -W(\ell)$ after applying gauge transform is of the same diagonal form as $C$, which allows to finish the proof the same way.

\textbf{(ii)} Result similar to Corollary~\ref{cor_sizes_of_Riesz_blocks} is also valid for boundary conditions of the form~\eqref{eq:separ.cond} if numbers $b_1, \ldots, b_n$ satisfy condition~\eqref{eq:b.2k-1<0<b.2k}. We just need to apply Lemma~\ref{lem:separ.regular}(i) instead of Lemma~\ref{lem:periodic.strict}(i). In fact, the blocks of Riesz basis decomposition with parentheses stated in Theorem~\ref{th:riesz.paren} have sizes at most $n/2$ in this case.

\textbf{(iii)} More generally, if the sequence $\L_0$ is a union of $N$ arithmetic progressions that lie on the lines parallel to the real axis, then the blocks of Riesz basis decomposition with parentheses stated in Theorem~\ref{th:riesz.paren} have sizes at most $N$.
\end{remark}
\section{Application to the Timoshenko beam model}
\label{sec:Timoshenko}
\subsection{Problem statement and reduction to Dirac type operator}
\label{subsec:Tim.general}
In this section we obtain some important geometric properties of the system
of root vectors of the dynamic generator of the Timoshenko
beam model. Consider the following linear system of two coupled
hyperbolic equations for $t \ge 0$
\begin{align}
 \label{eq:Tim.Ftt}
 I_{\rho}(x) \Phi_{tt} &= K(x)(W_x-\Phi) + (EI(x) \Phi_x)_x - p_1(x) \Phi_t, \quad x \in [0, \ell],\\
 \label{eq:Tim.Wtt}
 \rho(x) W_{tt} &= (K(x)(W_x-\Phi))_x - p_2(x) W_t, \qquad \qquad \qquad x \in [0, \ell].
\end{align}
The vibration of the Timoshenko beam of the length $\ell$
clamped at the left end is governed by the system
\eqref{eq:Tim.Ftt}--\eqref{eq:Tim.Wtt} subject to the following
boundary conditions for $t \ge 0$~\cite{Tim55}:
\begin{align}
 \label{eq:Tim.W0F0}
 W(0,t) = \Phi(0,t) &= 0, \\
 \label{eq:Tim.WLFLa1}
 \bigl(EI(x) \Phi_x(x,t) + \alp_1 \Phi_t(x,t)
 + \gam_1 W_t(x,t)\bigr)\bigr|_{x=\ell} &= 0, \\
 \label{eq:Tim.WLFLa2}
 \bigl(K(x)(W_x(x,t)-\Phi(x,t)) + \alp_2 W_t(x,t) + \gam_2 \Phi_t(x,t)\bigr)\bigr|_{x=\ell} &= 0.
\end{align}
Here $W(x,t)$ is the lateral displacement at a point $x$ and
time $t$, $\Phi(x,t)$ is the bending angle at a point $x$ and
time $t$, $\rho(x)$ is a mass density, $K(x)$ is the shear
stiffness of a uniform cross-section, $I_{\rho}(x)$ is the
rotary inertia, $EI(x)$ is the flexural rigidity at a point $x$,
$p_1(x)$ and $p_2(x)$ are locally distributed feedback
functions, $\alp_k, \gam_k \in \bC$, $k \in \{1,2\}$.
Boundary conditions at the right end contain as partial cases
most of the known boundary conditions if $\alp_1, \alp_2$
are allowed to be infinity.

Regarding the coefficients, we assume that they all are measurable functions satisfying the following general conditions for some $M > 1$,
\begin{equation} \label{eq:Tim.coef.cond}
 0 < M^{-1} \le \rho(x), I_{\rho}(x), K(x), EI(x) \le M,
 \quad x \in [0,\ell], \qquad p_1, p_2 \in L^1([0,\ell];\bC).
\end{equation}

Under these assumptions, the energy space associated with the
problem~\eqref{eq:Tim.Ftt}--\eqref{eq:Tim.WLFLa2} is
\begin{equation} \label{eq:cH.def}
 \cH := \wt{H}^1_0[0,\ell] \times L^2[0,\ell] \times \wt{H}^1_0[0,\ell] \times L^2[0,\ell],
\end{equation}
where $\wt{H}^1_0[0,\ell] := \{f \in W^{1,2}[0,\ell] :
f(0)=0\}$. The norm in the energy space is defined as follows:
\begin{equation} \label{eq:Tim.|y|H}
 \|y\|_{\cH}^2 = \int_0^\ell \bigl(EI|y_1'|^2+I_{\rho}|y_2|^2 + K|y_3'-y_1|^2+\rho|y_4|^2\bigr)dx,
 \quad y =\col(y_1,y_2,y_3,y_4).
\end{equation}
The problem~\eqref{eq:Tim.Ftt}--\eqref{eq:Tim.WLFLa2} can be
rewritten as
\begin{equation} \label{eq:Tim.yt=i.cLy}
 y_t = i \cL y, \quad y(x,t)|_{t=0} = y_0(x),
\end{equation}
where $y$ and $\cL$ are given by
\begin{equation} \label{eq:Tim.Ly.def}
 y = \begin{pmatrix} \Phi(x,t) \\ \Phi_t(x,t) \\ W(x,t) \\ W_t(x,t) \end{pmatrix}, \ \
 \cL \begin{pmatrix} y_1 \\ y_2 \\ y_3 \\ y_4 \end{pmatrix} = \frac{1}{i} \begin{pmatrix}
 y_2 \\ \frac{1}{I_{\rho}(x)}\Bigl(K(x)(y_3'-y_1) + \bigl(EI(x) y_1'\bigr)' - p_1(x) y_2\Bigr) \\
 y_4 \\ \frac{1}{\rho(x)}\Bigl(\bigl(K(x)(y_3'-y_1)\bigl)' - p_2(x) y_4\Bigr)
 \end{pmatrix}
\end{equation}
on the domain
\begin{align} \label{eq:Tim.dom.cL}
 && \dom(\cL) = \left\{ y = \col(y_1,y_2,y_3,y_4) :
 \ \ y_1, y_2, y_3, y_4 \in \wt{H}^1_0[0,\ell]\right., \nonumber \\
 && EI \cdot y_1' \in \AC[0,\ell], \qquad (EI\cdot y_1')' - p_1 y_2 \in L^2[0,\ell], \nonumber \\
 && K\cdot(y_3'-y_1) \in \AC[0,\ell], \qquad (K\cdot(y_3'-y_1))' - p_2 y_4 \in L^2[0,\ell], \nonumber \\
 && \bigl(EI \cdot y_1'\bigr)(\ell) + \alp_1 y_2(\ell) + \gam_1 y_4(\ell)= 0, \nonumber \\
 && \Bigl.\bigl(K \cdot (y_3'-y_1)\bigr)(\ell) + \alp_2 y_4(\ell) + \gam_2 y_2(\ell)= 0 \Bigr\}.
\end{align}

Timoshenko beam model is investigated in numerous papers (see~\cite{Tim55,KimRen87,MenZua00,Shub02,XuYung04,XuHanYung07,WuXue11,Shub11,LunMal15JST,LunMal16JMAA,Akian22} and the references therein). A number of stability, controllability, and optimization problems were studied. Note also that the general model~\eqref{eq:Tim.Ftt}--\eqref{eq:Tim.WLFLa2} of spatially non-homogenous Timoshenko beam with both boundary and locally distributed damping covers the cases studied by many authors. Geometric properties of the system of root vectors of the operator $\cL$ play important role in investigation of different properties of the problem~\eqref{eq:Tim.Ftt}--\eqref{eq:Tim.WLFLa2}.

Below we continue our investigation started in~\cite{LunMal15JST,LunMal16JMAA}, where we established completeness and the Riesz basis property with parentheses of the root vectors system of the operator $\cL$, without analyzing its spectrum.
In our previous papers we imposed the following additional algebraic assumption on $\cL$: the ratio of wave speeds $\frac{K(\cdot)}{\rho(\cdot)}$ and $\frac{EI(\cdot)}{I_\rho(\cdot)}$ is constant.
This assumption has to be added since in~\cite{LunMal15JST,LunMal16JMAA} we treated BVP~\eqref{eq:LQ.def.intro}--\eqref{eq:Uy=0.intro} with non-trivial potential matrix $Q$ and \emph{constant matrix} $B(x) = B$. Below we will establish completeness and the Riesz basis property with parentheses of the root vectors system of the operator $\cL$ \emph{without this algebraic assumption} and additionally establish asymptotic behavior of its eigenvalues.
Moreover, under additional assumptions ensuring that the eigenvalues of the operator $\cL$ are asymptotically separated, we will show that the system of root vectors of the operator $\cL$ forms a Riesz basis in $\cH$ (\emph{without parentheses}) and establish asymptotic behavior of the eigenvectors. Riesz basis property is essential for obtaining numerous stability and controllability properties.

As in our previous papers~\cite{LunMal15JST,LunMal16JMAA}, our approach to the spectral properties of the operator $\cL$ is based on the similarity reduction of $\cL$ to a special $4\times 4$ Dirac-type operator $L_U(Q)$ associated with appropriate BVP~\eqref{eq:LQ.def.intro}--\eqref{eq:Uy=0.intro}. To state the result we need some additional preparations. Let
\begin{align}
\label{eq:Tim.B}
 B(x) &:= \diag(-\beta_1(x), \beta_1(x), -\beta_2(x), \beta_2(x)),
 \qquad\text{where} \\
\label{eq:Tim.beta1.beta2}
 \beta_1(x) &:= \sqrt{\frac{I_{\rho}(x)}{EI(x)}},
 \qquad \beta_2(x) := \sqrt{\frac{\rho(x)}{K(x)}}, \qquad x \in [0,\ell].
\end{align}

Recall that with the matrix function $B(\cdot)$ one associates weighted vector $L^2$-space $\fH$ via formulas~\eqref{eq:fH.def}--\eqref{eq:fHk.def}. In the case of matrix function $B(\cdot)$ given by~\eqref{eq:Tim.B}, it takes the following form,
\begin{equation} \label{eq:Tim.fH}
 \fH = L^2_{\beta_1}[0,\ell] \oplus L^2_{\beta_1}[0,\ell]
 \oplus L^2_{\beta_2}[0,\ell] \oplus L^2_{\beta_2}[0,\ell].
\end{equation}
It follows from condition~\eqref{eq:Tim.coef.cond} that identity operator from $\fH$ to $L^2([0,\ell];\bC^4)$ is bounded and has a bounded inverse, i.e.\ the Hilbert spaces $\fH$ and $L^2([0,\ell];\bC^4)$ coincide algebraically and topologically.

Further, we set,
\begin{align}
\label{eq:Tim.Theta(x)}
 \Theta(x) &:= 2 \diag(h_1(x), h_1(x), h_2(x), h_2(x)),
 \qquad\text{where} \\
\label{eq:Tim.g1.g2.def}
 h_1(x) &:= \sqrt{EI(x) I_{\rho}(x)}, \qquad h_2(x):=\sqrt{K(x) \rho(x)},
 \qquad x \in [0,\ell].
\end{align}
In the sequel we assume that
\begin{equation} \label{eq:Tim.h1,h2.in.AC}
 h_1, h_2 \in \AC[0,\ell].
\end{equation}
It follows from~\eqref{eq:Tim.coef.cond}, definition~\eqref{eq:Tim.beta1.beta2} of $\beta_1$, $\beta_2$, and definition~\eqref{eq:Tim.g1.g2.def} of $h_1$, $h_2$ that
\begin{equation} \label{eq:Tim.beta.h}
 0 < M^{-1} \le \beta_1(x), \beta_2(x), h_1(x), h_2(x) \le M,
 \qquad x \in [0,\ell],
\end{equation}
with the same $M$ as in conditions~\eqref{eq:Tim.coef.cond}.

Under assumptions~\eqref{eq:Tim.coef.cond} and~\eqref{eq:Tim.h1,h2.in.AC} the
following matrix function $Q(\cdot)$ is well-defined and summable,
\begin{equation}
 \label{eq:Tim.Q(x)}
 Q(x) := \Theta^{-1}(x)
 \begin{pmatrix}
 p_1+h_1' & p_1-h_1' & h_2 & -h_2 \\
 p_1+h_1' & p_1-h_1' & h_2 & -h_2 \\
 -h_2 & -h_2 & p_2+h_2' & p_2-h_2' \\
 h_2 & h_2 & p_2+h_2' & p_2-h_2'
 \end{pmatrix}, \qquad x \in [0,\ell].
\end{equation}
Finally, let
\begin{equation} \label{eq:Tim.C.D}
 C = \begin{pmatrix}
 1 & 1 & 0 & 0 \\
 0 & 0 & 0 & 0 \\
 0 & 0 & 1 & 1 \\
 0 & 0 & 0 & 0
 \end{pmatrix}, \qquad
 D = \begin{pmatrix}
 0 & 0 & 0 & 0 \\
 \alp_1 - h_1(\ell) & \alp_1 + h_1(\ell) & \gam_1 & \gam_1 \\
 0 & 0 & 0 & 0 \\
 \gam_2 & \gam_2 & \alp_2 - h_2(\ell) & \alp_2 + h_2(\ell) \\
 \end{pmatrix}.
\end{equation}
\begin{proposition}[cf. Proposition 6.1 in~\cite{LunMal15JST}] \label{prop:Tim.similar}
Let measurable functions $\rho, I_{\rho}, K, EI, p_1, p_2, h_1, h_2$ satisfy conditions~\eqref{eq:Tim.coef.cond} and~\eqref{eq:Tim.h1,h2.in.AC}. Then the operator $\cL$ acting in the Hilbert space $\cH$ is similar to the $4 \times 4$ Dirac-type operator $L_U(Q)$ acting in the Hilbert space $\fH$ given by~\eqref{eq:Tim.fH}, where matrices $B(\cdot),C,D,Q(\cdot)$ are given by~\eqref{eq:Tim.B},~\eqref{eq:Tim.C.D} and~\eqref{eq:Tim.Q(x)}.
\end{proposition}
\begin{proof}
Introduce the following operator
\begin{equation} \label{eq:cU.def}
 \cU y = \col(EI(x)y_1',\ y_2,\ K(x)(y_3'-y_1),\ y_4), \qquad y = \col(y_1,y_2,y_3,y_4),
\end{equation}
that maps the Hilbert space $\cH$ given by~\eqref{eq:cH.def}
into $L^2([0,\ell]; \bC^4)$. Since $\frac{d}{dx}$ isometrically
maps $\wt{H}_0^1[0,\ell] = \{f \in W^{1,2}[0,\ell] : f(0)=0\}$
onto $L^2[0,\ell]$, it follows from
condition~\eqref{eq:Tim.coef.cond} that the operator $\cU$ is
bounded with a bounded inverse. It is easy to check that for
$y=\col(y_1,y_2,y_3,y_4)$
\begin{equation} \label{eq:Tim.cLU-1}
 \cL \, \cU^{-1} y =
 \frac{1}{i}\begin{pmatrix} y_2 \\ \frac{1}{I_{\rho}}(y_1'-p_1y_2+y_3) \\
 y_4 \\ \frac{1}{\rho}(y_3' - p_2y_4)\end{pmatrix}, \ \
 \wt{\cL} y := \cU \cL \, \cU^{-1} y = \frac{1}{i}\begin{pmatrix} EI \cdot y_2' \\ \frac{1}{I_{\rho}}(y_1'-p_1y_2+y_3) \\
 K \cdot (y_4'-y_2) \\ \frac{1}{\rho}(y_3' - p_2 y_4)\end{pmatrix},
\end{equation}
and
\begin{multline} \label{eq:Tim.dom.wtL}
 \dom(\wt{\cL}) = \cU \dom(\cL) = \bigl\{\bigr.y = \col(y_1,y_2,y_3,y_4)
 \in \AC([0,\ell]; \bC^4) : \ \ \wt{\cL}y \in L^2([0,\ell]; \bC^4), \\
 y_2(0) = y_4(0) = 0, \quad
 y_1(\ell) + \alp_1 y_2(\ell) + \gam_1 y_4(\ell) = 0, \quad
 y_3(\ell) + \alp_2 y_4(\ell) + \gam_2 y_2(\ell) = 0 \bigl.\bigr\}.
\end{multline}
Thus, the operator $\cL$ is similar to the operator $\wt{\cL}$,
\begin{equation}
 \wt{\cL}y = -i\wt{B}(x)y'+\wt{Q}(x)y, \qquad y \in \dom(\wt{\cL}),
\end{equation}
with the domain $\dom(\wt{\cL})$ given by~\eqref{eq:Tim.dom.wtL},
and the matrix functions $\wt{B}(\cdot)$, $\wt{Q}(\cdot)$, given
by
\begin{equation}
 \wt{B}(x) := \begin{pmatrix}
 0 & EI(x) & 0 & 0 \\
 \frac{1}{I_{\rho}(x)} & 0 & 0 & 0 \\
 0 & 0 & 0 & K(x) \\
 0 & 0 & \frac{1}{\rho(x)} & 0
 \end{pmatrix}, \qquad
 \wt{Q}(x) := i \begin{pmatrix}
 0 & 0 & 0 & 0 \\
 0 & \frac{p_1(x)}{I_{\rho}(x)} & -\frac{1}{I_{\rho}(x)} & 0 \\
 0 & K(x) & 0 & 0 \\
 0 & 0 & 0 & \frac{p_2(x)}{\rho(x)}
 \end{pmatrix}, \quad x \in [0,\ell].
\end{equation}
Note, that $\wt{Q} \in L^1([0,\ell]; \bC^{4 \times 4})$ in view of condition~\eqref{eq:Tim.coef.cond}. Next we diagonalize the matrix $\wt{B}(\cdot)$. Namely, setting
\begin{equation} \label{eq:wtcU.def}
 \wt{\cU} : y \to \wt{U}(x) y, \qquad
 \wt{U}(x) := \begin{pmatrix}
 -h_1(x) & h_1(x) & 0 & 0 \\
 1 & 1 & 0 & 0 \\
 0 & 0 & -h_2(x) & h_2(x) \\
 0 & 0 & 1 & 1
 \end{pmatrix},
\end{equation}
and noting that
\begin{equation}
 \wt{U}^{-1}(x) = \frac12 \begin{pmatrix}
 -\frac{1}{h_1(x)} & 1 & 0 & 0 \\
 \frac{1}{h_1(x)} & 1 & 0 & 0 \\
 0 & 0 & -\frac{1}{h_2(x)} & 1 \\
 0 & 0 & \frac{1}{h_2(x)} & 1
 \end{pmatrix},
\end{equation}
we easily get after straightforward calculations
\begin{equation} \label{eq:Tim.W1BW}
 \wt{U}^{-1}(x) \wt{B}(x) \wt{U}(x) = \diag\left(-\sqrt{\frac{EI(x)}{I_{\rho}(x)}},\sqrt{\frac{EI(x)}{I_{\rho}(x)}},
 -\sqrt{\frac{K(x)}{\rho(x)}},\sqrt{\frac{K(x)}{\rho(x)}}\right) = B(x)^{-1},
\end{equation}
Here we have used definition~\eqref{eq:Tim.g1.g2.def} of $h_1$,
$h_2$, and definition~\eqref{eq:Tim.beta1.beta2} of $\beta_1$, $\beta_2$.

Further, note that $\wt{U} \in \AC([0,\ell]; \bC^{4\times 4})$ since $h_1, h_2 \in \AC[0,\ell]$. Moreover, as noted earlier, $Q \in L^1([0,\ell]; \bC^{4 \times 4})$ under assumptions~\eqref{eq:Tim.coef.cond} and~\eqref{eq:Tim.h1,h2.in.AC}, where $Q(\cdot)$ is given by~\eqref{eq:Tim.Q(x)} and~\eqref{eq:Tim.Theta(x)}. Hence,
it is easily seen that
\begin{equation} \label{eq:Tim.W1QW}
 \wt{U}^{-1} \wt{Q} \wt{U} - i \wt{U}^{-1} \wt{B} \wt{U}'
 = -i B^{-1} Q.
\end{equation}
We consider the operator $\wt{\cU} : y \to \wt{U}(x) y$ acting from $\fH$ given by~\eqref {eq:Tim.fH} to $L^2([0,\ell]; \bC^4)$. It is clear from condition~\eqref{eq:Tim.coef.cond}, that $\wt{\cU}$ is bounded with a bounded inverse. Taking into account~\eqref{eq:Tim.W1BW} and~\eqref{eq:Tim.W1QW} we obtain for any $y \in \AC([0,\ell]; \bC^4)$ and satisfying
$\wt{U} y \in \dom(\wt{\cL})$ that
\begin{equation} \label{eq:Tim.wtLy}
 \wh{\cL} y := \wt{\cU}^{-1} \wt{\cL} \ \wt{\cU} y = -i B(x)^{-1} (y' + Q(x) y).
\end{equation}
Next, taking into account formulas~\eqref{eq:Tim.C.D} and
\eqref{eq:wtcU.def} for matrices $C$, $D$, and $\wt{U}(\cdot)$,
respectively, we derive that
\begin{equation} \label{eq:Tim.dom.whL}
 \dom(\wh{\cL}) = \{y \in \AC([0,\ell]; \bC^4) :
 \wh{\cL}y \in \fH,\ Cy(0)+Dy(\ell) = 0 \}.
\end{equation}
This directly implies that $L_U(Q) = \wh{\cL}$. Combining this identity with~\eqref{eq:Tim.cLU-1}, one concludes that $\cL$ is similar to
$L_U(Q)$.
\end{proof}
\subsection{Completeness and Riesz basis property with parentheses}
\label{subsec:Tim.riesz}
Applying Theorems~\ref{th:compl} and~\ref{th:riesz.paren} to the operator $L_U(Q)$ constructed in Proposition~\ref{prop:Tim.similar} we obtain the following result.
\begin{theorem} \label{th:Tim.compl.paren}
Let measurable functions $\rho, I_{\rho}, K, EI, p_1, p_2, h_1, h_2$ satisfy conditions~\eqref{eq:Tim.coef.cond} and~\eqref{eq:Tim.h1,h2.in.AC}. Set
\begin{equation} \label{eq:nux.def}
 \nu(x) := \frac{K(x)}{\rho(x)} - \frac{E(x)}{I_{\rho}(x)},
 \qquad x \in [0,\ell].
\end{equation}
Let function $\nu(\cdot)$ satisfy the following condition for some $\eps > 0$,
\begin{equation} \label{eq:nu=><0}
 \text{either} \quad \nu \equiv 0, \qquad\text{or}\qquad
 \nu(x) > \eps, \quad x \in [0,\ell], \qquad\text{or}\qquad
 \nu(x) < -\eps, \quad x \in [0,\ell].
\end{equation}
Let also
\begin{equation} \label{eq:Tim.a1!=h1,a1!=h2}
 (\alp_1 + h_1(\ell)) (\alp_2 + h_2(\ell)) \ne \gam_1 \gam_2 \quad\text{and}\quad
 (\alp_1 - h_1(\ell)) (\alp_2 - h_2(\ell)) \ne \gam_1 \gam_2.
\end{equation}

Then the system of root vectors of $\cL$ is complete, minimal, and forms a Riesz basis with parentheses in the Hilbert space $\cH$.
\end{theorem}
\begin{proof}
Consider the operator $L_U(Q)$ defined in Proposition~\ref{prop:Tim.similar}. It is clear from the form~\eqref{eq:Tim.B} of the matrix function $B(\cdot)$ that
$$
P_-= \diag(1,0,1,0) \quad\text{and}\quad P_+ = \diag(0,1,0,1),
$$
where ``projectors'' $P_{\pm}$ are given by~\eqref{eq:P.pm.def}. Combining this with expression~\eqref{eq:Tim.C.D} for the matrices $C$ and $D$ and with definition~\eqref{eq:TP.CD.def} of $J_P(C,D)$ yields
\begin{equation} \label{eq:Tim.TBCD}
 J_{P_+}(C,D) = \det\begin{pmatrix}
 1 & 0 & 0 & 0 \\
 0 & \alp_1+h_1(\ell) & 0 & \gam_1 \\
 0 & 0 & 1 & 0 \\
 0 & \gam_2 & 0 & \alp_2+h_2(\ell)
 \end{pmatrix} = (\alp_1 + h_1(\ell))(\alp_2 + h_2(\ell)) - \gam_1 \gam_2.
\end{equation}
Similarly one gets
$$
J_{P_-}(C,D) = (\alp_1-h_1(\ell)) (\alp_2-h_2(\ell)) - \gam_1 \gam_2.
$$
Condition~\eqref{eq:Tim.a1!=h1,a1!=h2} implies that $J_{\pm}(C,D) \ne 0$. Hence, boundary conditions $U(y)=Cy(0)+Dy(\ell)=0$ are regular.

It is clear from condition~\eqref{eq:Tim.coef.cond}, definition~\eqref{eq:Tim.beta1.beta2} of $\beta_1(\cdot)$, $\beta_2(\cdot)$ and definition~\eqref{eq:nux.def} of $\nu(\cdot)$ that
$$
 |\beta_1(x) - \beta_2(x)| = \frac{|\nu(x)|}{\beta_1(x) + \beta_2(x)}
 \in \left[\frac{1}{2M} |\nu(x)|, \ \frac{M}{2} |\nu(x)|\right],
 \qquad x \in [0,\ell].
$$
Hence, condition~\eqref{eq:nu=><0} implies that either $\beta_1 \equiv \beta_2$ or $\beta_1 - \beta_2$ does not change sign on $[0,\ell]$ and is separated from zero. Clearly this is valid for all other entries of the matrix function $B(\cdot)$. Hence it satisfies conditions~\eqref{eq:betak.alpk.Linf},~\eqref{eq:theta.exists}--\eqref{eq:nujk=><0} (note, that notations
$\beta_1$ and $\beta_2$ differ in these conditions). Moreover, conditons~\eqref{eq:Tim.coef.cond} and~\eqref{eq:Tim.h1,h2.in.AC} imply that $Q \in L^1([0,\ell]; \bC^{4 \times 4})$.
Therefore, Remark~\ref{rem:beta.vs.wtbeta}(ii), Theorem~\ref{th:compl}(ii) and Theorem~\ref{th:riesz.paren} imply that the system of root vectors of the operator $L_U(Q)$ is complete, minimal, and forms a Riesz basis with parentheses in the Hilbert space $\fH$. By Proposition~\ref{prop:Tim.similar}, the operator $\cL$ is similar to the operator $L_U(Q)$. Hence the system of root vectors of the operator $\cL$ has the same properties, which finishes the proof.
\end{proof}
\begin{remark}
Theorem~\ref{th:Tim.compl.paren} improves similar result from our previous paper~\cite{LunMal15JST} (Theorem 6.3) in several ways:

\textbf{(i)} Main improvement is replacing condition of wave speeds $\frac{K(\cdot)}{\rho(\cdot)}$ and $\frac{EI(\cdot)}{I_\rho(\cdot)}$ being proportional functions, with much more general condition~\eqref{eq:nu=><0} on the difference of these wave speeds.

\textbf{(ii)} To establish Riesz basis property with parentheses, in~\cite[Theorem 6.3]{LunMal15JST} we assumed ``smoothness'' conditions $p_1, p_2 \in L^\infty[0,\ell]$, $h_1, h_2 \in \Lip_1[0,\ell]$ and were only able to handle simpler case of boundary conditions when $\gam_1 = \gam_2 = 0$. Here we handle the most general boundary conditions under the most general conditions on functions $p_1, p_2, h_1, h_2$.

\textbf{(iii)} Finally, we replaced condition $\rho, I_{\rho}, K, EI \in C[0,\ell]$ with more general condition $\rho, I_{\rho}, K, EI \in L^{\infty}[0,\ell]$, allowing parameters of the model to have discontinuities. In fact, considerations in~\cite{LunMal15JST} work under these conditions as well if we note that $f \in \AC$, $g, g^{-1} \in \Lip_1$ implies that $f \circ g \in \AC$ (here $g^{-1}$ is the inverse of monotonous function $g$).
\end{remark}
If $\gam_1 = \gam_2 = 0$ we can significantly improve Theorem~\ref{th:Tim.compl.paren} by dropping cumbersome condition~\eqref{eq:nu=><0} on separation of wave speeds $\frac{K(\cdot)}{\rho(\cdot)}$ and $\frac{EI(\cdot)}{I_\rho(\cdot)}$.
\begin{theorem} \label{th:Tim.gam12=0}
Let measurable functions $\rho, I_{\rho}, K, EI, p_1, p_2, h_1, h_2$ satisfy conditions~\eqref{eq:Tim.coef.cond} and~\eqref{eq:Tim.h1,h2.in.AC}. Let also
\begin{equation} \label{eq:Tim.ak.ne.pm.hk}
 \alp_1 \ne \pm h_1(\ell), \qquad \alp_2 \ne \pm h_2(\ell),
 \qquad \gam_1 = \gam_2 = 0.
\end{equation}

Then the system of root vectors of $\cL$ is complete, minimal, and forms a Riesz basis with parentheses in the Hilbert space $\cH$.
\end{theorem}
\begin{proof}
Consider the operator $L_U(Q)$ defined in Proposition~\ref{prop:Tim.similar}. Since $\gam_1 = \gam_2 = 0$ we can represent it as bounded perturbation of the direct sum of two $2 \times 2$ Dirac type operators:
\begin{align}
\label{eq:LCDQ=Lbc1.Lbc2+wtQ}
 L_U(Q) &= L_1 \oplus L_2 + \wt{\cQ},
 \quad\text{where for}\quad k \in \{1,2\},\\
\label{eq:Lbck.Qk}
 (L_k y)(x) &:= -i B_k^{-1}(x) \bigl(y'(x) + Q_k(x) y(x)\bigr),
 \quad x \in [0,\ell], \quad y = \col(y_1, y_2) \in \dom L_k, \\
\nonumber
 \dom L_k & := \{y \in \AC([0,\ell]; \bC^2):
 L_k y \in L^2_{|\beta_k|}([0,\ell]; \bC^2), \\
\label{eq:bck}
 & \qquad y_1(0) + y_2(0) = (\alp_k - h_k(\ell)) y_1(\ell) +
 (\alp_k + h_k(\ell)) y_2(\ell) = 0\}, \\
\label{eq:Q1.Q2}
 B_k &:= \begin{pmatrix} -\beta_k & 0 \\ 0 & \beta_k \end{pmatrix}, \qquad
 Q_k := \frac{1}{2 h_k} \begin{pmatrix}
 p_k+h_k' & p_k-h_k' \\
 p_k+h_k' & p_k-h_k' \\
 \end{pmatrix}, \\
\label{eq:wtQt}
 (\wt{\cQ} y)(x) &= \wt{Q}(x) y(x), \qquad
 \wt{Q} = \Theta^{-1} \begin{pmatrix}
 0 & 0 & h_2 & -h_2 \\
 0 & 0 & h_2 & -h_2 \\
 -h_2 & -h_2 & 0 & 0 \\
 h_2 & h_2 & 0 & 0
 \end{pmatrix}.
\end{align}
It follows from~\eqref{eq:Tim.coef.cond} and~\eqref{eq:Tim.h1,h2.in.AC} that $Q_1, Q_2 \in L^1([0,\ell]; \bC^{2 \times 2})$ and $\wt{Q} \in L^{\infty}([0,\ell]; \bC^{2 \times 2})$. Due to conditions~\eqref{eq:Tim.a1!=h1,a1!=h2}, the operator $L_k$ is a $2 \times 2$ Dirac type operator with separated regular boundary conditions. Boundary conditions remain separated and regular after applying gauge transform from Lemma~\ref{lem:gauge}. Thus, by Lemma~\ref{lem:separ.regular}(ii), the $2 \times 2$ BVP corresponding to the operator $L_k$ is strictly regular according to Definition~\ref{def:bvp.strict}. Theorem~\ref{th:Riesz_basis} now implies that the system of its root vectors forms a Riesz basis in $L^2_{|\beta_k|}([0,\ell]; \bC^2)$ and its eigenvalues have a proper asymptotic, in particular, inequality~\eqref{eq:ln>cn} is satisfied for them. It is also clear that $L_k$ has finitely many associated vectors. Clearly, the direct sum $L := L_1 \oplus L_2$ has the same properties. Since $\wt{\cQ}$ is a bounded operator, the operator $L_U(Q)$ is a bounded perturbation of ``spectral'' operator $L$. Hence by Proposition~\ref{prop:Riesz.basis.abstract}, the system of root vectors of the operator $L_U(Q)$ forms a Riesz basis with parentheses in $\fH$. Since, by Proposition~\ref{prop:Tim.similar}, $\cL$ is similar to the operator $L_U(Q)$, the system of root vectors of $\cL$ forms a Riesz basis with parentheses in $\cH$.
\end{proof}
\begin{remark}
The proof above follows the proof of Theorem 8.2 in~\cite{LunMal16JMAA}, where this result was proved under additional assumption: the ratio of wave speeds $\frac{K(\cdot)}{\rho(\cdot)}$ and $\frac{EI(\cdot)}{I_\rho(\cdot)}$ is constant. Note that the in the proof of Theorem~\ref{th:Tim.gam12=0} none of results for Dirac type operators for $n>2$ were used. In fact, it follows from results of~\cite{LunMal16JMAA} by applying additional similarity transformation to $L_k$ that realizes a special change of variable that makes the matrix function $B_k$ above constant.
\end{remark}
\subsection{Asymptotic behavior of eigenvalues and Riesz basis property}
\label{subsec:Tim.asymp}
Going over to the asymptotic behavior of the eigenvalues of the operator $\cL$, first we restrict ourselves to the general case $\beta_1 \not\equiv \beta_2$. Recall that $\alp_1, \alp_2, \gam_1, \gam_2$ are parameters from boundary conditions~\eqref{eq:Tim.WLFLa1}--\eqref{eq:Tim.WLFLa2}, functions $\beta_1 = \sqrt{\frac{I_{\rho}}{EI}}$ and $\beta_2 = \sqrt{\frac{\rho}{K}}$ were defined in~\eqref{eq:Tim.beta1.beta2}, and functions $h_1 = \sqrt{EI \cdot I_{\rho}}$ and $h_2 = \sqrt{K \cdot \rho}$ were defined in~\eqref{eq:Tim.g1.g2.def}. Let us introduce some notations:
\begin{equation} \label{eq:bk.alpk.wk.def}
 b_k = \int_0^{\ell} \beta_k(t) dt, \qquad
 \alp_k^{\pm} := \alp_k \pm h_k(\ell), \qquad
 v_k^{\pm} = \(\frac{h_k(\ell)}{h_k(0)}\)^{\pm 1/2}, \quad k \in \{1,2\}.
\end{equation}
The following exponential polynomial plays a crucial role in establishing the asymptotic behavior of the eigenvalues of the operator $\cL$,
\begin{multline} \label{eq:Tim.wtDelta.def}
 \Delta^{\Tim}_0(\l)
 = (\alp_1^+ \alp_2^+ - \gam_1 \gam_2) v_1^+ v_2^+ \cdot e^{i \l (b_1 + b_2)}
 + (\alp_1^- \alp_2^- - \gam_1 \gam_2) v_1^- v_2^-
 \cdot e^{-i \l (b_1 + b_2)} \\
 - (\alp_1^+ \alp_2^- - \gam_1 \gam_2) v_1^+ v_2^- \cdot e^{i \l (b_1 - b_2)}
 - (\alp_1^- \alp_2^+ - \gam_1 \gam_2) v_1^- v_2^+
 \cdot e^{i \l (-b_1 + b_2)}, \quad \ \ \l \in \bC.
\end{multline}
Now we ready to state our main result on the asymptotic behavior of the eigenvalues of the operator $\cL$. For reader's convenience we state all involved conditions on the parameters $\rho(\cdot), I_{\rho}(\cdot), K(\cdot), EI(\cdot), p_1(\cdot), p_2(\cdot)$ and $\alp_1, \alp_2, \gam_1, \gam_2$ of the Timoshenko beam model~\eqref{eq:Tim.Ftt}--\eqref{eq:Tim.WLFLa2} without appealing to the previous formulas and notations (except definition~\eqref{eq:Tim.wtDelta.def} for $\Delta^{\Tim}_0(\cdot))$.
\begin{theorem} \label{th:Tim.asymp}
Let parameters $\rho, I_{\rho}, K, EI, p_1, p_2$ of the Timoshenko beam model~\eqref{eq:Tim.Ftt}--\eqref{eq:Tim.Wtt} be measurable functions and for some $M>1$ the following conditions hold,
\begin{align}
\label{eq:Tim.cond.Linf.asymp}
 & 0 < M^{-1} \le \rho(x), I_{\rho}(x), K(x), EI(x) \le M,
 \qquad x \in [0,\ell], \\
\label{eq:Tim.cond.L1.AC.asymp}
 & p_1, p_2 \in L^1([0,\ell]; \bC), \qquad
 h_1 := \sqrt{EI \cdot I_{\rho}} \in \AC[0,\ell],
 \qquad
 h_2 := \sqrt{K \cdot \rho} \in \AC[0,\ell].
\end{align}
Let also wave speeds $\frac{K(\cdot)}{\rho(\cdot)}$ and $\frac{EI(\cdot)}{I_\rho(\cdot)}$ be separated from each other.
i.e.\ for some $\eps>0$ the following condition holds,
\begin{equation} \label{eq:nu=><0.asymp}
 \text{either}\qquad
 \frac{K(x)}{\rho(x)} - \frac{E(x)}{I_{\rho}(x)} > \eps, \quad x \in [0,\ell], \qquad\text{or}\qquad
 \frac{K(x)}{\rho(x)} - \frac{E(x)}{I_{\rho}(x)} < -\eps, \quad x \in [0,\ell].
\end{equation}
Recall also that
\begin{equation} \label{eq:b1.b2.asymp}
 b_1 = \int_0^{\ell} \sqrt{\frac{I_{\rho}(t)}{EI(t)}} dt
 = \int_0^{\ell} \beta_1(t) dt, \qquad
 b_2 = \int_0^{\ell} \sqrt{\frac{\rho(t)}{K(t)}} dt
 = \int_0^{\ell} \beta_2(t) dt.
\end{equation}

\textbf{(i)} Let parameters $\alp_1, \alp_2, \gam_1, \gam_2 \in \bC$ from boundary conditions~\eqref{eq:Tim.WLFLa1}--\eqref{eq:Tim.WLFLa2} satisfy the following condition,
\begin{equation} \label{eq:Tim.a1!=h1,a1!=h2.asymp}
 (\alp_1 + h_1(\ell)) (\alp_2 + h_2(\ell)) \ne \gam_1 \gam_2 \quad\text{and}\quad
 (\alp_1 - h_1(\ell)) (\alp_2 - h_2(\ell)) \ne \gam_1 \gam_2.
\end{equation}
Then the dynamic generator $\cL$ of the general Timoshenko beam model~\eqref{eq:Tim.Ftt}--\eqref{eq:Tim.WLFLa2} has a countable sequence of eigenvalues $\L := \{\l_m\}_{m \in \bZ}$ counting multiplicity. The sequence $\L$ is incompressible $($see Definition~\ref{def:incompressible}$)$ and lies in the strip $\Pi_h = \{\l \in \bC : |\Im \l| \le h\}$ for some $h \ge 0$.

Moreover, exponential polynomial $\Delta^{\Tim}_0(\cdot)$ given by~\eqref{eq:Tim.wtDelta.def} has a countable sequence of zeros $\L_0 := \{\l_m^0\}_{m \in \bZ}$ counting multiplicity that satisfies the same properties, and both sequences $\L$ and $\L_0$ can be ordered in such a way
that the following sharp asymptotical formula holds
\begin{equation} \label{eq:lm=wtlm0+o}
 \l_m = \l_m^0 + o(1) = \frac{\pi m}{b_1 + b_2} + o(m) \quad\text{as}\quad m \to \infty.
\end{equation}

\textbf{(ii)} Let the following conditions hold,
\begin{equation} \label{eq:gam12=0}
 \gam_1 \gam_2 = 0, \qquad \alp_1 \ne \pm h_1(\ell),
 \qquad \alp_2 \ne \pm h_2(\ell).
\end{equation}
Then the sequence $\L$ of the eigenvalues of the operator $\cL$ is the union of two sequences asymptotically close to arithmetic progressions. Namely, $\L = \L_1 \cup \L_2$, where for $k \in \{1,2\}$ we have,
\begin{equation} \label{eq:Lk.Tim}
 \L_k := \{\l_{k,m}\}_{m \in \bZ}, \quad
 \l_{k,m} := \frac{\pi m}{b_k} - \frac{i \ln \tau_k}{2 b_k} + o(1),
 \qquad \tau_k := \frac{(\alp_k - h_k(\ell))h_k(0)}
 {(\alp_k + h_k(\ell)) h_k(\ell)} \ne 0,
 \qquad m \in \bZ.
\end{equation}
Moreover, the sequence $\L$ is asymptotically separated if and only if the following condition holds:
\begin{equation} \label{eq:strict.crit.Tim}
 \text{either}\quad b_1 \ln |\tau_2| \ne b_2 \ln |\tau_1|
 \quad\text{or}\quad \(
 \frac{b_1}{b_2} \in \bQ
 \quad\text{and}\quad
 \frac{b_1 \arg(\tau_2) - b_2 \arg(\tau_1)}{2 \pi \gcd(b_1, b_2)} \not \in \bZ\).
\end{equation}

\textbf{(iii)} Let the following conditions hold,
\begin{equation} \label{eq:alp1.alp2}
 \gam_1 \gam_2 \ne 0, \qquad
 \alp_1^2 = h_1^2(\ell) + \frac{h_1(\ell)}{h_2(\ell)}\gam_1 \gam_2, \qquad
 \alp_2 = \frac{h_2(\ell)}{h_1(\ell)} \alp_1.
\end{equation}
Then the sequence $\L = \{\l_m\}_{m \in \bZ}$ is asymptotically separated and the following sharp asymptotical formula holds,
\begin{equation} \label{eq:Tim.separ}
 \l_m := \frac{\pi m}{b_1+b_2} - \frac{i \ln \tau}{2 (b_1+b_2)} + o(1), \qquad
 \tau := \frac{(\alp_1 - h_1(\ell)) \cdot h_1(0) h_2(0)}{(\alp_1 + h_1(\ell)) \cdot h_1(\ell) h_2(\ell)} \ne 0,
 \qquad m \in \bZ.
\end{equation}

\textbf{(iv)} Let $b_1/b_2 \in \bQ$. Namely, $b_1 = n_1 b$, $b_2 = n_2 b$ for some $n_1, n_2 \in \bN$ and $b > 0$. Let also condition~\eqref{eq:Tim.a1!=h1,a1!=h2.asymp} holds. Then
$$
\Delta^{\Tim}_0(\l) = e^{-i \l (b_1 + b_2)} \cP(e^{i \l b}), \qquad \l \in \bC,
$$
where $\cP(\cdot)$ is a polynomial of degree $N := 2 (n_1+n_2)$ such that $\cP(0) \ne 0$. Let $z_1, \ldots, z_N \ne 0$ be its roots (counting multiplicity). Then the sequence $\L$ of the eigenvalues of the operator $\cL$ is the union of $N$ sequences asymptotically close to arithmetic progressions,
\begin{equation} \label{eq:Lk.Tim.poly}
 \L = \{\wt{\L}_k\}_{k=1}^N, \quad
 \wt{\L}_k := \{\wt{\l}_{k,m}\}_{m \in \bZ}, \quad
 \wt{\l}_{k,m} := \frac{2 \pi m}{b} - \frac{i \ln z_k}{b} + o(1),
 \qquad m \in \bZ, \quad k \in \onetoN.
\end{equation}
Moreover, the sequence $\L$ is asymptotically separated if and only if numbers $z_1, \ldots, z_n$ are distinct.
\end{theorem}
\begin{proof}
\textbf{(i)} It is clear, that conditions~\eqref{eq:Tim.cond.Linf.asymp}--\eqref{eq:Tim.cond.L1.AC.asymp} are the same as conditions~\eqref{eq:Tim.coef.cond} and~\eqref{eq:Tim.h1,h2.in.AC}. Further, condition~\eqref{eq:nu=><0.asymp} implies condition~\eqref{eq:nu=><0} on the wave speed difference $\nu(\cdot)$ given by~\eqref{eq:nux.def}. Finally, condition~\eqref{eq:Tim.a1!=h1,a1!=h2.asymp} is the same as~\eqref{eq:Tim.a1!=h1,a1!=h2}. Therefore, parameters of the Timoshenko beam model~\eqref{eq:Tim.Ftt}--\eqref{eq:Tim.WLFLa2} satisfy conditions of Proposition~\ref{prop:Tim.similar} and Theorem~\ref{th:Tim.compl.paren}.

By Proposition~\ref{prop:Tim.similar}, the operator $\cL$ is similar to the $4 \times 4$ Dirac-type operator $L_U(Q)$ acting in the Hilbert space $\fH$ given by~\eqref{eq:Tim.fH}, where matrix functions $B(\cdot)$ and $Q(\cdot)$ are given by~\eqref{eq:Tim.B} and~\eqref{eq:Tim.Q(x)}, equipped with the boundary conditions~\eqref{eq:Uy=0},
\begin{equation} \label{eq:Uy=0.Tim}
 U(y) = Cy(0)+Dy(\ell)=0,
\end{equation}
where matrices $C,D$ are given by~\eqref{eq:Tim.C.D}. Hence both operators have the same spectrum (counting multiplicity).

According to the proof of Theorem~\ref{th:Tim.compl.paren}, condition~\eqref{eq:Tim.a1!=h1,a1!=h2.asymp} implies regularity of boundary conditions~\eqref{eq:Uy=0.Tim}, while conditions~\eqref{eq:Tim.coef.cond} and~\eqref{eq:nu=><0} imply conditions~\eqref{eq:betak.alpk.Linf},~\eqref{eq:theta.exists}--\eqref{eq:nujk=><0} on the matrix function $B(\cdot)$. Moreover, conditions~\eqref{eq:Tim.coef.cond} and~\eqref{eq:Tim.h1,h2.in.AC} trivially imply that $Q \in L^1([0,\ell]; \bC^{4 \times 4})$. Therefore, Remark~\ref{rem:beta.vs.wtbeta}(ii) and Theorem~\ref{th:eigen.gen} imply the desired relation~\eqref{eq:lm=wtlm0+o}, and all other desired properties of the sequence $\L$ with the sequence $\wt{\L}_0$ in place of $\L_0$, where $\wt{\L}_0$ is the sequence of zeros of the modified characteristic determinant $\wt{\Delta}_0(\cdot)$ given by~\eqref{eq:wtDelta0}. To finish the proof, it is sufficient to show that $\wt{\Delta}_0(\cdot)$ is proportional to $\Delta^{\Tim}_0(\cdot)$ given by~\eqref{eq:Tim.wtDelta.def}.

Recall that
\begin{equation}
 \wt{\Delta}_0(\cdot) = \det(C + D W(\ell) \Phi_0(\ell, \cdot)),
\end{equation}
where $W(\cdot)$ is the solution of the Cauchy problem~\eqref{eq:W'+QW} that involves the block diagonal $Q_{\diag}$ of the matrix functions $Q$,
\begin{equation} \label{eq:W'+QW.Tim}
 W'(x) + Q_{\diag}(x) W(x) = 0, \quad x \in [0,\ell],
 \qquad W(0) = I_4.
\end{equation}
To this end, note that condition~\eqref{eq:nu=><0.asymp} implies that $\beta_1 \not\equiv \beta_2$. Hence matrix function $B(\cdot)$ given by~\eqref{eq:Tim.B} has simple spectrum. Therefore, its block matrix decomposition has all blocks of size one. This observation and formula~\eqref{eq:Tim.Q(x)} yield that
\begin{equation}
 Q_{\diag} = \diag\(
 \frac{p_1}{2h_1} + \frac{h_1'}{2h_1}, \ \
 \frac{p_1}{2h_1} - \frac{h_1'}{2h_1}, \ \
 \frac{p_2}{2h_2} + \frac{h_2'}{2h_2}, \ \
 \frac{p_2}{2h_2} - \frac{h_2'}{2h_2}\).
\end{equation}
First observe that since functions $h_1, h_2$ are positive and absolutely continuous, we have
\begin{equation}
 \exp\(\int_0^x \frac{h_k'(t)}{2h_k(t)} dt\)
 = \exp\( \frac{\ln h_k(x) - \ln h_k(0)}{2}\)
 = \(\frac{h_k(x)}{h_k(0)}\)^{1/2}, \qquad x \in [0,\ell], \quad k \in \{1,2\}.
\end{equation}
Taking into account this observation, solution $W(\cdot)$ of the Cauchy problem~\eqref{eq:W'+QW.Tim} has the following explicit form
\begin{equation}
 W(x) = \diag( W_1^-(x), W_1^+(x), W_2^-(x), W_2^+(x)), \qquad x \in [0,\ell].
\end{equation}
where
\begin{multline} \label{eq:wkpm.def}
 W_k^{\pm}(x) := \exp\(-\int_0^x \(\frac{p_k(t)}{2h_k(t)}
 \mp \frac{h_k'(t)}{2h_k(t)} \) dt\) \\
 = \(\frac{h_k(x)}{h_k(0)}\)^{\pm 1/2}
 \exp\(-\int_0^x \frac{p_k(t)}{2h_k(t)} dt\),
 \qquad x \in [0,\ell], \quad k \in \{1,2\}.
\end{multline}

To this end, note the fundamental matrix solution $\Phi_0(\cdot,\l)$ of the matrix equation $Y' = i B(x) Y$ with the matrix function $B(\cdot)$ given by~\eqref{eq:Tim.B} is of the following form,
\begin{equation} \label{eq:Tim.Phi0x}
 \Phi_0(x,\l) = \diag(e^{-i \l \rho_1(x)}, e^{i \l \rho_1(x)},
 e^{-i \l \rho_2(x)}, e^{i \l \rho_2(x)}),
 \qquad x \in [0,\ell], \qquad \l \in \bC,
\end{equation}
where
\begin{equation}
 \rho_k(x) := \int_0^x \beta_k(t) dt, \qquad x \in [0,\ell],
 \qquad k \in \{1,2\}.
\end{equation}
It follows from definition~\eqref{eq:Tim.beta1.beta2} of $\beta_1$, $\beta_2$ and definition~\eqref{eq:b1.b2.asymp} of $b_1$, $b_2$, that $b_k = \rho_k(\ell)$, $k \in \{1,2\}$.

Further, for brevity we set,
\begin{equation}
 w_k^{\pm} := W_k^{\pm}(\ell), \qquad
 \cE_k := \exp\(-\int_0^{\ell} \frac{p_k(t)}{2h_k(t)} dt\),
 \qquad k \in \{1,2\},
\end{equation}
It is clear from definition~\eqref{eq:bk.alpk.wk.def} of $v_k^{\pm}$ that $w_k^{\pm} = v_k^{\pm} \cE_k$, $k \in \{1,2\}$. With account of this notations and definition~\eqref{eq:bk.alpk.wk.def} of $\alp_k^{\pm}$, we derive from~\eqref{eq:Tim.C.D} and~\eqref{eq:Tim.Phi0x} that
\begin{equation} \label{eq:Tim.D.Wl}
 C + D W(\ell) \Phi_0(\ell, \l) = \begin{pmatrix}
 1 & 1 & 0 & 0 \\
 \alp_1^- w_1^- e^{-i \l b_1} & \alp_1^+ w_1^+ e^{i \l b_1} &
 \gam_1 w_2^- e^{-i \l b_2} & \gam_1 w_2^+ e^{i \l b_2} \\
 0 & 0 & 1 & 1 \\
 \gam_2 w_1^- e^{-i \l b_1} & \gam_2 w_1^+ e^{i \l b_1} &
 \alp_2^- w_2^- e^{-i \l b_2} & \alp_2^+ w_2^+ e^{i \l b_2} \\
 \end{pmatrix}.
\end{equation}
It now follows from~\eqref{eq:Tim.D.Wl} after straightforward calculations that
\begin{multline} \label{eq:Tim.wtDelta}
 \wt{\Delta}_0(\l)
 = (\alp_1^+ \alp_2^+ - \gam_1 \gam_2) w_1^+ w_2^+ \cdot e^{i \l (b_1 + b_2)}
 + (\alp_1^- \alp_2^- - \gam_1 \gam_2) w_1^- w_2^-
 \cdot e^{-i \l (b_1 + b_2)} \\
 - (\alp_1^+ \alp_2^- - \gam_1 \gam_2) w_1^+ w_2^- \cdot e^{i \l (b_1 - b_2)}
 - (\alp_1^- \alp_2^+ - \gam_1 \gam_2) w_1^- w_2^+
 \cdot e^{i \l (-b_1 + b_2)} \\
 = \(\alp_1^+ w_1^+ e^{i \l b_1} - \alp_1^- w_1^- e^{-i \l b_1}\)
 \(\alp_2^+ w_2^+ e^{i \l b_2} - \alp_2^- w_2^- e^{-i \l b_2}\) \\
 - \gam_1 \gam_2 \(w_1^+ e^{i \l b_1} - w_1^- e^{-i \l b_1}\)
 \(w_2^+ e^{i \l b_2} - w_2^- e^{-i \l b_2}\).
\end{multline}
It is clear from observation $w_k^{\pm} = v_k^{\pm} \cE_k$ above and definition~\eqref{eq:Tim.wtDelta.def} of $\Delta^{\Tim}_0(\cdot)$ that
\begin{equation} \label{eq:wtDelta=TimDelta}
\wt{\Delta}_0(\l) = \cE_1 \cE_2 \Delta^{\Tim}_0(\l), \qquad \l \in \bC,
\end{equation}
which finishes the proof of part (i).

\textbf{(ii)} If $\gam_1 \gam_2 = 0$ then with account of~\eqref{eq:wtDelta=TimDelta}, formula~\eqref{eq:Tim.wtDelta} simplifies to
\begin{equation} \label{eq:Tim.wtDelta.gam=0}
 \Delta^{\Tim}_0(\l) =
 \(\alp_1^+ v_1^+ e^{i \l b_1} - \alp_1^- v_1^- e^{-i \l b_1}\)
 \(\alp_2^+ v_2^+ e^{i \l b_2} - \alp_2^- v_2^- e^{-i \l b_2}\).
\end{equation}
Condition~\eqref{eq:gam12=0} on $\alp_1, \alp_2$ implies that $\alp_1^{\pm}, \alp_2^{\pm} \ne 0$. It is clear that arithmetic progressions $\L_1^0$ and $\L_2^0$,
\begin{equation} \label{eq:Lk0.Tim}
 \L_k^0 := \{\l_{k,m}^0\}_{m \in \bZ}, \quad
 \l_{k,m}^0 := \frac{2 \pi m -i \ln \tau_k^0}{2 b_k},
 \qquad \tau_k^0 := \frac{\alp_k^- v_k^-}{\alp_k^+ v_k^+} \ne 0,
 \qquad m \in \bZ, \quad k \in \{1,2\},
\end{equation}
are zeros of the first and second factor in the r.h.s.\ of~\eqref{eq:Tim.wtDelta.gam=0}, respectively. It is easily seen from definition~\eqref{eq:bk.alpk.wk.def} of $v_k^{\pm}$ that
\begin{equation} \label{eq:vk-/vk+}
 \frac{v_k^-}{v_k^+} = \frac{h_k(0)}{h_k(\ell)}, \qquad k \in \{1,2\}.
\end{equation}
Hence $\tau_k^0 = \tau_k$, where $\tau_k$ if given by~\eqref{eq:Lk.Tim}. Asymptotical formula~\eqref{eq:Lk.Tim} now follows from part (i) of the theorem.
Note that $\sigma_1 = 2 b_1$ and $\sigma_2 = 2 b_2$ in notations of Lemma~\ref{lem:separ.regular}. Hence Lemma~\ref{lem:separ.regular}(iii) implies that the sequence $\L$ is asymptotically separated if and only if condition~\eqref{eq:strict.crit.Tim} holds, which finishes the proof of part (ii).

\textbf{(iii)} It follows from~\eqref{eq:alp1.alp2} and definition~\eqref{eq:bk.alpk.wk.def} of $\alp_k^{\pm}$ that
\begin{multline} \label{eq:a1+a2-=g1g2}
 \alp_1^{\pm} \alp_2^{\mp}
 = (\alp_1 \pm h_1(\ell)) (\alp_2 \mp h_2(\ell))
 = (\alp_1 \pm h_1(\ell))
 \(\frac{h_2(\ell)}{h_1(\ell)} \alp_1 \mp h_2(\ell)\) \\
 = \frac{h_2(\ell)}{h_1(\ell)}
 (\alp_1 \pm h_1(\ell)) (\alp_1 \mp h_1(\ell))
 = \frac{h_2(\ell)}{h_1(\ell)} (\alp_1^2 - h_1^2(\ell))
 = \gam_1 \gam_2.
\end{multline}
I.e.\ $\alp_1^+ \alp_2^- = \alp_1^- \alp_2^+ = \gam_1 \gam_2$. Hence formula~\eqref{eq:Tim.wtDelta} simplifies to
\begin{equation} \label{eq:Tim.wtDelta.separ}
 \Delta^{\Tim}_0(\l)
 = (\alp_1^+ \alp_2^+ - \gam_1 \gam_2) v_1^+ v_2^+ \cdot e^{i \l (b_1 + b_2)}
 + (\alp_1^- \alp_2^- - \gam_1 \gam_2) v_1^- v_2^-
 \cdot e^{-i \l (b_1 + b_2)}, \qquad \l \in \bC.
\end{equation}
It follows that $\alp_1^- = \frac{\gam_1 \gam_2}{\alp_2^+}$, $\alp_2^- = \frac{\gam_1 \gam_2}{\alp_1^+}$. Hence
\begin{equation}
 \alp_1^- \alp_2^- - \gam_1 \gam_2
 = \frac{(\gam_1 \gam_2)^2}{\alp_1^+ \alp_2^+} - \gam_1 \gam_2
 = -\frac{\gam_1 \gam_2}{\alp_1^+ \alp_2^+} (\alp_1^+ \alp_2^+ - \gam_1 \gam_2).
\end{equation}
With account of this observation, formula~\eqref{eq:Tim.wtDelta.separ} simplifies further,
\begin{equation} \label{eq:Tim.wtDelta.separ2}
 \Delta^{\Tim}_0(\l)
 = (\alp_1^+ \alp_2^+ - \gam_1 \gam_2) \(v_1^+ v_2^+ \cdot e^{i \l (b_1 + b_2)}
 -\frac{\gam_1 \gam_2}{\alp_1^+ \alp_2^+} v_1^- v_2^-
 \cdot e^{-i \l (b_1 + b_2)}\), \qquad \l \in \bC.
\end{equation}
Since $\gam_1 \gam_2 \ne 0$, it follow from~\eqref{eq:a1+a2-=g1g2} that $\alp_1^{\pm}, \alp_2^{\pm} \ne 0$. Hence, using~\eqref{eq:a1+a2-=g1g2} again, taking into account definition~\eqref{eq:bk.alpk.wk.def} of $\alp_k^{\pm}$ and the fact that $h_1$ is positive function, we arrive at
\begin{equation} \label{eq:alp1+alp2-g1g2}
 \alp_1^+ \alp_2^+ - \gam_1 \gam_2
 = \alp_1^+ \alp_2^+ - \alp_1^- \alp_2^+
 = \alp_2^+ (\alp_1^+ - \alp_1^-) = 2 \alp_2^+ h_1(\ell) \ne 0.
\end{equation}
Combining~\eqref{eq:Tim.wtDelta.separ2} with~\eqref{eq:alp1+alp2-g1g2}, we see that $\Delta^{\Tim}_0(\cdot)$ is not identically zero and the sequence $\L_0$ of its zeros is the following arithmetic progression,
\begin{equation} \label{eq:Tim.separ0}
 \L_0 := \{\l_m^0\}_{m \in \bZ}, \quad
 \l_m^0 := \frac{2 \pi m -i \ln \tau_0}{2 (b_1+b_2)}, \qquad
 \tau_0 := \frac{\gam_1 \gam_2 \cdot v_1^- v_2^-}{\alp_1^+ \alp_2^+ \cdot v_1^+ v_2^+} \ne 0,
 \qquad m \in \bZ.
\end{equation}
It follows from~\eqref{eq:a1+a2-=g1g2} and~\eqref{eq:vk-/vk+} that
$$
 \tau_0 =
 \frac{\gam_1 \gam_2 \cdot v_1^- v_2^-}{\alp_1^+ \alp_2^+ \cdot v_1^+ v_2^+}
 = \frac{\alp_1^- \alp_2^+ \cdot h_1(0) h_2(0)}{\alp_1^+ \alp_2^+
 \cdot h_1(\ell) h_2(\ell)} = \tau,
$$
where $\tau$ is given by~\eqref{eq:Tim.separ}. Asymptotical formula~\eqref{eq:Tim.separ} now follows from~\eqref{eq:Tim.separ0} and part (i) of the theorem.

\textbf{(iv)} It is clear that the polynomial $\cP(\cdot)$ is of the following form,
\begin{multline} \label{eq:Tim.cPz}
 \cP(z)
 = (\alp_1^+ \alp_2^+ - \gam_1 \gam_2) v_1^+ v_2^+ \cdot z^{2 (n_1 + n_2)}
 + (\alp_1^- \alp_2^- - \gam_1 \gam_2) v_1^- v_2^- \\
 - (\alp_1^+ \alp_2^- - \gam_1 \gam_2) v_1^+ v_2^- \cdot z^{2 n_1}
 - (\alp_1^- \alp_2^+ - \gam_1 \gam_2) v_1^- v_2^+ \cdot z^{2 n_2},
 \quad \ \ \l \in \bC.
\end{multline}
It follows from the definition~\eqref{eq:bk.alpk.wk.def} of $\alp_k^{\pm}$ and $v_k^{\pm}$ and condition~\eqref{eq:Tim.a1!=h1,a1!=h2.asymp}, that the coefficient of $\cP(z)$ at $z^N$ is non-zero, $(\alp_1^+ \alp_2^+ - \gam_1 \gam_2) v_1^+ v_2^+ \ne 0$, and the coefficient of $\cP(z)$ at $z^0$ is non-zero, $(\alp_1^- \alp_2^- - \gam_1 \gam_2) v_1^- v_2^- \ne 0$. Hence
$$
\deg \cP = N = 2 (n_1 + n_2), \quad\text{and}\quad \cP(0) \ne 0.
$$
This implies that $\cP$ has $N$ non-zero roots $z_1, \ldots, z_N$ (counting multiplicity). Therefore, the sequence $\L_0$ of zeros of the exponential polynomial $\Delta^{\Tim}_0(\cdot)$ given by~\eqref{eq:Tim.wtDelta.def} is the union of $N$ arithmetic progressions,
\begin{equation} \label{eq:Lk.Tim.poly0}
 \L_0 = \{\wt{\L}_k^0\}_{k=1}^N, \quad
 \wt{\L}_k^0 := \{\wt{\l}_{k,m}^0\}_{m \in \bZ}, \quad
 \wt{\l}_{k,m}^0 := \frac{2 \pi m}{b} - \frac{i \ln z_k}{b},
 \qquad m \in \bZ, \quad k \in \onetoN.
\end{equation}
Part (i) of the theorem now finishes the proof.
\end{proof}
Now we are ready to formulate the main result on Riesz basis property \emph{(without parentheses)} for the dynamic generator $\cL$ of the Timoshenko beam model.
\begin{theorem} \label{eq:Tim.riesz}
Assume conditions of Theorem~\ref{th:Tim.asymp}(i) and let the sequence $\L_0$ of zeros of the exponential polynomial $\Delta^{\Tim}_0(\cdot)$ given by~\eqref{eq:Tim.wtDelta.def} be asymptotically separated (see Definition~\ref{def:strict.regular}(ii)). Then the system of root vectors of the operator $\cL$ forms a Riesz basis \emph{(without parentheses)} in the Hilbert space $\cH$. In particular, this is the case if one of the following conditions holds:

\textbf{(a)} $\gam_1 \gam_2 = 0$, $\alp_1 \ne \pm h_1(\ell)$, $\alp_2 \ne \pm h_2(\ell)$ and numbers $\tau_1$ and $\tau_2$ given by~\eqref{eq:Lk.Tim} satisfy condition~\eqref{eq:strict.crit.Tim}:

\textbf{(b)} $\gam_1 \gam_2 \ne 0$ and numbers $\alp_1$ and $\alp_2$ satisfy condition~\eqref{eq:alp1.alp2}.

\textbf{(c)} $b_1/b_2 \in \bQ$ and the polynomial $\cP(\cdot)$ given by~\eqref{eq:Tim.cPz} has $N = 2(n_1+n_2)$ distinct roots.
\end{theorem}
\begin{proof}
According to the proof of Theorem~\ref{th:Tim.asymp} operator $\cL$ is similar to $4 \times 4$ Dirac-type operator $L_U(Q)$ with matrix functions $B(\cdot)$ and $Q(\cdot)$ satisfying conditions~\eqref{eq:Bx.def}--\eqref{eq:betak-betaj<-eps} (after reordering entries of the $B(\cdot)$). It also follows from~\eqref{eq:wtDelta=TimDelta} and condition of the theorem that the modified characteristic determinant $\wt{\Delta}_0(\cdot)$ given by~\eqref{eq:wtDelta0} has asymptotically separated sequence of zeros (counting multiplicity). Hence BVP~\eqref{eq:LQ.def.reg}--\eqref{eq:Uy=0} is strictly regular according to Definition~\ref{def:bvp.strict}. Therefore, Theorem~\ref{th:Riesz_basis} implies that the system of root vectors of the operator $L_U(Q)$ forms a Riesz basis \emph{(without parentheses)} in the Hilbert space $\fH$. Similarity of the operators $L_U(Q)$ and $\cL$ implies that the system of root vectors of the operator $\cL$ forms a Riesz basis \emph{(without parentheses)} in the Hilbert space $\cH$. Moreover, according to parts (ii), (iii) and (iv) of Theorem~\ref{th:Tim.asymp}, each of the conditions (a), (b), (c) imply that the modified characteristic determinant $\wt{\Delta}_0(\cdot)$ has asymptotically separated sequence of zeros (counting multiplicity), which finishes the proof.
\end{proof}
If $\beta_1(\cdot) = \beta_2(\cdot) =: \beta(\cdot)$ we can obtain results similar to Theorems~\ref{th:Tim.asymp} and~\ref{eq:Tim.riesz}. In fact, we can obtain explicit criterion for eigenvalues of $\cL$ to be asymptotically separated, but in terms of solutions of certain $2 \times 2$ Cauchy problems that do not have explicit form. Indeed, in this case $B = \diag(-\beta, \beta, -\beta, \beta)$ and the block diagonal of $Q$ with respect to $B$ is of the following form,
\begin{equation} \label{eq:Tim.Qdiag}
 Q_{\diag} := \Theta^{-1}
 \begin{pmatrix}
 p_1+h_1' & 0 & h_2 & 0 \\
 0 & p_1-h_1' & 0 & -h_2 \\
 -h_2 & 0 & p_2+h_2' & 0 \\
 0 & h_2 & 0 & p_2-h_2'
 \end{pmatrix}, \qquad x \in [0,\ell],
\end{equation}
where $\Theta(\cdot)$ is given by~\eqref{eq:Tim.Theta(x)}.
Let $W$ be a solution of the following $4 \times 4$ Cauchy problem,
\begin{equation} \label{eq:Tim.W.beta.eq}
 W'(x) + Q_{\diag} W(x) = 0, \quad x \in [0,\ell], \qquad W(0) = I_4,
\end{equation}
Then the characteristic determinant $\Delta^{\Tim}_0(\cdot)$ has the following form
\begin{equation} \label{eq:Tim.DeltaTim.beta.eq}
 \Delta^{\Tim}_0(\l) = \det(C + D W(\ell) \Phi_0(\ell, \l)), \qquad \l \in \bC.
\end{equation}
One can easily see that the matrix function $W$ has a similar form to $Q_{\diag}$,
\begin{equation} \label{eq:Tim.Wdiag}
 W =: \begin{pmatrix}
 w_{11}^- & 0 & w_{12}^- & 0 \\
 0 & w_{11}^+ & 0 & w_{12}^+ \\
 w_{21}^- & 0 & w_{22}^- & 0 \\
 0 & w_{21}^+ & 0 & w_{22}^+
 \end{pmatrix}, \qquad
 W_{\pm} := \begin{pmatrix}
 w_{11}^{\pm} & w_{12}^{\pm} \\
 w_{21}^{\pm} & w_{22}^{\pm} \\
 \end{pmatrix},
\end{equation}
where matrix functions $W_{\pm}$ are solutions of the following $2 \times 2$ Cauchy problems,
\begin{align}
\label{eq:W-}
 W_-' + \frac12
 \begin{pmatrix}
 \frac{p_1 + h_1'}{h_1} & \frac{h_2}{h_1} \\
 -1 & \frac{p_2 + h_2'}{h_2} \\
 \end{pmatrix} W_- = 0, \qquad x \in [0,\ell], \qquad W_-(0) = I_2, \\
\label{eq:W+}
 W_+' + \frac12
 \begin{pmatrix}
 \frac{p_1 - h_1'}{h_1} & -\frac{h_2}{h_1} \\
 1 & \frac{p_2 - h_2'}{h_2} \\
 \end{pmatrix} W_+ = 0, \qquad x \in [0,\ell], \qquad W_+(0) = I_2.
\end{align}
This in turn implies that
\begin{equation} \label{eq:Tim.DWell}
 D W(\ell) =: \begin{pmatrix}
 0 & 0 & 0 & 0 \\
 d_{11}^- & d_{11}^+ & d_{12}^- & d_{12}^+ \\
 0 & 0 & 0 & 0 \\
 d_{21}^- & d_{21}^+ & d_{22}^- & d_{22}^+ \\
 \end{pmatrix},
\quad\text{where}\quad
\begin{cases}
d_{11}^{\pm} := \alp_1^{\pm} w_{11}^{\pm}(\ell) + \gam_2 w_{21}^{\pm}(\ell), \\
d_{12}^{\pm} := \alp_1^{\pm} w_{12}^{\pm}(\ell) + \gam_2 w_{22}^{\pm}(\ell), \\
d_{21}^{\pm} := \alp_2^{\pm} w_{21}^{\pm}(\ell) + \gam_1 w_{11}^{\pm}(\ell), \\
d_{22}^{\pm} := \alp_2^{\pm} w_{22}^{\pm}(\ell) + \gam_1 w_{12}^{\pm}(\ell).
\end{cases}
\end{equation}
Note that since $\beta_1 \equiv \beta_2$ then
$$
b_1 = b_2 = \int_0^{\ell} \sqrt{\frac{\rho(t)}{K(t)}} dt =: b > 0.
$$
With account of this observation, formulas~\eqref{eq:Tim.DWell},~\eqref{eq:Tim.C.D} and~\eqref{eq:Tim.Phi0x} imply that
\begin{equation}
 C + D W(\ell) \Phi_0(\ell, \l) = \begin{pmatrix}
 1 & 1 & 0 & 0 \\
 d_{11}^- e^{-i \l b} & d_{11}^+ e^{i \l b} &
 d_{12}^- e^{-i \l b} & d_{12}^+ e^{i \l b} \\
 0 & 0 & 1 & 1 \\
 d_{21}^- e^{-i \l b} & d_{21}^+ e^{i \l b} &
 d_{22}^- e^{-i \l b} & d_{22}^+ e^{i \l b} \\
 \end{pmatrix},
\end{equation}
which in turn implies that
\begin{multline} \label{eq:Tim.wtDelta.b1=b2}
 \Delta^{\Tim}_0(\l)
 = (d_{11}^+ d_{22}^+ - d_{12}^+ d_{21}^+) \cdot e^{2 i \l b}
 + (d_{11}^- d_{22}^- - d_{12}^- d_{21}^-) \cdot e^{-2 i \l b} \\
 - (d_{11}^+ d_{22}^- - d_{12}^- d_{21}^+)
 - (d_{11}^- d_{22}^+ - d_{12}^- d_{21}^+).
\end{multline}
This can be rewritten as follows,
\begin{equation} \label{eq:Tim.cPz.b}
 e^{-2 i \l b} \Delta^{\Tim}_0(\l) =: \cP(e^{2 i \l b}), \quad \l \in \bC,
 \qquad \cP(z) =: d_+ z^2 - d_0 z + d_-, \quad z \in \bC.
\end{equation}
I.e.\ $\cP(z)$ is a quadratic polynomial at $z$ and hence it has two (possibly equal) non-zero roots $z_1, z_2$. Hence the sequence of zeros of $\Delta^{\Tim}_0(\cdot)$ is the union of two arithmetic progressions that are asymptotically separated if and only if $z_1 \ne z_2$, which is equivalent to $d_0^2 \ne 4 d_+ d_-$ (the discriminant of $\cP$ is non-zero).
Summarizing all of the above observations and following the proofs of Theorems~\ref{th:Tim.asymp} and~\ref{eq:Tim.riesz} we can establish the following result.
\begin{theorem} \label{th:Tim.asymp.riesz}
Let parameters $\rho, I_{\rho}, K, EI, p_1, p_2$ of the Timoshenko beam model~\eqref{eq:Tim.Ftt}--\eqref{eq:Tim.Wtt} satisfy conditions~\eqref{eq:Tim.cond.Linf.asymp}--\eqref{eq:Tim.cond.L1.AC.asymp} and additionally
\begin{equation} \label{eq:nu=0}
 \frac{K(x)}{\rho(x)} = \frac{E(x)}{I_{\rho}(x)}, \qquad x \in [0,\ell],
 \qquad b := \int_0^{\ell} \sqrt{\frac{\rho(t)}{K(t)}} dt
 = \int_0^{\ell} \sqrt{\frac{E(t)}{I_{\rho}(t)}} dt,
\end{equation}
i.e.\ $\beta_1 \equiv \beta_2$.
Further, let parameters $\alp_1, \alp_2, \gam_1, \gam_2 \in \bC$ from boundary conditions~\eqref{eq:Tim.WLFLa1}--\eqref{eq:Tim.WLFLa2} satisfy ``regularity'' condition~\eqref{eq:Tim.a1!=h1,a1!=h2.asymp}. Then the following statements hold:

\textbf{(i)} Dynamic generator $\cL$ of the general Timoshenko beam model~\eqref{eq:Tim.Ftt}--\eqref{eq:Tim.WLFLa2} has a countable sequence of eigenvalues $\L := \{\l_m\}_{m \in \bZ}$ counting multiplicity.
This sequence is the union of two sequences asymptotically close to arithmetic progressions:
\begin{equation} \label{eq:Lk.Tim.b1=b2}
 \L = \L_1 \cup \L_2, \qquad
 \L_k := \{\l_{k,m}\}_{m \in \bZ}, \quad
 \l_{k,m} := \frac{\pi m}{b} - \frac{i \ln z_k}{2 b} + o(1),
 \qquad m \in \bZ, \quad k \in \{1,2\},
\end{equation}
where $z_1$ and $z_2$ are the roots of the polynomial $\cP(\cdot)$ given by~\eqref{eq:Tim.cPz.b}. Moreover, the system of root vectors of the operator $\cL$ forms a Riesz basis without parentheses in the Hilbert space $\cH$, where each block has a size at most two.

\textbf{(ii)} Sequence $\L = \{\l_m\}_{m \in \bZ}$ of eigenvalues of the dynamic generator $\cL$ is asymptotically separated (see Definition~\ref{def:strict.regular}(ii)) if and only if $z_1 \ne z_2$, or equivalently if
\begin{equation} \label{eq:Tim.Disc.cP}
 (d_{11}^+ d_{22}^- - d_{12}^- d_{21}^+ +
 d_{11}^- d_{22}^+ - d_{12}^- d_{21}^+)^2 \ne
 4 (d_{11}^+ d_{22}^+ - d_{12}^+ d_{21}^+) \cdot
 (d_{11}^- d_{22}^- - d_{12}^- d_{21}^-).
\end{equation}
In this case the system of root vectors of the operator $\cL$ forms a Riesz basis \emph{(without parentheses)} in the Hilbert space $\cH$.
\end{theorem}
\begin{proof}
Let us only comment on the Riesz basis with parentheses property. This property follows from Theorem~\ref{th:Tim.compl.paren}, if we note that condition~\eqref{eq:nu=0} implies that $\nu \equiv 0$, where $\nu(\cdot)$ is given by~\eqref{eq:nux.def}. In turn, Remark~\ref{rem:block.size}(iii) explains why the block sizes are at most two in our case.
\end{proof}
\begin{remark}
\textbf{(i)} One can easily see that
$$
 d_+ = d_{11}^+ d_{22}^+ - d_{12}^+ d_{21}^+
 = (\alp_1^+ \alp_2^+ - \gam_1 \gam_2)
 (w_{11}^+(\ell) w_{22}^+(\ell) - w_{12}^+(\ell) w_{21}^+(\ell))
 = (\alp_1^+ \alp_2^+ - \gam_1 \gam_2) \det W_+(\ell).
$$
A classical Liouville's formula (see~\eqref{eq:detPhi'}--\eqref{eq:detPhi}) implies that
$$
 \det W_+(\ell)
 = \prod_{k=1}^2 \exp\(-\int_0^{\ell} \(\frac{p_k(t)}{2h_k(t)}
 - \frac{h_k'(t)}{2h_k(t)} \) dt\).
$$
Similarly
$$
d_- = (\alp_1^- \alp_2^- - \gam_1 \gam_2) \det W_-(\ell), \qquad
\det W_-(\ell) = \prod_{k=1}^2 \exp\(-\int_0^{\ell} \(\frac{p_k(t)}{2h_k(t)}
 + \frac{h_k'(t)}{2h_k(t)} \) dt\).
$$
In particular,
$$
d_+ d_- = \bigl((\alp_1^- \alp_2^-)^2 - (\gam_1 \gam_2)^2\bigr)
\exp\(-\int_0^{\ell} \(\frac{p_1(t)}{h_1(t)} + \frac{p_2(t)}{h_2(t)} \) dt\)
$$
has a particularly simple explicit form. Unfortunately,
$$
 d_0 = d_{11}^+ d_{22}^- - d_{12}^- d_{21}^+ +
 d_{11}^- d_{22}^+ - d_{12}^- d_{21}^+
$$
does not have such an explicit form, which makes formula~\eqref{eq:Lk.Tim.b1=b2} for the eigenvalues asymptotic to be somewhat implicit in nature. It still represents a sharp practical formula if one uses numeric methods to solve the Cauchy problem.

\textbf{(ii)} Using Theorem~\ref{th:eigenvec} we can obtain explicit sharp asymptotic formulas for eigenvectors of the operator $\cL$, provided that characteristic determinant $\Delta^{\Tim}_0(\cdot)$ has asymptotically separated zeros. Namely, if $y$ is the eigenvector of the operator $L_U(Q)$, then $\mathbf{y} = \cU^{-1} \wt{\cU} y$ is the eigenvector of $\cL$, where similarity transforms $\cU$ and $\wt{\cU}$ are given by~\eqref{eq:cU.def} and~\eqref{eq:wtcU.def}, respectively. But explicit form of vectors $y$ and $\mathbf{y}$ is cumbersome and is omitted.
\end{remark}
\begin{remark}
\textbf{(i)} In connection with Theorem~\ref{th:Tim.gam12=0} we
mention the paper~\cite{Shub02} where the operator $\cL$ was
investigated under the following assumptions on the parameters
of the model:
\begin{equation} \label{eq:Tim.Shubov}
 EI, K \in W^{3,2}[0,\ell],
 \quad \rho, I_{\rho} \in W^{4,2}[0,\ell], \quad p_1 = p_2 = 0, \quad \gam_1 = \gam_2 = 0,
\end{equation}
The completeness of the root vectors was stated in~\cite{Shub02} under the condition~\eqref{eq:Tim.ak.ne.pm.hk} and the additional assumption
\begin{equation} \label{eq:Tim.nu!=1}
 I_{\rho}(x) K(x) \ne \rho(x) EI(x), \quad x \in [0,\ell],
\end{equation}
which in our notations (see~\eqref{eq:nux.def}) means that $\nu(x) \ne 1$, $x \in [0,\ell]$.

Unfortunately, the proof of the completeness in~\cite{Shub02} fails because of the incorrect application of the Keldysh theorem. Namely, the representation $\cL^{-1} = \cL_{00}^{-1} (I_{\cH} + T)$ used in~\cite{Shub02}, where $T$ is of finite rank bounded operator and $\cL_{00} = \cL_{00}^{*}$, fails since it leads to the inclusion $\dom(\cL) \subset \dom(\cL_{00})$, which holds if only if $\cL = \cL_{00}$.

Moreover, under conditions~\eqref{eq:Tim.Shubov},~\eqref{eq:Tim.nu!=1} and~\eqref{eq:Tim.ak.ne.pm.hk} the Riesz basis property for the system of root vectors of $\cL$ was stated in~\cite{Shub02}. The proof is based on the claim that under the above restrictions the eigenvalues of $\cL$ are asymptotically simple and separated. However, it is not the case. In our Theorem~\ref{th:Tim.asymp}(ii) (the case $\gam_1 \gam_2 =0$) we established a criterion for the sequence of eigenvalues of $\cL$ to be asymptotically simple and separated, and it is clear that not all values of $\alp_1 \ne \pm h_1(\ell)$ and $\alp_2 \ne \pm h_2(\ell)$ satisfy the condition~\eqref{eq:strict.crit.Tim}.
Note also that according to Theorem~\ref{th:Tim.gam12=0} the system of root vectors of the operator $\cL$ always forms a Riesz basis with parentheses under the assumptions~\eqref{eq:Tim.coef.cond},~\eqref{eq:Tim.h1,h2.in.AC} and~\eqref{eq:Tim.ak.ne.pm.hk}.

\textbf{(ii)} In connection with Theorem~\ref{th:Tim.compl.paren} we also mention the paper~\cite{XuYung04}. In this paper the operator $\cL$ was investigated under the following stronger assumptions on the parameters of the model:
\begin{equation}
 \label{eq:Tim.Xu} EI, K, \rho, I_{\rho} \ \text{are constant,} \quad p_1 = p_2 = 0, \quad
 \alp_1, \alp_2, \gam_1, \gam_2 \ge 0, \quad 4 \alp_1 \alp_2 \ge (\gam_1 + \gam_2)^2.
\end{equation}
The last condition in~\eqref{eq:Tim.Xu} ensures the dissipativity of the operator $\cL$. The completeness of the system of root vectors of the operator $\cL$ was proved in~\cite{XuYung04} under the restrictions~\eqref{eq:Tim.Xu} and~\eqref{eq:Tim.a1!=h1,a1!=h2}.
Note also that under additional assumptions, guarantying that the eigenvalues of $\cL$ are asymptotically simple and separated, it was proved in~\cite{XuYung04} that the root vectors of $\cL$ contains the Riesz basis.
So, our Theorems~\ref{th:Tim.compl.paren},~\ref{eq:Tim.riesz} and~\ref{th:Tim.asymp.riesz} generalize these results to the case of variable parameters $EI, K, \rho, I_{\rho}$ and broader class of boundary conditions, and improves it in the dissipative case.
\end{remark}
\textbf{Acknowledgement.} The publication has been prepared with the support of
the ``RUDN University Program 5-100''.

\end{document}